\documentclass[10pt,a4paper,english]{myamsart}

\usepackage[T1]{fontenc}
\usepackage[utf8]{inputenc}

\usepackage[english]{babel}

\usepackage{amsmath}
\usepackage{amssymb}
\usepackage{amsopn}                    
\usepackage{amsbsy}
\usepackage{mathtools} 
\usepackage{yhmath}

\usepackage{mathrsfs}
\usepackage{relsize}

\usepackage[toc,page]{appendix}

\usepackage[font=small,format=hang,labelfont={sf,bf}]{caption}
\usepackage[caption=false]{subfig}

\usepackage{epigraph, varwidth}

\usepackage{comment}

\usepackage{xargs}
\usepackage{xcolor}

\usepackage{xpatch}
\usepackage{xspace}
\usepackage{xpunctuate}
\usepackage[abbreviations]{foreign}
\usepackage{hyphenat}
\usepackage{setspace} 

\usepackage{xy}
\xyoption{arrow}
\xyoption{matrix}
\SilentMatrices
\xyoption{tips}
\SelectTips{cm}{10}
\xyoption{2cell}
\UseAllTwocells

\usepackage{tikz}
\usetikzlibrary{cd}
\usetikzlibrary{calc}
\usetikzlibrary{math}
\usetikzlibrary{arrows}
\usetikzlibrary{fit}
\usetikzlibrary{positioning}

\usetikzlibrary{decorations.pathmorphing, arrows.meta}

\newlength{\myline}
\newcommandx*{\triplearrow}[4][1=0, 2=1]{
  \draw[line width=\myline,double distance=3\myline,#3] #4;
  \draw[line width=\myline,shorten <=#1\myline,shorten >=#2\myline,#3] #4;
}
\tikzset{Rightarrow/.style={double equal sign distance,>={Implies},->},
triple/.style={-,preaction={draw,Rightarrow}},
quadruple/.style={preaction={draw,Rightarrow,shorten >=0pt},shorten >=1pt,-,double,double
distance=0.2pt}}

\usepackage{graphicx}

\usepackage{forest}

\usepackage{tane}

\usepackage{hyperref}

\makeindex

\usepackage{enumitem}
\setlist{nosep}
\setlist[enumerate, 1]{label={\rm (}\emph{\alph*}{\rm )}}
\setlist[enumerate, 2]{label={\rm (}\emph{\alph{enumi}.\arabic*}{\rm )}}

\renewcommand{\descriptionlabel}[1]{%
	\hspace\labelsep \upshape\bfseries #1:%
}

\theoremstyle{plain}
\newtheorem{thm}{Theorem}[subsection]
\newtheorem{prop}[thm]{Proposition}
\newtheorem{lemme}[thm]{Lemma}
\newtheorem{coro}[thm]{Corollary}

\theoremstyle{remark}
\newtheorem{rem}[thm]{Remark}

\newtheorem{exem}[thm]{Example}

\theoremstyle{definition}
\newtheorem{definition}[thm]{Definition}

\newtheorem{paragr}[thm]{}
\newtheorem*{notations}{Notations and Terminology}
\newtheorem*{remerciements}{Acknowledgements}

\newtheorem*{organisation-fr}{Plan de la Th\`{e}se}

\numberwithin{equation}{thm}

\usepackage{xpatch}

\newcommand\forlang\emph
\newcommand\ndef\emph
\newcommand\nbd\nobreakdash
\newcommand\eps\epsilon

\newcommand\mybreak{\par\ifdim\lastskip<\myskipamount
  \removelastskip\penalty-100\myskip\fi}
\newcommand\myskip{\vskip\myskipamount}
\newskip\myskipamount
\newcommand\badbreak\vskip

\makeatletter

\def\xpoint{\futurelet\@let@token\@xpoint}
\def\@xpoint{%
  \ifx\@let@token.\else
    .%
  \fi
  \@\xspace}

\makeatother

\newcommand\lp(
\newcommand\rp)

\newcommand\defssi{\overset{\text{\tiny def}}{\Longleftrightarrow}}
\DeclareMathAlphabet{\mathpzc}{OT1}{pzc}{m}{it}

\newcommand\quadtext[1]{\quad\text{#1}\quad}
\newcommand\quadmath[1]{\quadtext{$#1$}}

\newcommand\quadet{\quadtext{and}}

\newcommand\quaddefssi{\quadmath{\defssi}}

\renewcommand\le\leqslant
\renewcommand\leq\le
\renewcommand\ge\geqslant
\renewcommand\epsilon\varepsilon
\renewcommand\phi\varphi

\newcommand\longto\longrightarrow
\newcommand\ot\leftarrow
\newcommand\longot\longleftarrow
\newcommand\hookto\hookrightarrow
\newcommand\xto\xrightarrow
\newcommand\xot\xleftarrow

\newcommand\remtt[1]{\texttt{[#1]}}
\newcommand\todo[1]{\remtt{TODO : #1}}

\newcommand\e{\epsilon}

\newcommand\var\bullet

\newcommand\leN{\le_\N}

\newcommand{\sauf}{\mathchoice{\raise 1.8pt\hbox{${\scriptstyle\kern
2.5pt\smallsetminus\kern 2.5pt}$}}{\raise 1.8pt\hbox{${\scriptstyle\kern
2.5pt\smallsetminus\kern 2.5pt}$}}{\raise
1.8pt\hbox{${\scriptscriptstyle\kern 1.5pt\smallsetminus\kern
1.5pt}$}}{\raise 1.8pt\hbox{${\scriptscriptstyle\kern
1.5pt\smallsetminus\kern 1.5pt}$}}}

\newcommand\joint\star

\newcommand\vide\varnothing

\newcommand\Z{\mathbb{Z}}
\newcommand\N{\mathbb{N}}

\newcommand\Part{\mathcal{P}}

\newcommand\limind\varinjlim
\newcommand\limproj\varprojlim

\newcommand\A{\mathcal{A}}

\newcommand{\tA}{\tilde{A}}

\newcommand\V{$\mathcal{V}$}
\newcommand\Vn{\V\nbd}
\newcommand\G{\mathbb{G}}
\newcommand\Gr{\mathbb{O}}
\newcommand{\Hom}{\operatorname{\mathsf{Hom}}}

\newcommand{\Homi}{\operatorname{\kern.5truept\underline{\kern-.5truept\mathsf{Hom}\kern-.5truept}\kern1truept}}

\newcommand{\pref}[1]{{\widehat{ #1 }}}
\newcommand\id[1]{1_{#1}}
\newcommand\op\circ
\newcommand{\Ob}{\operatorname{\mathsf{Ob}}}

\newcommand{\Fl}{\operatorname{\mathsf{Cell}}}

\newcommand{\Ens}{{\mathcal{S}\mspace{-2.mu}\it{et}}}

\newcommand{\Ord}{\mathcal{O}\mspace{-2.mu}\it{rd}}
\newcommand\Pos\Ord

\newcommand{\Cat}{{\mathcal{C}\mspace{-2.mu}\it{at}}}
\newcommand{\nCat}[1]{{#1}\hbox{\protect\nbd-}\kern1pt\Cat}
\newcommand{\ooCat}{\nCat{\infty}}

\newcommand\oo{$\infty$\nbd}

\newcommand{\dCat}{\nCat{3}_{\cDelta}}
\newcommand{\lnCat}{\widetilde{\nCat{3}}}

\newcommand{\Glt}[2][\empty]{%
	\ifx\empty#1
	t^{#2}%
	\else
	t_{#1}^{#2}%
	\fi
}
\newcommand{\Gls}[2][\empty]{%
	\ifx\empty#1
	s^{#2}%
	\else
	s_{#1}^{#2}%
	\fi
}

\newcommand{\Glk}[2][\empty]{%
	\ifx\empty#1
	r^{#2}%
	\else
	r_{#1}^{#2}%
	\fi
}

\newcommand{\Tht}[2][\empty]{%
	\ifx\empty#1
	\tau^{#2}_{}%
	\else
	\tau^{#1}_{#2}%
	\fi
}
\newcommand{\Ths}[2][\empty]{%
	\ifx\empty#1
	\sigma^{#2}_{}%
	\else
	\sigma^{#1}_{#2}%
	\fi
}
\newcommand{\Thk}[2][\empty]{%
	\ifx\empty#1
	\rho^{#2}_{}%
	\else
	\rho^{#1}_{#2}%
	\fi
}

\newcommand\comp\ast

\newcommand{\ti}[1]{\tau_{\le #1}^{\mathrm i}}

\newcommand\Dn[1]{\mathrm{D}_{#1}}
\newcommand\disk{\varsigma}

\newcommand\Homcocat\Hom

\newcommand{\SN}{N_\infty}

\newcommand{\SNn}[1]{N_{#1}}
\newcommand{\Sc}{c_\infty}
\newcommand{\Nl}{\SNn{l}}
\newcommand{\cCl}{c_l}

\newcommand\dgn[2]{\sigma^{#1}_{#2}}
\newcommand\cDelta{\mathbf{\Delta}}

\newcommand\Deltan[1]{\varDelta^{#1}}
\newcommand\EnsSimp{\pref{\cDelta}}

\newcommand\nd[2]{{#1}^{\mathrm{nd}}_{#2}}

\newcommand\sd{\mathrm{Sd}}
\newcommand\Sd\sd

\newcommand\cO{c_\infty}
\newcommand\cOn[1]{c_{#1}}
\newcommand\cC{\mathsf{c}}
\newcommand\nN{\mathsf{N}}

\newcommand\Or{\mathcal{O}}
\newcommand\On[1]{\mathcal{O}_{#1}}
\newcommand\Onm[2]{\mathcal{O}_{#1}^{\leq #2}}

\newcommand{\tr}[2]{\mathchoice
  {#1\raise -1.8pt\vbox{\hbox{$\kern -.8pt/#2$}}}
  {#1\raise -1.8pt\vbox{\hbox{$\kern -.8pt/#2$}}\kern .8pt}
  {#1\raise -1.8pt\vbox{\hbox{$\scriptstyle\kern -.8pt /#2$}}}
  {#1\raise -1.8pt\vbox{\hbox{$\scriptscriptstyle\kern -.8pt /#2$}}}}

\newcommand{\trm}[2]{\mathchoice
  {#1\raise -1.8pt\vbox{\hbox{$\kern -.8pt\!\stackrel{\,\rm co}{/}\!\!#2$}}}
  {#1\raise -1.8pt\vbox{\hbox{$\kern -.8pt\!\stackrel{\,\rm co}{/}\!\!#2$}}\kern .8pt}
  {#1\raise -1.8pt\vbox{\hbox{$\scriptstyle\kern -.8pt\!\stackrel{\,\,\rm co}{/}\!\!#2$}}\kern .8pt}
  {TODO}}

\newcommand{\cotr}[2]{\mathchoice
  {\raise -1.8pt\vbox{\hbox{$#2\backslash$}}#1}
  {\raise -1.8pt\vbox{\hbox{$#2\backslash$}}#1}
  {\raise -1.8pt\vbox{\hbox{$\scriptstyle#2\backslash$}}#1}
  {\raise -1.8pt\vbox{\hbox{$\scriptscriptstyle#2\backslash$}}#1}}

\newcommand{\cotrm}[2]{\mathchoice
  {\raise -1.8pt\vbox{\hbox{$#2\!\stackrel{\!\rm co}{\backslash}$}}#1}
  {\raise -1.8pt\vbox{\hbox{$#2\!\stackrel{\!\rm co}{\backslash}$}}#1}
  {\raise -1.8pt\vbox{\hbox{$\scriptstyle#2\!\stackrel{\!\rm co}{\backslash}$}}#1}
  {TODO}}

\newcommand{\tru}[2]{\mathchoice
  {#1\raise -1.8pt\vbox{\hbox{$\kern -.8pt/^1#2$}}}
  {#1\raise -1.8pt\vbox{\hbox{$\kern -.8pt\overset{\,\scriptscriptstyle1}{/}#2$}}\kern .8pt}
  {#1\raise -1.8pt\vbox{\hbox{$\scriptstyle\kern -.8pt /^1#2$}}}
  {#1\raise -1.8pt\vbox{\hbox{$\scriptscriptstyle\kern -.8pt /^1#2$}}}}

\newcommand{\cotru}[2]{\mathchoice
  {\raise -1.8pt\vbox{\hbox{$#2^1\backslash$}}#1}
  {\raise -1.8pt\vbox{\hbox{$#2\overset{\!\scriptscriptstyle1}{\backslash}$}}#1}
  {\raise -1.8pt\vbox{\hbox{$\scriptstyle#2^1\backslash$}}#1}
  {\raise -1.8pt\vbox{\hbox{$\scriptscriptstyle#2^1\backslash$}}#1}}

\newcommand{\Cda}{\mathcal{C}_{\mathrm{ad}}}

\newcommand{\atom}[1]{\langle{#1}\rangle}

\newcommand{\tabld}[2]{\begin{pmatrix}#1^0_0 &\dots &#1^0_{#2-1}
  &#1^0_{#2}\cr\noalign{\vskip 3pt} #1^1_0 &\dots &#1^1_{#2-1}
  &#1^1_{#2}\end{pmatrix}}

\newcommand{\tabll}[2]{\begin{pmatrix}#1^0_0 &#1^0_1 &\dots &#1^0_{#2-1}
  &#1^0_{#2}\cr\noalign{\vskip 3pt} #1^1_0 &#1^1_1 &\dots &#1^1_{#2-1}
  &#1^1_{#2}\end{pmatrix}}
\newcommand{\supp}{\operatorname{supp}}
\newcommand\Zdec{\underline{\mathbb{Z'}\kern -2.5pt}\kern 2pt}

\newcommand\VCatd{{\Vcat{V}^{\kern 1pt\mathrm d}}}       
\newcommand\VCatg{{\Vcat{V}^{\kern 1pt\mathrm g}}}       
\newcommand\Vcat[1]{\mathbb{#1}}   
\newcommand\VHom{\operatorname{\kern.5truept\underline{\kern-.5truept\mathsf{Hom}\kern-1truept}\kern1.5truept}^{}}    

\newcommand\HomCdad{\operatorname{\kern.5truept\underline{\kern-.5truept\mathsf{Hom}\kern-1.5truept}\kern1.9truept}^{\mathrm{d}}_{\Cda}}   
\newcommand\HomCdag{\operatorname{\kern.5truept\underline{\kern-.5truept\mathsf{Hom}\kern-1.5truept}\kern1.9truept}^{\mathrm{g}}_{\Cda}}   
\author{Andrea Gagna}
\address{Universita Karlova\\ Matematicko-fyzik\'{a}ln\'{i} fakulta \\ Matematická sekce \\ Sokolovská 83 \\ 186 75 Praha 8 \\ Czech Republic}
\email{gagna@karlin.mff.cuni.cz}
\urladdr{https://sites.google.com/view/andreagagna/home}

\title
    {On a notion of oplax 3-functor}


\begin{document}

\begin{abstract}
	We introduce a notion of normalised oplax $3$-functor
	suitable for the elementary homotopy theory
	of strict $3$-categories, following the
	combinatorics of orientals. We show that any such
	morphism induces a morphism of simplicial sets
	between the Street nerves and we characterise
	those morphisms of simplicial sets coming from
	normalised oplax $3$-functors. This allows us
	to prove that normalised oplax $3$-functors
	compose. Finally we construct a strictification
	for normalised oplax $3$-functors whose source
	is a $1$-category without split-monos or split-epis.
\end{abstract}

\maketitle
\tableofcontents

 \setcounter{MaxMatrixCols}{20}
 \input{pent}
 \input{square}
 \input{square2}
 \tikzset{
between/.style args={#1 and #2}{
    at = ($(#1)!0.5!(#2)$)
},
betweenl/.style args={#1 and #2}{
    at = ($(#1)!0.35!(#2)$)
}
}
\tikzset{pics/.cd,
Pent/.style n args={4}{code={%
\begin{scope}[font=\footnotesize]
\foreach \XX [count=\r starting from 0] in {#3}
\node (\r) at (162 + \r * 72:#2) {$\XX$};
\draw[->] (0) -- node[midway,left] (ab) {$\ab$}  (1);
\draw[->] (1) -- node[midway,below] (bc) {$\bc$} (2);
\draw[->] (2) -- node[midway,right] (cd) {$\cd$} (3); 
\draw[->] (3) -- node[midway,above right] (de) {$\de$} (4); 
\draw[->] (0) -- node[midway,above left] (ea) {$\ae$} (4);
\ifcase#1     
    \draw[->] (0) -- node[midway,fill=white] (ac) {$\ac$}  (2);
    \draw[->] (0) -- node[midway,fill=white] (ad) {$\ad$}  (3);
    \node[between=1 and ac] {$\abc$};
    \node[betweenl=ad and 2] {$\acd$};
    \node[betweenl=ad and 4] {$\ade$};
\or
    \draw[->] (0) -- node[midway,fill=white] (ad) {$\ad$} (3);
    \draw[->] (1) -- node[midway,fill=white] (bd) {$\bd$}  (3);
    \node[betweenl=ad and 1] {$\abd$};
    \node[between=bd and 2] {$\bcd$};
    \node[betweenl=ad and 4] {$\ade$};
\or
    \draw[->] (1) -- node[midway,fill=white]  (bd) {$\bd$} (3);
    \draw[->] (1) -- node[midway,fill=white] (be) {$\be$} (4);
    \node[between=0 and be] {$\abe$};
    \node[betweenl=de and 1] {$\bde$};
    \node[between=bd and 2] {$\bcd$};
\or
    \draw[->] (1) -- node[midway,fill=white] (be) {$\be$} (4);
    \draw[->] (2) -- node[midway,fill=white] (ce) {$\ce$} (4);
    \node[between=0 and be] {$\abe$};
    \node[betweenl=bc and 4] {$\bce$};
    \node[between=ce and 3] {$\cde$};
\or
    \draw[->] (0) -- node[midway,fill=white] (ac) {$\ac$} (2);
    \draw[->] (2) -- node[midway,fill=white] (ce) {$\ce$} (4);
    \node[between=1 and ac] {$\abc$};
    \node[betweenl=ea and 2] {$\ace$};
    \node[between=ce and 3] {$\cde$};
\fi
\end{scope}
 }}}
\tikzset{pics/.cd,
Pentscript/.style n args={4}{code={%
\begin{scope}[font=\scriptsize]
\foreach \XX [count=\r starting from 0] in {#3}
\node (\r) at (162 + \r * 72:#2) {$\XX$};
\draw[->] (0) -- node[midway,left] (ab) {$\ab$}  (1);
\draw[->] (1) -- node[midway,below] (bc) {$\bc$} (2);
\draw[->] (2) -- node[midway,right] (cd) {$\cd$} (3); 
\draw[->] (3) -- node[midway,above right] (de) {$\de$} (4); 
\draw[->] (0) -- node[midway,above left] (ea) {$\ae$} (4);
\ifcase#1     
    \draw[->] (0) -- node[midway,fill=white] (ac) {$\ac$}  (2);
    \draw[->] (0) -- node[midway,fill=white] (ad) {$\ad$}  (3);
    \node[between=1 and ac] {$\abc$};
    \node[betweenl=ad and 2] {$\acd$};
    \node[betweenl=ad and 4] {$\ade$};
\or
    \draw[->] (0) -- node[midway,fill=white] (ad) {$\ad$} (3);
    \draw[->] (1) -- node[midway,fill=white] (bd) {$\bd$}  (3);
    \node[betweenl=ad and 1] {$\abd$};
    \node[between=bd and 2] {$\bcd$};
    \node[betweenl=ad and 4] {$\ade$};
\or
    \draw[->] (1) -- node[midway,fill=white]  (bd) {$\bd$} (3);
    \draw[->] (1) -- node[midway,fill=white] (be) {$\be$} (4);
    \node[between=0 and be] {$\abe$};
    \node[betweenl=de and 1] {$\bde$};
    \node[between=bd and 2] {$\bcd$};
\or
    \draw[->] (1) -- node[midway,fill=white] (be) {$\be$} (4);
    \draw[->] (2) -- node[midway,fill=white] (ce) {$\ce$} (4);
    \node[between=0 and be] {$\abe$};
    \node[betweenl=bc and 4] {$\bce$};
    \node[between=ce and 3] {$\cde$};
\or
    \draw[->] (0) -- node[midway,fill=white] (ac) {$\ac$} (2);
    \draw[->] (2) -- node[midway,fill=white] (ce) {$\ce$} (4);
    \node[between=1 and ac] {$\abc$};
    \node[betweenl=ea and 2] {$\ace$};
    \node[between=ce and 3] {$\cde$};
\fi
\end{scope}
 }}}

\section*{Introduction}

The homotopy theory of small categories was born
with the introduction by Grothendieck
of the nerve functor
\[ N \colon \Cat \to \EnsSimp \]
in~\cite{grothendieck_techniques_III},
where $\Cat$ is the category of small categories
and $\EnsSimp$ is the category of simplicial sets,
allowing us to define a class of weak equivalences in
$\Cat$: 
a functor is a weak equivalence precisely
when its nerve
is a simplicial weak homotopy equivalence.
We call these functors \ndef{Thomason equivalences} of $\Cat$.
The nerve functor preserves by definition the weak equivalences,
\ie maps Thomason equivalences to simplicial weak equivalences,
and therefore there is an induced functor
\[\bar{N} \colon \text{Ho}(\Cat) \longto \text{Ho}(\EnsSimp)\]
at the level of the homotopy categories.

The first striking result of this theory
appears in Illusie's thesis~\cite{cotangent} (who credits it
to Quillen) and states that this induced
functor $\bar{N} \colon \text{Ho}(\Cat) \longto \text{Ho}(\EnsSimp)$
is an equivalence of categories.
The homotopy inverse of the nerve functor $N \colon \Cat \to \EnsSimp$
is not induced by its left adjoint $c \colon \EnsSimp \to \Cat$, \ie the \ndef{categorical realisation} functor,
which behaves poorly homotopically, but instead by the \ndef{category of elements} functor
$i_{\Delta} \colon \EnsSimp \to \Cat$, mapping a simplicial set $X$ to its
category of elements $i_\Delta(X) = \cDelta/X$.
A careful study of the subtle homotopy theory
of small categories by 
Thomason~\cite{Cat_closed} led him to show
another important result: the existence
of a model category structure on $\Cat$ which is Quillen equivalent
to the Kan--Quillen model category structure
on simplicial sets.
This important result implies that
small categories and simplicial sets are not
only equivalent as homotopy categories, but
actually as weak $(\infty, 1)$-categories, \ie as homotopy theories.

One drawback of working with small categories as preferred
model for homotopy types is that there are no simple geometric
models of simplicial complexes.
For instance, the homotopy type of the two-dimensional
sphere~$S^2$ is often modelled by the poset
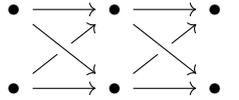
\[
 \begin{tikzcd}
  \bullet \ar[r] & \bullet \ar[r] & \bullet \\
  \bullet \ar[r] \ar[ur] & \bullet \ar[r] \ar[ur] & \bullet
  \ar[from=1-1, to=2-2, crossing over] \ar[from=1-2, to=2-3, crossing over]
 \end{tikzcd}\,,
\]
see for instance~\cite{FiniteSpaces}.
This is mainly due to the intrinsic $1$-dimensional shape
of categories. On the other hand, the homotopy type of $S^2$
can be easily modelled in a geometric fashion by a small $2$-category, namely
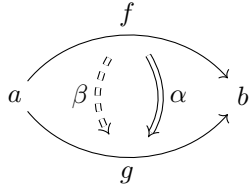
\begin{center}
 \begin{tikzpicture}[scale=1.5]
  \node (0) at (180:1) {$a$};
  \node (1) at (0:1) {$b$};
  \draw[->] (0) to [bend left=50] node [above] {$f$} (1);
  \draw[->] (0) to [bend right=50] node [below] {$g$} (1);
  \draw[double equal sign distance,>={Implies},->] (65:.4) to [bend left] 
  node [right] {$\alpha$} (-65:.4);
  \draw[dashed, double equal sign distance, >={Implies}, ->] (117:.38) to [bend right]
  node [left] {$\beta$} (-117:.38);
 \end{tikzpicture}
\end{center}
This suggests that strict higher categories may provide
a more convenient framework for setting up a categorical
model for homotopy types and a source of motivation for
generalising the homotopy theory of $\Cat$ to strict higher categories.
In fact, Ara and Maltsiniotis construct a functor $\Or$
assigning to any ordered simplicial complex an \oo-category,
see~\cite[§9]{AraMaltsiCondE}.

In a seminal article~\cite{Street}, Street introduced a nerve functor
\[
 \SN \colon \ooCat \to \EnsSimp
\]
for strict $\infty$-categories, allowing one to define
and study the homotopy theory of $\ooCat$, the category of small
strict $\infty$-categories, as well as of $\nCat{n}$,
the category of strict $n$-categories, for every positive integer $n$.
The class of weak equivalences of $\nCat{n}$ pulled back via the Street nerve
shall still be called \ndef{Thomason equivalences}.
This functor is homotopically meaningful, since for instance it
sends the above $2$\nbd-cat\-egorical model of~$S^2$ to a simplicial
set with the homotopy type of~$S^2$.

The particular case of small $2$-categories
was studied by Bullejos and Cegarra~\cite{BullejosCegarra},
Cegarra~\cite{Cegarra}, Chiche~\cite{chiche_homotopy} and del~Hoyo~\cite{del-Hoyo}.
Their approach stresses on the importance played by (normalised)
oplax $2$-functors. 
In fact, it was already noticed that oplax $2$-functors are geometrically meaningful,
see for instance~\cite[Section 10]{StreetCatStructures}.
This is consequence of the fact that
the Street nerve for $2$\nbd-cat\-egories $N_2 \colon \nCat{2} \to \EnsSimp$
is faithful but \emph{not} full; the set of morphisms of simplicial sets
between the nerves $N_2(A)$ and $N_2(B)$ of two small $2$-categories
$A$ and $B$ is in fact in bijection with the set of \ndef{normalised
oplax $2$-functors} from $A$ to $B$.
Ara and Maltsiniotis~\cite{AraMaltsiNThom} provide an abstract framework in which
to transfer the Kan--Quillen model category structure on simplicial sets
to strict $n$-categories and showed that
this is the case for small $2$-categories.
Their strategy makes use of normalised oplax $2$-functors in order to define some maps
needed for a homotopy cobase change property.

It is therefore natural to study a notion of
normalised oplax $n$-functor which could be
used to generalise the results listed above,
thus establishing a satisfactory homotopy theory
of $n$-categories. By this we mean showing
that $\nCat{n}$ can be equipped
with a Quillen model category structure
which is Quillen equivalent to the Kan--Quillen
model category structure on simplicial sets.
Providing a sensitive definition of such
a normalised oplax $n$-functor for the case $n = 3$,
which is the first
one not well understood,
is the aim of the present paper.

A normalised oplax $n$-functor $F \colon A \to B$, with $n\ge 1$, should roughly be
a morphism of $n$\nbd-graphs
which respects the identities on the nose, but that respects
compositions of arrows only up to oriented coherences. For example,
given a composition $a \xto{f} b \xto{g} c$ of two $1$-arrows of $A$,
the normalised oplax $n$-functor $F$ should provide:
\begin{itemize}
	\item a $1$-arrow $F(a) \xto{F(gf)} F(c)$ of $B$,
	\item two composable $1$-arrows $F(a) \xto{F(f)} F(b)$ and $F(b) \xto{F(g)} F(c)$ of $B$,
	\item a $2$-arrow $F(g, f) \colon F(gf) \Rightarrow F(g)F(f)$, which represents the coherence
	for the composition of $f$ and $g$.
\end{itemize}
We observed that a central tool for
the elementary homotopy theory of $2$-categories is the notion
of normalised oplax $2$-functor. Moreover, this turns out to be
a crucial ingredient in establishing the model category structure
\emph{à la} Thomason on $\nCat{2}$, too.

Normalised oplax $2$-functors can be composed and hence form
a category $\widetilde{\nCat{2}}$. There is a canonical
nerve functor $\widetilde N_2 \colon \widetilde{\nCat{2}} \to \EnsSimp$
extending the Street nerve for $2$-categories, that is, there is
a commutative triangle
\[
 \begin{tikzcd}[column sep=small]
  & \EnsSimp & \\
  \nCat{2} \ar[ur, "N_2"] \ar[rr] && \widetilde{\nCat{2}} \ar[lu, "\widetilde N_2"']
 \end{tikzcd}
\]
of functors, where the functor $\nCat{2} \to \widetilde{\nCat{2}}$ is simply
the embedding given by the fact that any $2$-functor is in particular
a normalised oplax $2$-functor.
The Street nerve $N_n$ is a faithful functor but not full for $n>1$.
In the $2$-categorical case, this deficiency is solved by
normalised oplax $2$-functors: the nerve $\widetilde N_2 \colon \widetilde{\nCat{2}} \to \EnsSimp$ is fully faithful.
Following this idea, Street 
proposes in~\cite{Conspectus} to define a normalised oplax $3$-functor from $A$ to $B$
as a simplicial morphism from $N_3(A)$ to $N_3(B)$.
A careful investigation of such a morphism
shows that this might not be an optimal definition since
in general simplicial morphisms between Street nerves of $3$\nbd-cat\-egories
fail to preserve the underlying $3$-graph.
Indeed, we analyse the case where $A$ is the ``categorical $2$-disk''
\[
	    	\begin{tikzcd}[column sep=4.5em]
		    	a
		    	\ar[r, bend left, "f", ""'{name=f}]
		    	\ar[r, bend right, "g"', ""{name=g}]
		    	\ar[Rightarrow, from=f, to=g, "\disk"]
		    	& a'
	    	\end{tikzcd}\ ,
\]
\ie the $2$-category with two parallel $1$-cells and a single $2$-cell between them,
and $B$ is the ``invertible categorical $3$-disk''
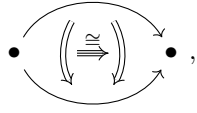
\[
	    	\begin{tikzcd}[column sep=4.5em]
	    	 \bullet
	    	 \ar[r, bend left=60, ""'{name=f}]
	    	 \ar[r, bend right=60, ""{name=g}]
	    	 \ar[Rightarrow, from=f, to=g, shift right=0.5em, bend right, shorten <=1mm, shorten >=1mm, ""{name=al}]
	    	 \ar[Rightarrow, from=f, to=g, shift left=0.5em, bend left, shorten <=1mm, shorten >=1mm, ""'{name=ar}]
	    	 \arrow[triple, from=al, to=ar, "\cong"]{}
	    	 & \bullet
	    	\end{tikzcd} ,
\]
\ie the $3$-category with two parallel $1$-cells, two parallel $2$-cells between them
and a single invertible $3$-cell between these $2$-cells, and we show that there are more
simplicial morphisms than expected between the respective Street nerves.
On the one hand, the $2$-category $A$ has no compositions and so
the normalised oplax $3$-functors from $A$ to $B$ should coincide with
the strict $3$-functors. On the other hand, there are simplicial morphisms
from $N_3(A)$ to $N_3(B)$ which do not come from the nerve of any strict $3$-functors.
This is a consequence of the fact that, for instance, there are two ways to
capture the $2$-cell $\disk$ of $A$ with a $2$-simplex of $N_3(A)$, namely
\[
    \begin{tikzcd}[column sep=small]
      &
       a'
       \ar[dr, equal, "1_{a'}"]
      & \\
      a
      \ar[rr, "f"', ""{name=f}]
      \ar[ru, "g"]
      &&
        a'
      \ar[Rightarrow, from=f, to=1-2, shorten >=3pt, "\disk"{near start}]
     \end{tikzcd}
     \quadet
     \begin{tikzcd}[column sep=small]
      &
       a
       \ar[dr, "g"]
      & \\
      a
      \ar[rr, "f"', ""{name=f}]
      \ar[ru, equal, "1_a"]
      &&
        a'
      \ar[Rightarrow, from=f, to=1-2, shorten >=3pt, "\disk"{near start}]
     \end{tikzcd}
     \ ,
    \]
and these two different ways are related by $3$-simplices, for instance
\begin{center}
 \begin{tikzpicture}[scale=1.6, font=\footnotesize]
     \squares{%
     	/squares/label/.cd,
     	0={$\bullet$}, 1={$\bullet$}, 2={$\bullet$}, 3=$\bullet$,
     	01={}, 12=$g$, 23={}, 02=$g$, 03=$f$, 13=$g$,
     	012={$=$}, 023=$\disk$, 123={$=$}, 013=$\disk$,
     	0123={$=$},
     	/squares/arrowstyle/.cd,
     	01={equal}, 23={equal},
     	012={phantom, description}, 123={phantom, description},
     	0123={phantom, description},
     	/squares/labelstyle/.cd,
     	012={anchor=center}, 123={anchor=center},
     	0123={anchor=center}
     }
\end{tikzpicture}\ ,
\end{center}
which are sent by any simplicial
morphism to $3$-simplices of $N_3(B)$ for which the principal $3$-cell is invertible, but non necessarily trivial.
Said otherwise, the different ways of encoding cells, or simple compositions of cells,
with simplices are linked together by higher simplices having the property that
the cell of greatest dimension is invertible; these higher simplices
act as invertible constraints for morphisms between Street nerves of $3$-categories
and it is therefore natural to imagine that a normalised oplax $3$-functor would
correspond to a simplicial morphism for which all these higher simplices acting
as constraints have \emph{trivial} greatest cell, instead of only invertible.
In order to determine a substantial set of these constraints,
we analyse the simplicial morphism canonically associated
to our notion of oplax normalised $3$-functor.
This provides a simplicial notion of oplax $3$-functor
preserving the underlying $3$-graph, that we call \ndef{simplicial oplax $3$-morphisms}.
We show that they compose and thus form a category whose objects are small $3$-categories.

The standard definition of a normalised oplax $2$-functor $F \colon A \to B$ has objects, $1$-cells, $2$-cells
and composition of $1$-cells as
datum, that is, to any object, $1$-arrow or $2$-arrow $x$ of $A$,
we associate an object, a $1$-arrow or a $2$-arrow, respectively, $F(x)$ of $B$
and for any pair $a \xto{f} b \xto{g} c$ of composable $1$-arrows of $A$, we
associate a $2$\nbd-arrow $F(g, f)$ of $B$ going from $F(gf)$ to $F(g)F(f)$.
These data must satisfy some normalisation conditions, for instance
$F(1_a) = 1_{F(a)}$, for any object $a$ of $A$, and $F(1_b, f) = 1_{F(f)}$, 
for any $1$-cell $f \colon a \to b$ of $A$,
and a cocycle condition, which is a coherence for the composition of three
$1$-arrows of $A$, and the vertical and horizontal
compatibility for $2$-arrows as coherences.
Another take on this notion is to see the data, the normalisations
and the coherences indexed by Joyal's cellular cat\-e\-gory~$\Theta_3$, whose objects are
trees of height at most $3$. Any tree has a dimension, given by the number of its edges,
and a normalised oplax $2$-functor can be defined as a set of maps index on the trees
$\treeDot$, $\treeL$, $\treeLL$ and $\treeV$ of dimension
at most $2$ for the data, which represents precisely objects, $1$-arrows, $2$-arrows and
compositions of two $1$-arrows; the same trees are the indices for the normalisation
conditions.
Finally the four trees $\treeW$, $\treeY$, $\treeVLeft$ and $\treeVRight$ of dimension $3$,
representing the composition of three composable $1$-arrows,
the vertical composition of two $2$-arrows
and the two possible whiskerings of a $2$-arrow with a $1$-arrow,
are the indices for the coherences.
More precisely, a normalised oplax $2$-functor $F \colon A \to B$ will consist
of a map $F_{\treeDot}$ from the objects of $A$ to the objects of~$B$,
a map $F_{\treeL}$ from the $1$-arrows of $A$ to the $1$-arrows of~$B$
respecting source and target, \ie for $f \colon a \to b$ in $A$ we get
$F_{\treeL}(f) \colon F_{\treeDot}(a) \to F_{\treeDot}(b)$ in~$B$,
a map $F_{\treeLL}$ from the $2$-arrows of $A$ to the $2$-arrows of~$B$
similarly respecting source and target and a map $F_{\treeV}$
that for any pair of composable $1$-arrows $a \xto{f} b \xto{g} c$ of $A$
associates a $2$-cell~$F_{\treeV}(g, f)$
\[
\begin{tikzcd}[column sep=small]
F_{\treeDot}(a)
\ar[rd, "F_{\treeLog}(f)"']
\ar[rr, "F_{\treeLog}(g f)"{name=gf}] &&
F_{\treeDot}(a'') \\
& F_{\treeDot}(a') \ar[ru, "F_{\treeL}(g)"'] & 
\ar[Rightarrow, shorten <=1.5mm, from=gf, to=2-2]
\end{tikzcd}
\]
as explained above. The coherence $F_{\treeW}(h, g, f)$ for a triple of composable
$1$-arrows of~$A$, say $a \xto{f} b \xto{g} c \xto{h} d$, can be represented by the following
diagram
\begin{center}
	\centering
	\begin{tikzpicture}[scale=2, font=\footnotesize, every label/.style={fill=white}]
	\squares{%
		/squares/label/.cd,
		0=$F_{\treeDot}(a)$, 1=$F_{\treeDot}(a')$, 2=$F_{\treeDot}(a'')$, 3=$F_{\treeDot}(a''')$, 
		01=${F_{\treeLog}(f)}$, 12=${F_{\treeL}(g)}$, 23=${F_{\treeL}(h)}$, 
		02=${F_{\treeL}(g\comp_0 f)}$, 03=${F_{\treeLog}(h\comp_0 g f)}$, 13=${F_{\treeL}(h g)}$,
		012=${F_{\treeV}(g, f)}$, 023={${F_{\treeV}(h, g  f)}$},
		123=${F_{\treeV}(h, g)}$, 013={${F_{\treeV}(h g, f)}$},
		0123={$F_{\treeW}(h, g, f)$},
		/squares/arrowstyle/.cd,
		0123={equal},
		/squares/labelstyle/.cd,
		012={below right = -1pt and -1pt},
		123={below left = -1pt and -1pt},
		023={swap, above left = -1pt and 3pt},
		013={swap, above right = -1pt and 3pt}
	}
	\end{tikzpicture}\ .
\end{center}
The tree $\treeY$ representing the vertical composition
of $2$-cells plays a key role, which in this low dimensional case is hidden
but becomes much clearer in the $3$-dimensional case.

We take this latter definition of normalised oplax $2$-functor, and the ``cellular'' point of view behind it,
as the starting point for a generalisation of this notion
for the case of $3$\nbd-cat\-e\-gories. Indeed, a normalised oplax $3$-functor shall consist
of a family of maps, the data, indexed by the trees of dimension at most $3$
\emph{except for $\treeY$},
subject to normalisation conditions indexed by these same trees as well as to
a set of coherences indexed by the trees of dimension $4$ joint with the tree $\treeY$.

Listing the tree $\treeY$, that is the only tree of dimension $3$ representing a composition of cells
which does not ``branch'' at height $0$, among the coherences is essential,
since the datum associated to such a tree must consist of a \emph{trivial cell}
(or invertible, but we do not follow this path) for the composition of two such normalised oplax
$3$-functors to be defined. It can be read as a condition of local strictness.
It is already crucial in showing that a normalised oplax $3$\nbd-func\-tor $F$
induces a canonical morphism of simplicial sets $\Nl(F)$, \ie that it can be pre-composed 
with normalised oplax $3$-functors with source a simplex; we show
that this induced morphism of simplicial sets is in fact a simplicial oplax $3$-morphism. 
Nevertheless, showing directly that
the composition of two normalised oplax $3$-functors is still a normalised oplax $3$-functor
is a very hard task. Indeed, the proof that a normalised oplax $3$-functor
induces a canonical simplicial oplax $3$-morphism boils down to
show that the coherence for the tree $\treeVV$, representing the
composition of four $1$-arrows, is satisfied by the ``obvious''
representative for the composition of normalised oplax $3$-functors;
this proof is highly non-trivial and involves a long and careful study
of pastings of all the coherences of the two composed normalised oplax $3$-functors.
This is just one of the 14 coherences that one would need to check
for the composition of normalised oplax $3$-functors to be well-defined.

As remarked above, a careful examination of $\Nl(F) \colon N_3(A) \to N_3(B)$,
the simplicial morphism  associated to a normalised
oplax $3$-functor $F \colon A \to B$,
reveals that certain non-trivial $3$-simplices
of $N_3(A)$ with trivial principal $3$-cell
are sent via $\Nl(F)$ to $3$-simplices of $N_3(B)$
where the principal cell is also trivial.
Such simplices were called constraints before
and if these constraints are taken as a property,
they allow us to give a simplicial definition
of a normalised oplax morphism between nerves of $3$-categories which preserves their underlying
$3$-graph. We called \emph{simplicial oplax $3$-morphisms} such morphisms of simplicial sets.
Hence, we come equipped with two notions
of normalised oplax morphisms for $3$-categories:
one that is cellular in spirit and the other that is simplicial.
The latter has the advantage that it is easy to see that it composes and forms a category;
while the advantage of the former is that it allows us to reason on cells or simple composition of cells
to define complicated morphisms, instead of describing it for objects of every dimension as
for a morphism of simplicial sets. We show how to associate a normalised oplax $3$-functor $\cCl(F)$
to any simplicial oplax $3$-morphism $F$. 
In order to do this, we have to check that the ``obvious''
data that we can associate to a simplicial oplax $3$-morphism
satisfies the normalisation conditions and the coherences of
a normalised oplax $3$-functor. The normalisations are simply
encoded in the degeneracies, while the coherences are non-trivially
encoded by appropriate $4$-simplices. As one might probably expect,
the coherence for the tree representing the horizontal composition
of $2$-cells is the hardest to prove.
We then show that this assignment going from simplicial to cellular is inverse to the map
giving a simplicial morphism to any normalised oplax $3$-functor.
Contrarily to what happens for the $2$-categorical case, this
is also a non-trivial task. For $F \colon N_3(A) \to N_3(B)$ a simplicial oplax $3$-morphism,
we are led to provide, for any $3$-simplex $x$ of $N_3(A)$,
an explicit description of
the principal $3$-cell of the $3$-simplex $\Nl\cCl(F)(x)$
in terms of the normalised oplax $3$-functor $\cCl(F)$.
This gives a bijection between the normalised oplax $3$\nbd-functors
from $A$ to $B$ and the simplicial oplax $3$\nbd-morphisms from $N_3(A)$
to $N_3(B)$.
We can then deduce
that normalised oplax $3$-functors compose and form a category and that this
category is isomorphic to that of simplicial oplax $3$\nbd-mor\-phisms.

The notion of oplax $3$-morphism already appears in the literature.
Indeed, Gordon--Power--Street in~\cite{GordonPowerStreet} and
Gurski in~\cite{GurskiCoherence} provide similar, although slightly different,
general definitions for
trimorphisms between tricategories, with oplax variants.
However, if we specialise to strict $3$-categories
we see that these notions are different from our.
In fact, the main difference lies in the oriented coherence,
by which we mean the datum, associated to the horizontal
composition of $2$-cells, that in our case
expresses instead a relation between the two
pieces of data associated to the two possible whiskerings
of a $2$-cell with a $1$-cell; the data of these
two whiskering are symmetric, as it is imposed by
the algebra of the orientals, which is incompatible
with the choice of a prescribed lax/oplax  orientation for the horizontal composition of $2$-cells, as it is imposed in the definition of trimorphism.

An important and motivating example of normalised oplax $3$-functor
is given by the ``sup'' morphism $\sup \colon \cDelta/N_3(A) \to A$, where $A$ is a $3$-category
and $\cDelta/N_3(A)$ is the category of elements of its Street nerve.
In the $1$-categorical case, this morphism is actually a $1$-functor
mapping an object $(\Deltan{n}, x)$ of $\Delta/N(A)$, where $x \colon \Deltan{n} \to A$
is a sequence of $n$ composable arrows of $A$, to the object $x(n)$ of $A$
and a morphism $f \colon (\Deltan{n}, y) \to (\Deltan{p}, y)$ of $\cDelta/N(A)$,
that we can depict by the triangle
\[
 \begin{tikzcd}[column sep=small]
  \Deltan{n} \ar[rr, "f"] \ar[rd, "x"'] && \Deltan{p} \ar[ld, "y"] \\
  & A &
 \end{tikzcd}\ ,
\]
to the arrow $y_{\{f(n), p\}} \colon y(f(n)) \to y(p)$. This
arrow can be seen as a functor
\[
\begin{tikzcd}
 \Deltan{1} \ar[rr, "{\{f(n), p\}}"] &&\Deltan{p} \ar[r, "y"] & A\ .
 \end{tikzcd}
\]
The sup $1$-functor is always a Thomason equivalence and it plays an important
role in the elementary homotopy theory of $1$-categories.
Del Hoyo~\cite{del-Hoyo} and Chiche~\cite{chiche_homotopy}
generalised and studied this $\sup$ morphism for the case of $2$-categories.
For instance, for a pair of composable $1$-arrows
$(\Deltan{m}, x) \xto{f} (\Deltan{n}, y) \xto{g} (\Deltan{p}, z)$ of~$\cDelta/N_2(A)$,
we assign the $2$-cell of $A$ given by the principal $2$-arrow of the $2$-functor
\[
 \begin{tikzcd}
 c_2N_2(\Deltan{2}) \ar[rrr, "{\{gf(m), g(n), p\}}"] &&& c_2N_2(\Deltan{p}) \ar[r, "z"] & A
 \end{tikzcd}\ .
\]
This normalised oplax $2$-functor proved to be crucial for the elementary homotopy theory of $2$-categories
and we provide a $3$-dimensional definition with our notion of normalised oplax $3$-functor.

We also study the ``strictification'' of such a morphism.
By this we mean the following general procedure: given a $1$-category $A$ and a $3$-category $B$,
there exists a $3$\nbd-cat\-e\-gory $\tilde A$ and a normalised oplax $3$-functor $\eta_A \colon A \to \tilde A$
such that the set of strict $3$-functors from $\tilde A$ to $B$
is in bijection with the set of normalised oplax $3$-functors from $A$ to $B$
and moreover this bijection is obtained by pre-composing by~$\eta_A$.
By the correspondence described above between normalised oplax $3$-functors
and simplicial oplax $3$-morphisms, the $3$-category $\tilde A$ is given
by $c_3N_3(A)$ and the morphism $\eta_A$ is just the unit $\eta_A \colon N_3(A) \to N_3c_3N_3(A)$.
A nice description for the $2$-categorical case has been given by del Hoyo~\cite{del-Hoyo},
so in particular we already know how to describe the $1$-cells of $\tilde A$.
We tackle this problem more generally and we provide a complete description of the
\oo-category $c_\infty \SN(A)$, for any $1$-category $A$ without split-monos or split-epis;
all posets and the subdivision $c \Sd N(C)$ of any $1$-category $C$ have such a property.
The objects of $c_\infty N(A)$ are the same as those of $A$ and, as we observed above,
we already know the $1$-cells, which are given by non-degenerate simplices of $N(A)$.
We define an \oo-category $c_\infty N(A)(f, g)$, for any pair of parallel
$1$-arrows $f$ and $g$ of $c_\infty N(A)$, as well as ``vertical compositions''
\[
 c_\infty N(A)(g, h) \times c_\infty N(A)(f, g) \to c_\infty N(A)(f, h)
\]
of $2$-cells and ``horizontal compositions''
\[
c_\infty N(A)(y, z) \times c_\infty N(A)(x, y) \to c_\infty N(A)(x, z)\,,
\]
where $x$, $y$ and $z$ are objects of $c_\infty N(A)$,
and we check that they satisfy all the axioms of \oo-category.
Finally, we prove that our definition of $c_\infty N(A)$ indeed satisfies
the expected universal property.

In an appendix, we recall some technicalities that we need in order
to develop this last section about strictifications.
In fact, we shall need some precise properties about the internal
Hom-$\infty$-categories of the orientals. We begin by recalling few
elements on Steiner's theory of augmented directed complexes.
We then give a brief glance at Joyal's $\Theta$ category and
finally we introduce the orientals together with some results
about them that we need in section~\ref{sec:tilde}.

\begin{notations}\label{notation}
The category of small strict $n$-category will be
denoted by $\nCat{n}$, for ay $1\leq n\leq \infty$.
For a definition of $\nCat{n}$, see Appendix~\ref{app:higher_cats}.
The simplex category shall be denoted by $\cDelta$
and $\EnsSimp$ shall denote the category of simplicial sets. The functor $\EnsSimp \to \Cat$
that to any simplicial set $X$ associate its
category of elements $\cDelta/X$ shall
be denoted by $i_\cDelta$.

All the \oo-categories will be strict, with
the composition operations denoted by $\comp_i$
and the identity of a cell $x$ denoted by $1_x$.
That is, if $A$ is an \oo-category, $0\leq i < j$ are integers and $x$ and $y$
are $j$-cells of $A$ such that the $i$-target $t_i(x)$ of $x$ is equal to the $i$-source $s_i(y)$
of $y$, then there exists a unique $j$-cell
$y \comp_i x$ of $A$ which is the $i$-composition
of $x$ and $y$; similarly, there is a $(j+1)$-cell
$1_x$ which is the identity of $x$. We shall often
call \emph{trivial} the identity cells of $A$.
We shall say that two $i$-cells $x$ and $y$ of $A$
are \emph{parallel} is they have same source and target;
if this is the case, we shall denote by $\Homi_A(x, y)$
the \oo-category whose $j$-cells, for $j \geq 0$,
are the $(i+j+1)$-cells of $A$ having $x$ as $i$-source
and $y$ as $i$-target.

The Street nerve from $n$-categories to simplicial sets shall always be denoted by $\SN \colon \nCat{n} \to \EnsSimp$, for any $1 \leq n \leq \infty$. This is justified by the fact that we can embed $n$-categories in \oo-categories. The left adjoint to the Street nerve shall be denoted
by $c_n \colon \EnsSimp \to \nCat{n}$.
\end{notations}

\begin{remerciements}
	I would like to thank Dimitri Ara for his support and guidance, who supervised my Ph.D.~thesis from which
	this work originated. I also want to thank Steve Lack
	for many interesting comments that overall improved
	the readability of this paper.
	I heartily thank Fosco Loregian and Dominic
	Verity for their invaluable \TeX{}nical help.
	Finally, I greatfully acknowledge the support of GA\v{C}R EXPRO 19-28628X.
\end{remerciements}

\section{Normalised oplax 3-functors}

 \subsection{Recall: normalised oplax 2-functors}
    
    We begin this chapter by recalling the classical notion of
    \ndef{normalised oplax $2$\nbd-func\-tor}.
    
    \begin{paragr}
     Let $A$ and $B$ be two $2$-categories.
     A \ndef{normalised oplax $2$-functor}\index{normalised oplax $2$-functor}
     $F \colon A \to B$ is given by:
     \begin{itemize}
      \item[-] a map $\Ob(A) \to \Ob(B)$ that
      to any object $x$ of $A$ associates an object
      $F(a)$ of $B$;
      \item[-] a map $\Fl_1(A) \to \Fl_1(B)$ that
      to any $1$-cell $f \colon x \to y$ of $A$ associates
      a $1$-cell $F(f) \colon F(x) \to F(y)$ of $B$;
      \item[-] a map $\Fl_2(A) \to \Fl_2(B)$ that
      to any $2$-cell $\alpha \colon f \to g$ of $A$
      associates a $2$-cell $F(\alpha) \colon F(f) \to F(g)$
      of $B$;
      \item[-] a map that to any composable $1$-cells
      $x \xrightarrow{f} y \xrightarrow{g} z$
       of $A$ associates a $2$-cell
       \[
        F(g, f) \colon F(g\comp_0 f) \to F(g) \comp_0 F(f)
       \]
       of $B$.
     \end{itemize}
     These data are subject to the following coherences:
     \begin{description}
      \item[normalisation] for any object $x$ of $A$ (resp.~any $1$-cell $f$ of $A$)
      we have $F(1_x) = 1_{F(x)}$ (resp.~$F(1_f) = 1_{F(f)}$); moreover for
      any $1$-cell $f\colon x \to y$ of $A$ we have
      \[F(1_y, f) = 1_{F(f)} = F(f, 1_x)\,;\]
      \item[cocycle] for any triple
      $x \xto{f} y \xto{g} z \xto{h} t$
      of composable $1$-cells of $A$ we have
      \[
       \bigl(F(h) \comp_0 F(g, f)\bigr) \comp_1 F(h, g\comp_0 f)
       = \bigl(F(h, g) \comp_1 F(f)\bigr) \comp_1 F(h\comp_0 g, f)\,;
      \]
      
      \item[vertical compatibility] for any pair
      \[
		 \begin{tikzcd}[column sep=4.5em]
		  a\phantom{'}
		  \ar[r, bend left=50, looseness=1.2, "f", ""{below, name=f}]
		  \ar[r, "g" description, ""{name=gu}, ""{below, name=gd}]
		  \ar[r, bend right=50, looseness=1.2, "h"', ""{name=h}]
		  \ar[Rightarrow, from=f, to=gu, "\alpha"]
		  \ar[Rightarrow, from=gd, to=h, "\beta"]&
		  a'
		 \end{tikzcd}
		\]
      of $1$-composable $2$-cells $\alpha$ and $\beta$ of $A$, we have
      $F(\beta\comp_1 \alpha) = F(\beta) \comp_1 F(\alpha)$;
      
      \item[horizontal compatibility] for any pair
      \[
      \begin{tikzcd}[column sep=4.5em]
        \bullet
        \ar[r, bend left, "f", ""{below, name=f1}]
        \ar[r, bend right, "f'"', ""{name=f2}]
        \ar[Rightarrow, from=f1, to=f2, "\alpha"]
        &
        \bullet
        \ar[r, bend left, "g", ""{below, name=g1}]
        \ar[r, bend right, "g'"', ""{name=g2}]
        \ar[Rightarrow, from=g1, to=g2, "\beta"]
        & \bullet
      \end{tikzcd}
     \]
     of $0$-composable $2$-cells $\alpha$ and $\beta$ of $A$,
     we have
     \[
      F(g', f') \comp_1 F(\beta \comp_0 \alpha) =
      \bigl( F(\beta) \comp_0 F(\alpha)\bigr) \comp_1 F(g, f)\,.
     \]
     \end{description}
    \end{paragr}
    
    \begin{paragr}
     The coherence given by the compatibility with respect to
     ``horizontal composition'' can be equivalently decomposed
     in two coherences, which correspond to the two possible
     ``whiskerings'' of a $2$-cell with a $1$-cell:
     \begin{description}
      \item[$\TreeVRight\ $] for any diagram
      \[
		 \begin{tikzcd}[column sep=4.5em]
		  a
		  \ar[r, bend left, "f", ""{below, name=f}]
		  \ar[r, bend right, "f'"', ""{name=fp}]
		  \ar[Rightarrow, from=f, to=fp, "\alpha"]
		  &
		  a' \ar[r,"g"] &
		  a''
		 \end{tikzcd}
		\]
		of $A$, we have
		\[
            F(g) \comp_0 F(\alpha) \comp_1 F(g, f)
             = 
            F(g, f') \comp_1 F(g \comp_0 \alpha)\,;
        \]
		\item[$\treeVLeft\ $] for any diagram
		\[
            \begin{tikzcd}[column sep=4.5em]
             a \ar[r, "f"] &
             a'
             \ar[r, bend left, "g", ""{below, name=g}]
		  \ar[r, bend right, "g'"', ""{name=gp}]&
		  a''
		  \ar[Rightarrow, from=g, to=gp, "\beta"]
            \end{tikzcd}
		\]
		of $A$, we have
		\[
		 F(g', f) \comp_1 F(\beta \comp_0 f)
		  =
		 F(\beta) \comp_0 F(f) \comp_1 F(g, f)\,.
		\]
     \end{description}
     These two coherences are a particular case of the horizontal coherence
     of the previous paragraph and reciprocally one checks immediately these
     two coherences joint with the vertical coherence imply the horizontal
     coherence.
     
     The advantage of this latter reformulation is that now
     the coherence datum of a normalised oplax $2$-functor
     is indexed over
     \begin{itemize}
      \item the trees of dimension $\le 2$ for the normalisation;
      \item the trees of dimension $3$ for the other coherences.
     \end{itemize}
     This will be the starting point in our generalisation towards
     a notion of normalised oplax $3$-functor.
     In order to clarify what we mean, we repropose the
     definition of normalised oplax $2$-functor
     with data and coherences indexed by trees.
    \end{paragr}
    
    \begin{paragr}
     A normalised oplax $2$-functor $F \colon A \to B$ consists
     of the following data:
     \begin{description}
		\item[$\TreeDot\:$] a map $F_{\treeDot}$ that
		to each object $a$ of $A$
		assigns an object $F_{\treeDot}(a)$ of $B$.
		
		\item[%
		\scalebox{0.3}{
			\begin{forest}
			for tree={%
				label/.option=content,
				grow'=north,
				content=,
				circle,
				fill,
				minimum size=3pt,
				inner sep=0pt,
				s sep+=15,
			}
			[[]]
			\end{forest}
		}] a map $F_{\treeLog}$ that to each $1$-cell $f \colon a \to a'$
		of $A$ assigns a $1$-cell $F_{\treeLog}(f) \colon F_{\treeDot}(a) \to F_{\treeDot}(a')$
		of $B$.
		
		\item[%
		\scalebox{0.3}{
			\begin{forest}
				for tree={%
					label/.option=content,
					grow'=north,
					content=,
					circle,
					fill,
					minimum size=3pt,
					inner sep=0pt,
					s sep+=15,
				}
				[[][]]
			\end{forest}
		}] a map $F_{\treeV}$ that to
			to each pair of $0$\hyp{}composable $1$\hyp{}cells
			\[
			 \begin{tikzcd}
			  a \ar[r, "f"] & a' \ar[r, "g"] &  a''
			 \end{tikzcd}
			\]
			of $A$ assigns a $2$-cell $F_{\treeV}(g, f)$
			\[
			 \begin{tikzcd}[column sep=small]
				 F_{\treeDot}(a)
				 \ar[rd, "F_{\treeLog}(f)"']
				 \ar[rr, "F_{\treeLog}(g\comp_0 f)"{name=gf}] &&
				 F_{\treeDot}(a'') \\
				 & F_{\treeDot}(a') \ar[ru, "F_{\treeL}(g)"'] & 
				 \ar[Rightarrow, shorten <=1.5mm, from=gf, to=2-2]
			 \end{tikzcd}
			\]
			of $B$, that is
			\[
			 F_{\treeV}(g, f) \colon F_{\treeLog}(g \comp_0 f) \to
			 F_{\treeLog}(g) \comp_0 F_{\treeLog}(f)
			\]

		\item[%
		\scalebox{0.3}{
			\begin{forest}
				for tree={%
					label/.option=content,
					grow'=north,
					content=,
					circle,
					fill,
					minimum size=3pt,
					inner sep=0pt,
					s sep+=15,
				}
				[ [ [] ] ]
			\end{forest}
		}] a map $F_{\treeLL}$
		that to each $2$-cell $\alpha \colon f \to g$ of $A$ associates
		a $2$-cell $F_{\treeLL}(\alpha) \colon F_{\treeLog}(f) \to F_{\treeLog}(g)$
		of $B$.
		\end{description}

		These data are subject to the following conditions of normalisation:
		\begin{description}
		\item[%
		\scalebox{0.3}{
			\begin{forest}
				for tree={%
					label/.option=content,
					grow'=north,
					content=,
					circle,
					fill,
					minimum size=3pt,
					inner sep=0pt,
					s sep+=15,
				}
				[]
			\end{forest}
		}
		] for any object $a$ of $A$, we have
		\[
			F_{\treeLog}(1_a) = 1_{F_{\treeDot}(a)}\,;
		\]
		
		\item[%
		\scalebox{0.3}{
			\begin{forest}
				for tree={%
					label/.option=content,
					grow'=north,
					content=,
					circle,
					fill,
					minimum size=3pt,
					inner sep=0pt,
					s sep+=15,
				}
				[ [] ]
			\end{forest}
		}
		] for any $1$-cell $f$ of $A$ we have
		\[
			F_{\treeLL}(1_f) = 1_{F_{\treeLog}(f)}\,;
		\]
		
		\item[%
		\scalebox{0.3}{
			\begin{forest}
				for tree={%
					label/.option=content,
					grow'=north,
					content=,
					circle,
					fill,
					minimum size=3pt,
					inner sep=0pt,
					s sep+=15,
				}
				[ [] [] ]
			\end{forest}
		}
		] for any $1$\nbd-cell $f \colon a \to a'$ of $A$, we have
		\[
		 F_{\treeV}(1_{a'}, f) = F_{\treeV}(f, 1_a) = 1_{F_{\treeL}(f)}\,.
		\]
		\end{description}

		Finally, we impose the following coherences:
		\begin{description}
		\item[%
		\scalebox{0.3}{
			\begin{forest}
				for tree={%
					label/.option=content,
					grow'=north,
					content=,
					circle,
					fill,
					minimum size=3pt,
					inner sep=0pt,
					s sep+=15,
				}
				[ [] [] [] ]
			\end{forest}
		}] for any triple of $0$-composable $1$ cells
		\[
		 \begin{tikzcd}[column sep=small]
		  a \ar[r, "f"] & a' \ar[r, "g"] & a'' \ar[r, "h"] & a'''
		 \end{tikzcd}
		\]
		of $A$ we have	
	\begin{center}
		\centering
		\begin{tikzpicture}[scale=2, font=\footnotesize, every label/.style={fill=white}]
		\squares{%
			/squares/label/.cd,
			0=$F_{\treeDot}(a)$, 1=$F_{\treeDot}(a')$, 2=$F_{\treeDot}(a'')$, 3=$F_{\treeDot}(a''')$, 
			01=${F_{\treeLog}(f)}$, 12=${F_{\treeL}(g)}$, 23=${F_{\treeL}(h)}$, 
			02=${F_{\treeL}(g\comp_0 f)}$, 03=${F_{\treeLog}(h\comp_0 g f)}$, 13=${F_{\treeL}(h g)}$,
			012=${F_{\treeV}(g, f)}$, 023={${F_{\treeV}(h, g  f)}$},
			123=${F_{\treeV}(h, g)}$, 013={${F_{\treeV}(h g, f)}$},
			0123={$F_{\treeW}(h, g, f)$},
			/squares/arrowstyle/.cd,
			0123={equal},
			/squares/labelstyle/.cd,
			012={below right = -1pt and -1pt},
			123={below left = -1pt and -1pt},
			023={above left = 1pt and 3pt},
			013={above right = 1pt and 3pt}
		}
		\end{tikzpicture}\ .
	\end{center}
		that is we impose the equality
		\[
            F_{\treeLog}(h) \comp_0 F_{\treeV}(g, f) \comp_1 F_{\treeV}(h, g \comp_0 f)
		=
            F_{\treeV}(h, g) \comp_0 F_{\treeLog}(f) \comp_1 F_{\treeV}(h \comp_0 g, f)\,;
		\]
		
		\item[%
		\scalebox{0.3}{
			\begin{forest}
				for tree={%
					label/.option=content,
					grow'=north,
					content=,
					circle,
					fill,
					minimum size=3pt,
					inner sep=0pt,
					s sep+=15,
				}
				[ [[][]] ]
			\end{forest}
		}] for any pair
		\[
		 \begin{tikzcd}[column sep=4.5em]
		  a\phantom{'}
		  \ar[r, bend left=50, looseness=1.2, "f", ""{below, name=f}]
		  \ar[r, "g" description, ""{name=gu}, ""{below, name=gd}]
		  \ar[r, bend right=50, looseness=1.2, "h"', ""{name=h}]
		  \ar[Rightarrow, from=f, to=gu, "\alpha"]
		  \ar[Rightarrow, from=gd, to=h, "\beta"]&
		  a'
		 \end{tikzcd}
		\]
      of $1$-composable $2$-cells $\alpha$ and $\beta$ of $A$, we impose
      \[F(\beta\comp_1 \alpha) = F(\beta) \comp_1 F(\alpha)\,;\]
		
		\item[%
		\scalebox{0.3}{
			\begin{forest}
				for tree={%
					label/.option=content,
					grow'=north,
					content=,
					circle,
					fill,
					minimum size=3pt,
					inner sep=0pt,
					s sep+=15,
				}
				[ [] [[]] ]
			\end{forest}
		}] for any whiskering
		\[
		 \begin{tikzcd}[column sep=4.5em]
		  a
		  \ar[r, bend left, "f", ""{below, name=f}]
		  \ar[r, bend right, "f'"', ""{name=fp}]
		  \ar[Rightarrow, from=f, to=fp, "\alpha"]
		  &
		  a' \ar[r,"g"] &
		  a''
		 \end{tikzcd}
		\]
		of $A$ we impose
		\[
            F_{\treeLog}(g) \comp_0 F_{\treeLL}(\alpha) \comp_1 F_{\treeV}(g, f)
            =
            F_{\treeV}(g, f') \comp_1 F_{\treeLL}(g \comp_0 \alpha)\,;
        \]

		\item[%
		\scalebox{0.3}{
			\begin{forest}
				for tree={%
					label/.option=content,
					grow'=north,
					content=,
					circle,
					fill,
					minimum size=3pt,
					inner sep=0pt,
					s sep+=15,
				}
				[ [[]] [] ]
			\end{forest}
		}] for any whiskering
		\[
            \begin{tikzcd}[column sep=4.5em]
             a \ar[r, "f"] &
             a'
             \ar[r, bend left, "g", ""{below, name=g}]
		  \ar[r, bend right, "g'"', ""{name=gp}]&
		  a''
		  \ar[Rightarrow, from=g, to=gp, "\beta"]
            \end{tikzcd}
		\]
		of $A$ we impose
		\[
		 F_{\treeV}(g', f) \comp_1 F_{\treeLL}(\beta \comp_0 f)
		 =
		 F_{\treeLL}(\beta) \comp_0 F_{\treeLog}(f) \comp_1 F_{\treeV}(g, f)\,.
		\]
    \end{description}
    \end{paragr}
    
    \begin{paragr}
     Given two normalised oplax $2$-functors $F \colon A \to B$
     and $G \colon B \to C$, there is an obvious candidate for
     the composition $GF \colon A \to C$ and one checks that
     this is still a normalised oplax $2$-functor; furthermore,
     the identity functor on a category is clearly an identity
     element for normalised oplax $2$-functor too. Hence,
     there is a category $\widetilde{\nCat{2}}$ with small
     $2$-categories as objects and normalised oplax $2$-functors
     as morphisms.
     
     The cosimplicial object $\Delta \to \nCat{2} \to \widetilde{\nCat{2}}$
     of $\widetilde{\nCat{2}}$ induces a nerve functor
     $\widetilde{N_2} \colon \widetilde{\nCat{2}} \to \EnsSimp$.
     For any $n \ge 0$, the normalised oplax $2$-functors
     $\Deltan{n} \to A$ correspond precisely to $2$\nbd-func\-tors
     $\Onm{n}{2} \to A$
     (see, for instance, \cite[\href{https://kerodon.net/tag/00BE}{Tag 00BE}]{kerodon}).
     Hence, we get a triangle diagram of functors
     \[
     \begin{tikzcd}
      & \EnsSimp & \\
      \nCat{2} \ar[ru, "N_2"] \ar[rr, hook] &&
      \widetilde{\nCat{2}} \ar[lu, "\widetilde{N_2}"']
     \end{tikzcd}
     \]
     which is commutative (up to a canonical isomorphism). Moreover,
     it is a standard fact that the functor $\widetilde{N_2}$
     is fully faithful (see, for instance~\cite{BullejosFaroBlanco}, \cite{LackPaoli}
     or~\cite[\href{https://kerodon.net/tag/00AU}{Tag 00AU}]{kerodon}).
    \end{paragr}

\subsection{Definition of normalised oplax 3-functor}

	Let $A$ and $B$ be two $3$-cat\-e\-go\-ries. We now give the
	definition of \ndef{normalised oplax $3$-functor}
	\index{normalised oplax $3$-functor}
	$F \colon A \to B$, which amounts to giving the structure
	of a family of maps indexed by the objects of $\Theta$
	of dimension at most $3$ subject to a family of relations
	indexed by the objects of $\Theta$ of dimension $4$
	as well as normalisation for identities of every dimension. A quick description of~$\Theta$ is provided
	in Appendix~\ref{app:theta}.
	The choice of the orientation for the structural maps is
	strongly guided by the algebra of the orientals and presents
	therefore a symmetry (or better, a duality) for symmetric trees of $\Theta$,
	see for instance the structural cells for the trees $\TreeVLeft$ and $\TreeVRight$
	denoting the two possible whiskerings of a $2$-cell with a $1$-cell.
	
	The arboreal rule for indexing structure, normalisation and coherence that we have stated
	right above has an exception. In fact, we are forced to list
	the tree $\TreeY$ denoting the vertical composition of two $2$-cells
	among the coherences. This imposes a local strictness on the lax $3$-functor,
	meaning that as a result a normalised $3$-functor will induce a strict $2$-functor
	on the hom-$2$-categories, and it is necessary in order to provide a reasonable
	coherence for the tree $\treeVLR$ representing the horizontal composition of two $2$-cells;
	we shall say more about this in a remark right after the definition.
	Another explanation for this choice, more simplicial in spirit, will be offered in
	section~\ref{sec:simplicial} (see in particular
	Remark~\ref{rem:treeY}).

\begin{paragr}[Data]\label{paragr:lax_3functor_cellular_data}
    A \ndef{normalised oplax $3$-functor} $F$ from $A$ to $B$ consists
    of:
	\begin{description}
		\item[$\TreeDot\ $] a map $F_{\treeDot}$ that
		to each object $a$ of $A$
		assigns an object $F_{\treeDot}(a)$ of $B$.
		
		\item[%
		\scalebox{0.3}{
			\begin{forest}
			for tree={%
				label/.option=content,
				grow'=north,
				content=,
				circle,
				fill,
				minimum size=3pt,
				inner sep=0pt,
				s sep+=15,
			}
			[[]]
			\end{forest}
		}] a map $F_{\treeLog}$ that to each $1$-cell $f \colon a \to a'$
		of $A$ assigns a $1$-cell $F_{\treeLog}(f) \colon F_{\treeDot}(a) \to F_{\treeDot}(a')$
		of $B$;
		
		\item[%
		\scalebox{0.3}{
			\begin{forest}
				for tree={%
					label/.option=content,
					grow'=north,
					content=,
					circle,
					fill,
					minimum size=3pt,
					inner sep=0pt,
					s sep+=15,
				}
				[[][]]
			\end{forest}
		}] a map $F_{\treeV}$ that to
			to each pair of $0$\hyp{}composable $1$\hyp{}cells
			\[
			 \begin{tikzcd}
			  a \ar[r, "f"] & a' \ar[r, "g"] &  a''
			 \end{tikzcd}
			\]
			of $A$ assigns a $2$-cell $F_{\treeV}(g, f)$
			\[
			 \begin{tikzcd}[column sep=small]
				 F_{\treeDot}(a)
				 \ar[rd, "F_{\treeLog}(f)"']
				 \ar[rr, "F_{\treeLog}(g\comp_0 f)"{name=gf}] &&
				 F_{\treeDot}(a'') \\
				 & F_{\treeDot}(a') \ar[ru, "F_{\treeL}(g)"'] & 
				 \ar[Rightarrow, shorten <=1.5mm, from=gf, to=2-2]
			 \end{tikzcd}
			\]
			of $B$, that is
			\[
			 F_{\treeV}(g, f) \colon F_{\treeLog}(g \comp_0 f) \to
			 F_{\treeLog}(g) \comp_0 F_{\treeLog}(f)\,;
			\]

		\item[%
		\scalebox{0.3}{
			\begin{forest}
				for tree={%
					label/.option=content,
					grow'=north,
					content=,
					circle,
					fill,
					minimum size=3pt,
					inner sep=0pt,
					s sep+=15,
				}
				[ [ [] ] ]
			\end{forest}
		}] a map $F_{\treeLL}$
		that to each $2$-cell $\alpha \colon f \to g$ of $A$ associates
		a $2$-cell $F_{\treeLL}(\alpha) \colon F_{\treeLog}(f) \to F_{\treeLog}(g)$
		of $B$;
		
		\item[%
		\scalebox{0.3}{
			\begin{forest}
				for tree={%
					label/.option=content,
					grow'=north,
					content=,
					circle,
					fill,
					minimum size=3pt,
					inner sep=0pt,
					s sep+=15,
				}
				[ [] [] [] ]
			\end{forest}
		}] a map $F_{\treeW}$,
		that to each triple of $0$-composable $1$ cells
		\[
		 \begin{tikzcd}[column sep=small]
		  a \ar[r, "f"] & a' \ar[r, "g"] & a'' \ar[r, "h"] & a'''
		 \end{tikzcd}
		\]
		of $A$ associates a $3$-cell $F_{\treeW}(h, g, f)$
		\begin{center}
			\centering
			\begin{tikzpicture}[scale=1.8, font=\footnotesize, every label/.style={fill=white}]
			\squares{%
				/squares/label/.cd,
				0=$F_{\treeDot}(a)$, 1=$F_{\treeDot}(a')$, 2=$F_{\treeDot}(a'')$, 3=$F_{\treeDot}(a''')$, 
				01=${F_{\treeLog}(f)}$, 12=${F_{\treeL}(g)}$, 23=${F_{\treeL}(h)}$, 
				02=${F_{\treeL}(g\comp_0 f)}$, 03=${F_{\treeLog}(h\comp_0 g f)}$, 13=${F_{\treeL}(h g)}$,
				012=${F_{\treeV}(g, f)}$, 023={${F_{\treeV}(h, g  f)}$},
				123=${F_{\treeV}(h, g)}$, 013={${F_{\treeV}(h g, f)}$},
				0123={$F_{\treeW}(h, g, f)$},
				/squares/labelstyle/.cd,
				012={below right = -1pt and -1pt},
				123={below left = -1pt and -1pt},
				023={above left = 1pt and 3pt},
				013={above right = 1pt and 3pt}
			}
			\end{tikzpicture}\ ,
		\end{center}
		that is the $3$-cell $F_{\treeW}(h, g, f)$ has
		\[
            F_{\treeLog}(h) \comp_0 F_{\treeV}(g, f) \comp_1 F_{\treeV}(h, g \comp_0 f)
		\]
		as source and
		\[
            F_{\treeV}(h, g) \comp_0 F_{\treeLog}(f) \comp_1 F_{\treeV}(h \comp_0 g, f)
		\]
		as target;
		
		\item[%
		\scalebox{0.3}{
			\begin{forest}
				for tree={%
					label/.option=content,
					grow'=north,
					content=,
					circle,
					fill,
					minimum size=3pt,
					inner sep=0pt,
					s sep+=15,
				}
				[ [] [[]] ]
			\end{forest}
		}] a map $F_{\treeVRight}$
		that to any whiskering
		\[
		 \begin{tikzcd}[column sep=4.5em]
		  a
		  \ar[r, bend left, "f", ""{below, name=f}]
		  \ar[r, bend right, "f'"', ""{name=fp}]
		  \ar[Rightarrow, from=f, to=fp, "\alpha"]
		  &
		  a' \ar[r,"g"] &
		  a''
		 \end{tikzcd}
		\]
		of $A$ associates a $3$-cell 
		\[
            F_{\treeVRight}(g, \alpha) \colon
            F_{\treeLog}(g) \comp_0 F_{\treeLL}(\alpha) \comp_1 F_{\treeV}(g, f)
            \to
            F_{\treeV}(g, f') \comp_1 F_{\treeLL}(g \comp_0 \alpha)
        \]
        of $B$;

		\item[%
		\scalebox{0.3}{
			\begin{forest}
				for tree={%
					label/.option=content,
					grow'=north,
					content=,
					circle,
					fill,
					minimum size=3pt,
					inner sep=0pt,
					s sep+=15,
				}
				[ [[]] [] ]
			\end{forest}
		}] a map $F_{\treeVLeft}$,
		that to each whiskering
		\[
            \begin{tikzcd}[column sep=4.5em]
             a \ar[r, "f"] &
             a'
             \ar[r, bend left, "g", ""{below, name=g}]
		  \ar[r, bend right, "g'"', ""{name=gp}]&
		  a''
		  \ar[Rightarrow, from=g, to=gp, "\beta"]
            \end{tikzcd}
		\]
		of $A$ associates a $3$-cell
		\[
		 F_{\treeVLeft}(\beta, f) \colon
		 F_{\treeV}(g', f) \comp_1 F_{\treeLL}(\beta \comp_0 f)
		 \to
		 F_{\treeLL}(\beta) \comp_0 F_{\treeLog}(f) \comp_1 F_{\treeV}(g, f)
		\]
		of $B$;

		\item[%
		\scalebox{0.3}{
			\begin{forest}
				for tree={%
					label/.option=content,
					grow'=north,
					content=,
					circle,
					fill,
					minimum size=3pt,
					inner sep=0pt,
					s sep+=15,
				}
				[ [[[]]] ]
			\end{forest}
		}] a map $F_{\treeLLL}$,
		that to any $3$-cell $\gamma \colon \alpha \to \alpha'$ of $A$,
		\ie any tree of $A_{\treeLLL}$, associates
		a $3$-cell \[F_{\treeLLL}(\gamma) \colon F_{\treeLL}(\alpha) \to F_{\treeLL}(\alpha')\]
		of $B$.
	\end{description}
	\end{paragr}

	\begin{paragr}[Normalisation]\label{paragr:lax_3functor_cellular_norm}
		The normalisation requires the following constraints:
		\begin{description}
		\item[%
		\scalebox{0.3}{
			\begin{forest}
				for tree={%
					label/.option=content,
					grow'=north,
					content=,
					circle,
					fill,
					minimum size=3pt,
					inner sep=0pt,
					s sep+=15,
				}
				[]
			\end{forest}
		}
		] for any object $a$ of $A$, we have
		\[
			F_{\treeLog}(1_a) = 1_{F_{\treeDot}(a)}\,;
		\]
		
		\item[%
		\scalebox{0.3}{
			\begin{forest}
				for tree={%
					label/.option=content,
					grow'=north,
					content=,
					circle,
					fill,
					minimum size=3pt,
					inner sep=0pt,
					s sep+=15,
				}
				[ [] ]
			\end{forest}
		}
		] for any $1$-cell $f$ of $A$ we have
		\[
			F_{\treeLL}(1_f) = 1_{F_{\treeLog}(f)}\,;
		\]
		
		\item[%
		\scalebox{0.3}{
			\begin{forest}
				for tree={%
					label/.option=content,
					grow'=north,
					content=,
					circle,
					fill,
					minimum size=3pt,
					inner sep=0pt,
					s sep+=15,
				}
				[ [] [] ]
			\end{forest}
		}
		] for any $1$\nbd-cell $f \colon a \to a'$ of $A$, we have
		\[
		 F_{\treeV}(1_{a'}, f) = F_{\treeV}(f, 1_a) = 1_{F_{\treeL}(f)}\,;
		\]

		\item[%
		\scalebox{0.3}{
			\begin{forest}
				for tree={%
					label/.option=content,
					grow'=north,
					content=,
					circle,
					fill,
					minimum size=3pt,
					inner sep=0pt,
					s sep+=15,
				}
				[ [ [] ] ]
			\end{forest}
		}
		] for any $2$-cell $\alpha$ of $A$ we have
		\[
		F_{\treeLLL}(1_\alpha) = 1_{F_{\treeLL}(\alpha)}\,.
		\]

		\item[%
		\scalebox{0.3}{
			\begin{forest}
				for tree={%
					label/.option=content,
					grow'=north,
					content=,
					circle,
					fill,
					minimum size=3pt,
					inner sep=0pt,
					s sep+=15,
				}
				[ [] [] [] ]
			\end{forest}
		}
		]
		for any pair $a \xto{f} a' \xto{g} a''$ of composable $1$-cell of $A$, we have
		\[
		 F_{\treeW}(g, f, 1_a) = F_{\treeW}(g, 1_{a'}, f) = F_{\treeW}(1_{a''}, g, f) = 1_{F_{\treeV}(g, f)}\,;
		\]

		\item[%
						\scalebox{0.3}{
							\begin{forest}
								for tree={%
									label/.option=content,
									grow'=north,
									content=,
									circle,
									fill,
									minimum size=3pt,
									inner sep=0pt,
									s sep+=15,
								}
								[ [] [[]] ]
							\end{forest}
						}] for any pair $a \xto{f} a' \xto{g} a''$ of composable $1$\nbd-cells
						of~$A$, we have
						\[
						F_{\treeVRight}(g, 1_f) = 1_{F_{\treeV}(g, f)}\,,
						\]
						and for any $2$-cell $\alpha \colon f \to f'$ of $A$, we have
						\[
						F_{\treeVRight}(1_{a'}, \alpha) = 1_{F_{\treeLL}(\alpha)}\,;
						\]

		\item[%
						\scalebox{0.3}{
							\begin{forest}
								for tree={%
									label/.option=content,
									grow'=north,
									content=,
									circle,
									fill,
									minimum size=3pt,
									inner sep=0pt,
									s sep+=15,
								}
								[ [[]] [] ]
							\end{forest}
						}] for any pair $a \xto{f} a' \xto{g} a''$ of composable $1$\nbd-cells
						of~$A$, we have
						\[
						F_{\treeVLeft}(1_g, f) = 1_{F_{\treeV}(g, f)}\,,
						\]
						and for any $2$-cell $\beta \colon g \to g'$ of $A$, we have
						\[
						F_{\treeVLeft}(\beta, 1_{a'}) = 1_{F_{\treeLL}(\beta)}\,;
						\]

		\end{description}
		
		\end{paragr}

%
%
%
%
%
%
%
%
%
%

	\begin{paragr}[Coherences]\label{paragr:lax_3functor_cellular_coherences}
	These maps are subject to the following relations:
	\begin{description}
	\item[%
     \scalebox{0.3}{
     	\begin{forest}
     		for tree={%
     			label/.option=content,
     			grow'=north,
     			content=,
     			circle,
     			fill,
     			minimum size=3pt,
     			inner sep=0pt,
     			s sep+=15,
     		}
     		[ [[][]] ]
     	\end{forest}
     }] for any pair of $1$-composable $2$-cells
     \[
     \begin{tikzcd}[column sep=4.5em]
     a\phantom{'}
     \ar[r, bend left=50, looseness=1.2, "f", ""{below, name=f}]
     \ar[r, "g" description, ""{name=gu}, ""{below, name=gd}]
     \ar[r, bend right=50, looseness=1.2, "h"', ""{name=h}]
     \ar[Rightarrow, from=f, to=gu, "\alpha"]
     \ar[Rightarrow, from=gd, to=h, "\beta"]&
     a'
     \end{tikzcd}
     \]
     of $A$, \ie to any tree of $A_{\treeY}$, we have an equality
     \[
     F_{\treeLL}(\beta) \comp_1 F_{\treeLL}(\alpha)
     =
     F_{\treeLL}(\beta \comp_1 \alpha)
     \]
     in $B$. We shall sometimes write $F_{\treeY}(\beta, \alpha)$
     for the identity $3$-cell of this $2$-cell above.

	 \item[%
        \scalebox{0.3}{
            \begin{forest}
                for tree={%
                    label/.option=content,
                    grow'=north,
                    content=,
                    circle,
                    fill,
                    minimum size=3pt,
                    inner sep=0pt,
                    s sep+=15,
                }
                [
                    [] [] [] []
                ]
            \end{forest}
        }
     ] For any quadruple
     \[
      \begin{tikzcd}
       \bullet \ar[r, "f"] & \bullet \ar[r, "g"] & \bullet \ar[r, "h"] & \bullet \ar[r, "i"] & \bullet
      \end{tikzcd}
     \]
     of $0$-composable $1$-cells of $A$ we impose that the $3$-cells
     \begin{gather*}
       F_{\treeV}(i, h) \comp_0 F_{\treeLog}(g) \comp_0 F_{\treeLog}(f)
       \comp_1 F_{\treeW}(ih, g, h) \\
       \comp_2\\
       F_{\treeLog}(i) \comp_0 F_{\treeLog}(h) \comp_0 F_{\treeV}(g, f)
       \comp_1 F_{\treeW}(i, h, gf)
     \end{gather*}
     and
     \begin{gather*}
       F_{\treeW}(i, h, g) \comp_0 F_{\treeLog}(f) \comp_1 F_{\treeV}(i\comp_0 h \comp_0 g, f)\\
       \comp_2\\
       F_{\treeLog}(i) \comp_0 F_{\treeV}(h, g) \comp_0 F_{\treeLog}(f) \comp_1
       F_{\treeW}(i, h \comp_0 g, f)\\
       \comp_2 \\
       F_{\treeLog}(i) \comp_0 F_{\treeW}(h, g, f) \comp_1 F_{\treeV}(i, h\comp_0 g \comp_0 f)
     \end{gather*}
     of $B$ are equal.
     
     \item[%
        \scalebox{0.3}{
            \begin{forest}
                for tree={%
                    label/.option=content,
                    grow'=north,
                    content=,
                    circle,
                    fill,
                    minimum size=3pt,
                    inner sep=0pt,
                    s sep+=15,
                }
                [
                    [ [] ] [] []
                ]
            \end{forest}
        }
     ] For any triple
     \[
      \begin{tikzcd}[column sep=4.5em]
       \bullet \ar[r, "f"] & \bullet \ar[r, "g"] &
       \bullet
       \ar[r, bend left, "h", ""{below, name=h}]
       \ar[r, bend right, "h'"', ""{name=h2}]
       \ar[Rightarrow, from=h, to=h2, "\alpha"]
       & \bullet
      \end{tikzcd}
     \]
     of $0$-composable cells $f$, $g$ and $\alpha$ of $A$
     we impose the $3$-cells
     \begin{gather*}
        F_{\treeLL}(\alpha) \comp_0 F_{\treeLog}(g) \comp_0 F_{\treeLog}(f)
        \comp_1 F_{\treeW}(h, g, f)\\
        \comp_2\\
        F_{\treeLog}(h') \comp_0 F_{\treeV}(g, f) \comp_1 F_{\treeVLeft}(\alpha, g \comp_0 f)
     \end{gather*}
     and
     \begin{gather*}
        F_{\treeVLeft}(\alpha, g) \comp_0 F_{\treeLog}(f) \comp_1 F_{\treeV}(h\comp_0 g, f)\\
        \comp_2\\
        F_{\treeV}(h', g) \comp_0 F_{\treeLog}(f) \comp_1 F_{\treeVLeft}(\alpha \comp_0 g, f)\\
        \comp_2\\
        F_{\treeW}(h', g, f) \comp_1 F_{\treeLL}(\alpha \comp_0 g \comp_0 f)
     \end{gather*}
     of $B$ to be equal.
     
     \item[%
        \scalebox{0.3}
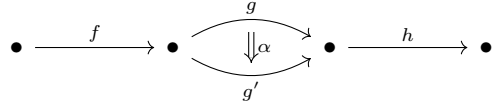
{
            \begin{forest}
                for tree={%
                    label/.option=content,
                    grow'=north,
                    content=,
                    circle,
                    fill,
                    minimum size=3pt,
                    inner sep=0pt,
                    s sep+=15,
                }
                [
                    [] [ [] ] []
                ]
            \end{forest}
        }
     ] For any triple
     \[
      \begin{tikzcd}[column sep=4.5em]
       \bullet \ar[r, "f"] &
       \bullet
       \ar[r, bend left, "g", ""{below, name=g}]
       \ar[r, bend right, "g'"', ""{name=g2}]
       \ar[Rightarrow, from=g, to=g2, "\alpha"] &
       \bullet \ar[r, "h"] & \bullet
      \end{tikzcd}
     \]
     of $0$-composable cells $f$, $\alpha$ and $h$ of $A$,
     we impose that the $3$-cells
     \begin{gather*}
        F_{\treeV}(h, g') \comp_0 F_{\treeLog}(f)
        \comp_1 F_{\treeVLeft}(h\comp_0 \alpha, f)\\
        \comp_2\\
        F_{\treeW}(h, g', f) \comp_1 F_{\treeLL}(h\comp_0 \alpha \comp_0 f)\\
        \comp_2\\
        F_{\treeLog}(h) \comp_0 F_{\treeV}(g', f) \comp_1
        F_{\treeVRight}(h, \alpha \comp_0 f)
     \end{gather*}
     and
     \begin{gather*}
        F_{\treeVRight}(h, \alpha) \comp_0 F_{\treeLog}(f)
        \comp_1 F_{\treeV}(h \comp_0 g, f)\\
        \comp_2\\
        F_{\treeLog}(h) \comp_0 F_{\treeLL}(\alpha) \comp_0 F_{\treeLog}(f)
        \comp_1 F_{\treeW}(h, g, f)\\
        \comp_2\\
        F_{\treeLog}(h) \comp_0 F_{\treeVLeft}(\alpha, f)
        \comp_1 F_{\treeV}(h, g\comp_0 f)
     \end{gather*}
     of $B$ to be equal.
     
     \item[%
        \scalebox{0.3}{
            \begin{forest}
                for tree={%
                    label/.option=content,
                    grow'=north,
                    content=,
                    circle,
                    fill,
                    minimum size=3pt,
                    inner sep=0pt,
                    s sep+=15,
                }
                [
                    [] [] [ [] ]
                ]
            \end{forest}
        }
     ] For any triple
     
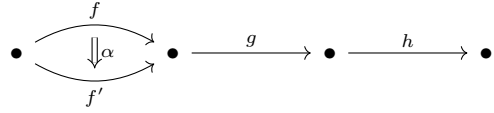
\[
      \begin{tikzcd}[column sep=4.5em]
       \bullet
       \ar[r, bend left, "f", ""{below, name=f}]
       \ar[r, bend right, "f'"', ""{name=f2}]
       \ar[Rightarrow, from=f, to=f2, "\alpha"] &
       \bullet \ar[r, "g"] &
       \bullet \ar[r, "h"] & \bullet
      \end{tikzcd}
     \]
     of $0$-composable cells $\alpha$, $g$ and $h$ of $A$
     we impose the $3$-cells
     \begin{gather*}
        F_{\treeV}(h, g) \comp_0 F_{\treeLog}(f')
        \comp_1 F_{\treeVRight}(h\comp_0 g, \alpha)\\
        \comp_2\\
        F_{\treeLog}(h) \comp_0 F_{\treeLog}(g) \comp_0 F_{\treeLL}(\alpha)
        \comp_1 F_{\treeW}(h, g, f)
     \end{gather*}
     and
     \begin{gather*}
        F_{\treeW}(h, g, f') \comp_1 F_{\treeLL}(h \comp_0 g \comp_0 \alpha)\\
        \comp_2\\
        F_{\treeLog}(h) \comp_0 F_{\treeV}(g, f')
        \comp_1 F_{\treeVRight}(h, g \comp_0 \alpha)\\
        \comp_2\\
        F_{\treeLog}(h) \comp_0 F_{\treeVRight}(g, \alpha)
        \comp_1 F_{\treeV}(h, g \comp_0 f)
     \end{gather*}
     of $B$ to be equal.
     
    \item[%
        \scalebox{0.3}{
            \begin{forest}
                for tree={%
                    label/.option=content,
                    grow=north,
                    content=,
                    circle,
                    fill,
                    minimum size=3pt,
                    inner sep=0pt,
                    s sep+=15,
                }
                [
                    [
                        [] []
                    ]
                    []
                ]
            \end{forest}
        }
     ] For any triple
     
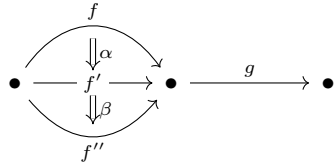
\[
      \begin{tikzcd}[column sep=4.5em]
       \bullet
       \ar[r, bend left=55, looseness=1.3, "f", ""{below, name=f1}]
       \ar[r, "f'"{description}, ""{name=f2u}, ""{below, name=f2d}]
       \ar[r, bend right=50, looseness=1.3, "f''"', ""{name=f3}]
       \ar[Rightarrow, from=f1, to=f2u, "\alpha"]
       \ar[Rightarrow, from=f2d, to=f3, "\beta"]
       &
       \bullet \ar[r, "g"] & \bullet
      \end{tikzcd}
     \]
     of cells $\alpha$, $\beta$ and $g$ of $A$
     we impose the equality of the $3$-cells
     \[
      F_{\treeV}(g, f'') \comp_1 F_{\treeVRight}(g, \beta \comp_1 \alpha)
     \]
     and
     \[
      F_{\treeVRight}(g, \beta) \comp_1 F_{\treeLL}(g \comp_0 \alpha)\ 
        \comp_2\ 
        F_{\treeLog}(g) \comp_0 F_{\treeLL}(\beta)
        \comp_1 F_{\treeVRight}(g, \alpha)
     \]
     of $B$.
     
     \item[%
        \scalebox{0.3}{
            \begin{forest}
                for tree={%
                    label/.option=content,
                    grow'=north,
                    content=,
                    circle,
                    fill,
                    minimum size=3pt,
                    inner sep=0pt,
                    s sep+=15,
                }
                [
                    [
                        [] []
                    ]
                    []
                ]
            \end{forest}
        }
     ] For any triple
     
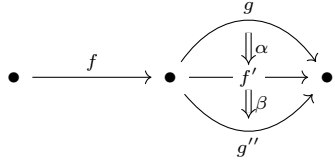
\[
      \begin{tikzcd}[column sep=4.5em]
        \bullet \ar[r, "f"] &
        \bullet
        \ar[r, bend left=55, looseness=1.3, "g", ""{below, name=g1}]
       \ar[r, "f'"{description}, ""{name=g2u}, ""{below, name=g2d}]
       \ar[r, bend right=50, looseness=1.3, "g''"', ""{name=g3}]
       \ar[Rightarrow, from=g1, to=g2u, "\alpha"]
       \ar[Rightarrow, from=g2d, to=g3, "\beta"]
       &
       \bullet
      \end{tikzcd}
     \]
     of cells $\alpha$, $\beta$ and $g$ of $A$ we impose
     that the $3$-cells
     \[
        F_{\treeVLeft}(\beta \comp_1 \alpha, f)
     \]
     and
     \[
        F_{\treeLL}(\beta) \comp_0 F_{\treeLog}(f)
        \comp_1 F_{\treeVLeft}(\alpha, f)\ 
        \comp_2\ 
        F_{\treeVLeft}(\beta, f) \comp_1 F_{\treeLL}(\alpha \comp_0 f)
     \]
     of $B$ to be equal.
     
     \item[%
        \scalebox{0.3}{
            \begin{forest}
                for tree={%
                    label/.option=content,
                    grow'=north,
                    content=,
                    circle,
                    fill,
                    minimum size=3pt,
                    inner sep=0pt,
                    s sep+=15,
                }
                [
                    [ [] ]
                    [ [] ]
                ]
            \end{forest}
        }
     ] Notice first that for any pair
     
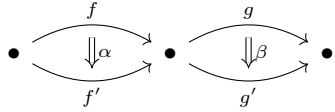
\[
      \begin{tikzcd}[column sep=4.5em]
        \bullet
        \ar[r, bend left, "f", ""{below, name=f1}]
        \ar[r, bend right, "f'"', ""{name=f2}]
        \ar[Rightarrow, from=f1, to=f2, "\alpha"]
        &
        \bullet
        \ar[r, bend left, "g", ""{below, name=g1}]
        \ar[r, bend right, "g'"', ""{name=g2}]
        \ar[Rightarrow, from=g1, to=g2, "\beta"]
        & \bullet
      \end{tikzcd}
     \]
     of $0$-composable $2$-cells $\alpha$ and $\beta$ of $A$,
     we have an equality of $2$-cells
     \[
      \begin{tikzcd}[column sep=-5.5em]
       \null &
       F_{\treeLL}(g' \comp_0 \alpha \comp_1  \beta \comp_0 f)
       \ar[ldd, equal, "{\treeY}"']
       = F_{\treeLL}(\beta \comp_0 f' \comp_1 g \comp_0 \alpha)
       \ar[rdd, equal, "{\treeY}"]
       & \null \\ \\
       F_{\treeLL}(g' \comp_0 \alpha) \comp_1 F_{\treeLL}(\beta \comp_0 f)
       & \null &
       F_{\treeLL}(\beta \comp_0 f') \comp_1 F_{\treeLL}(g \comp_0 \alpha)
      \end{tikzcd}
     \]
     of $B$, where the equality in the higher row is just the exchange law.
     We shall denote by $F_{\text{ex}}(\beta, \alpha)$ the identity $3$-cell
     going from $F_{\treeLL}(g' \comp_0 \alpha) \comp_1 F_{\treeLL}(\beta \comp_0 f)$
     to $F_{\treeLL}(\beta \comp_0 f') \comp_1 F_{\treeLL}(g \comp_0 \alpha)$.

     For any pair of $2$-cells $\alpha$ and $\beta$ of $A$
     as above, we impose the $3$-cells
     \[
        F_{\treeLL}(\beta) \comp_0 F_{\treeLog}(f')
        \comp_1 F_{\treeVRight}(g, \alpha)\ 
        \comp_2\ 
        F_{\treeLog}(g') \comp_0 F_{\treeLL}(\alpha)
        \comp_1 F_{\treeVLeft}(\beta, f)
     \]
     and
     \[
        F_{\treeVLeft}(\beta, f') \comp_1 F_{\treeLL}(g\comp_0 \alpha)\ 
        \comp_2\ 
        F_{\treeV}(g', f') \comp_1 F_{\text{ex}}(\beta, \alpha)\ 
        \comp_2\ 
        F_{\treeVRight}(g', \alpha) \comp_1 F_{\treeLL}(\beta \comp_0 f)
     \]
     of $B$ to be equal. Since $F_{\text{ex}}(\beta, \alpha)$ is a trivial
     $3$-cell, this coherence is actually imposing the equality between
     the $3$-cells
     \[
      F_{\treeLL}(\beta) \comp_0 F_{\treeLog}(f')
        \comp_1 F_{\treeVRight}(g, \alpha)\ 
        \comp_2\ 
        F_{\treeLog}(g') \comp_0 F_{\treeLL}(\alpha)
        \comp_1 F_{\treeVLeft}(\beta, f)
     \]
     and
     \[
        F_{\treeVLeft}(\beta, f') \comp_1 F_{\treeLL}(g\comp_0 \alpha)\ 
        \comp_2\ 
        F_{\treeVRight}(g', \alpha) \comp_1 F_{\treeLL}(\beta \comp_0 f)
     \]
     of $B$.

     \item[%
     \scalebox{0.3}{
     	\begin{forest}
     		for tree={%
     			label/.option=content,
     			grow'=north,
     			content=,
     			circle,
     			fill,
     			minimum size=3pt,
     			inner sep=0pt,
     			s sep+=15,
     		}
     		[
     		[
     		[] [] []
     		]
     		]
     	\end{forest}
     }
     ] For any
     triple
     
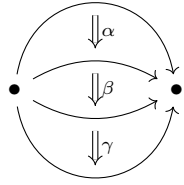
\[
     \begin{tikzcd}[column sep=4.7em]
     \bullet
     \ar[r, bend left=80, looseness=1.6, ""{below, name=1}]
     \ar[r, bend left, ""{name=2u}, ""{below, name=2d}]
     \ar[r, bend right, ""{name=3u}, ""{below, name=3d}]
     \ar[r, bend right=80, looseness=1.6, ""{name=4}]
     \ar[Rightarrow, from=1, to=2u, "\alpha"]
     \ar[Rightarrow, from=2d, to=3u, "\beta"]
     \ar[Rightarrow, from=3d, to=4, "\gamma"] &
     \bullet
     \end{tikzcd}
     \]
     of $1$-composable $2$-cells $\alpha$, $\beta$ and $\gamma$
     of $A$ we have the equalities between the identity $3$-cell
     \[
     F_{\treeY}(\gamma, \beta\comp_1 \alpha) \comp_2
     F_{\treeLL}(\gamma) \comp_1 F_{\treeY}(\beta, \alpha)
     \]
     and the identity $3$-cell
     \[
     F_{\treeY}(\gamma\comp_1 \beta, \alpha) \comp_2
     F_{\treeY}(\gamma, \beta) \comp_1 F_{\treeLL}(\alpha)
     \]
     of $B$.

     \item[%
     \scalebox{0.3}{
     	\begin{forest}
     		for tree={%
     			label/.option=content,
     			grow=north,
     			content=,
     			circle,
     			fill,
     			minimum size=3pt,
     			inner sep=0pt,
     			s sep+=15,
     		}
     		[
     		[
     		[
     		[] []
     		]
     		]
     		]
     	\end{forest}
     }
     ] For any pair
     
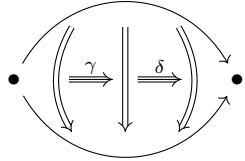
\[
     \begin{tikzcd}[column sep=7em]
     \bullet
     \ar[r, bend left=60, looseness=1.2, "\phantom{bullet}"{below, name=1}]
     \ar[r, bend right=60, looseness=1.2, "\phantom{bullet}"{name=3}]
     \ar[Rightarrow, from=1, to=3, shift right=4ex, bend right, ""{name=beta1}]
     \ar[Rightarrow, from=1, to=3, ""'{name=beta2d}, ""{name=beta2u}]
     \ar[Rightarrow, from=1, to=3, shift left=4ex, bend left, ""'{name=beta3}]
     \arrow[triple, from=beta1, to=beta2d, "\gamma"]{}
     \arrow[triple, from=beta2u, to=beta3, "\delta"]{}
     &
     \bullet
     \end{tikzcd}
     \]
     of $2$-composable $3$-cells $\gamma$ and $\delta$ of $A$ we impose
     the equality
     \[
     F_{\treeLLL}(\delta \comp_2 \gamma) = F_{\treeLLL}(\delta) \comp_2 F_{\treeLLL}(\gamma)
     \]
     between these two $3$-cells of $B$.

     \item[%
     \scalebox{0.3}{
     	\begin{forest}
     		for tree={%
     			label/.option=content,
     			grow=north,
     			content=,
     			circle,
     			fill,
     			minimum size=3pt,
     			inner sep=0pt,
     			s sep+=15,
     		}
     		[
     		[
     		[] [ [] ]
     		]
     		]
     	\end{forest}
     }
     ] For any pair
     \[
     \begin{tikzcd}[column sep=5em]
     \bullet
     \ar[r, bend left=60, looseness=1.2, ""{below, name=1}]
     \ar[r, ""{name=2u}, ""{below, name=2d}]
     \ar[r, bend right=60, looseness=1.2, ""{name=3}]
     \ar[Rightarrow, from=1, to=2u, "\alpha"]
     \ar[Rightarrow, from=2d, to=3, shift right=2.6ex, ""{name=beta1}]
     \ar[Rightarrow, from=2d, to=3, shift left=2.6ex, ""'{name=beta2}]
     \arrow[triple, from=beta1, to=beta2, "\gamma"]{}
     &
     \bullet
     \end{tikzcd}
     \]
     of $1$-composable cells $\alpha$ and $\gamma$ of $A$, we impose that
     the $3$-cells
     \[
     F_{\treeLLL}(\gamma) \comp_1 F_{\treeLL}(\alpha)
     \]
     and
     \[
     F_{\treeLLL}(\gamma \comp_1 \alpha)
     \]
     of $B$ are equal.

     \item[%
     \scalebox{0.3}{
     	\begin{forest}
     		for tree={%
     			label/.option=content,
     			grow=north,
     			content=,
     			circle,
     			fill,
     			minimum size=3pt,
     			inner sep=0pt,
     			s sep+=15,
     		}
     		[
     		[
     		[ [] ] []
     		]
     		]
     	\end{forest}
     }
     ] For any pair
     \[
     \begin{tikzcd}[column sep=5em]
     \bullet
     \ar[r, bend left=60, looseness=1.2, ""{below, name=1}]
     \ar[r, ""{name=2u}, ""{below, name=2d}]
     \ar[r, bend right=60, looseness=1.2, ""{name=3}]
     \ar[Rightarrow, from=2d, to=3, "\beta"]
     \ar[Rightarrow, from=1, to=2u, shift right=2.6ex, ""{name=beta1}]
     \ar[Rightarrow, from=1, to=2u, shift left=2.6ex, ""'{name=beta2}]
     \arrow[triple, from=beta1, to=beta2, "\gamma"]{}
     &
     \bullet
     \end{tikzcd}
     \]
     of $1$-composable cells $\gamma$ and $\beta$ of $A$, we impose that
     the $3$-cells
     \[
     F_{\treeLL}(\beta) \comp_1 F_{\treeLLL}(\gamma)
     \]
     and
     \[
     F_{\treeLLL}(\beta \comp_1 \gamma)
     \]
     of $B$ are equal.

    \item[%
     \scalebox{0.3}{
     	\begin{forest}
     		for tree={%
     			label/.option=content,
     			grow=north,
     			content=,
     			circle,
     			fill,
     			minimum size=3pt,
     			inner sep=0pt,
     			s sep+=15,
     		}
     		[
     		[]
     		[ [ [] ] ]
     		]
     	\end{forest}
     }
     ] For any pair
     \[
     \begin{tikzcd}[column sep=5em]
     \bullet
     \ar[r, "f"]
     &
     \bullet
     \ar[r, bend left=60, looseness=1.2, "g", "\phantom{bullet}"'{name=1}]
     \ar[r, bend right=60, looseness=1.2, "g'"', "\phantom{bullet}"{name=3}]
     \ar[Rightarrow, from=1, to=3, shift right=2ex, bend right, ""{name=beta1}]
     \ar[Rightarrow, from=1, to=3, shift left=2ex, bend left, ""'{name=beta3}]
     \arrow[triple, from=beta1, to=beta3, "\Gamma"]{}
     &
     \bullet
     \end{tikzcd}
     \]
     of $0$-composable cells $f$ and $\Gamma$ of $A$, we impose the equality
     \[
     F_{\treeV}(g', f) \comp_1 F_{\treeLLL}(\Gamma \comp_0 f) = F_{\treeLLL}(\Gamma) \comp_0 F_{\treeLog}(f) \comp_1 F_{\treeV}(g, f)
     \]
     between these two $3$-cells of $B$.

     \item[%
     \scalebox{0.3}{
     	\begin{forest}
     		for tree={%
     			label/.option=content,
     			grow=north,
     			content=,
     			circle,
     			fill,
     			minimum size=3pt,
     			inner sep=0pt,
     			s sep+=15,
     		}
     		[
	     		[
		     		[ [] ]
	     		]
	     		[]
     		]
     	\end{forest}
     }
     ] For any pair
     \[
     \begin{tikzcd}[column sep=5em]
     \bullet
     \ar[r, bend left=60, looseness=1.2, "f", "\phantom{bullet}"'{name=1}]
     \ar[r, bend right=60, looseness=1.2, "f'"', "\phantom{bullet}"{name=3}]
     \ar[Rightarrow, from=1, to=3, shift right=2ex, bend right, ""{name=beta1}]
     \ar[Rightarrow, from=1, to=3, shift left=2ex, bend left, ""'{name=beta3}]
     \arrow[triple, from=beta1, to=beta3, "\Gamma"]{}
     &
     \bullet
     \ar[r, "g"]
     &
     \bullet
     \end{tikzcd}
     \]
     of $0$-composable cells $\Gamma$ and $g$ of $A$, we impose the equality
     \[
	     F_{\treeV}(g, f') \comp_1 F_{\treeLLL}(g \comp_0 \Gamma) = F_{\treeLog}(g) \comp_0 F_{\treeLLL}(\Gamma) \comp_1 F_{\treeV}(g, f)
     \]
     between these two $3$-cells of $B$.     
	\end{description}

	\begin{rem}
		Let us clarify a bit better why we need the data for the tree $\treeY$
		representing the vertical composition of two $2$-cells to be trivial.
		In fact, one would expect that
		for any pair of $1$-composable $2$-cells $x$ and $y$ of $A$,
		a (normalised) oplax $3$-functor would associate a $3$-cell
		\[ F_{\treeY}(x, y) \colon F_{\treeLL}(x \comp_1 y) \to F_{\treeLL}(x) \comp_1 F_{\treeLL}(y). \] 
		
		At the same time, for any pair
		
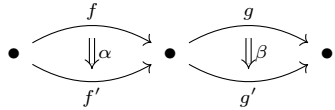
\[
		\begin{tikzcd}[column sep=4.5em]
		\bullet
		\ar[r, bend left, "f", ""{below, name=f1}]
		\ar[r, bend right, "f'"', ""{name=f2}]
		\ar[Rightarrow, from=f1, to=f2, "\alpha"]
		&
		\bullet
		\ar[r, bend left, "g", ""{below, name=g1}]
		\ar[r, bend right, "g'"', ""{name=g2}]
		\ar[Rightarrow, from=g1, to=g2, "\beta"]
		& \bullet
		\end{tikzcd}
		\]
		of $0$-composable $2$-cells $\alpha$ and $\beta$ of $A$,
		the coherence for the tree $\treeVLR$ should express
		a relationship among $F_{\treeVRight}(g\comp_0 \alpha)$,
		$F_{\treeVLeft}(\beta \comp_0 f)$, $F_{\treeVRight}(g'\comp_0 \alpha)$
		and $F_{\treeVLeft}(\beta \comp_0 f')$.
		On the one hand, we can compose
		\[
		 \bigl(F_{\treeLL}(\beta) \comp_0 F_{\treeLog}(f')
		 \comp_1 F_{\treeVRight}(g, \alpha)\bigr)\ 
		 \comp_2\ 
		 \bigl(F_{\treeLog}(g') \comp_0 F_{\treeLL}(\alpha)
		 \comp_1 F_{\treeVLeft}(\beta, f)\bigr),
		\]
		getting a $3$-cell from
		\[
		 F_{\treeL}(g') \comp_0 F_{\treeLL}(\alpha) \comp_1 F_{\treeV}(g', f) \comp_1 F_{\treeLL}(\beta \comp_0 f)
		\]
		to
		\[
		 F_{\treeLL}(\beta)\comp_0 F_{\treeL}(f') \comp_1 F_{\treeV}(g, f') \comp_1 F_{\treeLL}(g \comp_0 \alpha).
		\]
		On the other hand, we have the $3$-cell
		\begin{equation}\label{cell:VR}
		 F_{\treeVRight}(g', \alpha) \comp_1 F_{\treeLL}(\beta \comp_0 f), \tag{A}
		\end{equation}
		with source
		\[
		 F_{\treeL}(g') \comp_0 F_{\treeLL}()\alpha) \comp_1 F_{\treeV}(g', f) \comp_1 F_{\treeLL}(\beta \comp_0 f)
		\]
		and target
		\[
		 F_{\treeV}(g', f') \comp_1  F_{\treeL}(g'\comp_0 \alpha) \comp_1 F_{\treeLL}(\beta \comp_0 f),
		\]
		as well as the $3$-cell
		\begin{equation}\label{cell:VL}
		 F_{\treeVLeft}(\beta, f') \comp_1 F_{\treeLL}(g\comp_0 \alpha), \tag{B}
		\end{equation}
		with source
		\[
		 F_{\treeV}(g', f') \comp_1  F_{\treeL}(\beta\comp_0 f') \comp_1 F_{\treeLL}(g \comp_0 \alpha)
		\]
		and target
		\[
		 F_{\treeLL}(\beta)\comp_0 F_{\treeL}(f') \comp_1 F_{\treeV}(g, f') \comp_1 F_{\treeLL}(g \comp_0 \alpha).
		\]
		The $3$-cells \eqref{cell:VR} and \eqref{cell:VL} are \emph{not} composable,
		since the target of the first one and the source of the second one
		are respectively the bottom left and the bottom right $2$-cells of
		the following diagram
		\[
		\begin{tikzcd}[column sep=-13.5em]
		\null &
		F_{\treeV}(g', f') \comp_1 F_{\treeLL}(g' \comp_0 \alpha \comp_1  \beta \comp_0 f)
		\ar[ldd, triple, "{\treeY}"']
		= F_{\treeV}(g', f') \comp_1 F_{\treeLL}(\beta \comp_0 f' \comp_1 g \comp_0 \alpha)
		\ar[rdd, triple, "{\treeY}"]
		& \null \\ \\
		F_{\treeV}(g', f') \comp_1 F_{\treeLL}(g' \comp_0 \alpha) \comp_1 F_{\treeLL}(\beta \comp_0 f)
		& \null &
		F_{\treeV}(g', f') \comp_1 F_{\treeLL}(\beta \comp_0 f') \comp_1 F_{\treeY}(g \comp_0 \alpha)
		\end{tikzcd}
		\]
		of $B$, where the equality in the higher row is just the exchange law.
		Unless the data for the tree $\treeY$ is trivial, it is impossible to
		provide a relationship among the $3$-cells listed above and involving
		the whiskerings.
	\end{rem}
	
	\begin{rem}
	 Gurski defines in~\cite{GurskiCoherence} a notion of lax trimorphism
	 between tricategories, see Definition~4.11 of~\loccit; one can easily
	 adapt and dualise suitably the definition and get a notion of
	 normalised oplax trimorphism between tricategories.
	 Consider two strict $3$-categories $A$ and $B$, 
	 a normalised oplax $3$-functor $F \colon A \to B$ and
	 a normalised oplax trimorphism $G \colon A \to B$.
	 There are two main
	 differences between Gurski's notion of normalised oplax trimorphism
	 and the notion of normalised oplax $3$-functor that we presented above.
	 \begin{itemize}
	  \item The first difference concerns the tree~$\treeY$. Gurski's
	  notion requires that, for any pair $(\beta, \alpha)$ of $1$-composable
	  $2$-cells of~$A$, there is a $3$-cell of~$B$ going from
	  $G(\beta \comp_1 \alpha)$ to $G(\beta) \comp_1 G(\alpha)$
	  which is \emph{not} invertible in general. It is essential
	  in our definition of normaised oplax $3$-functor that the
	  tree~$\treeY$ appears in the coherences. This is slightly
	  unnatural even from our arboreal point of view and it may be
	  reasonable to actually impose this condition on Gurski's
	  normalised oplax trimorphisms, in light of a comparison with
	  normalised oplax $3$-functors. Gurski calls this condition
      \emph{local strictness}. Notice that it implies in particular
      that~$F$ as well as~$G$ induce strict $2$-functors on the
      hom-$2$-categories.
	  
	  \item The second difference is deeper and somehow irreconcilable.
	  For any pair $(\beta, \alpha)$ of $0$-composable $2$-cells of~$A$,
	  the normalised oplax trimorphism provides a $3$-cell
	  \[
	   G_{\treeVLR}(\beta, \alpha) \colon 
	    G(t_1(\beta), t_1(\alpha)) \comp_1 G(\beta \comp_0 \alpha) \to
	    \bigl(G(\beta) \comp_0 G(\alpha)\bigr) \comp_1 G(s_1(\beta), s_1(\alpha))\,.
	  \]
	  This is incompatible with the algebra of the orientals,
	  as we shall better explain in the following section.
	  In fact, a normalised oplax $3$-functor has the tree~$\treeVLR$
	  as a coherence and instead the trees~$\treeVLeft$ and~$\treeVRight$
	  as part of the data. But these pieces of data are symmetric
	  (or better dual), hence they cannot fit as a particular case
	  of a single lax or oplax datum for~$\treeVLR$.
	 \end{itemize}

	\end{rem}

	\begin{exem}\label{exem:sup}
		Let $C$ be a $3$-category. We now define a normalised oplax $3$-functor
		$\sup \colon i_{\cDelta}(N_{3}(C)) \to C$ (cf.~\hyperref[notation]{Notations and Terminology}).
		\begin{description}
			\item[$\TreeDot\ \ $] The map $\sup_{\treeDot}$ is defined by mapping
			an object $(a, x)$, where $x \colon \On{m} \to C$, to $x(\atom{m})$.
			
			\item[$\TreeLog\ \ \:$] The map $\sup_{\treeLog}$ assigns to any
			morphism $f \colon (a, x) \to (b, y)$
			of $i_{\cDelta}(N_{3}(C))$, where $x \colon \On{m} \to C$ and
			$y \colon \On{n} \to C$, the $1$-cell $x'(\atom{f(m), n})$ of $C$.
			
			\item[$\TreeV\ $] The map $\sup_{\treeV}$ assigns to any
			pair of composable morphisms
			\[
			\begin{tikzcd}
			(a, x) \ar[r, "f"] & (b, y) \ar[r, "g"] &  (c, z)
			\end{tikzcd}
			\]
			of $i_{\cDelta}(N_{3}(C))$,
			with $x \colon \On{m} \to C$, $y \colon \On{n} \to C$ and $z \colon \On{p} \to C$,
			the $2$-cell $\sup_{\treeV}(g, f)$ of $C$ given by
			\[
			z\bigl(\atom{gf(m), g(n), p}\bigr)
			\]
			
			\item[$\TreeW\ $] The map $\sup_{\treeW}$,
			assigns to any triple of composable morphisms
			\[
			\begin{tikzcd}[column sep=small]
			(a,x) \ar[r, "f"] & (b, y) \ar[r, "g"] & (c, z) \ar[r, "h"] & (d, t)
			\end{tikzcd}
			\]
			of $i_{\cDelta}(N_{3}(C))$,
			with $x \colon \On{m} \to C$, $y \colon \On{n} \to C$, $z \colon \On{p} \to C$ and
			$t \colon \On{q} \to C$,
			the $3$-cell $\sup_{\treeW}(h, g, f)$ of $C$
			given by
			\[
			t\bigl(\atom{hgf(m), hg(n), h(p), q}\bigr)
			\]
		\end{description}
		
		Notice that by definition we have that $1_{\sup_{\treeDot}(a, x)}$
		is precisely $\sup_{\treeL}(1_{(a, x)})$ and that
		the other conditions of normalisation are equally trivial by definition.
		We now check the
		coherence for the tree $\treeVV$.
		
		Consider four composable morphisms of~$i_{\cDelta}(N_{3}(C))$
		\[
		\begin{tikzcd}[column sep=small]
		(a, x) \ar[r, "f"] &
		(b, y) \ar[r, "g"] &
		(c, z) \ar[r, "h"] &
		(d, t) \ar[r, "i"] &
		(e, w)
		\end{tikzcd}\ ,
		\]
		with $x \colon \On{m} \to C$, $y \colon \On{n} \to C$, $z \colon \On{p} \to C$,
		$t \colon \On{q} \to C$ and $w \colon \On{r} \to C$.
		We have to show that the $3$-cells
		\begin{gather*}
		\textstyle
		\sup_{\treeW}(i, h, g) \comp_0 \sup_{\treeLog}(f) \comp_1\sup_{\treeV}(i\comp_0 h \comp_0 g, f)\\
		\comp_2\\
		\textstyle
		\sup_{\treeLog}(i) \comp_0 \sup_{\treeV}(h, g) \comp_0 \sup_{\treeLog}(f) \comp_1
		\sup_{\treeW}(i, h \comp_0 g, f)\\
		\comp_2 \\
		\textstyle
		\sup_{\treeLog}(i) \comp_0 \sup_{\treeW}(h, g, f) \comp_1 \sup_{\treeV}(i, h\comp_0 g \comp_0 f)
		\end{gather*}
		and
		\begin{gather*}
		\textstyle
		\sup_{\treeV}(i, h) \comp_0 \sup_{\treeLog}(g) \comp_0 \sup_{\treeLog}(f)
		\comp_1 \sup_{\treeW}(ih, g, f) \\
		\comp_2\\
		\textstyle
		\sup_{\treeLog}(i) \comp_0 \sup_{\treeLog}(h) \comp_0 \sup_{\treeV}(g, f)
		\comp_1 \sup_{\treeW}(i, h, gf)\,.
		\end{gather*}
		But these two $3$-cells are precisely the target and the source of
		the main $4$-cell of $\On{4}$ via the \oo-functor $\On{4} \xto{\phi} \On{r} \xto{w} C$,
		where
		\[
		\phi = \On{\{ihgf(m), ihg(n), ih(p), i(q), r\}}\,.
		\]
	\end{exem}
\end{paragr}

\subsection{From cellular to simplicial}\label{section:cellular-to-simplicial}

In this subsection we shall show that to any normalised oplax $3$-functor
$F \colon B \to C$ there is a canonically associated
morphism $\Nl(F) \colon \SN(B) \to \SN(C)$ of simplicial sets.
Throughout this section we shall sometimes
write oplax $3$-functor for normalised oplax $3$-functor,
since no confusion is possible.

\begin{paragr}\label{paragr:def_cellular_to_simplicial}
 Let $A$ be a $1$-category and $B$ and $C$ be two $3$\nbd-cat\-egories.
 Fix two oplax $3$\nbd-func\-tors $F \colon A \to B$ and $G \colon B \to C$.
 We now define a candidate $GF$ for the composition of $F$ and $G$
 and we dedicate the rest of the subsection to prove the coherences.
 
 Since $A$ is a $1$-category, the amount of data
 that we have to provide in order to define an
 oplax $3$-functor, \ie the trees of dimension
 less than~$3$, is limited to the trees $\TreeDot$, $\TreeLog$, $\TreeV$
 and $\TreeW$. The $3$-functor $GF$ is defined as follows:
 
	\begin{description}
		\item[%
		\scalebox{0.3}{
			\begin{forest}
				for tree={%
					label/.option=content,
					grow'=north,
					content=,
					circle,
					fill,
					minimum size=3pt,
					inner sep=0pt,
					s sep+=15,
				}
				[ ]
			\end{forest}
		}
		] The map $GF_{\treeDot}$ assigns
		to any object $a$ of $A$ the object $G_{\treeDot}(F_{\treeDot}(a))$ of $C$.
		
		\item[%
		\scalebox{0.3}{
			\begin{forest}
				for tree={%
					label/.option=content,
					grow'=north,
					content=,
					circle,
					fill,
					minimum size=3pt,
					inner sep=0pt,
					s sep+=15,
				}
				[ [] ]
			\end{forest}
		}
		] The map $GF_{\treeLog}$ assigns to any $1$-cell $f \colon a \to a'$
		of $A$, \ie any tree of $A_{\treeLog}$,
		the $1$-cell
		\[G_{\treeLog}(F_{\treeLog}(f)) \colon GF_{\treeDot}(a) \to GF_{\treeDot}(a')\]
		of $C$.
		
		\item[%
		\scalebox{0.3}{
			\begin{forest}
				for tree={%
					label/.option=content,
					grow'=north,
					content=,
					circle,
					fill,
					minimum size=3pt,
					inner sep=0pt,
					s sep+=15,
				}
				[ [][] ]
			\end{forest}
		}
		] The map $GF_{\treeV}$ assigns to any
		pair of $0$\hyp{}composable $1$\hyp{}cells
			\[
			 \begin{tikzcd}
			  a \ar[r, "f"] & a' \ar[r, "g"] &  a''
			 \end{tikzcd}
			\]
			of $A$ the $2$-cell $GF_{\treeV}(g, f)$
			\[
			 \begin{tikzcd}[row sep=1.35em]
                a^{\phantom\prime}
                \ar[rr, bend left=75, "{GF_{\treeLog}(gf)}", ""'{name=f}]
                \ar[rr, "{G_{\treeLog}(F_{\treeLog}(g)F_{\treeLog}(f))}"{description, name=g}]
                \ar[rd, bend right, "{GF_{\treeLog}(f)}"']
                &&
                    a''
                \\
                & a'
                    \ar[ru, bend right, "{GF_{\treeLog}(g)}"']
                &
                \ar[Rightarrow, from=f, to=g, shorten <=1mm, shorten >=2mm, "\alpha"]
                \ar[Rightarrow, from=g, to=2-2, shorten <=1mm, pos=0.4, "\beta" near end]
            \end{tikzcd}
			\]
			of $C$, where $\alpha = G_{\treeLL}(F_{\treeV}(g, f))$
			and $\beta = G_{\treeV}(F_{\treeLog}(g), F_{\treeLog}(f))$, that is
			\[
			 G_{\treeV}(F_{\treeLog}(g), F_{\treeLog}(f)) \comp_1 G_{\treeLL}(F_{\treeV}(g, f))\,.
			\]

		\item[%
		\scalebox{0.3}{
			\begin{forest}
				for tree={%
					label/.option=content,
					grow'=north,
					content=,
					circle,
					fill,
					minimum size=3pt,
					inner sep=0pt,
					s sep+=15,
				}
				[ [][][] ]
			\end{forest}
		}
		] The map $GF_{\treeW}$,
		assigns to any triple of $0$-composable $1$ cells
		\[
		 \begin{tikzcd}[column sep=small]
		  a \ar[r, "f"] & a' \ar[r, "g"] & a'' \ar[r, "h"] & a'''
		 \end{tikzcd}
		\]
		of $A$ the $3$-cell $GF_{\treeW}(h, g, f)$
		defined as
		\begin{gather*}
			G_{\treeV}(F_{\treeLog}h, F_{\treeLog}g) \comp_0 GF_{\treeLog}(f) \comp_1
			G_{\treeVLeft}(F_{\treeV}(h, g), F_{\treeLog}f) \comp_1
			G_{\treeLL}(F_{\treeV}(hg, f)\\
			\comp_2\\
			G_{\treeW}(F_{\treeLog}h, F_{\treeLog}g, F_{\treeLog}f) \comp_1
			G_{\treeLLL}(F_{\treeW}(h, g, f)) \\
			\comp_2\\
			GF_{\treeLog}(h) \comp_0 G_{\treeV}(F_{\treeLog}g, F_{\treeLog}f) \comp_1
			G_{\treeVRight}(F_{\treeLog}h, F_{\treeV}(g, f)) \comp_1
			G_{\treeLL}(F_{\treeV}(h, gf))\,.
		\end{gather*}
		Notice that for the $3$-cell $G_{\treeLLL}(F_{\treeW}(h, g, f))$ in the second line
		we are implicitly using the coherence $G_{\treeY}$, as we use the equality
		\[
		 G_{\treeLL}(F_{\treeLog}h \comp_0 F_{\treeV}(g, f)) \comp_1
		 G_{\treeLL}(F_{\treeV}(h, gf)) =
		 G_{\treeLL}(F_{\treeLog}h \comp_0 F_{\treeV}(g, f) \comp_1 F_{\treeV}(h, gf))
		\]
		as well as the equality
		\[
		 G_{\treeLL}(  F_{\treeV}(h, g) \comp_0 F_{\treeLog}f) \comp_1
		 G_{\treeLL}(F_{\treeV}(hg, f)) =
		 G_{\treeLL}(F_{\treeV}(h, g) \comp_0 F_{\treeLog}f \comp_1 F_{\treeV}(hg, f))
		\]
		of $2$-cells of $A$, which are respectively the source and target of the $3$-cell
		$G_{\treeLLL}(F_{\treeW}(h, g, f))$.
	\end{description}
\end{paragr}

\begin{paragr}
	The conditions of normalisation for the composite $GF$ are tedious but straightforward.
	We give here a few examples and we leave the other similar verifications to the reader.
	
	\begin{description}
	\item[%
	\scalebox{0.3}{
		\begin{forest}
			for tree={%
				label/.option=content,
				grow'=north,
				content=,
				circle,
				fill,
				minimum size=3pt,
				inner sep=0pt,
				s sep+=15,
			}
			[ [] ] 
		\end{forest}
	}
	] For any $0$-cell $a$ of $A$ we have
	\[
	GF_{\treeLL}(1_a) = s^0_1(GF_{\treeL}(a)) = 1_{GF_{\treeL}(a)}\,.
	\]
	
	\item[%
	\scalebox{0.3}{
		\begin{forest}
			for tree={%
				label/.option=content,
				grow'=north,
				content=,
				circle,
				fill,
				minimum size=3pt,
				inner sep=0pt,
				s sep+=15,
			}
			[ [][] ] 
		\end{forest}
	}
	] For any $1$-cell $f \colon a \to a'$ of $A$, we have
	\[
	 G_{\treeV}\bigl(F_{\treeL}(f), F_{\treeL}(1_a)\bigr) = 
	 G_{\treeV}\bigl(F_{\treeL}(f), 1_{F_{\treeDot}(a)}\bigr) = 1_{G_{\treeL}(F_{\treeL}(f))}
	 = 1_{GF_{\treeL}(f)}
	\]
	and also
	\[
	 G_{\treeLL}\bigl(F_{\treeV}(f, 1_a)\bigr) =
	 G_{\treeLL}(1_{F_{\treeL}(f)}) = 1_{G_{\treeL}(F_{\treeL}(f))}
	 = 1_{GF_{\treeL}(f)}\,,
	\]
	so that
	\[
	 GF_{\treeV}(f, 1_a) = 1_{GF_{\treeL}(f)} \comp_1 1_{GF_{\treeL}(f)} = 1_{GF_{\treeL}(f)}\,.
	\]
	
	\item[%
	\scalebox{0.3}{
		\begin{forest}
			for tree={%
				label/.option=content,
				grow'=north,
				content=,
				circle,
				fill,
				minimum size=3pt,
				inner sep=0pt,
				s sep+=15,
			}
			[ [][][] ] 
		\end{forest}
	}
	] For any pair $a \xto{f} a' \xto{g} a''$ of $1$-cells of $A$, we have
	\[
	 G_{\treeVLeft}\bigl(F_{\treeV}(g, f), F_{\treeL}(1_a)\bigr) = 
	 G_{\treeVLeft}\bigl(F_{\treeV}(g, f), 1_{F_{\treeDot}(a)}\bigr) =
	 1_{G_{\treeLL}(F_{\treeV}(g, f))}
	\]
	and
	\[
	 G_{\treeW}\bigl(F_{\treeL}(g), F_{\treeL}(f), F_{\treeL}(1_a)\bigr) =
	 G_{\treeW}\bigl(F_{\treeL}(g), F_{\treeL}(f), 1_{F_{\treeDot}(a)}\bigr) =
	 1_{G_{\treeLL}(F_{\treeL}(g), F_{\treeL}(f))}
	\]
	and
	\[
	 G_{\treeLLL}\bigl(F_{\treeW}(g, f, 1_a)\bigr) = 
	 G_{\treeLLL}(1_{F_{\treeV}(g, f)}) = 1_{G_{\treeLL}(F_{\treeV}(g, f))}
	\]
	and
	\[
	 G_{\treeVRight}\bigl(F_{\treeL}(g), F_{\treeV}(f, 1_a)\bigr) =
	 G_{\treeVRight}\bigl(F_{\treeL}(g), 1_{F_{\treeL}(f)}\bigr) =
	 1_{F_{\treeV}(g, f)}\,.
	\]
	Hence, we get that
	\[
	 GF_{\treeW}(g, f, 1_a) =
	 1_{G_{\treeLL}(F_{\treeL}(g), F_{\treeL}(f))} \comp_1 1_{G_{\treeLL}(F_{\treeV}(g, f))}
	 = 1_{GF_{\treeV}(g, f)}\,.
	\]
	\end{description}
\end{paragr}

\begin{paragr}\label{paragr:pentagon_coherence}
	Since $A$ is a $1$-category, the only coherence we have to prove is the
	coherence associated to the tree
	\scalebox{0.3}{
		\begin{forest}
			for tree={%
				label/.option=content,
				grow'=north,
				content=,
				circle,
				fill,
				minimum size=3pt,
				inner sep=0pt,
				s sep+=15,
			}
			[
			[] [] [] []
			]
		\end{forest}
	}.
	Consider four composable $1$\nbd-cells of~$A$
	\[
	\begin{tikzcd}[column sep=small]
	 \bullet \ar[r, "f"] &
	 \bullet \ar[r, "g"] &
	 \bullet \ar[r, "h"] &
	 \bullet \ar[r, "i"] &
	 \bullet
	\end{tikzcd}\ .
	\]
	We have to show that the $3$-cells
		 \begin{gather*}
		    GF_{\treeW}(i, h, g) \comp_0 GF_{\treeLog}(f) \comp_1 GF_{\treeV}(i\comp_0 h \comp_0 g, f)\\
		     \comp_2\\
		     GF_{\treeLog}(i) \comp_0 GF_{\treeV}(h, g) \comp_0 GF_{\treeLog}(f) \comp_1
		     GF_{\treeW}(i, h \comp_0 g, f)\\
		     \comp_2 \\
		     GF_{\treeLog}(i) \comp_0 GF_{\treeW}(h, g, f) \comp_1 GF_{\treeV}(i, h\comp_0 g \comp_0 f)
	     \end{gather*}
	     and
	     \begin{gather*}
	     GF_{\treeV}(i, h) \comp_0 GF_{\treeLog}(g) \comp_0 GF_{\treeLog}(f)
	     \comp_1 GF_{\treeW}(ih, g, f) \\
	     \comp_2\\
	     GF_{\treeLog}(i) \comp_0 GF_{\treeLog}(h) \comp_0 GF_{\treeV}(g, f)
	     \comp_1 GF_{\treeW}(i, h, gf)
	     \end{gather*}
	     of $C$ are equal. The five $3$-cells involved in this compositions are:
	     \begin{enumerate}
	     	\item\label{item:h-g-f} the $3$-cell $GF_{\treeW}(h, g, f)$ of $C$, which is defined as
	     	\begin{gather*}
	     	G_{\treeV}(F_{\treeLog}h, F_{\treeLog}g) \comp_0 GF_{\treeLog}(f) \comp_1
	     	G_{\treeVLeft}(F_{\treeV}(h, g), F_{\treeLog}f) \comp_1
	     	G_{\treeLL}(F_{\treeV}(hg, f))\\
	     	\comp_2\\
	     	G_{\treeW}(F_{\treeLog}h, F_{\treeLog}g, F_{\treeLog}f) \comp_1
	     	G_{\treeLLL}(F_{\treeW}(h, g, f)) \\
	     	\comp_2\\
	     	GF_{\treeLog}(h) \comp_0 G_{\treeV}(F_{\treeLog}g, F_{\treeLog}f) \comp_1
	     	G_{\treeVRight}(F_{\treeLog}h, F_{\treeV}(g, f)) \comp_1
	     	G_{\treeLL}(F_{\treeV}(h, gf))\,,
	     	\end{gather*}
	     	which is a suitably whiskered $1$-composition of the $3$-cells
	     	\begin{enumerate}
	     		\item\label{item:h-g-f-1} $G_{\treeVRight}(F_{\treeLog}h, F_{\treeV}(g, f))$,
	     		\item\label{item:h-g-f-2} $G_{\treeW}(F_{\treeLog}h, F_{\treeLog}g, F_{\treeLog}f) \comp_1
	     		G_{\treeLLL}(F_{\treeW}(h, g, f))$,
	     		\item\label{item:h-g-f-3} $G_{\treeVLeft}(F_{\treeV}(h, g), F_{\treeLog}f)$,
	     	\end{enumerate}
	     	of $C$;
	     	
	     	\item\label{item:i-hg-f} the $3$-cell $GF_{\treeW}(i, h \comp_0 g, f)$ of $C$, which is defined as
	     	\begin{gather*}
	     	G_{\treeV}(F_{\treeLog}i, F_{\treeLog}hg) \comp_0 GF_{\treeLog}(f) \comp_1
	     	G_{\treeVLeft}(F_{\treeV}(i, hg), F_{\treeLog}f) \comp_1
	     	G_{\treeLL}(F_{\treeV}(ihg, f))\\
	     	\comp_2\\
	     	G_{\treeW}(F_{\treeLog}i, F_{\treeLog}hg, F_{\treeLog}f) \comp_1
	     	G_{\treeLLL}(F_{\treeW}(i, hg, f)) \\
	     	\comp_2\\
	     	GF_{\treeLog}(i) \comp_0 G_{\treeV}(F_{\treeLog}hg, F_{\treeLog}f) \comp_1
	     	G_{\treeVRight}(F_{\treeLog}i, F_{\treeV}(hg, f)) \comp_1
	     	G_{\treeLL}(F_{\treeV}(i, hgf))\,,
	     	\end{gather*}
	     	which is a suitably whiskered $1$-composition of the $3$-cells
	     	\begin{enumerate}
	     		\item\label{item:i-hg-f-1}\label{item:b1} $G_{\treeVRight}(F_{\treeLog}i, F_{\treeV}(hg, f))$,
	     		\item\label{item:i-hg-f-2}\label{item:b2} $G_{\treeW}(F_{\treeLog}i, F_{\treeLog}hg, F_{\treeLog}f) \comp_1
	     		G_{\treeLLL}(F_{\treeW}(i, hg, f))$,
	     		\item\label{item:i-hg-f-3}\label{item:b3} $G_{\treeVLeft}(F_{\treeV}(i, hg), F_{\treeLog}f)$,
	     	\end{enumerate}
	     	of $C$;
	     	
	     	\item\label{item:i-h-g} the $3$-cell $GF_{\treeW}(i, h, g)$ of $C$, which is defined as
	     	\begin{gather*}
	     	G_{\treeV}(F_{\treeLog}i, F_{\treeLog}h) \comp_0 GF_{\treeLog}(g) \comp_1
	     	G_{\treeVLeft}(F_{\treeV}(i, h), F_{\treeLog}g) \comp_1
	     	G_{\treeLL}(F_{\treeV}(ih, g))\\
	     	\comp_2\\
	     	G_{\treeW}(F_{\treeLog}i, F_{\treeLog}h, F_{\treeLog}g) \comp_1
	     	G_{\treeLLL}(F_{\treeW}(i, h, g)) \\
	     	\comp_2\\
	     	GF_{\treeLog}(i) \comp_0 G_{\treeV}(F_{\treeLog}h, F_{\treeLog}g) \comp_1
	     	G_{\treeVRight}(F_{\treeLog}i, F_{\treeV}(h, g)) \comp_1
	     	G_{\treeLL}(F_{\treeV}(i, hg))\,,
	     	\end{gather*}
	     	which is a suitably whiskered $1$-composition of the $3$-cells
	     	\begin{enumerate}
	     		\item\label{item:i-h-g-1} $G_{\treeVRight}(F_{\treeLog}i, F_{\treeV}(h, g))$,
	     		\item\label{item:i-h-g-2} $G_{\treeW}(F_{\treeLog}i, F_{\treeLog}h, F_{\treeLog}g) \comp_1
	     		G_{\treeLLL}(F_{\treeW}(i, h, g))$,
	     		\item\label{item:i-h-g-3} $G_{\treeVLeft}(F_{\treeV}(i, h), F_{\treeLog}g)$,
	     	\end{enumerate}
	     	of $C$;
	     	
	     	\item\label{item:i-h-gf} the $3$-cell $GF_{\treeW}(i, h, g)$ of $C$, which is defined as
	     	\begin{gather*}
	     	G_{\treeV}(F_{\treeLog}i, F_{\treeLog}h) \comp_0 GF_{\treeLog}(gf) \comp_1
	     	G_{\treeVLeft}(F_{\treeV}(i, h), F_{\treeLog}gf) \comp_1
	     	G_{\treeLL}(F_{\treeV}(ih, gf))\\
	     	\comp_2\\
	     	G_{\treeW}(F_{\treeLog}i, F_{\treeLog}h, F_{\treeLog}gf) \comp_1
	     	G_{\treeLLL}(F_{\treeW}(i, h, gf)) \\
	     	\comp_2\\
	     	GF_{\treeLog}(i) \comp_0 G_{\treeV}(F_{\treeLog}h, F_{\treeLog}gf) \comp_1
	     	G_{\treeVRight}(F_{\treeLog}i, F_{\treeV}(h, gf)) \comp_1
	     	G_{\treeLL}(F_{\treeV}(i, hgf))\,,
	     	\end{gather*}
	     	which is a suitably whiskered $1$-composition of the $3$-cells
	     	\begin{enumerate}
	     		\item\label{item:i-h-gf-1} $G_{\treeVRight}(F_{\treeLog}i, F_{\treeV}(h, gf))$,
	     		\item\label{item:i-h-gf-2} $G_{\treeW}(F_{\treeLog}i, F_{\treeLog}h, F_{\treeLog}gf) \comp_1
	     		G_{\treeLLL}(F_{\treeW}(i, h, gf))$,
	     		\item\label{item:i-h-gf-3} $G_{\treeVLeft}(F_{\treeV}(i, h), F_{\treeLog}gf)$,
	     	\end{enumerate}
	     	of $C$;
	     	
	     	\item\label{item:ih-g-f} the $3$-cell $GF_{\treeW}(ih, g, f)$ of $C$, which is defined as
	     	\begin{gather*}
	     	G_{\treeV}(F_{\treeLog}ih, F_{\treeLog}g) \comp_0 GF_{\treeLog}(f) \comp_1
	     	G_{\treeVLeft}(F_{\treeV}(ih, g), F_{\treeLog}f) \comp_1
	     	G_{\treeLL}(F_{\treeV}(ihg, f))\\
	     	\comp_2\\
	     	G_{\treeW}(F_{\treeLog}ih, F_{\treeLog}g, F_{\treeLog}f) \comp_1
	     	G_{\treeLLL}(F_{\treeW}i(h, g, f)) \\
	     	\comp_2\\
	     	GF_{\treeLog}(ih) \comp_0 G_{\treeV}(F_{\treeLog}g, F_{\treeLog}f) \comp_1
	     	G_{\treeVRight}(F_{\treeLog}ih, F_{\treeV}(g, f)) \comp_1
	     	G_{\treeLL}(F_{\treeV}(ih, gf))\,,
	     	\end{gather*}
	     	which is a suitably whiskered $1$-composition of the $3$-cells
	     	\begin{enumerate}
	     		\item\label{item:ih-g-f-1} $G_{\treeVRight}(F_{\treeLog}ih, F_{\treeV}(g, f))$,
	     		\item\label{item:ih-g-f-2} $G_{\treeW}(F_{\treeLog}ih, F_{\treeLog}g, F_{\treeLog}f) \comp_1
	     		G_{\treeLLL}(F_{\treeW}(ih, g, f))$,
	     		\item\label{item:ih-g-f-3} $G_{\treeVLeft}(F_{\treeV}(ih, g), F_{\treeLog}f)$,
	     	\end{enumerate}
	     	of $C$.
	     \end{enumerate}
	     
	     In summary, we have to show that the pentagon
	     
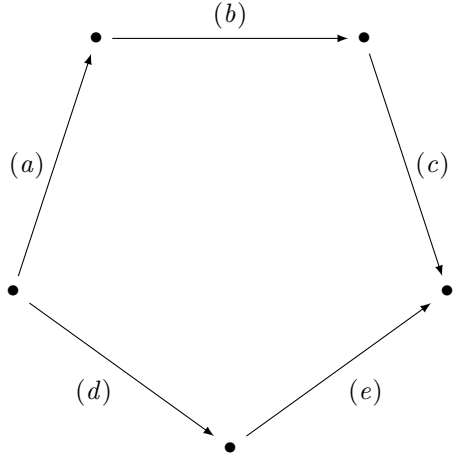
\begin{center}
		  \begin{tikzpicture}[scale=1.5]
		   \foreach \i in {0,1,2,3,4} {
		   	\tikzmath{\a = 270 - (72 * \i);}
		   	\node (\i) at (\a:2) {$\bullet$};
		   }
		   \draw[->, >=latex] (1) -- node [left] {\ref{item:h-g-f}} (2);
		   \draw[->, >=latex] (2) -- node [above] {\ref{item:i-hg-f}} (3);
		   \draw[->, >=latex] (3) -- node [right] {\ref{item:i-h-g}} (4);
		   \draw[->, >=latex] (1) -- node [below left] {\ref{item:i-h-gf}} (0);
		   \draw[->, >=latex] (0) -- node [below right] {\ref{item:ih-g-f}} (4);
		  \end{tikzpicture}
	     \end{center}
	     of $2$-compositions of $3$-cells of $C$ is commutative. Using the decomposition of each of these
	     $3$-cells as suitably whiskered $1$-composition of other $3$-cells of $C$,
	     we have to show that the following diagram
	     
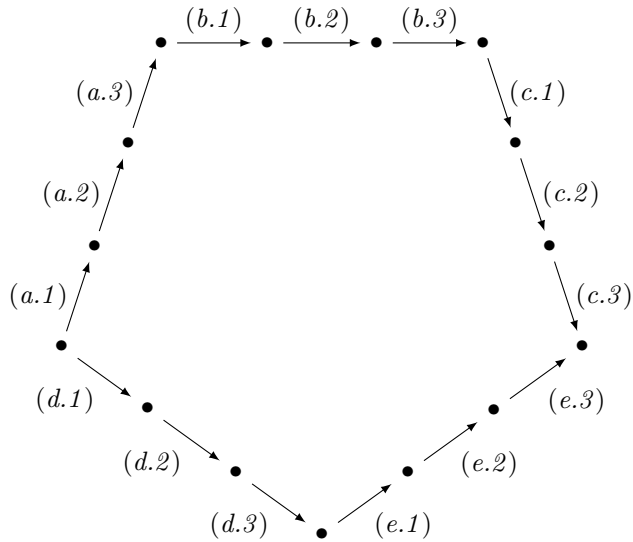
\begin{center}
	     	\begin{tikzpicture}[scale=1.2]
	     	\foreach \i in {0,1,2,3,4} {
	     		\tikzmath{\a = 270 - (72 * \i);}
	     		\node (e\i) at (\a:3) {$\bullet$};
	     	}
	     	\foreach \j in {1,2,3,4} {
	     		\tikzmath{\a = 270 - (72 * \j);}
	     		\pgfmathtruncatemacro\bj{\j-1}
	     		\node (e\j-a) at ($(e\bj)!0.33!(e\j)$) {$\bullet$};
	     		\node (e\j-b) at ($(e\bj)!0.67!(e\j)$) {$\bullet$};
	     	}
	     	\node (e0-a) at ($(e0)!0.33!(e4)$) {$\bullet$};
	     	\node (e0-b) at ($(e0)!0.66!(e4)$) {$\bullet$};
	     	\draw[->, >=latex] (e1) to node [left] {\ref{item:h-g-f-1}} (e2-a);
	     	\draw[->, >=latex] (e2-a) to node [left] {\ref{item:h-g-f-2}} (e2-b);
	     	\draw[->, >=latex] (e2-b) to node [left] {\ref{item:h-g-f-3}} (e2);
	     	\draw[->, >=latex] (e2) to node [above] {\ref{item:i-hg-f-1}} (e3-a);
	     	\draw[->, >=latex] (e3-a) to node [above] {\ref{item:i-hg-f-2}} (e3-b);
	     	\draw[->, >=latex] (e3-b) to node [above] {\ref{item:i-hg-f-3}} (e3);
	     	\draw[->, >=latex] (e3) to node [right] {\ref{item:i-h-g-1}} (e4-a);
	     	\draw[->, >=latex] (e4-a) to node [right] {\ref{item:i-h-g-2}} (e4-b);
	     	\draw[->, >=latex] (e4-b) to node [right] {\ref{item:i-h-g-3}} (e4);
	     	\draw[->, >=latex] (e1) to node [below left] {\ref{item:i-h-gf-1}} (e1-b);
	     	\draw[->, >=latex] (e1-b) to node [below left] {\ref{item:i-h-gf-2}} (e1-a);
	     	\draw[->, >=latex] (e1-a) to node [below left] {\ref{item:i-h-gf-3}} (e0);
	     	\draw[->, >=latex] (e0) to node [below right] {\ref{item:ih-g-f-1}} (e0-a);
	     	\draw[->, >=latex] (e0-a) to node [below right] {\ref{item:ih-g-f-2}} (e0-b);
	     	\draw[->, >=latex] (e0-b) to node [below right] {\ref{item:ih-g-f-3}} (e4);
	      \end{tikzpicture}
	     \end{center}
	     of $2$-compositions of $3$-cells of $C$ commutes;
	     notice that in the latter diagram the referenced $3$-cells of $C$ are not
	     $2$-composable: we are making the abuse of denoting each arrow of the diagram
	     with the reference to a particular $3$-cell, without the suitable whiskerings
	     making all these $3$-cells $2$-composable. These whiskerings are written explicitly
	     above.
\end{paragr}

\begin{paragr}
		In order to show that the diagram of the previous paragraph is commutative,
		we shall decompose it in several smaller diagrams and we shall show that each
		of them is commutative. This decomposition is displayed in figure~\ref{fig:diagram_composition}.
		There is a duality involving the diagram numbered with ($n$) and with ($n'$)
		and the one commutes if and only if the other one does.
		We shall illustrate this phenomenon with the diagrams (1) and (1') and (2) and (2'),
		but then we will limit ourself to prove the commutativity of the diagrams of the type
		($n$), leaving the diagrams of type ($n'$) to the reader.
\end{paragr}
\begin{figure}
	\centering
	\begin{tikzpicture}[scale=2]
	\foreach \i in {0,1,2,3,4} {
		\tikzmath{\a = 270 - (72 * \i);}
		\node (i\i) at (\a:1) {i\i};
		\node (e\i) at (\a:3) {e\i};
	}
	\foreach \j in {1,2,3,4} {
		\tikzmath{\a = 270 - (72 * \j);}
		\pgfmathtruncatemacro\bj{\j-1}
		\node (m\j) at (\a:2) {m\j};
		\node (e\j-a) at ($(e\bj)!0.33!(e\j)$) {e\j-a};
		\node (e\j-b) at ($(e\bj)!0.67!(e\j)$) {e\j-b};
	}
	\node (e0-a) at ($(e0)!0.33!(e4)$) {e0-a};
	\node (e0-b) at ($(e0)!0.66!(e4)$) {e0-b};
	%
	%
	\draw[->, >=latex] (e1) to (e2-a);
	\draw[->, >=latex] (e2-a) to (m1);
	\draw[->, >=latex] (e1) to (e1-b);
	\draw[->, >=latex] (e1-b) to (m1);
	\node at ($(e1)!0.5!(m1)$) {(1)};
	\draw[->, >=latex] (m4) to (e4-b);
	\draw[->, >=latex] (e4-b) to (e4);
	\draw[->, >=latex] (m4) to (e0-b);
	\draw[->, >=latex] (e0-b) to (e4);
	\node at ($(m4)!0.5!(e4)$) {(1')};
	%
	%
	\draw[->, >=latex] (m1) to (i1);
	\draw[->, >=latex] (e2-a) to (i1);
	\node at (190:1.8) {(2)};
	\draw[->, >=latex] (i4) to (m4);
	\draw[->, >=latex] (i4) to (e4-b);
	\node at (-10:1.8) {(2')};
	%
	%
	\draw[->, >=latex] (e2-a) to (e2-b);
	\draw[->, >=latex] (e2-b) to (i2);
	\draw[->, >=latex] (i1) to (i2);
	\node at ($(e2-a)!0.5!(i2)$) {(3)};
	\draw[->, >=latex] (e4-a) to (e4-b);
	\draw[->, >=latex] (i3) to (e4-a);
	\draw[->, >=latex] (i3) to (i4);
	\node at ($(i3)!0.5!(e4-b)$) {(3')};
	%
	%
	\draw[->, >=latex] (e2-b) to (m2);
	\draw[->, >=latex] (m2) to (i2);
	\node at (134:1.8) {(4)};
	\draw[->, >=latex] (i3) to (m3);
	\draw[->, >=latex] (m3) to (e4-a);
	\node at (46:1.8) {(4')};
	%
	%
	\draw[->, >=latex] (e2-b) to (e2);
	\draw[->, >=latex] (e2) to (e3-a);
	\draw[->, >=latex] (m2) to (e3-a);
	\node at ($(m2)!0.5!(e2)$) {(5)};
	\draw[->, >=latex] (e3-b) to (e3);
	\draw[->, >=latex] (e3) to (e4-a);
	\draw[->, >=latex] (e3-b) to (m3);
	\node at ($(m3)!0.5!(e3)$) {(5')};
	%
	%
	\draw[->, >=latex] (e3-a) to (e3-b);
	\draw[->, >=latex] (i2) to (i3);
	\node at ($(m2)!0.5!(m3)$) {(6)};
	%
	%
	\draw[->, >=latex] (i1) to (i0);
	\draw[->, >=latex] (e1-b) to (e1-a);
	\draw[->, >=latex] (e1-a) to (i0);
	\node at ($(e1-b)!0.5!(i0)$) {(7)};
	\draw[->, >=latex] (i0) to (i4);
	\draw[->, >=latex] (i0) to (e0-a);
	\draw[->, >=latex] (e0-a) to (e0-b);
	\node at ($(i0)!0.5!(e0-b)$) {(7')};
	%
	%
	\draw[->, >=latex] (e1-a) to (e0);
	\draw[->, >=latex] (e0) to (e0-a);
	\node at ($(i0)!0.6!(e0)$) {(8)};
	%
	%
	\node at (0:0) {(9)};
	\end{tikzpicture}
	\caption{The diagram for the coherence $\protect\treeVV$.}
	\label{fig:diagram_composition}
\end{figure}

\begin{paragr}[1]
	Consider the diagram (1), where again we abuse of notation by forgetting
	the suitable whiskerings making these $3$-cells of $C$ actually $2$-composable:
	
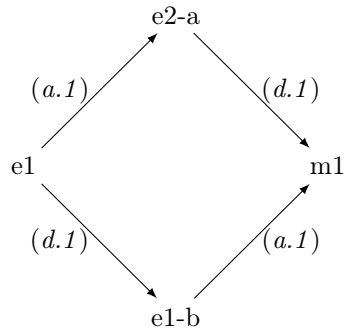
\begin{center}
	 \begin{tikzpicture}
		\node (e1) at (180:2) {e1};
		\node (e2-a) at (90:2) {e2-a};
		\node (m1) at (0:2) {m1};
		\node (e1-b) at (270:2) {e1-b};
		\draw[->, >=latex] (e1) -- node [left] {\ref{item:h-g-f-1}} (e2-a);
		\draw[->, >=latex] (e1) -- node [left] {\ref{item:i-h-gf-1}} (e1-b);
		\draw[->, >=latex] (e2-a) -- node [right] {\ref{item:i-h-gf-1}} (m1);
		\draw[->, >=latex] (e1-b) -- node [right] {\ref{item:h-g-f-1}} (m1);
	 \end{tikzpicture}
	\end{center}
	To be precise, the $3$-cell from e1 to e2-a is
	\begin{gather*}
	 GF_{\treeLog}(i) \comp_0 GF_{\treeLog}(h) \comp_0 G_{\treeV}(F_{\treeLog}g, F_{\treeLog}f) \comp_1
	 G_{\treeVRight}(F_{\treeLog}h, F_{\treeV}(g, f))
	 \\\comp_1\\
	 GF_{\treeLog}(i) \comp_0 G_{\treeLL}(F_{\treeV}(h, gf))\bigr) \comp_1 
	 GF_{\treeV}(i, h\comp_0 g \comp_0 f)
	\end{gather*}
	and the $3$-cell from e1-b to m1 is
	\begin{gather*}
	 GF_{\treeLog}(i) \comp_0 GF_{\treeLog}(h) \comp_0 G_{\treeV}(F_{\treeLog}g, F_{\treeLog}f) \comp_1
	 G_{\treeVRight}(F_{\treeLog}h, F_{\treeV}(g, f))
	 \\ \comp_1 \\
	 G_{\treeV}(F_{\treeLog}i, F_{\treeLog}(h)\comp_0 F_{\treeLog}(gf)) \comp_1
	 G_{\treeLL}(F_{\treeLog}i \comp_0 F_{\treeV}(h, gf)) \comp_1
	 G_{\treeLL}(F_{\treeLL}(i, hgf))\,,
	\end{gather*}
	while the $3$-cell from e1 to e1-b is
	\begin{gather*}
	 GF_{\treeLog}(i) \comp_0 \bigl(GF_{\treeLog}(h) \comp_0 GF_{\treeV}(g, f) \comp_1
	 G_{\treeV}(F_{\treeLog}h, F_{\treeLog}gf)\bigr)
	 \\ \comp_1 \\
	 G_{\treeVRight}(F_{\treeLog}i, F_{\treeV}(h, gf)) \comp_1
	 G_{\treeLL}(F_{\treeV}(i, hgf))
	\end{gather*}
	and the $3$-cell from e2-a to m1 is
	\begin{gather*}
	 GF_{\treeLog}(i) \comp_0 \bigl(GF_{\treeLog}(h) \comp_0 G_{\treeV}(F_{\treeLog}g, F_{\treeLog}f) \comp_1
	 G_{\treeV}(F_{\treeLog}h, F_{\treeLog}g\comp_0 F_{\treeLog}f) \comp_1
	 G_{\treeLL}(F_{\treeLog}h \comp_0 F_{\treeV}(g, f))\bigr)
	 \\\comp_1\\
	 G_{\treeVRight}(F_{\treeLog}i, F_{\treeV}(h, gf)) \comp_1
	 G_{\treeLL}(F_{\treeV}(i, hgf))\,.
	\end{gather*}
	The commutativity of the diagram is simply an instance of the interchange law.
	Indeed, the source  and the target of the $3$-cell of $C$
	\[
	 GF_{\treeLog}(h) \comp_0 G_{\treeV}(F_{\treeLog}g, F_{\treeLog}f)
	 \comp_1
	 G_{\treeVRight}(F_{\treeLog}h, F_{\treeV}(g, f))
	\]
	are the $2$-cells of $C$
	\[
	 GF_{\treeLog}(h) \comp_0 GF_{\treeV}(g, f) \comp_1
	 G_{\treeV}(F_{\treeLog}h, F_{\treeLog}gf)
	\]
	and
	\[
	 GF_{\treeLog}(h) \comp_0 G_{\treeV}(F_{\treeLog}g, F_{\treeLog}f) \comp_1
	 G_{\treeV}(F_{\treeLog}h, F_{\treeLog}g\comp_0 F_{\treeLog}f) \comp_1
	 G_{\treeLL}(F_{\treeLog}h \comp_0 F_{\treeV}(g, f))
	\]
	respectively;
	while the source and target of the $3$-cell
	\[
	 G_{\treeVRight}(F_{\treeLog}i, F_{\treeV}(h, gf))\comp_1
	 G_{\treeLL}(F_{\treeV}(i, hgf))
	\]
	of $C$ are the $2$-cells
	\[
	 GF_{\treeLog}(i) \comp_0 G_{\treeLL}(F_{\treeV}(h, gf))\bigr) \comp_1 
	 GF_{\treeV}(i, h\comp_0 g \comp_0 f)
	\]
	and
	\[
	 G_{\treeV}(F_{\treeLog}i, F_{\treeLog}(h)\comp_0 F_{\treeLog}(gf)) \comp_1
	 G_{\treeLL}(F_{\treeLog}i \comp_0 F_{\treeV}(h, gf)) \comp_1
	 G_{\treeLL}(F_{\treeLL}(i, hgf))
	\]
	respectively.
\end{paragr}

\begin{paragr}[1']
	Consider the diagram (1'), where we abuse of notation by forgetting
	the suitable whiskerings making these $3$-cells of $C$ actually $2$-composable:
	
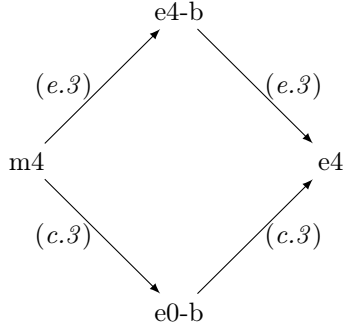
\begin{center}
	 \begin{tikzpicture}
		\node (m4) at (180:2) {m4};
		\node (e4-b) at (90:2) {e4-b};
		\node (e4) at (0:2) {e4};
		\node (e0-b) at (270:2) {e0-b};
		\draw[->, >=latex] (m4) -- node [left] {\ref{item:ih-g-f-3}} (e4-b);
		\draw[->, >=latex] (m4) -- node [left] {\ref{item:i-h-g-3}} (e0-b);
		\draw[->, >=latex] (e4-b) -- node [right] {\ref{item:ih-g-f-3}} (e4);
		\draw[->, >=latex] (e0-b) -- node [right] {\ref{item:i-h-g-3}} (e4);
	 \end{tikzpicture}
	\end{center}
	To be precise, the $3$-cell from e4-b to e4 is
	\begin{gather*}
	 G_{\treeV}(F_{\treeLog}i, F_{\treeLog}h) \comp_0
	 GF_{\treeLog}(g) \comp_0 GF_{\treeLog}(f)  \comp_1
	 G_{\treeVLeft}(F_{\treeV}(i, h), F_{\treeLog}g)
	 \\\comp_1\\
	  G_{\treeLL}(F_{\treeV}(ih, g))\bigr) \comp_0 GF_{\treeLog}(f) \comp_1 
	 GF_{\treeV}(i \comp_0 h \comp_0 g, f)
	\end{gather*}
	and the $3$-cell from m4 to e0-b is
	\begin{gather*}
	 G_{\treeV}(F_{\treeLog}i, F_{\treeLog}h) \comp_0
	 GF_{\treeLog}(g) \comp_0 GF_{\treeLog}(f)  \comp_1
	 G_{\treeVLeft}(F_{\treeV}(i, h), F_{\treeLog}g)
	 \\ \comp_1 \\
	 G_{\treeV}(F_{\treeLog}i \comp_0 F_{\treeLog}(hg), F_{\treeLog}(f)) \comp_1
	 G_{\treeLL}(F_{\treeV}(ih, g) \comp_0 F_{\treeLog}f) \comp_1
	 G_{\treeLL}(F_{\treeLL}(ihg, f))\,,
	\end{gather*}
	while the $3$-cell from e0-b to e4 is
	\begin{gather*}
     \bigl(GF_{\treeV}(i, h)  \comp_0 GF_{\treeLog}(g) \comp_1
	 G_{\treeV}(F_{\treeLog}ih, F_{\treeLog}g)\bigr) \comp_0 GF_{\treeLog}(i)
	 \\ \comp_1 \\
	 G_{\treeVLeft}(F_{\treeV}(ih, g), F_{\treeLog}f) \comp_1
	 G_{\treeLL}(F_{\treeV}(ihg, f))
	\end{gather*}
	and the $3$-cell from m4 to e4-b is
	\begin{gather*}
	 \bigl(G_{\treeV}(F_{\treeLog}i, F_{\treeLog}h) \comp_0  GF_{\treeLog}(g) \comp_1
	 G_{\treeV}(F_{\treeLog}i \comp_0 F_{\treeLog}h, F_{\treeLog}g) \comp_1
	 G_{\treeLL}(F_{\treeV}(i, h) \comp_0 F_{\treeLog}g)\bigr) \comp_0  GF_{\treeLog}(f)
	 \\\comp_1\\
	 G_{\treeVLeft}( F_{\treeV}(ih, g), F_{\treeLog}f) \comp_1
	 G_{\treeLL}(F_{\treeV}(ihg, f))\,.
	\end{gather*}
	Analogously to the previous case,
	the commutativity of the diagram is simply an instance of the interchange law.
\end{paragr}

\begin{paragr}[2]
	Consider the diagram (2), where we always adopt the same notational abuse:
	
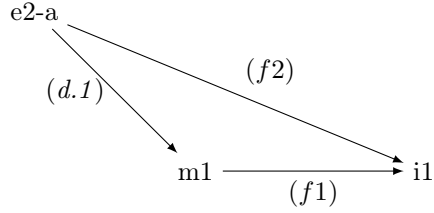
\begin{center}
	 \begin{tikzpicture}[scale=2]
		\node (e2-a) at (135:1.5) {e2-a};
		\node (m1) at (0:0) {m1};
		\node (i1) at (0:1.5) {i1};
		\draw[->, >=latex] (e2-a) -- node [left] {\ref{item:i-h-gf-1}} (m1);
		\draw[->, >=latex] (m1) -- node [below]
			{$(f1)$} (i1);
		\draw[->, >=latex] (e2-a) -- node [above right]
			{$(f2)$} (i1);
	 \end{tikzpicture}
	\end{center}
	where
	\begin{equation}
	\tag*{$(f1)$}
	 G_{\treeVRight}(F_{\treeLog}i, F_{\treeLog} h \comp_0 F_{\treeV}(g, f))
	 \label{item:f1}
	\end{equation}
	is the principal $3$-cell of ($f1$) and
	\begin{equation}
	\tag*{$(f2)$}
	 G_{\treeVRight}(F_{\treeLog}i, F_{\treeLog}h \comp_0 F_{\treeV}(g, f) \comp_1 F_{\treeV}(h, gf)  )
	 \label{item:f2}
	\end{equation}
	is the principal $3$-cell of ($f2$). More precisely, the $3$-cell of $C$ from e2-a to m1 is
	\begin{gather*}
	 GF_{\treeLog}(i) \comp_0 \bigl(GF_{\treeLog}(h) \comp_0 G_{\treeV}(F_{\treeLog}g, F_{\treeLog}f) \comp_1
	 G_{\treeV}(F_{\treeLog}h, F_{\treeLog}g\comp_0 F_{\treeLog}f) \bigr)
	 \\\comp_1\\
	 GF_{\treeLog}(i) \comp_0 G_{\treeLL}(F_{\treeLog}h \comp_0 F_{\treeV}(g, f)) \comp_1 G_{\treeVRight}(F_{\treeLog}i, F_{\treeV}(h, gf))
	 \\\comp_1\\
	 G_{\treeLL}(F_{\treeV}(i, hgf))\,,
	\end{gather*}
	the $3$-cell from m1 to i1 is
	\begin{gather*}
	 GF_{\treeLog}(i) \comp_0 \bigl(GF_{\treeLog}(h) \comp_0 G_{\treeV}(F_{\treeLog}g, F_{\treeLog}f) \comp_1
	 G_{\treeV}(F_{\treeLog}h, F_{\treeLog}g\comp_0 F_{\treeLog}f)\bigr)
	 \\\comp_1\\
	 G_{\treeVRight}(F_{\treeLog}i, F_{\treeLog} h \comp_0 F_{\treeV}(g, f))
	 \comp_1 G_{\treeLL}(F_{\treeLog}i \comp_0 F_{\treeV}(h, gf))
	 \\\comp_1\\
	  G_{\treeLL}(F_{\treeV}(i, hgf))
	\end{gather*}
	and the $3$-cell from e2-a to i1 is
	\begin{gather*}
	 GF_{\treeLog}(i) \comp_0 \bigl(GF_{\treeLog}(h) \comp_0 G_{\treeV}(F_{\treeLog}g, F_{\treeLog}f) \comp_1
	 G_{\treeV}(F_{\treeLog}h, F_{\treeLog}g\comp_0 F_{\treeLog}f)\bigr)
	 \\\comp_1\\
	 G_{\treeVRight}\bigl(F_{\treeLog}i, F_{\treeLog}h \comp_0 F_{\treeV}(g, f) \comp_1 F_{\treeV}(h, gf)\bigr)
	 \\\comp_1\\
	 G_{\treeLL}(F_{\treeV}(i, hgf))\,.
	\end{gather*}
	The $3$-cells appearing in the middle lines of these $1$-compositions are precisely
	the $3$-cells of the coherence for $G$ for the tree
	\scalebox{0.3}{
		\begin{forest}
			for tree={%
				label/.option=content,
				grow=north,
				content=,
				circle,
				fill,
				minimum size=3pt,
				inner sep=0pt,
				s sep+=15,
			}
			[
			 [ [] [] ]
			 []
			]
		\end{forest}
	} for the pasting diagram
	
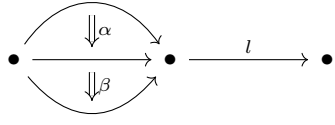
\[
	\begin{tikzcd}[column sep=4.5em]
	\bullet
	\ar[r, bend left=55, looseness=1.3, "", ""{below, name=f1}]
	\ar[r, ""{name=f2u}, ""{below, name=f2d}]
	\ar[r, bend right=50, looseness=1.3, ""', ""{name=f3}]
	\ar[Rightarrow, from=f1, to=f2u, "\alpha"]
	\ar[Rightarrow, from=f2d, to=f3, "\beta"]
	&
	\bullet \ar[r, "l"] & \bullet
	\end{tikzcd}
	\]
	where $\alpha = F_{\treeV}(h, gf)$, $\beta = F_{\treeLog}(h)\comp_0 F_{\treeV}(g, f)$
	and $l = F_{\treeLog}(i)$. Hence the diagram commutes.
\end{paragr}

\begin{paragr}[2']
	Consider the diagram (2'), where we always adopt the same notational abuse:
	
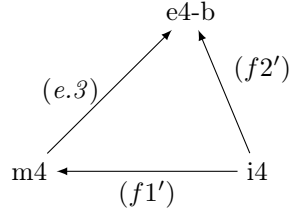
\begin{center}
	 \begin{tikzpicture}[scale=2]
		\node (e4-b) at (45:1.5) {e4-b};
		\node (m4) at (0:0) {m4};
		\node (i4) at (0:1.5) {i4};
		\draw[->, >=latex] (m4) -- node [left] {\ref{item:ih-g-f-3}} (e4-b);
		\draw[->, >=latex] (i4) -- node [below]
			{$(f1')$} (m4);
		\draw[->, >=latex] (i4) -- node [above right]
			{$(f2')$} (e4-b);
	 \end{tikzpicture}
	\end{center}
	where the $3$-cell of $C$ from m4 to e4-b is
	\begin{gather*}
	 \bigl(G_{\treeV}(F_{\treeLog}i, F_{\treeLog}h) \comp_0  GF_{\treeLog}(g)\comp_1
	 G_{\treeV}(F_{\treeLog}i \comp_0 F_{\treeLog}h, F_{\treeLog}g) \bigr) \comp_0 GF_{\treeLog}(f) 
	 \\\comp_1\\
	 G_{\treeLL}(F_{\treeV}(i, h) \comp_0 F_{\treeLog}g) \comp_0 GF_{\treeLog}(f) \comp_1 G_{\treeVLeft}(F_{\treeV}(ih, g), F_{\treeLog}f)
	 \\\comp_1\\
	 G_{\treeLL}(F_{\treeV}(ihg, f))\,,
	\end{gather*}
	the $3$-cell from i4 to m4 is
	\begin{gather*}
	 \bigl(G_{\treeV}(F_{\treeLog}i, F_{\treeLog}h) \comp_0 GF_{\treeLog}(g) \comp_1
	 G_{\treeV}(F_{\treeLog}i \comp_0 F_{\treeLog}h, F_{\treeLog}g)\bigr) \comp_0 GF_{\treeLog}(f)
	 \\\comp_1\\
	 G_{\treeVLeft}(F_{\treeV}(i, h) \comp_0 F_{\treeLog}g, F_{\treeLog}f)
	 \comp_1 G_{\treeLL}(F_{\treeV}(ih, g) \comp_0 F_{\treeLog}f)
	 \\\comp_1\\
	  G_{\treeLL}(F_{\treeV}(ihg, f))
	\end{gather*}
	and the $3$-cell from i4 to e4-b is
	\begin{gather*}
	 \bigl(G_{\treeV}(F_{\treeLog}i, F_{\treeLog}h)  \comp_0 GF_{\treeLog}(g)\comp_1
	 G_{\treeV}(F_{\treeLog}i \comp_0 F_{\treeLog}h, F_{\treeLog}g)\bigr) \comp_0 GF_{\treeLog}(f)
	 \\\comp_1\\
	 G_{\treeVLeft}\bigl(F_{\treeV}(ih, g) \comp_1 F_{\treeV}(i, h) \comp_0 F_{\treeLog}g, F_{\treeLog}f\bigr)
	 \\\comp_1\\
	 G_{\treeLL}(F_{\treeV}(ihg, f))\,.
	\end{gather*}
	The $3$-cells appearing in the middle lines of these $1$-compositions are precisely
	the $3$-cells of the coherence for $G$ for the tree
	\scalebox{0.3}{
		\begin{forest}
			for tree={%
				label/.option=content,
				grow=north,
				content=,
				circle,
				fill,
				minimum size=3pt,
				inner sep=0pt,
				s sep+=15,
			}
			[
			 []
			 [ [] [] ]
			]
		\end{forest}
	} for the pasting diagram
	
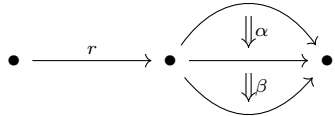
\[
	\begin{tikzcd}[column sep=4.5em]
	\bullet \ar[r, "r"] & \bullet
	\ar[r, bend left=55, looseness=1.3, "", ""{below, name=f1}]
	\ar[r, ""{name=f2u}, ""{below, name=f2d}]
	\ar[r, bend right=50, looseness=1.3, ""', ""{name=f3}]
	\ar[Rightarrow, from=f1, to=f2u, "\alpha"]
	\ar[Rightarrow, from=f2d, to=f3, "\beta"]
	&
	\bullet
	\end{tikzcd}
	\]
	where $\alpha = F_{\treeV}(ih, g)$, $\beta = F_{\treeV}(i, h)\comp_0 F_{\treeLog}(g)$
	and $r = F_{\treeLog}(f)$. Hence the diagram commutes.
\end{paragr}

\begin{paragr}[3]
 Consider the diagram (3):
 \begin{center}
  \begin{tikzpicture}
  \node (e2-b) at (150:2) {e2-b};
  \node (e2-a) at (210:2) {e2-a};
  \node (i2) at (30:2) {i2};
  \node (i1) at (-30:2) {i1};
  \draw[->, >=latex] (e2-a) -- node [left] {\ref{item:h-g-f-2}} (e2-b);
  \draw[->, >=latex] (e2-a) -- node [below]
    {\ref{item:f2}} (i1);
  \draw[->, >=latex] (e2-b) -- node [above]
    {$(f3)$} (i2);
  \draw[->, >=latex] (i1) -- node [right]
  {$(g3)$} (i2);
  \end{tikzpicture}
 \end{center}
 where
 \begin{equation}
 \tag*{$(f3)$}
  G_{\treeVRight}(F_{\treeLog}i, F_{\treeV}(h, g)\comp_0F_{\treeLog}f \comp_1 F_{\treeV}(hg, f))
  \label{item:f3}
 \end{equation}
 is the principal $3$-cell of ($f3$) and
 \begin{equation}
 \tag*{$(g3)$}
  GF_{\treeLog}(i)\comp_0 G_{\treeW}(F_{\treeLog}h, F_{\treeLog}g, F_{\treeLog}f) 
  \comp_1
  G_{\treeV}(F_{\treeLog}i, F_{\treeLog}h\,F_{\treeLog}g\,F_{\treeLog}f)
  \comp_1
   G_{\treeLLL}(F_{\treeLog}(i) \comp_0 F_{\treeW}(h, g, f))
   \label{item:g3}
 \end{equation}
 is the principal $3$-cell of ($g3$). More precisely, the $3$-cell from e2-a to e2-b is
    \begin{gather*}
     G_{\treeW}(F_{\treeLog}h, F_{\treeLog}g, F_{\treeLog}f) \comp_1
	 G_{\treeLLL}(F_{\treeW}(h, g, f))
     \\ \comp_1 \\
     GF_{\treeV}(i, hgf)\,;
    \end{gather*}
    the $3$-cell from e2-a to i1 is
	\begin{gather*}
	 GF_{\treeLog}(i) \comp_0 \bigl(GF_{\treeLog}(h) \comp_0 G_{\treeV}(F_{\treeLog}g, F_{\treeLog}f) \comp_1
	 G_{\treeV}(F_{\treeLog}h, F_{\treeLog}g\comp_0 F_{\treeLog}f)\bigr)
	 \\ \comp_1 \\
	 G_{\treeVRight}(F_{\treeLog}i, F_{\treeV}(h, gf) \comp_1 F_{\treeLog}h \comp_0 F_{\treeV}(g, f))
	 \\ \comp_1 \\
	 G_{\treeLL}(F_{\treeV}(i, hgf))\,;
	\end{gather*}
	the $3$-cell from e2-b to i2 is
	\begin{gather*}
	 GF_{\treeLog}(i) \comp_0
	 \bigl(G_{\treeV}(F_{\treeLog}h F_{\treeLog}g)\comp_0 GF_{\treeLog}(f)
	 \comp_1 G_{\treeV}(F_{\treeLog}h\,F_{\treeLog}g, F_{\treeLog}f)\bigr)
	 \\ \comp_1 \\
	 G_{\treeVRight}\bigl(F_{\treeLog}(i), F_{\treeV}(hg,f) \comp_1 F_{\treeV}(h, g) \comp_0 F_{\treeLog}(f)\bigr)
	 \\ \comp_1 \\
	 G_{\treeLL}(F_{\treeV}(i, ghf))\,;
	\end{gather*}
	the $3$-cell from i1 to i2 is
	\begin{gather*}
	 GF_{\treeL}(i) \comp_0 G_{\treeW}\bigl(F_{\treeL}(h), F_{\treeL}(g), F_{\treeL}(f)\bigr)
	 \\ \comp_1 \\
	 G_{\treeV}\bigl(F_{\treeL}(i), F_{\treeL}(h) \comp_0 F_{\treeL}(g) \comp_0 F_{\treeL}(f)\bigr)
	 \\ \comp_1 \\
	 G_{\treeLLL}\bigl(F_{\treeL}(i) \comp_0 F_{\treeW}(h, g, f)\bigr)
	 \\ \comp_1 \\
	 G_{\treeLL}\bigl(F_{\treeV}(i, hgf)\bigr)\,.
	\end{gather*}
	The coherence for the tree
	\scalebox{0.3}{
		\begin{forest}
			for tree={%
				label/.option=content,
				grow=north,
				content=,
				circle,
				fill,
				minimum size=3pt,
				inner sep=0pt,
				s sep+=15,
			}
			[
			 [ [ [] ] ]
			 [ ]
			]
		\end{forest}
	}
	applied to the pasting diagram
	
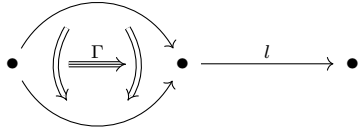
\[
	 \begin{tikzcd}[column sep=5em]
     \bullet
     \ar[r, bend left=60, looseness=1.2, "", "\phantom{bullet}"'{name=1}]
     \ar[r, bend right=60, looseness=1.2, ""', "\phantom{bullet}"{name=3}]
     \ar[Rightarrow, from=1, to=3, shift right=2ex, bend right, ""{name=beta1}]
     \ar[Rightarrow, from=1, to=3, shift left=2ex, bend left, ""'{name=beta3}]
     \arrow[triple, from=beta1, to=beta3, "\Gamma"]{}
     &
     \bullet
     \ar[r, "l"]
     &
     \bullet
     \end{tikzcd}
	\]
	where $\Gamma = F_{\treeW}(h, g, f)$ and $l = F_{\treeL}(i)$
	gives the equality
	\begin{gather*}
	 G_{\treeVRight}\bigl(F_{\treeL}(i), F_{\treeV}(h, g) \comp_0 F_{\treeL}(f)
	 \comp_1 F_{\treeV}(hg, f)\bigr)
	 \\ \comp_2 \\
	 GF_{\treeL}(i) \comp_0 G_{\treeLLL}\bigl(F_{\treeW}(h, g, f)\bigr)
	 \\ = \\
	 G_{\treeLLL}\bigl(F_{\treeL}(i) \comp_0 F_{\treeW}(h, g, f)\bigr)
	 \\ \comp_2 \\
	 G_{\treeVRight}\bigl(F_{\treeL}(i), F_{\treeL}(h) \comp_0 F_{\treeV}(g, f)
	 \comp_1 F_{\treeV}(h, gf) \bigr)
	\end{gather*}
	of $3$-cells of $C$; $1$-precomposing, \ie whiskering by $\comp_1$, this equality on both sides by
	the $2$-cell $G_{\treeLL}(F_{\treeV}(i, hgf))$ and 
	$1$-post-composing, \ie whiskering by $\comp_1$ again,
	both sides by the $3$-cell
	\[
	 GF_{\treeL}(i) \comp_0  G_{\treeW}\bigl(F_{\treeL}(h), F_{\treeL}(g), F_{\treeL}(f)\bigr)\,,
    \]
    the two sides of the equality are precisely the two paths of diagram (3),
    which is therefore commutative.
\end{paragr}

\begin{paragr}[4]
    Consider the diagram (4)
    
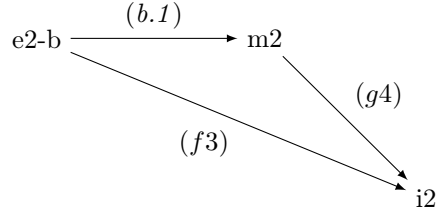
\begin{center}
     \begin{tikzpicture}[scale=2]
      \node (m2) at (0:0) {m2};
      \node (e2-b) at (180:1.5) {e2-b};
      \node (i2) at (-45:1.5) {i2};
      \draw [->, >=latex] (e2-b) -- node [below left] {\ref{item:f3}} (i2);
      \draw [->, >=latex] (e2-b) -- node [above] {\ref{item:b1}} (m2);
      \draw [->, >=latex] (m2) -- node [above right] {$(g4)$} (i2);
     \end{tikzpicture}
    \end{center}
    where the principal $3$-cell of ($g4$) is
    \begin{equation}
     \tag*{$(g4)$}
     \label{item:g4}
     G_{\treeVRight}\bigl(F_{\treeL}(i), F_{\treeV}(h, g) \comp_0 F_{\treeL}(f)\bigr)\,.
    \end{equation}
    More precisely, we have already seen that the $3$-cell
    from e2-b to i2 is set to be
    \begin{gather*}
	 GF_{\treeLog}(i) \comp_0
	 \bigl(G_{\treeV}(F_{\treeLog}(h) F_{\treeLog}(g))\comp_0 GF_{\treeLog}(f)
	 \comp_1 G_{\treeV}(F_{\treeLog}(h)\,F_{\treeLog}(g), F_{\treeLog}(f))\bigr)
	 \\ \comp_1 \\
	 G_{\treeVRight}\bigl(F_{\treeLog}(i), F_{\treeV}(hg,f) \comp_1 F_{\treeV}(h, g) \comp_0 F_{\treeLog}(f)\bigr)
	 \\ \comp_1 \\
	 G_{\treeLL}(F_{\treeV}(i, ghf))\,;
	\end{gather*}
	the $3$-cell from m2 to i2 is
	\begin{gather*}
	 GF_{\treeLog}(i) \comp_0
	 \bigl(G_{\treeV}(F_{\treeL}(h), F_{\treeL}(g))\comp_0 GF_{\treeL}(f)
	 \comp_1 G_{\treeV}(F_{\treeL}(h)\,F_{\treeL}(g), F_{\treeL}(f))\bigr)
	 \\ \comp_1 \\
	 G_{\treeVRight}\bigl(F_{\treeL}(i), F_{\treeV}(h, g)\comp_0F_{\treeL}(f)\bigr)
	 \\ \comp_1 \\
	 G_{\treeLL}\bigl(F_{\treeL}(i) \comp_0 F_{\treeV}(hg, f)\bigr)
	 \\ \comp_1 \\ G_{\treeLL}\bigl(F_{\treeV}(i, hgf)\bigr)\,;
	\end{gather*}
	the $3$-cell of $C$ from e2-b to m2 is
	\begin{gather*}
	 GF_{\treeLog}(i) \comp_0
	 \bigl(G_{\treeV}(F_{\treeL}(h), F_{\treeL}(g))\comp_0 GF_{\treeL}(f)
	 \comp_1 G_{\treeV}(F_{\treeL}(h)\,F_{\treeL}(g), F_{\treeL}(f))\bigr)
	 \\ \comp_1 \\
	 GF_{\treeLog}(i) \comp_0 G_{\treeL}(F_{\treeV}(h, g)\comp_0 F_{\treeL}(f)\bigr)
	 \\ \comp_1 \\
     G_{\treeVRight}(F_{\treeLog}(i), F_{\treeV}(hg, f))
     \\ \comp_1 \\
     G_{\treeLL}(F_{\treeV}(i, hgf))\,.
	\end{gather*}
	The coherence for the tree
	\scalebox{0.3}{
		\begin{forest}
			for tree={%
				label/.option=content,
				grow=north,
				content=,
				circle,
				fill,
				minimum size=3pt,
				inner sep=0pt,
				s sep+=15,
			}
			[
			 [ [] [] ] []
			]
		\end{forest}
	}
	applied to the pasting diagram
	
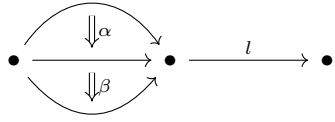
\[
	 \begin{tikzcd}[column sep=4.5em]
       \bullet
       \ar[r, bend left=55, looseness=1.3, ""{below, name=f1}]
       \ar[r, ""{name=f2u}, ""{below, name=f2d}]
       \ar[r, bend right=50, looseness=1.3, ""{name=f3}]
       \ar[Rightarrow, from=f1, to=f2u, "\alpha"]
       \ar[Rightarrow, from=f2d, to=f3, "\beta"]
       &
       \bullet \ar[r, "l"] & \bullet
      \end{tikzcd}
	\]
	of $C$, where $\alpha = F_{\treeV}(hg, f)$,
	$\beta = F_{\treeV}(h, g)\comp_0 F_{\treeL}(f)$
	and $l = F_{\treeL}(i)$ gives the equality
	\begin{gather*}
	 G_{\treeVRight}\bigl(F_{\treeL}(i), F_{\treeV}(h, g) \comp_0 F_{\treeL}(f)
	 \comp_1 F_{\treeV}(hg, f)\bigr)
	 \\ = \\
	 G_{\treeV}\bigl(F_{\treeL}(i), F_{\treeV}(h, g)\comp_0 F_{\treeL}(f)\bigr)
	 \comp_1 G_{\treeLL}\bigl(F_{\treeL}(i)\comp_0 F_{\treeV}(hg, f)\bigr)
	 \\ \comp_2 \\
	 GF_{\treeL}(i) \comp_0 G_{\treeLL}\bigl(F_{\treeV}(hg, f)\bigr)
	 \comp_1 G_{\treeVRight}\bigl(F_{\treeL}(i), F_{\treeV}(h, g) \comp_0 F_{\treeL}(f)\bigr)
	\end{gather*}
	of $3$-cells of $C$; $1$-precomposing,\ie whiskering by $\comp_1$,
	both members of the equality by the $2$-cell $G_{\treeLL}\bigl(F_{\treeV}(i, hgf)\bigr)$
	and $1$-post-composing, \ie whiskering by $\comp_1$, by the $2$-cell
	\[
	 GF_{\treeLog}(i) \comp_0
	 \bigl(G_{\treeV}(F_{\treeL}(h), F_{\treeL}(g))\comp_0 GF_{\treeL}(f)
	 \comp_1 G_{\treeV}(F_{\treeL}(h)\,F_{\treeL}(g), F_{\treeL}(f))\bigr)
	\]
	we get precisely the $3$-cells of diagram (4), which therefore commutes.
\end{paragr}

\begin{paragr}[5]
    Consider the diagram (5)
    
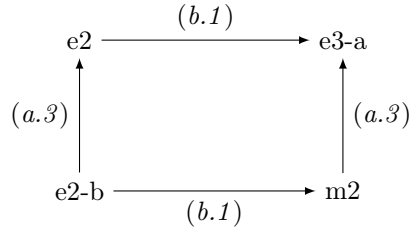
\begin{center}
     \begin{tikzpicture}
        \node (e2) at (150:2) {e2};
        \node (e2-b) at (210:2) {e2-b};
        \node (e3-a) at (30:2) {e3-a};
        \node (m2) at (-30:2) {m2};
        \draw[->, >=latex] (e2-b) -- node [left] {\ref{item:h-g-f-3}} (e2);
        \draw[->, >=latex] (e2-b) -- node [below]
            {\ref{item:b1}} (m2);
        \draw[->, >=latex] (e2) -- node [above]
            {\ref{item:b1}} (e3-a);
        \draw[->, >=latex] (m2) -- node [right]
        {\ref{item:h-g-f-3}} (e3-a);
     \end{tikzpicture}\ .
    \end{center}
    More precisely, we already know that the $3$-cell from
    e2 to e3-a is
    \begin{gather*}
     GF_{\treeLog}(i) \comp_0
	 \bigl(GF_{\treeV}(h, g)\comp_0 GF_{\treeL}(f)\bigr)
	 \\ \comp_1 \\
	 GF_{\treeLog}(i) \comp_0 G_{\treeV}\bigl(F_{\treeL}(hg), F_{\treeL}(f)\bigr)
	 \\ \comp_1 \\
     G_{\treeVRight}(F_{\treeLog}(i), F_{\treeV}(hg, f))
     \\ \comp_1 \\
     G_{\treeLL}(F_{\treeV}(i, hgf))\,,
    \end{gather*}
    the $3$-cell from e2-b to e2 is
    \begin{gather*}
     GF_{\treeL}(i) \comp_0 G_{\treeV}\bigl(F_{\treeLog}(h), F_{\treeLog}(g)\bigr) \comp_0 GF_{\treeLog}(f)
     \\ \comp_1 \\
     GF_{\treeL}(i) \comp_0  G_{\treeVLeft}(F_{\treeV}(h, g), F_{\treeLog}f)
     \\ \comp_1 \\
     G_{\treeLL}\bigl(F_{\treeV}(h, g) \comp_0 F_{\treeL}(f)\bigr)
     \\ \comp_1 \\
     GF_{\treeV}(i, hgf)
    \end{gather*}
    and the $3$-cell from e2-b to m2 is
    \begin{gather*}
	 GF_{\treeLog}(i) \comp_0
	 \bigl(G_{\treeV}(F_{\treeL}(h), F_{\treeL}(g))\comp_0 GF_{\treeL}(f)\bigr)
	 \\ \comp_1 \\
	 GF_{\treeLog}(i) \comp_0 \bigl(G_{\treeV}(F_{\treeL}(h)\,F_{\treeL}(g), F_{\treeL}(f)) \comp_1 G_{\treeL}(F_{\treeV}(h, g)\comp_0 F_{\treeL}(f)\bigr)
	 \\ \comp_1 \\
     G_{\treeVRight}(F_{\treeLog}(i), F_{\treeV}(hg, f))
     \\ \comp_1 \\
     G_{\treeLL}(F_{\treeV}(i, hgf))\,;
	\end{gather*}
	the $3$-cell from m2 to e3-a is
	\begin{gather*}
	 GF_{\treeL}(i) \comp_0 G_{\treeV}\bigl(F_{\treeLog}(h), F_{\treeLog}(g)\bigr) \comp_0 GF_{\treeLog}(f)
     \\ \comp_1 \\
     GF_{\treeL}(i) \comp_0  G_{\treeVLeft}(F_{\treeV}(h, g), F_{\treeLog}f)
     \\ \comp_1 \\
     G_{\treeV}\bigl(F_{\treeL}(i), F_{\treeL}(hg) \comp_0 F_{\treeL}(f)\bigr)
     \\ \comp_1 \\
     G_{\treeLL}\bigl(F_{\treeL}(i) \comp_0 F_{\treeV}(hg, f) \bigr) \comp_1 G_{\treeLL}\bigl(F_{\treeV}(i, hgf)\bigr)\,.
	\end{gather*}
	It is clear from this explicit description of the $3$-cells involved
	that diagram (5) is commutative by virtue of the interchange law.
\end{paragr}

\begin{paragr}[6]
 Consider the diagram (6)
 
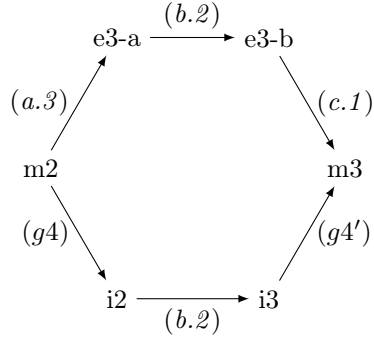
\begin{center}
  \begin{tikzpicture}[scale=2]
   \node (m3) at (0:1) {m3};
   \node (e3-b) at (60:1) {e3-b};
   \node (e3-a) at (120:1) {e3-a};
   \node (m2) at (180:1) {m2};
   \node (i2) at (240:1) {i2};
   \node (i3) at (300:1) {i3};
   \draw[->, >=latex] (m2) -- node [left] {\ref{item:h-g-f-3}} (e3-a);
   \draw[->, >=latex] (e3-a) -- node [above] {\ref{item:b2}} (e3-b);
   \draw[->, >=latex] (e3-b) -- node [right] {\ref{item:i-h-g-1}} (m3);
   \draw[->, >=latex] (m2) -- node [left] {\ref{item:g4}} (i2);
   \draw[->, >=latex] (i2) -- node [below] {\ref{item:b2}} (i3);
   \draw[->, >=latex] (i3) -- node [right] {$(g4')$} (m3);
  \end{tikzpicture}
 \end{center}
 where the principal $3$-cell of ($g4'$) is
 \begin{equation}
  \tag*{$(g4')$}
  \label{item:g4p}
  G_{\treeVLeft} \bigl(F_{\treeL}(i) \comp_0 F_{\treeV}(h, g), F_{\treeL}(f)\bigr)\,.
 \end{equation}
More precisely, the $3$-cell from i3 to m3 is
\begin{gather*}
 GF_{\treeL}(i) \comp_0 G_{\treeV}\bigl(F_{\treeL}(h), F_{\treeL}(g)\bigr)
 \comp_0 GF_{\treeL}(f)
 \\ \comp_1 \\
 G_{\treeV}\bigl(F_{\treeL}(i), F_{\treeL}(h) \comp_0 F_{\treeL}(g)\bigr)
 \comp_0 GF_{\treeL}(f)
  \comp_1
 G_{\treeVLeft} \bigl(F_{\treeL}(i) \comp_0 F_{\treeV}(h, g), F_{\treeL}(f)\bigr)
 \\ \comp_1 \\
 G_{\treeLL}\bigl(F_{\treeV}(i, hg)\comp_0F_{\treeL}(f)\bigr)
  \comp_1 
 G_{\treeLL}\bigl(F_{\treeV}(ihg, f)\bigr)\,;
\end{gather*}
the $3$-cell from i2 to i3 is
\begin{gather*}
 GF_{\treeL}(i) \comp_0 G_{\treeV}\bigl(F_{\treeL}(h), F_{\treeL}(g)\bigr)
 \comp_0 GF_{\treeL}(f)
 \\ \comp_1 \\
 G_{\treeW}\bigl(F_{\treeL}(i), F_{\treeL}(h)\comp_0 F_{\treeL}(g), F_{\treeL}(f)\bigr)
 \comp_1 G_{\treeLL}\bigl(F_{\treeL}(h) \comp_0 F_{\treeL}(h, g)\comp_0 F_{\treeL}(f)\bigr)
 \\ \comp_1 \\
 G_{\treeLLL}\bigl(F_{\treeW}(i, hg, f)\bigr)\,;
\end{gather*}
we already know that the $3$-cell from m2 to i2 is
\begin{gather*}
     GF_{\treeLog}(i) \comp_0 G_{\treeV}\bigl(F_{\treeL}(h), F_{\treeL}(g)\bigr)\comp_0 GF_{\treeL}(f)
	 \\ \comp_1 \\
	 GF_{\treeLog}(i) \comp_0 G_{\treeV}\bigl(F_{\treeL}(h)\comp_0 F_{\treeL}(g), F_{\treeL}(f)\bigr)
	  \comp_1 
	 G_{\treeVRight}\bigl(F_{\treeL}(i), F_{\treeV}(h, g)\comp_0 F_{\treeL}(f)\bigr)
	 \\ \comp_1 \\
	 G_{\treeLL}\bigl(F_{\treeL}(i) \comp_0 F_{\treeV}(hg, f)\bigr)
	  \comp_1  G_{\treeLL}\bigl(F_{\treeV}(i, hgf)\bigr)\,;
\end{gather*}
the $3$-cell from e3-b to m3 is
\begin{gather*}
 GF_{\treeLog}(i) \comp_0 G_{\treeV}\bigl(F_{\treeL}(h), F_{\treeL}(g)\bigr)\comp_0 GF_{\treeL}(f)
	 \\ \comp_1 \\
	 G_{\treeVRight}\bigl(F_{\treeL}(i), F_{\treeV}(h, g)\bigr) \comp_0 GF_{\treeL}(f)
	 \comp_1 G_{\treeV}\bigl(F_{\treeL}(i) \comp_0 F_{\treeL}(h, g), F_{\treeL}(f)\bigr)
	 \\ \comp_1 \\
 G_{\treeLL}\bigl(F_{\treeV}(i, hg)\comp_0F_{\treeL}(f)\bigr)
  \comp_1 
 G_{\treeLL}\bigl(F_{\treeV}(ihg, f)\bigr)\,;
\end{gather*}
we already know that the $3$-cell from e3-a to e3-b is
\begin{gather*}
 GF_{\treeL}(i) \comp_0 G_{\treeV}\bigl(F_{\treeL}(h), F_{\treeL}(g)\bigr)
 \comp_0 GF_{\treeL}(f)
 \\ \comp_1 \\
 GF_{\treeL}(i) \comp_0 G_{\treeLL}\bigl(F_{\treeL}(h, g)\bigr) \comp_0 F_{\treeL}(f)
 \comp_1
 G_{\treeW}\bigl(F_{\treeL}(i), F_{\treeL}(hg), F_{\treeL}(f)\bigr)
 \\ \comp_1 \\
 G_{\treeLLL}\bigl(F_{\treeW}(i, hg, f)\bigr)\,;
\end{gather*}
we also know that the $3$-cell from m2 to e3-a is
\begin{gather*}
    GF_{\treeL}(i) \comp_0 G_{\treeV}\bigl(F_{\treeLog}(h), F_{\treeLog}(g)\bigr) \comp_0              GF_{\treeLog}(f)
     \\ \comp_1 \\
     GF_{\treeL}(i) \comp_0  G_{\treeVLeft}\bigl(F_{\treeV}(h, g), F_{\treeLog}(f)\bigr)
      \comp_1 
     G_{\treeV}\bigl(F_{\treeL}(i), F_{\treeL}(hg) \comp_0 F_{\treeL}(f)\bigr)
     \\ \comp_1 \\
     G_{\treeLL}\bigl(F_{\treeL}(i) \comp_0 F_{\treeV}(hg, f) \bigr) \comp_1 G_{\treeLL}\bigl(F_{\treeV}(i, hgf)\bigr)\,.
\end{gather*}
Notice that there is a complete duality between the $3$-cell from m2 to e3-a and
the $3$-cell from e3-b to m3 and also between the $3$-cell from m2 to i2
and the $3$-cell from i3 to m3.
The coherence for the tree
\scalebox{0.3}{
		\begin{forest}
			for tree={%
				label/.option=content,
				grow=north,
				content=,
				circle,
				fill,
				minimum size=3pt,
				inner sep=0pt,
				s sep+=15,
			}
			[
			 []
			 [ [] ]
			 []
			]
		\end{forest}
}
applied to the pasting diagram
\[
    \begin{tikzcd}[column sep=4.5em]
    \bullet \ar[r, "k"] &
    \bullet
     \ar[r, bend left, ""{below, name=g}]
     \ar[r, bend right, ""{name=g2}]
     \ar[Rightarrow, from=g, to=g2, "\alpha"] &
    \bullet \ar[r, "l"] & \bullet
    \end{tikzcd}
\]
of $C$ with $k= F_{\treeL}(f)$, $\alpha = F_{\treeV}(h, g)$
and $l = F_{\treeL}(i)$ gives the equality between the
$3$\nbd-cell
\begin{gather*}
 G_{\treeV}\bigl(F_{\treeL}(i), F_{\treeL}(h) \comp_0 F_{\treeL}(g)\bigr)
 \comp_0 GF_{\treeL}(f)
  \comp_1
 G_{\treeVLeft} \bigl(F_{\treeL}(i) \comp_0 F_{\treeV}(h, g), F_{\treeL}(f)\bigr)
 \\ \comp_1 \\
 G_{\treeW}\bigl(F_{\treeL}(i), F_{\treeL}(h)\comp_0 F_{\treeL}(g), F_{\treeL}(f)\bigr)
 \comp_1 G_{\treeLL}\bigl(F_{\treeL}(h) \comp_0 F_{\treeL}(h, g)\comp_0 F_{\treeL}(f)\bigr)
 \\ \comp_1 \\
 GF_{\treeLog}(i) \comp_0 G_{\treeV}\bigl(F_{\treeL}(h)\comp_0 F_{\treeL}(g), F_{\treeL}(f)\bigr)
  \comp_1 
 G_{\treeVRight}\bigl(F_{\treeL}(i), F_{\treeV}(h, g)\comp_0 F_{\treeL}(f)\bigr)
 \end{gather*}
 of $C$ and the $3$-cell
 \begin{gather*}
 G_{\treeVRight}\bigl(F_{\treeL}(i), F_{\treeV}(h, g)\bigr) \comp_0 GF_{\treeL}(f)
	 \comp_1 G_{\treeV}\bigl(F_{\treeL}(i) \comp_0 F_{\treeL}(h, g), F_{\treeL}(f)\bigr)
 \\ \comp_1 \\
 GF_{\treeL}(i) \comp_0 G_{\treeLL}\bigl(F_{\treeL}(h, g)\bigr) \comp_0 F_{\treeL}(f)
 \comp_1
 G_{\treeW}\bigl(F_{\treeL}(i), F_{\treeL}(hg), F_{\treeL}(f)\bigr)
 \\ \comp_1 \\
 GF_{\treeL}(i) \comp_0  G_{\treeVLeft}\bigl(F_{\treeV}(h, g), F_{\treeLog}(f)\bigr)
      \comp_1 
     G_{\treeV}\bigl(F_{\treeL}(i), F_{\treeL}(hg) \comp_0 F_{\treeL}(f)\bigr)
\end{gather*}
of $C$.
We get the diagram (6) by $1$-precomposing, \ie whiskering by $\comp_1$,
both terms of this equality with the $2$-cell
\[
 GF_{\treeL}(i) \comp_0 G_{\treeV}\bigl(F_{\treeLog}(h), F_{\treeLog}(g)\bigr) \comp_0              GF_{\treeLog}(f)
\]
 of $C$ and by $1$-post-composing with the $3$-cell $G_{\treeLLL}\bigl(F_{\treeW}(i, hg, f)\bigr)$, \ie a ``vertical composition'' of $3$-cells. Hence, the diagram is commutative.
\end{paragr}

\begin{paragr}[7]
 Consider the diagram (7)
 
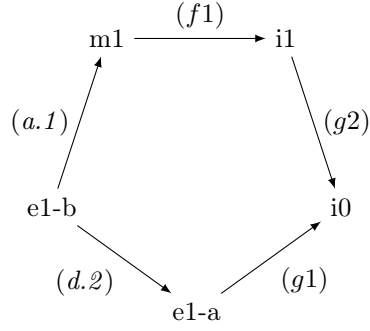
\begin{center}
  \begin{tikzpicture}
   \node (e1-a) at (270:2) {e1-a};
   \node (e1-b) at (198:2) {e1-b};
   \node (m1) at (126:2) {m1};
   \node (i1) at (54:2) {i1};
   \node (i0) at (-18:2) {i0};
   \draw[->,>=latex] (e1-b) -- node [below left] {\ref{item:i-h-gf-2}} (e1-a);
   \draw[->,>=latex] (e1-a) -- node [below right] {$(g1)$} (i0);
   \draw[->,>=latex] (e1-b) -- node [left] {\ref{item:h-g-f-1}} (m1);
   \draw[->,>=latex] (m1) -- node [above] {\ref{item:f1}} (i1);
   \draw[->,>=latex] (i1) -- node [right] {$(g2)$} (i0);
  \end{tikzpicture}
 \end{center}
 where
 \begin{equation}
 \tag*{$(g1)$}
  \label{item:g1}
  G_{\treeVRight}\bigl(F_{\treeL}(i)\comp_0 F_{\treeL}(h), F_{\treeV}(g, f)\bigr)
 \end{equation}
 is the principal $3$-cell of ($g1$) and
 \begin{equation}
  \tag*{$(g2)$}
  \label{item:g2}
  G_{\treeW}\bigl(F_{\treeL}(i), F_{\treeL}(h), F_{\treeL}(g)\comp_0 F_{\treeL}(f)\bigr)
 \end{equation}
 is the principal $3$-cell of ($g2$). More precisely, the $3$-cell of $C$ from e1-a to i0 is
 \begin{gather*}
  GF_{\treeL}(i)\comp_0 GF_{\treeL}(h) \comp_0 G_{\treeV}\bigl(F_{\treeL}(g), F_{\treeL}(f)\bigr)
  \\ \comp_1 \\
  G_{\treeV}\bigl(F_{\treeL}(i), F_{\treeL}(h)\bigr) \comp_0
  G_{\treeL}\bigl(F_{\treeL}(g)\comp_0 F_{\treeL}(f)\bigr) \comp_1
  G_{\treeVRight}\bigl(F_{\treeL}(i)\comp_0 F_{\treeL}(h), F_{\treeV}(g, f)\bigr)
  \\ \comp_1 \\
  G_{\treeLL}\bigl(F_{\treeV}(i, h)\comp_0 F_{\treeL}(gf)\bigr)
   \comp_1
  G_{\treeLL}\bigl(F_{\treeV}(ih, gf)\bigr)\,;
 \end{gather*}
 the $3$-cell from e1-b to e1-a, as we know, is
 \begin{gather*}
  GF_{\treeL}(i)\comp_0 GF_{\treeL}(h) \comp_0 G_{\treeV}\bigl(F_{\treeL}(g), F_{\treeL}(f)\bigr)
  \\ \comp_1 \\
  GF_{\treeL}(i) \comp_0 GF_{\treeL}(h) \comp_0 G_{\treeLL}\bigl(F_{\treeV}(g, f)\bigr)
   \comp_1
  G_{\treeW}\bigl(F_{\treeL}(i), F_{\treeL}(h), F_{\treeL}(gf)\bigr)
  \\ \comp_1 \\
  G_{\treeLLL}\bigl(F_{\treeW}(i, h, gf)\bigr)\,;
 \end{gather*}
 the $3$-cell from i1 to i0 is
 \begin{gather*}
  GF_{\treeL}(i)\comp_0 GF_{\treeL}(h) \comp_0 G_{\treeV}\bigl(F_{\treeL}(g), F_{\treeL}(f)\bigr)
  \\ \comp_1 \\
  G_{\treeW}\bigl(F_{\treeL}(i), F_{\treeL}(h), F_{\treeL}(g)\comp_0 F_{\treeL}(f)\bigr)
   \comp_1
  G_{\treeLL}\bigl(F_{\treeL}(i)\comp_0 F_{\treeL}(h) \comp_0 F_{\treeV}(g, f)\bigr)
  \\ \comp_1 \\
  G_{\treeLLL}\bigl(F_{\treeW}(i, h, gf)\bigr)\,;
 \end{gather*}
 the $3$-cell from m1 to i1, as we know, is
 \begin{gather*}
  GF_{\treeL}(i)\comp_0 GF_{\treeL}(h) \comp_0 G_{\treeV}\bigl(F_{\treeL}(g), F_{\treeL}(f)\bigr)
  \\ \comp_1 \\
   GF_{\treeL}(i)\comp_0 G_{\treeV}\bigl(F_{\treeLog}(h), F_{\treeLog}(g)\comp_0 F_{\treeLog}(f)\bigr)
	 \comp_1
    G_{\treeVRight}\bigl(F_{\treeLog}(i), F_{\treeLog}(h) \comp_0 F_{\treeV}(g, f)\bigr)
	\\ \comp_1 \\
  G_{\treeLL}\bigl(F_{\treeV}(i, h)\comp_0 F_{\treeL}(gf)\bigr)
   \comp_1
  G_{\treeLL}\bigl(F_{\treeV}(ih, gf)\bigr)\,;
 \end{gather*}
 finally we know that the $3$-cell from e1-b to m1 is
 \begin{gather*}
 GF_{\treeL}(i)\comp_0 GF_{\treeL}(h) \comp_0 G_{\treeV}\bigl(F_{\treeL}(g), F_{\treeL}(f)\bigr)
  \\ \comp_1 \\
	 GF_{\treeL}(i) \comp_0 G_{\treeVRight}\bigl(F_{\treeLog}(h), F_{\treeV}(g, f)\bigr)
	  \comp_1 
	 G_{\treeV}\bigl(F_{\treeLog}(i), F_{\treeLog}(h)\comp_0 F_{\treeLog}(gf)\bigr)
	 \\ \comp_1 \\
	 G_{\treeLL}\bigl(F_{\treeV}(i, h)\comp_0 F_{\treeL}(gf)\bigr)
   \comp_1
  G_{\treeLL}\bigl(F_{\treeV}(ih, gf)\bigr)\,.
 \end{gather*}
 The coherence for the tree
 \scalebox{0.3}{
		\begin{forest}
			for tree={%
				label/.option=content,
				grow'=north,
				content=,
				circle,
				fill,
				minimum size=3pt,
				inner sep=0pt,
				s sep+=15,
			}
			[
			 []
			 []
			 [ [] ]
			]
		\end{forest}
}
applied to the pasting diagram
\[
      \begin{tikzcd}[column sep=4.5em]
       \bullet
       \ar[r, bend left, ""{below, name=f}]
       \ar[r, bend right, ""{name=f2}]
       \ar[Rightarrow, from=f, to=f2, "\alpha"] &
       \bullet \ar[r, "k"] &
       \bullet \ar[r, "l"] & \bullet
      \end{tikzcd}
\]
of $C$, with $\alpha = F_{\treeV}(g, f)$, $k = F_{\treeL}(h)$
and $l = F_{\treeL}(i)$ give the equality
\begingroup
\allowdisplaybreaks
\begin{gather*}
 G_{\treeV}\bigl(F_{\treeL}(i), F_{\treeL}(h)\bigr) \comp_0
  G_{\treeL}\bigl(F_{\treeL}(g)\comp_0 F_{\treeL}(f)\bigr) \comp_1
  G_{\treeVRight}\bigl(F_{\treeL}(i)\comp_0 F_{\treeL}(h), F_{\treeV}(g, f)\bigr) \nobreak
  \\ \comp_1  \nobreak \\
  GF_{\treeL}(i) \comp_0 GF_{\treeL}(h) \comp_0 G_{\treeLL}\bigl(F_{\treeV}(g, f)\bigr)
   \comp_1
  G_{\treeW}\bigl(F_{\treeL}(i), F_{\treeL}(h), F_{\treeL}(gf)\bigr)  \nobreak
  \\ =  \nobreak \\
  G_{\treeW}\bigl(F_{\treeL}(i), F_{\treeL}(h), F_{\treeL}(g)\comp_0 F_{\treeL}(f)\bigr)
   \comp_1
  G_{\treeLL}\bigl(F_{\treeL}(i)\comp_0 F_{\treeL}(h) \comp_0 F_{\treeV}(g, f)\bigr)  \nobreak
  \\ \comp_1  \nobreak \\
  GF_{\treeL}(i)\comp_0 G_{\treeV}\bigl(F_{\treeLog}(h), F_{\treeLog}(g)\comp_0 F_{\treeLog}(f)\bigr)
	 \comp_1
    G_{\treeVRight}\bigl(F_{\treeLog}(i), F_{\treeLog}(h) \comp_0 F_{\treeV}(g, f)\bigr)  \nobreak
	\\ \comp_1   \nobreak \\
  GF_{\treeL}(i) \comp_0 G_{\treeVRight}\bigl(F_{\treeLog}(h), F_{\treeV}(g, f)\bigr)
   \comp_1 
  G_{\treeV}\bigl(F_{\treeLog}(i), F_{\treeLog}(h)\comp_0 F_{\treeLog}(gf)\bigr)
\end{gather*}
\endgroup
of $3$-cells of $C$. If we $1$-post-compose, \ie we whisker with $\comp_1$,
with the $2$-cell
\[
 GF_{\treeL}(i)\comp_0 GF_{\treeL}(h) \comp_0 G_{\treeV}\bigl(F_{\treeL}(g), F_{\treeL}(f)\bigr)
\]
of $C$ and we $1$-precompose with the $3$-cell
$G_{\treeLLL}\bigl(F_{\treeW}(i, h, gf)\bigr)$,
\ie we perform a ``vertical composition''
of $3$-cells, both members of the equality above,
then we get precisely diagram (7), which therefore commutes.
\end{paragr}

\begin{paragr}[8]
	Consider the diagram (8)
	\begin{center}
		\begin{tikzpicture}
		\node (e1-a) at (180:2) {e1-a};
		\node (i0) at (90:2) {i0};
		\node (e0-a) at (0:2) {e0-a};
		\node (e0) at (270:2) {e0};
		\draw[->, >=latex] (e1-a) -- node [left] {\ref{item:g1}} (i0);
		\draw[->, >=latex] (e1-a) -- node [left] {\ref{item:i-h-g-3}} (e0);
		\draw[->, >=latex] (i0) -- node [right] {$(g1')$} (e0-a);
		\draw[->, >=latex] (e0) -- node [right] {\ref{item:ih-g-f-1}} (e0-a);
		\end{tikzpicture}
	\end{center}
where the principal $3$-cell of ($g1'$) is
\begin{equation}
G_{\treeVLeft}\bigl(F_{\treeV}(i, h), F_{\treeL}(g)\comp_0 F_{\treeL}(f)\bigr)\,.
\end{equation}
More precisely, the $3$-cell of $C$ from e0 to e0-a, as we know, is
\begin{gather*}
 GF_{\treeL}(i)\comp_0 GF_{\treeL}(h) \comp_0 G_{\treeV}\bigl(F_{\treeL}(g), F_{\treeL}(f)\bigr)
 \\ \comp_1 \\
 G_{\treeLL}\bigl(F_{\treeV}(i, h)\bigr) \comp_0 G_{\treeL}\bigl(F_{\treeL}(g) \comp_0 F_{\treeL}(f)\bigr)
 \comp_1 G_{\treeVLeft}\bigl(F_{\treeV}(i, h), F_{\treeL}(g)\comp_0 F_{\treeL}(f)\bigr)
 \\ \comp_1 \\
 G_{\treeLL}\bigl(F_{\treeV}(ih, gf)\bigr)\,;
\end{gather*}
the $3$-cell from e1-a to e0 is, as we know,
\begin{gather*}
 GF_{\treeL}(i)\comp_0 GF_{\treeL}(h) \comp_0 G_{\treeV}\bigl(F_{\treeL}(g), F_{\treeL}(f)\bigr)
 \\ \comp_1 \\
 G_{\treeL}\bigl(F_{\treeL}(i) \comp_0 F_{\treeL}(h)\bigr) \comp_0 G_{\treeLL}\bigl(F_{\treeV}(g, f)\bigr)
 \comp_1 G_{\treeVLeft}\bigl(F_{\treeV}(i, h) \comp_0 F_{\treeL}(gf)\bigr)
 \\ \comp_1 \\
 G_{\treeLL}\bigl(F_{\treeV}(ih, gf)\bigr)\,;
\end{gather*}
the $3$-cell from e0 to e0-a is
\begin{gather*}
 GF_{\treeL}(i)\comp_0 GF_{\treeL}(h) \comp_0 G_{\treeV}\bigl(F_{\treeL}(g), F_{\treeL}(f)\bigr)
 \\ \comp_1 \\
 G_{\treeVLeft}\bigl(F_{\treeV}(i, h), F_{\treeL}(g) \comp_0 F_{\treeL}(f)\bigr)
 \comp_1 G_{\treeLL}\bigl(F_{\treeL}(ih) \comp_0 F_{\treeV}(g, f)\bigr)
 \\ \comp_1 \\
 G_{\treeLL}\bigl(F_{\treeV}(ih, gf)\bigr)\,;
\end{gather*}
finally the $3$-cell from e1-a to e0, as we know, is
\begin{gather*}
 GF_{\treeL}(i)\comp_0 GF_{\treeL}(h) \comp_0 G_{\treeV}\bigl(F_{\treeL}(g), F_{\treeL}(f)\bigr)
 \\ \comp_1 \\
 G_{\treeVRight}\bigl(F_{\treeL}(i) \comp_0 F_{\treeL}(h), F_{\treeV}(g, f)\bigr)
 \comp_1 G_{\treeLL}\bigl(F_{\treeV}(i, h) \comp_0 F_{\treeL}(gf)\bigr)
	\\ \comp_1 \\
	G_{\treeLL}\bigl(F_{\treeV}(ih, gf)\bigr)\,.
\end{gather*}
The coherence for the tree
\scalebox{0.3}{
	\begin{forest}
		for tree={%
			label/.option=content,
			grow'=north,
			content=,
			circle,
			fill,
			minimum size=3pt,
			inner sep=0pt,
			s sep+=15,
		}
		[
		[ [] ]
		[ [] ]
		]
	\end{forest}
}
applied to the pasting diagram
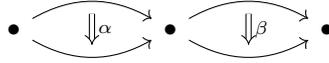
\[
\begin{tikzcd}[column sep=4.5em]
\bullet
\ar[r, bend left, ""{below, name=f1}]
\ar[r, bend right, ""{name=f2}]
\ar[Rightarrow, from=f1, to=f2, "\alpha"]
&
\bullet
\ar[r, bend left, ""{below, name=g1}]
\ar[r, bend right, ""{name=g2}]
\ar[Rightarrow, from=g1, to=g2, "\beta"]
& \bullet
\end{tikzcd}
\]
of $C$, where $\alpha = F_{\treeV}(g, f)$ and $\beta = F_{\treeV}(h, i)$,
gives us the equality
\begingroup
\allowdisplaybreaks
\begin{gather*}
 G_{\treeLL}\bigl(F_{\treeV}(i, h)\bigr) \comp_0 G_{\treeL}\bigl(F_{\treeL}(g) \comp_0 F_{\treeL}(f)\bigr)
 \comp_1 G_{\treeVLeft}\bigl(F_{\treeV}(i, h), F_{\treeL}(g)\comp_0 F_{\treeL}(f)\bigr) \nobreak
 \\ \comp_1 \nobreak\\
 G_{\treeL}\bigl(F_{\treeL}(i) \comp_0 F_{\treeL}(h)\bigr) \comp_0 G_{\treeLL}\bigl(F_{\treeV}(g, f)\bigr)
 \comp_1 G_{\treeVLeft}\bigl(F_{\treeV}(i, h) \comp_0 F_{\treeL}(gf)\bigr)
 \\ = \\
 G_{\treeVLeft}\bigl(F_{\treeV}(i, h), F_{\treeL}(g) \comp_0 F_{\treeL}(f)\bigr)
 \comp_1 G_{\treeLL}\bigl(F_{\treeL}(ih) \comp_0 F_{\treeV}(g, f)\bigr) \nobreak
 	\\ \comp_1  \nobreak \\
 G_{\treeVRight}\bigl(F_{\treeL}(i) \comp_0 F_{\treeL}(h), F_{\treeV}(g, f)\bigr)
 \comp_1 G_{\treeLL}\bigl(F_{\treeV}(i, h) \comp_0 F_{\treeL}(gf)\bigr)
\end{gather*}
\endgroup
of $3$-cells of $C$. By $1$-post-composing with the $2$-cell
\[
 GF_{\treeL}(i)\comp_0 GF_{\treeL}(h) \comp_0 G_{\treeV}\bigl(F_{\treeL}(g), F_{\treeL}(f)\bigr)
\]
of $C$ and $1$-precomposing with the $2$-cell $G_{\treeLL}\bigl(F_{\treeV}(ih, gf)\bigr)$
we get the commutativity of the diagram (8).
\end{paragr}

\begin{paragr}[9]
Consider the diagram (9)
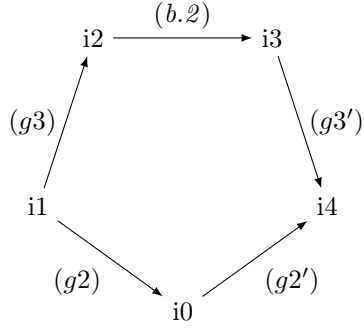
\begin{center}
	\begin{tikzpicture}
	\node (i0) at (270:2) {i0};
	\node (i1) at (198:2) {i1};
	\node (i2) at (126:2) {i2};
	\node (i3) at (54:2) {i3};
	\node (i4) at (-18:2) {i4};
	\draw[->,>=latex] (i1) -- node [below left] {\ref{item:g2}} (i0);
	\draw[->,>=latex] (i0) -- node [below right] {$(g2')$} (i4);
	\draw[->,>=latex] (i1) -- node [left] {\ref{item:g3}} (i2);
	\draw[->,>=latex] (i2) -- node [above] {\ref{item:b2}} (i3);
	\draw[->,>=latex] (i3) -- node [right] {$(g3')$} (i4);
	\end{tikzpicture}
\end{center}
where
\begin{gather*}
 G_{\treeW}\bigl(F_{\treeL}(i)\comp_0 F_{\treeL}(h), F_{\treeL}(g), F_{\treeL}(f)\bigr)
 \\\comp_1 \\
 \notag G_{\treeLLL}\bigl(F_{\treeW}(ih, g, f)\bigr)
\end{gather*}
is the principal $3$-cell of ($g2'$) and
\begin{gather*}
 G_{\treeW}\bigl(F_{\treeL}(i), F_{\treeL}(h)\comp_0 F_{\treeL}(g), F_{\treeL}(f)\bigr)
 \\ \comp_1 \\
 \notag G_{\treeLLL}\bigl(F_{\treeW}(i, hg, f)\bigr)
\end{gather*}
is the principal $3$-cell of ($g3'$). More precisely, the $3$-cell of $C$ from i0 to i4 is
\begin{gather*}
 G_{\treeV}\bigl(F_{\treeL}(i), F_{\treeL}(h)\bigr) \comp_0 GF_{\treeL}(g) \comp_0 GF_{\treeL}(f)
 \comp_1 G_{\treeW}\bigl(F_{\treeL}(i)\comp_0 F_{\treeL}(h), F_{\treeL}(g), F_{\treeL}(f)\bigr)
 \\ \comp_1 \\
 G_{\treeLL}\bigl(F_{\treeV}(i, h) \comp_0 F_{\treeL}(g) \comp_0 F_{\treeL}(f) \bigr)
 \comp_1 G_{\treeLLL}\bigl(F_{\treeW}(ih, g, f)\bigr)\,;
\end{gather*}
the $3$-cell from i1 to i0, as we know, is
\begin{gather*}
 GF_{\treeL}(i)\comp_0 GF_{\treeL}(h) \comp_0 G_{\treeV}\bigl(F_{\treeL}(g), F_{\treeL}(f)\bigr)
  \comp_1 
 G_{\treeW}\bigl(F_{\treeL}(i), F_{\treeL}(h), F_{\treeL}(g)\comp_0 F_{\treeL}(f)\bigr)
 \\ \comp_1 \\
 G_{\treeLL}\bigl(F_{\treeL}(i)\comp_0 F_{\treeL}(h) \comp_0 F_{\treeV}(g, f)\bigr)
  \comp_1 
 G_{\treeLLL}\bigl(F_{\treeW}(i, h, gf)\bigr)\,;
\end{gather*}
the $3$-cell from i3 to i4 is
\begin{gather*}
 G_{\treeW}\bigl(F_{\treeL}(i), F_{\treeL}(h),  F_{\treeL}(g)\bigr) \comp_0 GF_{\treeL}(f)
 \comp_1 G_{\treeV}\bigl(F_{\treeL}(i)\comp_0 F_{\treeL}(H) \comp_0 F_{\treeL}(g), F_{\treeL}(f)\bigr)
 \\ \comp_1 \\
 G_{\treeLLL}\bigl(F_{\treeW}(i, h, g) \comp_0 F_{\treeL}(f)\bigr) \comp_1
 G_{\treeLL}\bigl(F_{\treeV}(ihg, f)\bigr)\,;
\end{gather*}
the $3$-cell from i2 to i3, as we know, is
\begin{gather*}
 GF_{\treeL}(i) \comp_0 G_{\treeV}\bigl(F_{\treeL}(h), F_{\treeL}(g)\bigr)
 \comp_0 GF_{\treeL}(f)
  \comp_1
 G_{\treeW}\bigl(F_{\treeL}(i), F_{\treeL}(h)\comp_0 F_{\treeL}(g), F_{\treeL}(f)\bigr)
 \\ \comp_1 \\
 G_{\treeLL}\bigl(F_{\treeL}(h) \comp_0 F_{\treeL}(h, g)\comp_0 F_{\treeL}(f)\bigr)
  \comp_1 
 G_{\treeLLL}\bigl(F_{\treeW}(i, hg, f)\bigr)\,;
\end{gather*}
finally we also already know that the $3$-cell from i1 to i2 is
\begin{gather*}
 GF_{\treeL}(i) \comp_0 G_{\treeW}\bigl(F_{\treeL}(h), F_{\treeL}(g), F_{\treeL}(f)\bigr)
  \comp_1 
 G_{\treeV}\bigl(F_{\treeL}(i), F_{\treeL}(h) \comp_0 F_{\treeL}(g) \comp_0 F_{\treeL}(f)\bigr)
 \\ \comp_1 \\
 G_{\treeLLL}\bigl(F_{\treeL}(i) \comp_0 F_{\treeW}(h, g, f)\bigr)
  \comp_1 
 G_{\treeLL}\bigl(F_{\treeV}(i, hgf)\bigr)\,.
\end{gather*}
Notice that diagram (9) is actually formed by ``vertical compositions'' of
these $3$-cells; that is, the first lines of these $3$-cells are $2$-composable
and the same holds for the $3$-cells of the second line. Observe that in order to
$2$-compose the $3$-cell from i1 to i0 with the $3$-cell from i0 to i4 we use
the relations
\begin{gather*}
 G_{\treeLL}\bigl(F_{\treeL}(i)\comp_0 F_{\treeL}(h) \comp_0 F_{\treeV}(g, f)\bigr)
 \comp_1
 G_{\treeLL}\bigl(F_{\treeV}(i, h) \comp_0 F_{\treeL}(g\comp_0 f)\bigr)
 \\ = \\
 G_{\treeLL}\bigl(F_{\treeL}(i)\comp_0 F_{\treeL}(h) \comp_0 F_{\treeV}(g, f)
	 \comp_1 F_{\treeV}(i, h) \comp_0 F_{\treeL}(g \comp_0 f)\bigr)
 \\ = \\
 G_{\treeLL}\bigl(F_{\treeV}(i, h) \comp_0 F_{\treeL}(g) \comp_0 F_{\treeL}(f)
	 \comp_1 F_{\treeL}(i \comp_0 h) \comp_0 F_{\treeV}(g, f)\bigr)
 \\ = \\
 G_{\treeLL}\bigl(F_{\treeV}(i, h) \comp_0 F_{\treeL}(g) \comp_0 F_{\treeL}(f)\bigr)
 \comp_1 G_{\treeLL}\bigl(F_{\treeL}(i \comp_0 h) \comp_0 F_{\treeV}(g, f)\bigr)\,,
\end{gather*}
where the first and the last equality are instances of the coherence for the tree
\scalebox{0.3}{
	\begin{forest}
		for tree={%
			label/.option=content,
			grow'=north,
			content=,
			circle,
			fill,
			minimum size=3pt,
			inner sep=0pt,
			s sep+=15,
		}
		[ 
		 [ [] ]
		 [ [] ]
		]
	\end{forest}
},
\ie the $0$-composition of $2$-cells, and the equality in the middle is simply given
by the interchange law.
Now, the coherence for the tree
\scalebox{0.3}{
	\begin{forest}
		for tree={%
			label/.option=content,
			grow'=north,
			content=,
			circle,
			fill,
			minimum size=3pt,
			inner sep=0pt,
			s sep+=15,
		}
		[
		 [][][][]
		]
	\end{forest}
}
applied to the pasting diagram
\[
\begin{tikzcd}
\bullet \ar[r, "f"] & \bullet \ar[r, "g"] & \bullet \ar[r, "h"] & \bullet \ar[r, "i"] & \bullet
\end{tikzcd}
\]
of $0$-composable $1$-cells of $A$ gives us the equality
\begingroup
\allowdisplaybreaks
\begin{gather*}
 F_{\treeV}(i, h) \comp_0 F_{\treeLog}(g) \comp_0 F_{\treeLog}(f)
 \comp_1 F_{\treeW}(ih, g, h) \nobreak \\
 \comp_2 \nobreak \\
 F_{\treeLog}(i) \comp_0 F_{\treeLog}(h) \comp_0 F_{\treeV}(g, f)
 \comp_1 F_{\treeW}(i, h, gf) \nobreak
 \\ = \\
 F_{\treeW}(i, h, g) \comp_0 F_{\treeLog}(f) \comp_1 F_{\treeV}(i\comp_0 h \comp_0 g, f) \nobreak \\
 \comp_2 \nobreak \\
 F_{\treeLog}(i) \comp_0 F_{\treeV}(h, g) \comp_0 F_{\treeLog}(f) \comp_1
 F_{\treeW}(i, h \comp_0 g, f) \nobreak \\
 \comp_2 \nobreak \\
 F_{\treeLog}(i) \comp_0 F_{\treeW}(h, g, f) \comp_1 F_{\treeV}(i, h\comp_0 g \comp_0 f)
\end{gather*}
\endgroup
of $3$-cells of $B$.
Applying $G_{\treeLLL}$ to both terms of this equality and using the coherences
for the trees
\scalebox{0.3}{
	\begin{forest}
		for tree={%
			label/.option=content,
			grow'=north,
			content=,
			circle,
			fill,
			minimum size=3pt,
			inner sep=0pt,
			s sep+=15,
		}
		[
		 [[[][]]]
		]
	\end{forest}
},
\scalebox{0.3}{
	\begin{forest}
		for tree={%
			label/.option=content,
			grow'=north,
			content=,
			circle,
			fill,
			minimum size=3pt,
			inner sep=0pt,
			s sep+=15,
		}
		[
		[[[]][]]
		]
	\end{forest}
} and
\scalebox{0.3}{
	\begin{forest}
		for tree={%
			label/.option=content,
			grow'=north,
			content=,
			circle,
			fill,
			minimum size=3pt,
			inner sep=0pt,
			s sep+=15,
		}
		[
		[[][[]]]
		]
	\end{forest}
}
we get the following equality
\begingroup
\allowdisplaybreaks
\begin{gather*}
 G_{\treeLL}\bigl( F_{\treeV}(i, h) \comp_0 F_{\treeLog}(g) \comp_0 F_{\treeLog}(f)\bigr)
 \comp_1 G_{\treeLLL}\bigl(F_{\treeW}(ih, g, h) \bigr) \nobreak
 \\ \comp_2  \nobreak \\
 G_{\treeLL}\bigl( F_{\treeLog}(i) \comp_0 F_{\treeLog}(h) \comp_0 F_{\treeV}(g, f)	\bigr)
 \comp_1 G_{\treeLLL}\bigl(F_{\treeW}(i, h, gf) \bigr) \nobreak
 \\ = \\
 G_{\treeLLL}\bigl( F_{\treeW}(i, h, g) \comp_0 F_{\treeLog}(f)\bigr)
 \comp_1 G_{\treeLL}\bigl(F_{\treeV}(i\comp_0 h \comp_0 g, f) \bigr) \nobreak
 \\ \comp_2  \nobreak\\
 G_{\treeLL}\bigl( F_{\treeLog}(i) \comp_0 F_{\treeV}(h, g) \comp_0 F_{\treeLog}(f)\bigr) \comp_1
 G_{\treeLLL}\bigl(F_{\treeW}(i, h \comp_0 g, f) \bigr) \nobreak
 \\ \comp_2 \nobreak \\
 G_{\treeLLL}\bigl( F_{\treeLog}(i) \comp_0 F_{\treeW}(h, g, f)\bigr)
 \comp_1 G_{\treeLL} \bigl(F_{\treeV}(i, h\comp_0 g \comp_0 f)\bigr)\,;
\end{gather*}
\endgroup
applying instead the coherence for the tree
\scalebox{0.3}{
	\begin{forest}
		for tree={%
			label/.option=content,
			grow'=north,
			content=,
			circle,
			fill,
			minimum size=3pt,
			inner sep=0pt,
			s sep+=15,
		}
		[
		[][][][]
		]
	\end{forest}
}
to the pasting scheme given by the four $0$-composable $1$-cells of $B$
\[
\begin{tikzcd}
\bullet \ar[r, "F_{\treeL}(f)"] & \bullet \ar[r, "F_{\treeL}(g)"] &
\bullet \ar[r, "F_{\treeL}(h)"] & \bullet \ar[r, "F_{\treeL}(i)"] & \bullet
\end{tikzcd}
\]
we get the equality
\begingroup
\allowdisplaybreaks
\begin{gather*}
 G_{\treeV}\bigl(F_{\treeL}(i), F_{\treeL}(h)\bigr) \comp_0 GF_{\treeL}(g) \comp_0 GF_{\treeL}(f)
 \comp_1 G_{\treeW}\bigl(F_{\treeL}(i)\comp_0 F_{\treeL}(h), F_{\treeL}(g), F_{\treeL}(f)\bigr) \nobreak
 \\ \comp_1 \nobreak\\
 GF_{\treeL}(i)\comp_0 GF_{\treeL}(h) \comp_0 G_{\treeV}\bigl(F_{\treeL}(g), F_{\treeL}(f)\bigr)
 \comp_1 
 G_{\treeW}\bigl(F_{\treeL}(i), F_{\treeL}(h), F_{\treeL}(g)\comp_0 F_{\treeL}(f)\bigr) \nobreak
 \\ = \\
 G_{\treeW}\bigl(F_{\treeL}(i), F_{\treeL}(h), F_{\treeL}(g)\bigr) \comp_0 GF_{\treeL}(f)
 \comp_1 G_{\treeV}\bigl(F_{\treeL}(i)\comp_0 F_{\treeL}(H) \comp_0 F_{\treeL}(g), F_{\treeL}(f)\bigr) \nobreak
 \\ \comp_2 \nobreak\\
 GF_{\treeL}(i) \comp_0 G_{\treeV}\bigl(F_{\treeL}(h), F_{\treeL}(g)\bigr)
 \comp_0 GF_{\treeL}(f)
 \comp_1
 G_{\treeW}\bigl(F_{\treeL}(i), F_{\treeL}(h)\comp_0 F_{\treeL}(g), F_{\treeL}(f)\bigr) \nobreak
 \\ \comp_1 \nobreak \\
 GF_{\treeL}(i) \comp_0 G_{\treeW}\bigl(F_{\treeL}(h), F_{\treeL}(g), F_{\treeL}(f)\bigr)
 \comp_1 
 G_{\treeV}\bigl(F_{\treeL}(i), F_{\treeL}(h) \comp_0 F_{\treeL}(g) \comp_0 F_{\treeL}(f)\bigr)\,.
\end{gather*}
\endgroup
Notice that the $1$-composition line by line of the $3$-cells of these two equalities give
precisely the $3$-cells defining diagram (9), which by the interchange law is therefore commutative.
\end{paragr}

The previous paragraph ends the proof of the coherence presented
in paragraph~\ref{paragr:pentagon_coherence}, hence achieving the
following proposition.

\begin{prop}
	Let $F \colon A \to B$ and $G \colon B \to C$ be two normalised oplax $3$-functors.
	If $A$ is a $1$-category, then the data of $GF$ defined in paragraph~\ref{paragr:def_cellular_to_simplicial}
	define a normalised oplax $3$-functor.
\end{prop}

\begin{rem}
	It is suggestive, in light of the coherences described above, to decorate the polygons of figure~\eqref{fig:diagram_composition}
	with Stasheff trees. We give such a representation
	in figure~\eqref{fig:diagram_composition_stasheff}
	(the unlabelled polygons being commutative by exchange law),
	but we do not pursue this informal approach any further.
\end{rem}

\begin{figure}
	\centering
	\begin{tikzpicture}[scale=2]
	\foreach \i in {0,1,2,3,4} {
		\tikzmath{\a = 270 - (72 * \i);}
		\node (i\i) at (\a:1) {i\i};
		\node (e\i) at (\a:3) {e\i};
	}
	\foreach \j in {1,2,3,4} {
		\tikzmath{\a = 270 - (72 * \j);}
		\pgfmathtruncatemacro\bj{\j-1}
		\node (m\j) at (\a:2) {m\j};
		\node (e\j-a) at ($(e\bj)!0.33!(e\j)$) {e\j-a};
		\node (e\j-b) at ($(e\bj)!0.67!(e\j)$) {e\j-b};
	}
	\node (e0-a) at ($(e0)!0.33!(e4)$) {e0-a};
	\node (e0-b) at ($(e0)!0.66!(e4)$) {e0-b};
	%
	%
	\draw[->, >=latex] (e1) to (e2-a);
	\draw[->, >=latex] (e2-a) to (m1);
	\draw[->, >=latex] (e1) to (e1-b);
	\draw[->, >=latex] (e1-b) to (m1);
	%
	\draw[->, >=latex] (m4) to (e4-b);
	\draw[->, >=latex] (e4-b) to (e4);
	\draw[->, >=latex] (m4) to (e0-b);
	\draw[->, >=latex] (e0-b) to (e4);
	%
	%
	\draw[->, >=latex] (m1) to (i1);
	\draw[->, >=latex] (e2-a) to (i1);
	\node at (190:1.8) {$\sttwo$};
	\draw[->, >=latex] (i4) to (m4);
	\draw[->, >=latex] (i4) to (e4-b);
	\node at (-10:1.8) {$\sttwop$};
	%
	%
	\draw[->, >=latex] (e2-a) to (e2-b);
	\draw[->, >=latex] (e2-b) to (i2);
	\draw[->, >=latex] (i1) to (i2);
	\node at ($(e2-a)!0.5!(i2)$) {$\stthree$};
	\draw[->, >=latex] (e4-a) to (e4-b);
	\draw[->, >=latex] (i3) to (e4-a);
	\draw[->, >=latex] (i3) to (i4);
	\node at ($(i3)!0.5!(e4-b)$) {$\stthreep$};
	%
	%
	\draw[->, >=latex] (e2-b) to (m2);
	\draw[->, >=latex] (m2) to (i2);
	\node at (134:1.8) {$\stfour$};
	\draw[->, >=latex] (i3) to (m3);
	\draw[->, >=latex] (m3) to (e4-a);
	\node at (46:1.8) {$\stfourp$};
	%
	%
	\draw[->, >=latex] (e2-b) to (e2);
	\draw[->, >=latex] (e2) to (e3-a);
	\draw[->, >=latex] (m2) to (e3-a);
	%
	\draw[->, >=latex] (e3-b) to (e3);
	\draw[->, >=latex] (e3) to (e4-a);
	\draw[->, >=latex] (e3-b) to (m3);
	%
	%
	\draw[->, >=latex] (e3-a) to (e3-b);
	\draw[->, >=latex] (i2) to (i3);
	\node at ($(m2)!0.5!(m3)$) {$\stsix$};
	%
	%
	\draw[->, >=latex] (i1) to (i0);
	\draw[->, >=latex] (e1-b) to (e1-a);
	\draw[->, >=latex] (e1-a) to (i0);
	\node at ($(e1-b)!0.5!(i0)$) {$\stseven$};
	\draw[->, >=latex] (i0) to (i4);
	\draw[->, >=latex] (i0) to (e0-a);
	\draw[->, >=latex] (e0-a) to (e0-b);
	\node at ($(i0)!0.5!(e0-b)$) {$\stsevenp$};
	%
	%
	\draw[->, >=latex] (e1-a) to (e0);
	\draw[->, >=latex] (e0) to (e0-a);
	\node at ($(i0)!0.6!(e0)$) {$\steight$};
	%
	%
	\node at (0:0) {$\stnine$};
	\end{tikzpicture}
	\caption{The diagram for the coherence $\protect\treeVV$, again.}
	\label{fig:diagram_composition_stasheff}
\end{figure}

\begin{prop}
    Let $u \colon A \to A'$, $F \colon A \to B$ and $G \colon B \to C$
    be normalised oplax $3$-functors, where $A$ and $a'$ are $1$-categories.
    Then $G(Fu) = (GF)u$.
\end{prop}

\begin{proof}
    Notice that a normalised oplax $3$-functor $u \colon A' \to A$
    between $1$\nbd-categories is simply a $1$-functor and
    one immediately checks that the equality
    $G(Fu) = (GF)u$ of normalised oplax $3$-functor
    is verified.
\end{proof}

\begin{lemme}
	Let $B$ be a $3$-category.
	The $n$-simplices of $\SN(B)$ are in bijection with the
	oplax normalised $3$-functors $\Deltan{n} \to B$.
\end{lemme}

\begin{proof}
	For any $3$-category $B$ the simplicial set $\SN(B)$
	is $4$-coskeletal (see~\cite{Street}) and so we have to check that
	the set of normalised oplax $3$-functors $x \colon \Deltan{n} \to B$
	are in bijection with the set of morphisms of simplicial sets
	$x \colon \text{Sk}_4(\Deltan{n}) \to \SN(B)$; it is enough to define
	the latter ones on the $i$-simplices of $\Deltan{n}$, with $i=0,1, \dots, 4$.
	But such a definition corresponds precisely to the data
	$\treeDot$, $\treeL$, $\TreeV$ and $\TreeW$ joint with the
	coherence $\treeVV$ of a normalised oplax $3$-functor from
	the $1$-category $\Deltan{i}$ to the $3$-category $B$.
\end{proof}

\begin{thm}\label{thm:cellular-simplicial}
 Let $G \colon B \to C$ be a normalised oplax $3$-functor.
 Then there is a morphism of simplicial sets $\Nl(G) \colon \SN(B) \to \SN(C)$,
 where, for any $n  \ge 0$, an $n$\nbd-simplex of $\SN(B)$ corresponding to
 a normalised oplax $3$-functor $x \colon \Deltan{n} \to B$ is sent to
 the $n$-simplex of $C$ corresponding to the normalised oplax $3$-functor
 $Gx \colon \Deltan{n}\to C$.
\end{thm}

\begin{proof}
 To any normalised oplax $3$-functor $G \colon B \to C$,
 we have seen along this section that we can define a composition
 $Gx \colon \Deltan{n} \to C$ which is still a normalised oplax $3$\nbd-functor
 and so it canonically corresponds to an $n$-simplex $\Nl G_n(x)$ of $\SN(C)$.
 The functoriality of this correspondence, given by the preceding proposition,
 implies the naturality of the functions $\Nl G_n$, which assemble to a morphism
 of simplicial sets $\Nl(G) \colon \SN(B) \to \SN(C)$.
\end{proof}

\begin{exem}\label{exem:sup_we}
	Let $C$ be a small $3$-category and consider the normalised oplax
	$3$-functor $\sup \colon i_{\cDelta}(N_3(C))\to C$ defined
	in Example~\ref{exem:sup}. One checks that the
	associated morphism of simplicial sets
	$\Nl(\sup)\colon N i_{\cDelta}(N_3(C)) \to N_3(C)$ coincides
	with the morphism of simplicial sets called $\tau_{N_3(C)}$
	in paragraph~7.3.14 of~\cite{CisinskiHigherCats}.
	Hence, Proposition~7.3.15 of~\loccit implies that
	$\Nl(\sup)$ is a simplicial weak equivalence.
\end{exem}

\section{The simplicial definition}\label{sec:simplicial}

	It is expected that a good notion of normalised oplax $3$-functor
	would satisfy the following property: for any $3$-categories $A$ and $B$,
	the set of normalised oplax $3$\nbd-func\-tors from $A$ to $B$ is in bijection
	with the set of simplicial morphisms from $N_3(A)$ to $N_3(B)$.
	Nevertheless, a careful investigation of this latter notion
	shows that they might not be optimal as they fail to preserve
	the underlying $3$-graph.
	Indeed, we will analyse the case where $A$ is the ``$2$-disk'', \ie
	the $2$-category with two parallel $1$-cells and a single $2$-cell between them,
	and $B$ is the ``invertible $3$-disk'', \ie
	the $3$-category with two parallel $1$-cells, two parallel $2$-cells between them
	and a single invertible $3$-cell between these $2$-cells, and we show that there are more
	simplicial morphisms than expected between the respective Street nerves.
	On the one hand, the $2$-category $A$ has no compositions and so
	the normalised oplax $3$-functors from $A$ to $B$ should coincide with
	the strict $3$-functors. On the other hand, there are simplicial morphisms
	from $N_3(A)$ to $N_3(B)$ which do not come from the nerve of strict $3$-functors.
	This is a consequence of the fact that, for instance, there are two ways to
	capture the $2$-cell of $A$ with a $2$-simplex of $N_3(A)$ and these two
	different ways are related by $3$-simplices which are sent by any simplicial
	morphism $N_3(A) \to N_3(B)$ to $3$-simplices of $N_3(B)$ for which the main $3$-cell is invertible.
	Said otherwise, the different ways to encode cells, or simple compositions of cells,
	with simplices are linked together by higher simplices with the property
	of having the cell of greatest dimension invertible; these higher simplices
	act as invertible constraints for morphisms between Street nerves of $3$-categories
	and it is therefore natural to imagine that a normalised oplax $3$-functor would
	correspond to a simplicial morphism for which all these higher simplices acting
	as constraints have \emph{trivial} greatest cell, instead of only invertible. In order to identify such constraints,
	we shall examine in further detail the nerve $\Nl(F)$ of any normalised oplax $3$-functor $F$.
	
\subsection{Case study: {$\Dn{2}$}}
    
    \begin{paragr}
    	Consider the $2$-category $\Dn{2}$
    	\[
	    	\begin{tikzcd}[column sep=4.5em]
		    	\disk^0_0
		    	\ar[r, bend left, "\disk^0_1", ""'{name=f}]
		    	\ar[r, bend right, "\disk^1_1"', ""{name=g}]
		    	\ar[Rightarrow, from=f, to=g, "\disk"]
		    	& \disk^1_0
	    	\end{tikzcd}\ .
    	\]
    	We know from paragraphs~\ref{paragr:encode_2cell} that the simplicial set~$\SN(\Dn{2})$
    	has at least two non-degenerate $2$-simplices $\disk_l$ and $\disk_r$,
    	at least two non-degenerate $3$\hyp{}simplices that we shall call
    	$\tau_{\text{u}}$ and $\tau_{\text{d}}$
    	and also at least two non-degenerate $4$\hyp{}simplices
    	that we shall name $x_\tau$ and $y_\tau$.
    	In fact, Ozornova and Rovelli have shown in~\cite{OzornovaRovelliDisk} that
    	these are the only non-degenerate $i$-simplices, for $i=2, 3, 4$.
    	Therefore we get an explicit description of the $3$\hyp{}category $\ti{3}\cO\SN(\Dn{2})$
    	given by
    	\[
	    	\begin{tikzcd}[column sep=4.5em]
	    	 \disk^0_0
	    	 \ar[r, bend left=60, "\disk^0_1", ""'{name=f}]
	    	 \ar[r, bend right=60, "\disk^1_1"', ""{name=g}]
	    	 \ar[Rightarrow, from=f, to=g, shift right=0.5em, bend right, shorten <=1mm, shorten >=1mm, "\disk_l"', ""{name=al}]
	    	 \ar[Rightarrow, from=f, to=g, shift left=0.5em, bend left, shorten <=1mm, shorten >=1mm, "\disk_r", ""'{name=ar}]
	    	 \arrow[triple, from=al, to=ar, "\tau_{\text{d}}", "\cong"']{}
	    	 & \disk^1_0
	    	\end{tikzcd}
    	\]
    	where the inverse of the $3$-cell $\tau_{\text{d}}$ is given by the $3$-cell
    	$\tau_{\text{u}} \colon \disk_r \to \disk_l$.
    	We shall call $\Dn 3^\sharp$
    	this $3$-category. This is motivated by the fact that the $2$-skeleton of this
    	$3$-category is equal to that of~$\Dn{3}$,
    	but the top dimensional cell is invertible.
    \end{paragr}
    
    \begin{paragr}
    	There are no compositions of cells in the $3$-category $\Dn{2}$
    	and therefore a good notion of oplax $3$-functor $F$ with source
    	$\Dn{2}$ and target a $3$\hyp{}category $B$ should coincide
    	with a strict $3$\hyp{}functor, since there is no composition
    	to ``laxify'' in $\Dn{2}$. This is not the case if we set
    	the oplax $3$-functors from $A$ to $B$, where $A$ and $B$ are two
    	small $3$-categories, to be the set
    	\[
	    	\Hom_{\EnsSimp}(\SN(A), \SN(B)) \cong \Hom_{\nCat{3}}(\cOn{3}\SN (A), B)\,.
    	\]
    	Indeed, let $A = \Dn{2}$ and $B = \Dn{3}^\sharp$ and let us restrict our attention
    	to the $3$\hyp{}functors mapping $\disk^\eps_1$ to $\disk^\eps_1$, for $\eps=0, 1$,
    	\ie mapping the top cell $\disk$ of $\Dn 2$
    	to a non-trivial $2$-cell of $\Dn{3}^\sharp$.
    	We have precisely two such $3$\hyp{}functors: one sends $\disk$ to $\disk_l$ and
    	the other to $\disk_r$. Nevertheless, is we consider the $3$-functors in 
    	\[
	    	\Hom_{\nCat{3}}(\cOn{3}\SN (A), B)
	    	\cong \Hom_{\nCat{3}}({\Dn{3}}^\sharp, {\Dn{3}}^\sharp)
    	\]
    	mapping $\disk^\eps_1$ to $\disk^\eps_1$, then we count four of them
    	and they are determined by their behaviour with respect to the $3$-cell $\tau_{\text{d}}$:
    	there are two of them sending $\tau_{\text{d}}$ to the identity of $\disk_l$
    	and to the identity of $\disk_r$ respectively, which are the (mates of the) nerve of
    	the $3$-functors from $\Dn 2$ to ${\Dn{3}}^\sharp$ we considered above; furthermore,
    	there are two $3$-functors, corresponding to the automorphisms of the $3$-category
    	$\Dn{3}^\sharp$, mapping $\tau_{\text{d}}$ to itself and to $\tau_{\text{u}}$ respectively.
    \end{paragr}

	\subsection{The nerve of a normalised oplax 3-functor}\label{subsec:constraints}
    Let $A$ and $B$ be two small $3$-categories
    and consider a morphism $F \colon \SN(A) \to \SN(B)$
    of simplicial sets. In this subsection we shall study
    some of the constraints to which the morphism $F$
    is subject. As explained above, here by \emph{constraint}
    we mean an invertible and not necessarily trivial
    cell of $B$, normally a $3$-cell, which
    is the principal cell of a $3$-simplex
    $F(x)$, where on the other hand the $3$-cell of $A$
    defined by~$x$ is a trivial cell of~$A$.
    The term constraints is due to the fact that, when trying
    to extract a cellular form of oplax $3$-functor from
    such a morphism, these particular $3$-simplices
    act as additional data which do not respect the underlying
    $3$-graphs or as invertible coherences.
    
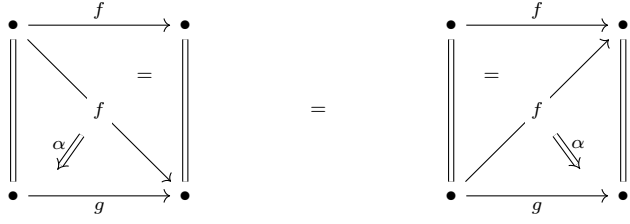
    \begin{paragr}
    Any object $a$ of $A$, that is $0$-simplex of $\SN(A)$,
    is mapped to an object $F(a)$ of $B$ and any $1$-cell
    $f \colon a \to a'$ of $A$, that is $1$-simplex of $\SN(A)$,
    is mapped to a $1$-cell $F(f) \colon F(a) \to F(a')$
    of $B$.
    
    Encoding the behaviour of higher cells in a morphism between
    the nerve of two $3$\nbd-cat\-e\-gories requires choices and
    leads to a web of coherences which increasingly become hard
    to control. The prototypical example of such a phenomenon
    is given by the way a simplicial morphism between nerves
    of $3$-categories encodes a $2$-cell. There are two different
    way of encoding a $2$-cell $\alpha$ of $A$ as a $2$-simplex.
    The main $2$-cells of $B$ of the images under $F$ of these
    two $2$-simplices are possibly two different $2$-cells; this
    can be read as the fact that simplicial morphisms do not respect
    the underlying $3$-graph in general. Nonetheless, these two
    different $2$-cells of $B$ can be proven to be linked one another
    by an invertible $3$-cell of $B$. This is described
    in detail in the next paragraph.
    \end{paragr}

        \begin{paragr}\label{paragr:encode_2cell}
    Consider a $2$-cell
    \[
     \begin{tikzcd}[column sep=4.5em]
        a
        \ar[r, bend left, "f", ""{below, name=f}]
        \ar[r, bend right, ""{name=g}, "g"']
        \ar[Rightarrow, from=f, to=g, "\alpha"]
        & a'
     \end{tikzcd}
    \]
    of $A$. The simplicial set $\SN(A)$ encodes 
    the cell $\alpha$ in two different $2$-simplices,
    namely
    \[
    \alpha_l :=
     \begin{tikzcd}[column sep=small]
      a
      \ar[rr, "f", ""{below, name=f}]
      \ar[rd, equal, "1_a"']
      &&
        a'
      \\
      &
       a
       \ar[ur, "g"']
      &
      \ar[Rightarrow, from=f, to=2-2, shorten >=3pt, "\alpha"]
     \end{tikzcd}
     \quadet
     \alpha_r :=
     \begin{tikzcd}[column sep=small]
      a
      \ar[rr, "f", ""{below, name=f}]
      \ar[rd, "g"']
      &&
        a'
      \\
      &
       a'
       \ar[ur, equal, "1_{a'}"']
      &
      \ar[Rightarrow, from=f, to=2-2, shorten >=2pt, "\alpha"]
     \end{tikzcd}\ .
    \]
    These two $2$-simplices of $\SN(A)$ are linked together
    by the following two non-degenerate $3$\hyp{}simplices
    \begin{center}
     \begin{tikzpicture}[scale=1.6, font=\footnotesize]
     \squares{%
     	/squares/label/.cd,
     	0=$\bullet$, 1=$\bullet$, 2=$\bullet$, 3=$\bullet$,
     	01={}, 12=$g$, 23={}, 02=$g$, 03=$f$, 13=$g$,
     	012=${=}$, 023=$\alpha$, 123=${=}$, 013=$\alpha$,
     	0123=${=}$,
     	/squares/arrowstyle/.cd,
     	01={equal}, 23={equal},
     	012={phantom, description}, 123={phantom, description},
     	0123={phantom, description},
     	/squares/labelstyle/.cd,
     	012={anchor=center}, 123={anchor=center},
     	0123={anchor=center}
     }
     \end{tikzpicture}
    \end{center}
    and
    \begin{center}
    	\begin{tikzpicture}[scale=1.6, font=\footnotesize]
    	\squares{%
    		/squares/label/.cd,
    		0=$\bullet$, 1=$\bullet$, 2=$\bullet$, 3=$\bullet$,
    		01={}, 12=$g$, 23={}, 02=$f$, 03=$f$, 13=$f$,
    		012=$\alpha$, 023=${=}$, 123=$\alpha$, 013=${=}$,
    		0123=${=}$,
    		/squares/arrowstyle/.cd,
    		01={equal}, 23={equal},
    		023={phantom, description}, 013={phantom, description},
    		0123={phantom, description},
    		/squares/labelstyle/.cd,
    		013={anchor=center}, 023={anchor=center},
    		0123={anchor=center}
    	}
    	\end{tikzpicture}\ .
    \end{center}
    The images under $F$ of these two~$3$-sim\-plices of $\SN(A)$ give
    the following two $3$-simplices of $\SN(B)$:
    \begin{center}
    	\begin{tikzpicture}[scale=1.6, font=\footnotesize]
    	\squares{%
    		/squares/label/.cd,
    		0=$\bullet$, 1=$\bullet$, 2=$\bullet$, 3=$\bullet$,
    		01={}, 12=$F(g)$, 23={}, 02=$F(g)$, 03=$F(f)$, 13=$F(g)$,
    		012=${=}$, 023=$F(\alpha_r)$, 123=${=}$, 013=$F(\alpha_l)$,
    		0123=$\tau_{\text{u}}(\alpha)$,
    		/squares/arrowstyle/.cd,
    		01={equal}, 23={equal},
    		012={phantom, description}, 123={phantom, description},
    		/squares/labelstyle/.cd,
    		012={anchor=center}, 123={anchor=center}
    	}
    	\end{tikzpicture}
    \end{center}
    and
    \begin{center}
    	\begin{tikzpicture}[scale=1.6, font=\footnotesize]
    	\squares{%
    		/squares/label/.cd,
    		0=$\bullet$, 1=$\bullet$, 2=$\bullet$, 3=$\bullet$,
    		01={}, 12=$F(g)$, 23={}, 02=$F(f)$, 03=$F(f)$, 13=$F(f)$,
    		012=$F(\alpha_l)$, 023=${=}$, 123=$F(\alpha_r)$, 013=${=}$,
    		0123=$\tau_{\text{d}}(\alpha)$,
    		/squares/arrowstyle/.cd,
    		01={equal}, 23={equal},
    		023={phantom, description}, 013={phantom, description},
    		/squares/labelstyle/.cd,
    		013={anchor=center}, 023={anchor=center},
    		012={below right}, 123={below left}
    	}
    	\end{tikzpicture}\ .
    \end{center}
    \end{paragr}
    
    \begin{rem}
     If $B$ is a $2$-category, then the $2$-cells $F(\alpha_l)$ and
     $F(\alpha_r)$ coincide.
    \end{rem}

    \begin{paragr}\label{paragr:tau_invertible}
	The $3$-cells $\tau_\text{u}(\alpha)$ and $\tau_\text{d}(\alpha)$
	of $B$ described in the preceding paragraph
	turn out to be connected by two non-degenerate $4$-simplices
	of $\SN(B)$, that we shall call $x_\tau$ and $y_\tau$.
	The first one, displayed in diagram~\eqref{fig:tau1}, witnesses the relation
	\[\tau_\text{d}(\alpha)\comp_2 \tau_\text{u}(\alpha) = 1_{F(\alpha_r)}\,.\]
	
	The second one, displayed in diagram~\eqref{fig:tau2}, witnesses instead the relation \[\tau_\text{u}(\alpha)\comp_2 \tau_\text{d}(\alpha) = 1_{F(\alpha_l)}\,,\] so that in fact $\tau_\text{u}(\alpha)$ and $\tau_\text{d}(\alpha)$
	are two invertible $3$-cells of $B$.
	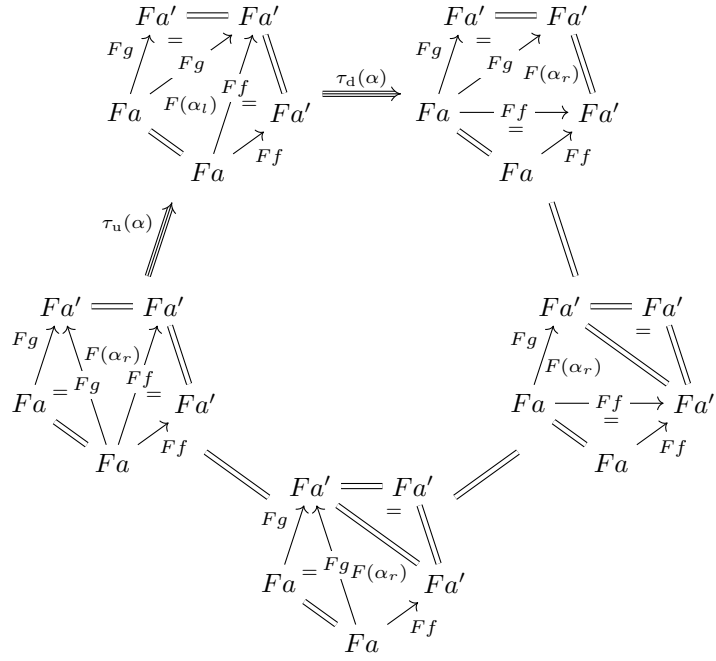
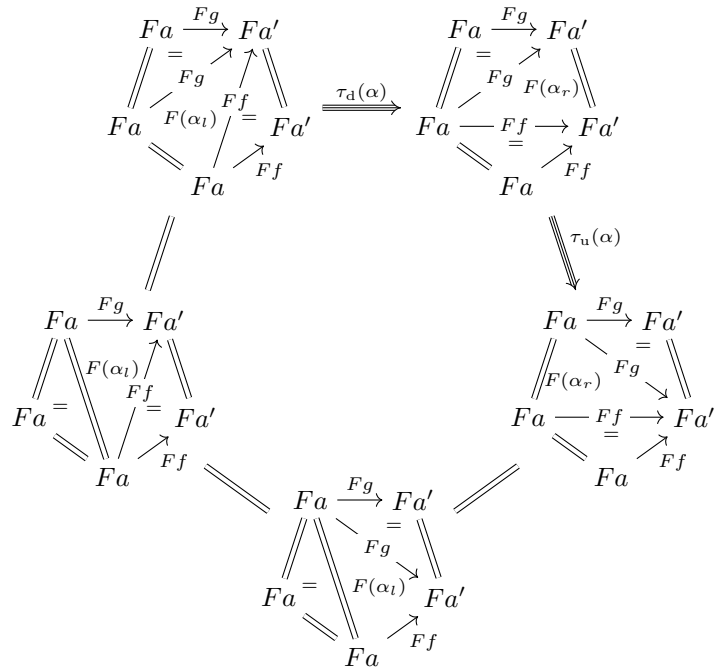
\begin{figure}
		\centering
		\subfloat[][The $4$-simplex $x_\tau$.]{%
			\begin{tikzpicture}[scale=1.15]\label{fig:tau1}
			\pentagon{%
				/pentagon/label/.cd,
				0=$Fa$, 1=$Fa$, 2=$Fa'$, 3=$Fa'$, 4=$Fa'$,
				01={}, 12=$Fg$, 23={}, 34={}, 04=$Ff$,
				02=$Fg$, 03=$Ff$, 13=$Fg$, 14=$Ff$, 24={},
				012=${=}$, 034=${=}$, 023=$F(\alpha_r)$, 123=${=}$, 134=$F(\alpha_r)$,
				014=${=}$, 024=$F(\alpha_r)$, 234=${=}$, 013=$F(\alpha_l)$, 124=$F(\alpha_r)$,
				0123=$\tau_\text{u}(\alpha)$, 0124={}, 0134=$\tau_\text{d}(\alpha)$,
				0234={}, 1234={},
				01234={},
				/pentagon/arrowstyle/.cd,
				01={equal}, 23={equal}, 34={equal}, 24={equal}, 03={pos=0.55},
				012={phantom, description}, 034={phantom, description},
				023={phantom, description}, 123={phantom, description},
				014={phantom, description}, 234={phantom, description},
				134={phantom, description}, 024={phantom, description},
				013={phantom, description}, 124={phantom, description},
				0124={equal}, 0234={equal}, 1234={equal},
				01234={phantom, description},
				/pentagon/labelstyle/.cd,
				24={}, 
				012={anchor=center}, 034={anchor=center},
				023={anchor=center}, 123={anchor=center},
				014={anchor=center}, 234={anchor=center},
				134={anchor=center}, 024={anchor=center},
				013={anchor=center}, 124={anchor=center}
			}
			\end{tikzpicture}
		}\\
		\subfloat[][The $4$-simplex $y_\tau$.]{%
			\begin{tikzpicture}[scale=1.15]\label{fig:tau2}
			\pentagon{%
				/pentagon/label/.cd,
				0=$Fa$, 1=$Fa$, 2=$Fa$, 3=$Fa'$, 4=$Fa'$,
				01={}, 12={}, 23=$Fg$, 34={}, 04=$Ff$,
				02={}, 03=$Ff$, 13=$Fg$, 14=$Ff$, 24=$Fg$,
				012=${=}$, 034=${=}$, 023=$F(\alpha_l)$, 123=${=}$, 134=$F(\alpha_r)$,
				014=${=}$, 024=$F(\alpha_l)$, 234=${=}$, 013=$F(\alpha_l)$, 124=$F(\alpha_r)$,
				0123={}, 0124={}, 0134=$\tau_\text{d}(\alpha)$,
				0234={}, 1234=$\tau_\text{u}(\alpha)$,
				01234={},
				/pentagon/arrowstyle/.cd,
				01={equal}, 12={equal}, 34={equal}, 02={equal}, 03={pos=0.55},
				012={phantom, description}, 034={phantom, description},
				023={phantom, description}, 123={phantom, description},
				014={phantom, description}, 234={phantom, description},
				134={phantom, description}, 024={phantom, description},
				013={phantom, description}, 124={phantom, description},
				0124={equal}, 0234={equal}, 0123={equal},
				01234={phantom, description},
				/pentagon/labelstyle/.cd,
				02={}, 
				012={anchor=center}, 034={anchor=center},
				023={anchor=center}, 123={anchor=center},
				014={anchor=center}, 234={anchor=center},
				134={anchor=center}, 024={anchor=center},
				013={anchor=center}, 124={anchor=center}
			}
			\end{tikzpicture}
		}
		\caption{The $4$-simplices governing $\tau_{\text{u}}$ and $\tau_{\text{d}}$.}
	\end{figure}
   \end{paragr}

   \begin{paragr}\label{paragr:cond-i}
    Let $F \colon A \to B$ be a normalised oplax $3$-functor.
    Consider a $2$-cell $\alpha \colon f \to g$ of $A$ and the two normalised
    oplax $3$-functors
    $L \colon \Deltan{2} \to A$
    and $R \colon \Deltan{2} \to A$
    defined by mapping
    \begin{align*}
    \atom{01} \mapsto 1_{s_0(\alpha)} = a &,& \atom{01} \mapsto g,\\
    \atom{12} \mapsto g &,& \atom{12} \mapsto 1_{t_0(\alpha)} = a',\\
    \atom{02} \mapsto f &,& \atom{02} \mapsto f,\\
    \atom{12}\comp_0 \atom{01} \mapsto g &,& \atom{12}\comp_0 \atom{01} \mapsto g,\\
    L_{\treeV}(\atom{12}, \atom{01}) = \alpha &,&
    R_{\treeV}(\atom{12}, \atom{01} = \alpha,
    \end{align*}
    respectively, that we can depict as
    
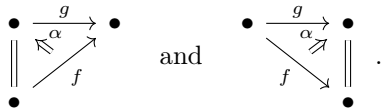
\[
    \begin{tikzcd}
    \bullet  \ar[r, "g", ""'{name=g}] & \bullet \\
    \bullet \ar[u, equal] \ar[ur, "f"', ""{name=f}]
    \ar[Rightarrow, from=f, to=1-1, shorten >=1pt, "\alpha"']
    \end{tikzcd}
    \quadet
    \begin{tikzcd}
    \bullet
    \ar[r, "g", ""'{name=g}] \ar[rd, "f"', ""{name=f}] &
    \bullet \ar[d,equal] \\
    & \bullet 
    \ar[Rightarrow, from=f, to=1-2, shorten >=1pt, "\alpha"]
    \end{tikzcd}\ .
    \]
    Now, we can view $L$ and $R$ as the $3$-face and $0$-face
    respectively of the normalised oplax $3$-functor
    $T \colon \Deltan{3} \to A$ that we can represent as
    \begin{center}
    \begin{tikzpicture}
    \squares{%
        /squares/label/.cd,
        0=$\bullet$, 1=$\bullet$, 2=$\bullet$, 3=$\bullet$,
        12=$g$, 02=$f$, 03=$f$, 13=$f$,
        012=$\alpha$, 023=${=}$, 123=$\alpha$, 013=${=}$,
        0123=$1_{\alpha}$,
        /squares/arrowstyle/.cd,
        01={equal}, 23={equal},
        023={phantom, description}, 013={phantom, description},
        /squares/labelstyle/.cd,
        023={anchor=center}, 013={anchor=center}
        }
    \end{tikzpicture}
    \end{center}
    The conditions of normalisations impose that the image
    under $F$ of such a diagram of $A$, \ie the image of the
    normalised oplax $3$-functor $FT \colon \Deltan{3} \to B$,
    must be
    \begin{center}
    \begin{tikzpicture}[scale=1.5]
    \squares{%
        /squares/label/.cd,
        0=$\bullet$, 1=$\bullet$, 2=$\bullet$, 3=$\bullet$,
        12=$F_{\treeL}(g)$, 02=$F_{\treeL}(f)$, 03=$F_{\treeL}(f)$,
        13=$F_{\treeL}(f)$,
        012=$F_{\treeLL}(\alpha)$, 023=${=}$, 123=$F_{\treeLL}(\alpha)$, 013=${=}$,
        0123=${\Gamma}$,
        /squares/arrowstyle/.cd,
        01={equal}, 23={equal},
        023={phantom, description}, 013={phantom, description},
        /squares/labelstyle/.cd,
        023={anchor=center}, 013={anchor=center}
        }
    \end{tikzpicture}\ ,
    \end{center}
    where
    \[
    \Gamma = FT_{\treeW}(\atom{23}, \atom{12}, \atom{01})\,.
    \]
    Now, the four $3$-cells of $B$ appearing in the definition of $\Gamma$ are (see~\ref{paragr:def_cellular_to_simplicial})
    \[
    F_{\treeVLeft}(\alpha, 1_a) = 1_{F_{\treeLL}(\alpha)}\,,
    \]
    \[
    F_{\treeW}(1_{a'}, g, 1_a) = 1_{F_{\treeL}(g)}\,,
    \]
    and
    \[
    F_{\treeLLL}(1_\alpha) = 1_{F_{\treeLL}(\alpha)}
    \]
    and
    \[
    F_{\treeVRight}(1_{a'}, \alpha) = 1_{F_{\treeLL}(\alpha)}\,.
    \]
    
    Hence, for any $2$-cell $\alpha$ of $A$, we have that
    the $3$-cells $\tau_{\text{d}}(\alpha)$ and $\tau_{\text{u}}(\alpha)$ of $B$ associated to 
    the morphism of simplicial sets
    $\Nl(F)\colon \SN(A) \to \SN(B)$ are both trivial.
\end{paragr}

    \begin{paragr}\label{paragr:encode_comp_1cells}
     Consider two composable $1$-cells
     \[
      \begin{tikzcd}
       a \ar[r, "f"] &
       a' \ar[r, "g"] &
       a''
      \end{tikzcd}
     \]
     of $A$. The simplicial set $\SN(A)$ encodes the composition of $f$ and $g$
     with the $2$-simplex
     \[
      \begin{tikzcd}[column sep=small]
       a
       \ar[rr, "g\comp_1 f", ""{below, name=gf}]
       \ar[rd, "f"']
       &&
         a''
       \\
       &
        a'
        \ar[ur, "g"'] &
       \ar[phantom, from=gf, to=2-2, shorten >=3pt, "=" description]
      \end{tikzcd}\ .
     \]
     The morphism $F$ maps this $2$-simplex to a $2$-simplex
     \[
      \begin{tikzcd}[column sep=small]
       Fa
       \ar[rr, "F(g\comp_1 f)", ""{below, name=gf}]
       \ar[rd, "Ff"']
       &&
         Fa''
       \\
       &
        a'
        \ar[ur, "Fg"'] &
       \ar[Rightarrow, from=gf, to=2-2, shorten >=3pt]
      \end{tikzcd}\ ,
     \]
     where we call $F_{g, f}$ the $2$-cell of $B$ filling the triangle,
     \ie having $F(g \comp_0 f)$ as source and $Fg \comp_0 Ff$ as target;
     we shall often write $Fgf$ for the $1$-cell $F(g \comp_0 f)$ of $B$.
     
     If the morphism of simplicial sets $F$ is the nerve
     of a normalised oplax $3$-funtor $G \colon A \to B$,
     then by definition $F_{g, f} = G_{\treeV}(g, f)$
     (see~\ref{paragr:def_cellular_to_simplicial}).
    \end{paragr}

    \begin{paragr}\label{paragr:encode_triangle}
     Consider a $2$-cell
     \[
      \begin{tikzcd}[column sep=small, row sep=2pt]
        a
        \ar[rr, bend left=30, "f", ""{below, name=f}]
        \ar[rd, bend right=20, "g"']
        &&
          a''
        \\
        &
         a'
         \ar[ru, bend right=20, "h"']
        &
        \ar[Rightarrow, from=f, to=2-2, "\alpha"]
     \end{tikzcd}
     \]
     of $A$. The simplicial set $\SN(A)$ can encode the cell $\alpha$
     with the $2$-simplices $\alpha_l$ and $\alpha_r$ described above
     in paragraph~\ref{paragr:encode_2cell}, but also with
     the $2$-simplex
     \[
      \bar{\alpha} :=
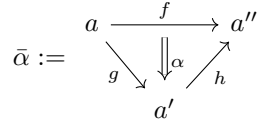
      \begin{tikzcd}[column sep=small]
       a
       \ar[rr, "f", ""{below, name=f}]
       \ar[rd, "g"']
       &&
         a''
       \\
       &
        a'
        \ar[ur, "h"'] &
       \ar[Rightarrow, from=f, to=2-2, shorten >=3pt, "\alpha"]
      \end{tikzcd}
     \]
     of $\SN(A)$. These three $2$-simplices of $\SN(A)$ are
     tied together by the two $3$-simplices described in paragraph~\ref{paragr:encode_2cell},
     but also by the following two $3$-simplices
     
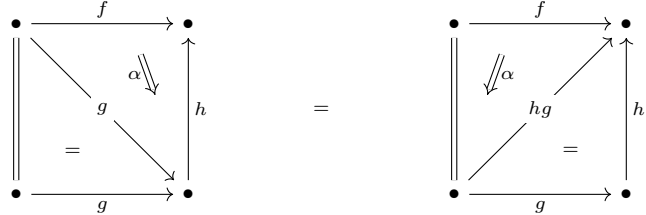
\begin{center}
     	\begin{tikzpicture}[scale=1.6, font=\footnotesize]
     	\squares{%
     		/squares/label/.cd,
     		0=$\bullet$, 1=$\bullet$, 2=$\bullet$, 3=$\bullet$,
     		01={}, 12=$g$, 23=$h$, 02=$g$, 03=$f$, 13=$hg$,
     		012=${=}$, 023=$\alpha$, 123=${=}$, 013=$\alpha$,
     		0123=${=}$,
     		/squares/arrowstyle/.cd,
     		01={equal},
     		012={phantom, description}, 123={phantom, description},
     		0123={phantom, description},
     		/squares/labelstyle/.cd,
     		012={anchor=center}, 123={anchor=center},
     		0123={anchor=center}
     	}
     	\end{tikzpicture}
     \end{center}
     and
     
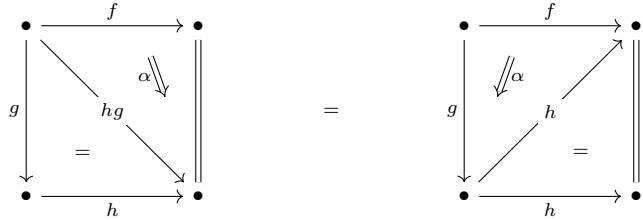
\begin{center}
     	\begin{tikzpicture}[scale=1.6, font=\footnotesize]
     	\squares{%
     		/squares/label/.cd,
     		0=$\bullet$, 1=$\bullet$, 2=$\bullet$, 3=$\bullet$,
     		01=$g$, 12=$h$, 23={}, 02=$hg$, 03=$f$, 13=$h$,
     		012=${=}$, 023=$\alpha$, 123=${=}$, 013=$\alpha$,
     		0123=${=}$,
     		/squares/arrowstyle/.cd,
     		23={equal},
     		012={phantom, description}, 123={phantom, description},
     		0123={phantom, description},
     		/squares/labelstyle/.cd,
     		012={anchor=center}, 123={anchor=center},
     		0123={anchor=center}
     	}
     	\end{tikzpicture}
     \end{center}
     of $\SN(A)$, whose image under $F$ gives the following two $3$-simplices of $\SN(B)$:
     \begin{center}
     	\begin{tikzpicture}[scale=1.6, font=\footnotesize]
     	\squares{%
     		/squares/label/.cd,
     		0=$\bullet$, 1=$\bullet$, 2=$\bullet$, 3=$\bullet$,
     		01={}, 12=$Fg$, 23=$Fh$, 02=$Fg$, 03=$Ff$, 13=$Fhg$,
     		012=${=}$, 023=$F(\bar\alpha)$, 123=$F_{g, f}$, 013=$F(\alpha_l)$,
     		0123=$\gamma_{\text{l}}(\alpha)$,
     		/squares/arrowstyle/.cd,
     		01={equal},
     		012={phantom, description},
     		/squares/labelstyle/.cd,
     		012={anchor=center}
     	}
     	\end{tikzpicture}
     \end{center}
     and
     \begin{center}
     	\begin{tikzpicture}[scale=1.6, font=\footnotesize]
     	\squares{%
     		/squares/label/.cd,
     		0=$\bullet$, 1=$\bullet$, 2=$\bullet$, 3=$\bullet$,
     		01=$Fg$, 12=$Fg$, 23={}, 02=$Fhg$, 03=$Ff$, 13=$Fh$,
     		012=$F_{g, f}$, 023=$F(\alpha_r)$, 123=${=}$, 013=$F(\bar\alpha)$,
     		0123=$\gamma_{\text{r}}(\alpha)$,
     		/squares/arrowstyle/.cd,
     		23={equal},
     		123={phantom, description},
     		/squares/labelstyle/.cd,
     		123={anchor=center}
     	}
     	\end{tikzpicture}\ .
     \end{center}
    \end{paragr}
    
    \begin{rem}
     If $B$ is a $2$-category, then the $2$-cells $F_{g, f}\comp_1 F(\alpha_l)$,
     $F_{g, f} \comp_1 F(\alpha_r)$ and $F(\bar\alpha)$ coincide.
    \end{rem}

    \begin{paragr}\label{paragr:gamma_invertible}
    The $3$-cells $\gamma_\text{l}(\alpha)$ and $\gamma_\text{r}(\alpha)$
    of $B$ described in the preceding paragraph
    turn out to be connected by two non-degenerate $4$-simplices
    of $\SN(B)$, that we shall call $x_\gamma$ and $y_\gamma$.
    The first one, displayed in diagram~\ref{fig:gamma1}, witnesses the relation
    \[
      \gamma_\text{l}(\alpha) \comp_2 \gamma_\text{r}(\alpha) = F_{g, f} \comp_1 \tau_\text{u}(\alpha)\,.
    \]
    The second one, displayed in diagram~\ref{fig:gamma2}, witnesses instead the relation
    \[
      \gamma_\text{r}(\alpha) \comp_2 (F_{g, f} \comp_1 \tau_\text{d}(\alpha)) \comp_2 \gamma_\text{l}(\alpha)
      = 1_{F\bar\alpha}\,.
    \]
    We already know by paragraph~\ref{paragr:tau_invertible} that
    $\tau_\text{u}(\alpha)$ and $\tau_\text{d}(\alpha)$
    are two invertible $3$-cells of $B$, inverses of each others.
    Hence we obtain that $\gamma_\text{l}(\alpha)$ and
    $\gamma_\text{r}(\alpha)$ are invertible $3$-cells of $B$,
    with $F_{g, f}\comp_1 \tau_\text{d}(\alpha) \comp_2 \gamma_\text{l}(\alpha)$
    as inverse of $\gamma_\text{r}(\alpha)$.
	\begin{figure}
	\centering
	\subfloat[][The $4$-simplex $x_\gamma$.]{%
	 \begin{tikzpicture}[scale=1.2]\label{fig:gamma1}
	 \pentagon{%
	 	/pentagon/label/.cd,
	 	0=$Fa$, 1=$Fa$, 2=$Fa'$, 3=$Fa''$, 4=$Fa''$,
	 	01={}, 12=$Fg$, 23=$Fh$, 34={}, 04=$Ff$,
	 	02=$Fg$, 03=$Fhg$, 13=$Fhg$, 14=$Fhg$, 24=$Fh$,
	 	012=${=}$, 034=$F\alpha_r$, 023=$F_{g, f}$, 123=${F_{g,f}}$, 134=${=}$,
	 	014=${F\alpha_l}$, 024=$F\bar\alpha$, 234=${=}$, 013=${=}$, 124=$F_{g,f}$,
	 	0123={}, 0124=$\gamma_{\text{l}}(\alpha)$,
	 	0134={$F_{g, f}\comp_1\tau_\text{u}(\alpha)$},
	 	0234=$\gamma_{\text{r}}(\alpha)$, 1234={},
	 	01234={},
	 	/pentagon/arrowstyle/.cd,
	 	01={equal}, 34={equal}, 03={pos=0.55},
	 	012={phantom, description}, 034={phantom, description},
	 	023={phantom, description}, 123={phantom, description},
	 	014={phantom, description}, 234={phantom, description},
	 	134={phantom, description}, 024={phantom, description},
	 	013={phantom, description}, 124={phantom, description},
	 	0123={equal}, 1234={equal},
	 	01234={phantom, description},
	 	/pentagon/labelstyle/.cd,
	 	012={anchor=center}, 034={anchor=center},
	 	023={anchor=center}, 123={anchor=center},
	 	014={anchor=center}, 234={anchor=center},
	 	134={anchor=center}, 024={anchor=center},
	 	013={anchor=center}, 124={anchor=center}
	 }
	 \end{tikzpicture}
	}\\
	\subfloat[][The $4$-simplex $y_\tau$.]{%
		\begin{tikzpicture}[scale=1.2]\label{fig:gamma2}
		\pentagon{%
			/pentagon/label/.cd,
			0=$Fa$, 1=$Fa$, 2=$Fa'$, 3=$Fa''$, 4=$Fa''$,
            01={}, 12=$Fg$, 23=$Fh$, 34={}, 04=$Ff$,
            02=$Fg$, 03=$Ff$, 13=$Fhg$, 14=$Fhg$, 24=$Fh$,
			012=${=}$, 034=${=}$, 023=$F\bar\alpha$, 123=$F_{g,f}$,
			134=$F\alpha_r$, 014=${=}$, 024=$F\bar\alpha$,
			234=${=}$, 013=$F\alpha_l$, 124=$F\bar\alpha$,
			0123=$\gamma_{\text{l}}(\alpha)$, 0124={},
			0134={$F_{g, f}\comp_1\tau_\text{d}(\alpha)$},
			0234={}, 1234=$\gamma_\text{r}(\alpha)$,
			01234={},
			/pentagon/arrowstyle/.cd,
			01={equal}, 34={equal}, 03={pos=.55},
			012={phantom, description}, 034={phantom, description},
			023={phantom, description}, 123={phantom, description},
			014={phantom, description}, 234={phantom, description},
			134={phantom, description}, 024={phantom, description},
			013={phantom, description}, 124={phantom, description},
			0234={equal}, 0124={equal},
			01234={phantom, description},
			/pentagon/labelstyle/.cd,
			012={anchor=center}, 034={anchor=center},
			023={anchor=center}, 123={anchor=center},
			014={anchor=center}, 234={anchor=center},
			134={anchor=center}, 024={anchor=center},
			013={anchor=center}, 124={anchor=center}
		}
		\end{tikzpicture}
	}
	\caption{The $4$-simplices governing $\gamma_{\text{l}}$ and $\gamma_{\text{r}}$.}
	\end{figure}
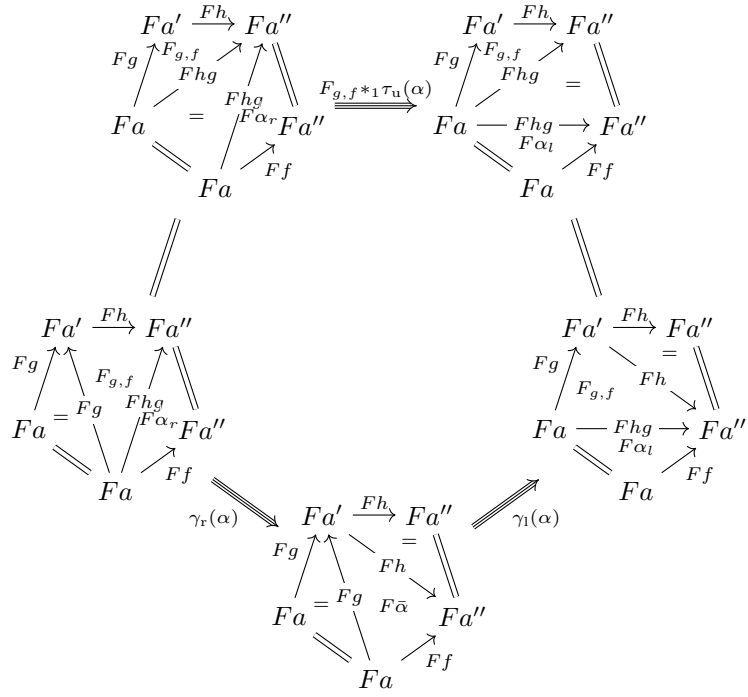
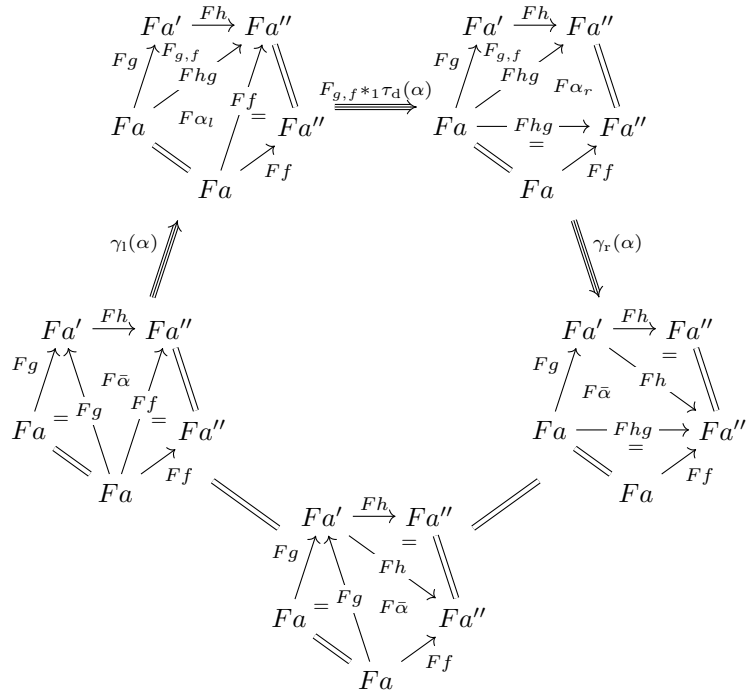
    \end{paragr}

\begin{paragr}\label{paragr:cond-ii}
	Let $F \colon A \to B$ be a normalised oplax $3$-functor.
    Consider the normalised oplax $3$-functor $T \colon \Deltan{3} \to A$
    given by
    
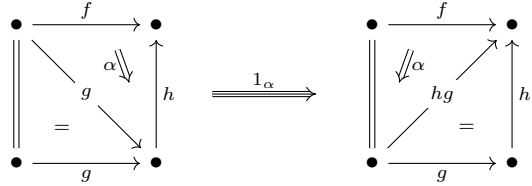
\begin{center}
    \begin{tikzpicture}[scale=1.3]
    \squares{%
        /squares/label/.cd,
        01={}, 12={$g$}, 23={$h$}, 03={$f$}, 02={$g$}, 13={$hg$},
        012={$=$}, 023={$\alpha$}, 013={$\alpha$}, 123={$=$},
        0123={$1_{\alpha}$},
        /squares/arrowstyle/.cd,
        01={equal}, 012={phantom, description}, 123={phantom, description},
        /squares/labelstyle/.cd,
        012={anchor=center}, 123={anchor=center}
    }
    \end{tikzpicture}\ .
    \end{center}
    The conditions of normalisations impose that the image under $F$ of $T$ is
    
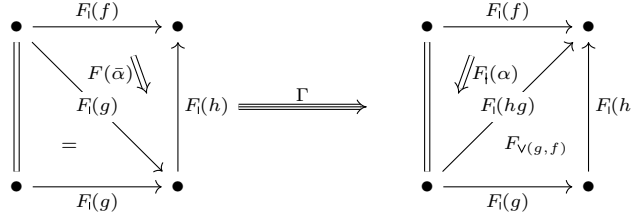
\begin{center}
        \begin{tikzpicture}[scale=1.5]
        \squares{%
            /squares/label/.cd,
            01={}, 12={$F_{\treeL}(g)$}, 23={$F_{\treeL}(h)$}, 03={$F_{\treeL}(f)$},
            02={$F_{\treeL}(g)$}, 13={$F_{\treeL}(hg)$},
            012={$=$}, 023={$F(\bar\alpha)$}, 013={$F_{\treeLL}(\alpha)$}, 123={$F_{\treeV(g, f)}$},
            0123={$\Gamma$},
            /squares/arrowstyle/.cd,
            01={equal}, 012={phantom, description}, 123={phantom, description},
            /squares/labelstyle/.cd,
            012={anchor=center}, 123={anchor=center}
        }
        \end{tikzpicture}\,,
    \end{center}
    Moreover, the four main $3$-cells of $\Gamma$ are by definition (see~\ref{paragr:def_cellular_to_simplicial}):
    \[
    F_{\treeVLeft}(1_{hg}, 1_a) = 1_{F_{\treeLL}(hg)}\,,
    \]
    \[
    F_{\treeW}(h, g, 1_a) = 1_{F_{\treeV}(g, f)}\,,
    \]
    \[
    F_{\treeLLL}(1_{\alpha}) = 1_{F_{\treeLL}(\alpha)}
    \]
    and
    \[
    F_{\treeVRight}(h, 1_g) = 1_{F_{\treeV(g, f)}}\,.
    \]
    Hence the $3$-cell $\Gamma$ is trivial.
    This is equivalent to saying that for any diagram
    
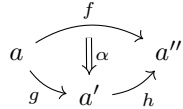
\[
    \begin{tikzcd}[column sep=small, row sep=2pt]
    a
    \ar[rr, bend left=30, "f", ""{below, name=f}]
    \ar[rd, bend right=20, "g"']
    &&
    a''
    \\
    &
    a'
    \ar[ru, bend right=20, "h"']
    &
    \ar[Rightarrow, from=f, to=2-2, "\alpha"]
    \end{tikzcd}
    \]
    of $A$, the $3$-cells $\gamma_{\text{l}}(\alpha)$
    and $\gamma_{\text{r}}(\alpha)$ associated
    to the morphism of simplicial set
    $\Nl(F) \colon \SN(A) \to \SN(B)$
    are trivial.
\end{paragr}

    \begin{paragr}\label{paragr:sigma}
     Consider two $1$-composable $2$-cells
     
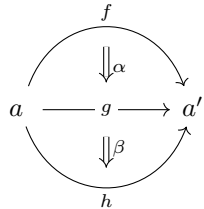
\[
      \begin{tikzcd}[column sep=4.8em]
       a^{\phantom\prime}
       \ar[r, bend left=70, looseness=1.4, "f", ""'{name=f}]
       \ar[r, "g"{description, name=g}]
       \ar[r, bend right=70, looseness=1.4, ""{name=h}, "h"']
       \ar[Rightarrow, from=f, to=g, shorten <=1mm, shorten >= 2mm, "\alpha"]
       \ar[Rightarrow, from=g, to=h, shorten <=2mm, shorten >= 1mm, "\beta"]
       & a'
      \end{tikzcd}
     \]
     of $A$. We have a $3$-simplex $\sigma_{\alpha, \beta}$
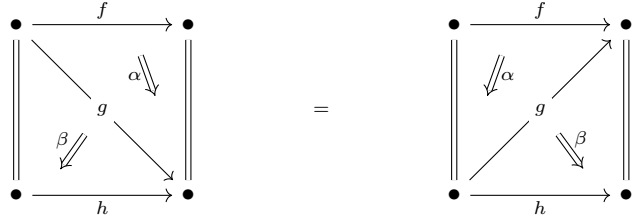
     \begin{center}
      \begin{tikzpicture}[scale=1.6]
       \squares{
        /squares/label/.cd,
        0=$\bullet$, 1=$\bullet$, 2=$\bullet$, 3=$\bullet$,
        01={}, 12=$h$, 23={}, 03=$f$, 02=$g$, 13=$g$,
        012=$\beta$, 023=$\alpha$,
        013=$\alpha$, 123=$\beta$,
        0123=${=}$,
        /squares/arrowstyle/.cd,
        01={equal}, 23={equal},
        0123={phantom, description},
        /squares/labelstyle/.cd,
        0123={anchor=center}
       }
      \end{tikzpicture}
     \end{center}
     of $\SN(A)$ whose image under $F$ is given by the following $3$\hyp{}simplex
     $F(\sigma_{\alpha, \beta})$ of $\SN(B)$
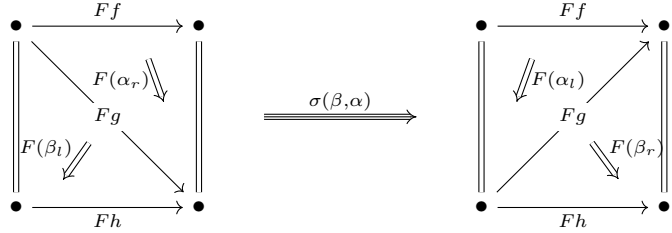
     \begin{center}
      \begin{tikzpicture}[scale=1.7]
       \squares{
        /squares/label/.cd,
        0=$\bullet$, 1=$\bullet$, 2=$\bullet$, 3=$\bullet$,
        01={}, 12=$Fh$, 23={}, 03=$Ff$, 02=$Fg$, 13=$Fg$,
        012=$F(\beta_l)$, 023=$F(\alpha_r)$,
        013=$F(\alpha_l)$, 123=$F(\beta_r)$,
        0123=${\sigma(\beta, \alpha)}$,
        /squares/arrowstyle/.cd,
        01={equal}, 23={equal}
       }
      \end{tikzpicture}\ .
     \end{center}
     There is a close relationship between the $3$-cell $\sigma(\beta, \alpha)$
     and the $3$-cells $\tau_{\text{u}}(\alpha)$ and $\tau_{\text{d}}(\beta)$
     of $B$, as displayed by the $4$-simplex $x_\sigma$ in figure~\ref{fig:sigma}.
     In particular, $\sigma(\beta, \alpha)$ is an invertible $3$-cell of $B$.
     Being more precise, we have
     \[
      \sigma(\beta, \alpha) = \tau_{\text{d}}(\beta) \comp_1 \tau_{\text{u}}(\alpha)\,.
     \]
     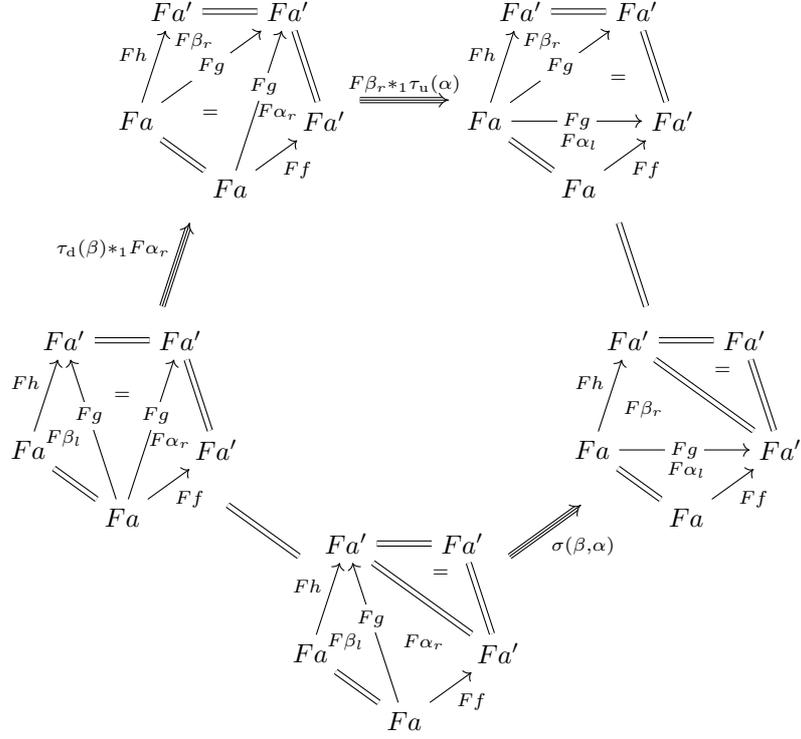
\begin{figure}
      \centering
      \begin{tikzpicture}[scale=1.3]
       \pentagon{
            /pentagon/label/.cd,
            0=$Fa$, 1=$Fa$, 2=$Fa'$, 3=$Fa'$, 4=$Fa'$,
            01={}, 12=$Fh$, 23={}, 34={}, 04=$Ff$,
            02=$Fg$, 03=$Fg$, 13=$Fg$, 14=$Fg$, 24={},
            012=$F\beta_l$, 034=$F\alpha_r$, 023={$=$}, 123=$F\beta_r$,
            134={$=$}, 014=$F\alpha_l$, 024=$F\alpha_r$,
            234={$=$}, 013={$=$}, 124=$F\beta_r$,
            0123={$\tau_{\text{d}}(\beta)\comp_1 F\alpha_r$},
            0124={$\sigma(\beta, \alpha)$},
            0134={$F\beta_r\comp_1\tau_\text{u}(\alpha)$},
            0234={}, 1234={},
            01234={},
            /pentagon/arrowstyle/.cd,
            01={equal}, 23={equal}, 34={equal},
            03={pos=.6}, 02={pos=0.6}, 24={equal},
            012={phantom, description}, 034={phantom, description},
            023={phantom, description}, 123={phantom, description},
            014={phantom, description}, 234={phantom, description},
            134={phantom, description}, 024={phantom, description},
            013={phantom, description}, 124={phantom, description},
            0234={equal}, 1234={equal},
            01234={phantom, description},
            /pentagon/labelstyle/.cd,
            24={},
            012={anchor=center}, 034={anchor=center},
            023={anchor=center}, 123={anchor=center},
            014={anchor=center}, 234={anchor=center},
            134={anchor=center}, 024={anchor=center},
            013={anchor=center}, 124={anchor=center}
       }
      \end{tikzpicture}
    \caption{The $4$-simplex $x_\sigma$}
    \label{fig:sigma}
    \end{figure}
    \end{paragr}

    \begin{rem}\label{rem:treeY}
    	If $F \colon A\to B$ be a normalised oplax $3$-functor,
    	then it follows by paragraph~\ref{paragr:cond-i} and by the
    	relation of the previous paragraph that given
    	any pair $(\beta, \alpha)$ of $1$-composable
    	$2$-cells of $A$, the $3$-cell $\sigma(\beta, \alpha)$
    	of~$B$ associated to the morphism of simplicial sets
    	$\Nl(F) \colon \SN(A) \to \SN(B)$ is \emph{trivial}.
    	This can be taken as a further justification for
    	the choice of listing the datum associated to
    	the tree $\treeY$, representing the vertical composition
    	of $2$-cells, as a coherence and not as a structural cell
    	in the definition of normalised oplax $3$-functor.
    	Indeed, any preferred direction (lax/oplax) would be
    	incompatible with the combinatorics dictated by
    	the simplicial sets; more precisely, it would be
    	irreconcilable with the combinatorics of the orientals
    	and thus with the data encoded by morphisms of simplicial
    	sets between Street nerve of $3$-categories.
    \end{rem}

    \begin{paragr}\label{paragr:encode_horizontal_comp}
     The images of vertical compositions of cells encode a great deal
     of information into the form of coherences, \ie invertible
     cells of $B$. An example of critical importance
     for the following sections is presented in this paragraph.
     Consider two $1$-composable $2$-cells
     
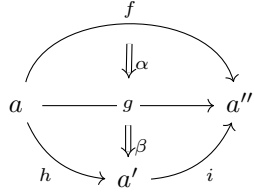
\[
      \begin{tikzcd}[row sep=1.35em]
       a^{\phantom\prime}
       \ar[rr, bend left=75, "f", ""'{name=f}]
       \ar[rr, "g"{description, name=g}]
       \ar[rd, bend right, "h"']
       &&
         a''
       \\
       & a'
         \ar[ru, bend right, "i"']
       &
       \ar[Rightarrow, from=f, to=g, shorten <=1mm, shorten >=2mm, "\alpha"]
       \ar[Rightarrow, from=g, to=2-2, shorten <=1mm, pos=0.4, "\beta" near end]
      \end{tikzcd}
     \]
     of $A$, to which we can associate the following two $3$-simplices
     
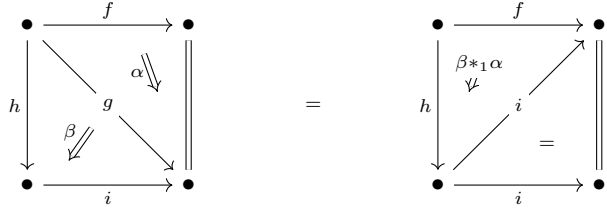
\begin{center}
     	\begin{tikzpicture}[scale=1.5]
     	\squares{%
     		/squares/label/.cd,
     		01={$h$}, 12={$i$}, 23={}, 03={$f$}, 02={$g$}, 13={$i$},
     		012={$\beta$}, 023={$\alpha$}, 013={$\beta\comp_1\alpha$}, 123={$=$},
     		0123={$=$},
     		/squares/arrowstyle/.cd,
     		23={equal}, 123={phantom, description}, 0123={phantom, description},
     		/squares/labelstyle/.cd,
     		013={description, near start},
     		123={anchor=center}, 0123={anchor=center}
     	}
     	\end{tikzpicture}
     \end{center}
     and
     
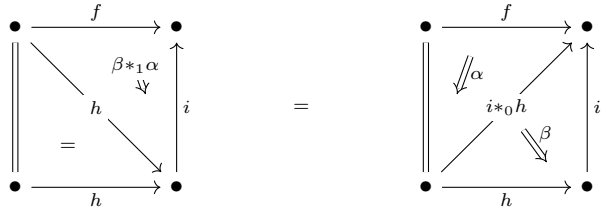
\begin{center}
     	\begin{tikzpicture}[scale=1.5]
     	\squares{%
     		/squares/label/.cd,
     		01={}, 12={$h$}, 23={$i$}, 03={$f$}, 02={$h$}, 13={$i\comp_0 h$},
     		012={$=$}, 023={$\beta\comp_1 \alpha$}, 013={$\alpha$}, 123={$\beta$},
     		0123={$=$},
     		/squares/arrowstyle/.cd,
     		01={equal}, 012={phantom, description}, 0123={phantom, description},
     		/squares/labelstyle/.cd,
     		023={description, near start},
     		012={anchor=center}, 0123={anchor=center}
     	}
     	\end{tikzpicture}
     \end{center}
     of $\SN(A)$. These are mapped under $F$ to the following two
     $3$\hyp{}simplices of $\SN(B)$:
     \begin{center}
     	\begin{tikzpicture}[scale=1.65]
     	\squares{%
     		/squares/label/.cd,
     		01={$Fh$}, 12={$Fi$}, 23={}, 03={$Ff$}, 02={$Fg$}, 13={$Fi$},
     		012={$F(\bar\beta)$}, 023={$F(\alpha_r)$}, 013={$F(\overline{\beta\comp_1\alpha})$}, 123={$=$},
     		0123={$\eps_\text{l}(\beta, \alpha)$},
     		/squares/arrowstyle/.cd,
     		23={equal}, 123={phantom, description},
     		/squares/labelstyle/.cd,
     		013={description, near start = 1pt and 1pt},
     		123={anchor=center}
     	}
     	\end{tikzpicture}
     \end{center}
     and
     \begin{center}
     	\begin{tikzpicture}[scale=1.65]
     	\squares{%
     		/squares/label/.cd,
     		01={}, 12={$Fh$}, 23={$Fi$}, 03={$Ff$}, 02={$Fh$}, 13={$F(i\comp_0 h)$},
     		012={$=$}, 023={$F(\overline{\beta\comp_1 \alpha})$}, 013={$F(\alpha_l)$}, 123={$F(\bar\beta)$},
     		0123={$\eps_\text{r}(\beta, \alpha)$},
     		/squares/arrowstyle/.cd,
     		01={equal}, 012={phantom, description},
     		/squares/labelstyle/.cd,
     		023={description, near start = 1pt and 1pt},
     		012={anchor=center}
     	}
     	\end{tikzpicture}\ .
     \end{center}
    \end{paragr}
    
    \begin{figure}
    	\centering
    	\subfloat[][The $4$-simplex $x_\eps$.]{%
    	 \begin{tikzpicture}[scale=1.25, font=\footnotesize]\label{fig:eps1}
    		\pentagon{%
    		 /pentagon/label/.cd,
    		 0={$Fa$}, 1={$Fa$}, 2={$Fa'$}, 3={$Fa''$}, 4={$Fa''$},
    		 01={}, 12={$Fh$}, 23={$Fi$}, 34={}, 04={$Ff$},
    		 02={$Fh$}, 03={$Fg$}, 13={$Fhi$}, 14={$Fg$}, 24={$Fi$},
    		 012={$=$}, 023={$F\bar\beta$}, 034={$F\alpha_r$},
    		 013={$F\beta_l$}, 123={$F_{i,h}$}, 014={$F\alpha_l$},
    		 134={$F\beta_r$}, 124={$F\bar\beta$},
    		 024={$F\overline{\beta\comp_1\alpha}$}, 234={$=$},
    		 0123={$\gamma_\text{l}(\beta)\comp_1 F\alpha_r$},
    		 0134={$\sigma(\beta, \alpha)\comp_1 F_{i, h}$},
    		 1234={$\gamma_\text{r}(\beta) \comp_1 F\alpha_l$},
    		 0234={$\eps_{\text{l}}(\beta, \alpha)$},
    		 0124={$\eps_{\text{r}}(\beta, \alpha)$},
    		 01234={},
    		 /pentagon/arrowstyle/.cd,
    		 01={equal}, 34={equal}, 02={pos=0.55}, 03={pos=0.6},
    		 012={phantom, description}, 034={phantom, description},
    		 023={phantom, description}, 123={phantom, description},
    		 014={phantom, description}, 234={phantom, description},
    		 134={phantom, description}, 024={phantom, description},
    		 013={phantom, description}, 124={phantom, description},
    		 01234={phantom, description},
    		 /pentagon/labelstyle/.cd,
    		 012={anchor=center}, 034={anchor=center},
    		 023={anchor=center}, 123={anchor=center},
    		 014={anchor=center}, 234={anchor=center},
    		 134={anchor=center}, 024={anchor=center},
    		 013={anchor=center}, 124={anchor=center}
    		}
    	 \end{tikzpicture}
    	} \\
    	\subfloat[][The $4$-simplex $y_\eps$.]{%
    		\begin{tikzpicture}[scale=1.25, font=\footnotesize]\label{fig:eps2}
    		\pentagon{%
    			/pentagon/label/.cd,
    			0={$Fa$}, 1={$Fa$}, 2={$Fa'$}, 3={$Fa''$}, 4={$Fa''$},
    			01={}, 12={$Fh$}, 23={$Fi$}, 34={}, 04={$Ff$},
    			02={$Fh$}, 03={$Ff$}, 13={$Fg$}, 14={$Ff$}, 24={$Fi$},
    			012={$=$}, 023={$F\overline{\beta\comp_1\alpha}$}, 034={$=$},
    			013={$F\alpha_l$}, 123={$F\bar\beta$}, 014={$=$},
    			134={$F\alpha_r$}, 124={$F\overline{\beta\comp_1\alpha}$},
    			024={$F\overline{\beta\comp_1\alpha}$}, 234={$=$},
    			0123={$\eps_{\text{r}}(\beta, \alpha)$},
    			0134={$F\bar\beta \comp_1 \tau_{\text{d}}(\alpha)$},
    			1234={$\eps_{\text{l}}(\beta, \alpha)$},
    			0234={},
    			0124={},
    			01234={},
    			/pentagon/arrowstyle/.cd,
    			01={equal}, 34={equal}, 02={pos=0.55}, 03={pos=0.6},
    			012={phantom, description}, 034={phantom, description},
    			023={phantom, description}, 123={phantom, description},
    			014={phantom, description}, 234={phantom, description},
    			134={phantom, description}, 024={phantom, description},
    			013={phantom, description}, 124={phantom, description},
    			0234={equal}, 0124={equal},
    			01234={phantom, description},
    			/pentagon/labelstyle/.cd,
    			012={anchor=center}, 034={anchor=center},
    			023={anchor=center}, 123={anchor=center},
    			014={anchor=center}, 234={anchor=center},
    			134={anchor=center}, 024={anchor=center},
    			013={anchor=center}, 124={anchor=center}
    		}
    		\end{tikzpicture}
    	}
    	\caption{The $4$-simplices governing $\eps_\text{l}(\beta, \alpha)$
    		and $\eps_\text{l}(\beta, \alpha)$.}
    	\label{fig:eps}
    \end{figure}
    
   \begin{paragr}\label{paragr:eps_invertible}
     Under the assumptions of the preceding paragraph,
     we can construct the two $4$-simplices $x_\eps$ and $y_\eps$
     of $\SN(B)$ displayed in figure~\ref{fig:eps}.
     
     We have the following equalities of $3$-cells of $B$:
     \[
      \begin{split}
       & \gamma_{\text{r}}(\beta)\comp_1 F\alpha_l \comp_2
      F_{i, h}  \comp_1 \sigma(\beta, \alpha) \comp_2
       \gamma_{\text{l}}(\beta)\comp_1 F\alpha_r\\
       =\ & \gamma_{\text{r}}(\beta)\comp_1 F\alpha_l \comp_2 (F_{i, h}  \comp_1
       \tau_{\text{d}}(\beta) \comp_1 \tau_{\text{u}}(\alpha)  )
         \comp_2
       \gamma_{\text{l}}(\beta)\comp_1 F\alpha_r \\
       =\ & (\gamma_{\text{r}}(\beta)\comp_1 F_{i, h} \comp_2 \tau_{\text{d}}(\beta) \comp_2\gamma_{\text{l}}(\beta) )
       \comp_1 \tau_{\text{u}}(\alpha) \\
       =\ & 1_{F\beta_r\comp_1 F\alpha_l} \comp_1 \tau_{\text{u}}(\alpha) = F\bar\beta \comp_1 \tau_{\text{u}}(\alpha)\,,
      \end{split}
     \]
     where the first equality follows by paragraph~\ref{paragr:sigma},
     the second one by the exchange law and the third one by paragraph~\ref{paragr:gamma_invertible}.
     Hence the $4$\hyp{}simplex $x_\eps$ depicted in figure~\ref{fig:eps1} witnesses the relation
     \[
        \eps_{\text{r}}(\beta, \alpha) \comp_2 \eps_{\text{l}}(\beta, \alpha) = F\bar\beta \comp_1 \tau_{\text{u}}(\alpha)\,,
     \]
     that by paragraph~\ref{paragr:tau_invertible} is equivalent to saying that the $3$-cell
     \[
        F\bar\beta \comp_1 \tau_{\text{d}}(\alpha)
        \comp_2
        \eps_{\text{r}}(\beta, \alpha) \comp_2 \eps_{\text{l}}(\beta, \alpha)
     \]
     is precisely the identity of the $2$-cell
     \[
        F_{i,h} \comp_1 F\bar\beta \comp_1 F\alpha_l\,.
     \]
     Moreover, the $4$-simplex $y_\eps$ depicted in figure~\ref{fig:eps2}
     gives us that the $3$-cell
     \[
      \eps_{\text{l}}(\beta, \alpha) \comp_2
      F\beta_r \comp_1 \tau_{\text{d}}(\alpha)
        \comp_2
        \eps_{\text{r}}(\beta, \alpha)
     \]
     of $B$ is an identity cell, too. Therefore both the $3$-cells
     $\eps_{\text{l}}(\beta, \alpha)$ and $\eps_{\text{r}}(\beta, \alpha)$
     are invertible.
    \end{paragr}

    \begin{paragr}\label{paragr:encode_simple_horizontal_comp}
     Consider two $1$-composable $2$-cells $\alpha$ and $\beta$ of $A$
     as in paragraph~\ref{paragr:sigma}
     and the two $3$-simplices
     \begin{center}
      \begin{tikzpicture}[scale=1.5,  font=\footnotesize]
       \squares{%
       	 /squares/label/.cd,
       	 01={}, 12={$h$}, 23={}, 03={$f$}, 02={$g$}, 13={$h$},
       	 012={$\beta$}, 023={$\alpha$}, 013={$\beta\comp_1 \alpha$}, 123={$=$},
       	 0123={$=$},
       	 /squares/arrowstyle/.cd,
       	 01={equal}, 23={equal},
       	 123={phantom, description}, 0123={phantom, description},
       	 /squares/labelstyle/.cd,
       	 013={description, near start},
       	 123={anchor=center}, 0123={anchor=center}
       	}
      \end{tikzpicture}
     \end{center}
    and
     \begin{center}
     	\begin{tikzpicture}[scale=1.5,  font=\footnotesize]
     	\squares{%
     		/squares/label/.cd,
     		01={}, 12={$h$}, 23={}, 03={$f$}, 02={$h$}, 13={$h$},
     		012={$=$}, 023={$\beta\comp_1\alpha$}, 013={$\alpha$}, 123={$\beta$},
     		0123={$=$},
     		/squares/arrowstyle/.cd,
     		01={equal}, 23={equal},
     		012={phantom, description}, 0123={phantom, description},
     		/squares/labelstyle/.cd,
     		023={description, near start},
     		012={anchor=center}, 0123={anchor=center}
     	}
     	\end{tikzpicture}
     \end{center}
     of $\SN(A)$. It follows immediately from the previous
     paragraph that the main $3$-cells of the images by $F$ of the two
     $3$-simplices above is trivial.
     
     There are two others $3$-simplices of $\SN(A)$ that is natural to consider,
     that is
     \begin{center}
     	\begin{tikzpicture}[scale=1.5,  font=\footnotesize]
     	\squares{%
     		/squares/label/.cd,
     		01={}, 12={$h$}, 23={}, 03={$f$}, 02={$g$}, 13={$f$},
     		012={$\beta$}, 023={$\alpha$}, 013={$=$}, 123={$\beta\comp_1 \alpha$},
     		0123={$=$},
     		/squares/arrowstyle/.cd,
     		01={equal}, 23={equal},
     		013={phantom, description}, 0123={phantom, description},
     		/squares/labelstyle/.cd,
     		123={description, near start},
     		013={anchor=center}, 0123={anchor=center}
     	}
     	\end{tikzpicture}
     \end{center}
     and
     \begin{center}
     	\begin{tikzpicture}[scale=1.5,  font=\footnotesize]
     	\squares{%
     		/squares/label/.cd,
     		01={}, 12={$h$}, 23={}, 03={$f$}, 02={$h$}, 13={$f$},
     		012={$\beta\comp_1 \alpha$}, 023={$=$}, 013={$\alpha$}, 123={$\beta$},
     		0123={$=$},
     		/squares/arrowstyle/.cd,
     		01={equal}, 23={equal},
     		023={phantom, description}, 0123={phantom, description},
     		/squares/labelstyle/.cd,
     		012={description, near start},
     		023={anchor=center}, 0123={anchor=center}
     	}
     	\end{tikzpicture}
     \end{center}
     Let us denote by $\omega_{\text{r}}(\beta, \alpha)$ and
     $\omega_{\text{l}}(\beta, \alpha)$ respectively the main
     $3$-cells of $B$ of the images under $F$ of the $3$-simplices of $\SN(A)$ above.
     Both these $3$-cells of $B$ are invertible. The proof for $\omega_{\text{r}}(\beta, \alpha)$
     is given by the $4$-simplex $x_{\omega}$ of $\SN(B)$ depicted
     in figure~\ref{fig:omega_r} and the proof for $\omega_{\text{l}}(\beta, \alpha)$
     is completely similar and we leave it to the reader.
     One can actually check that the four $3$-cells $\eps_{\text{l}}(\beta, \alpha)$,
     $\eps_{\text{r}}(\beta, \alpha)$, $\omega_{\text{l}}(\beta, \alpha)$ and $\omega_{\text{r}}(\beta, \alpha)$ are all tied together by a $5$-simplex of $B$, that we will not need
     and so we are not going to describe.
     
     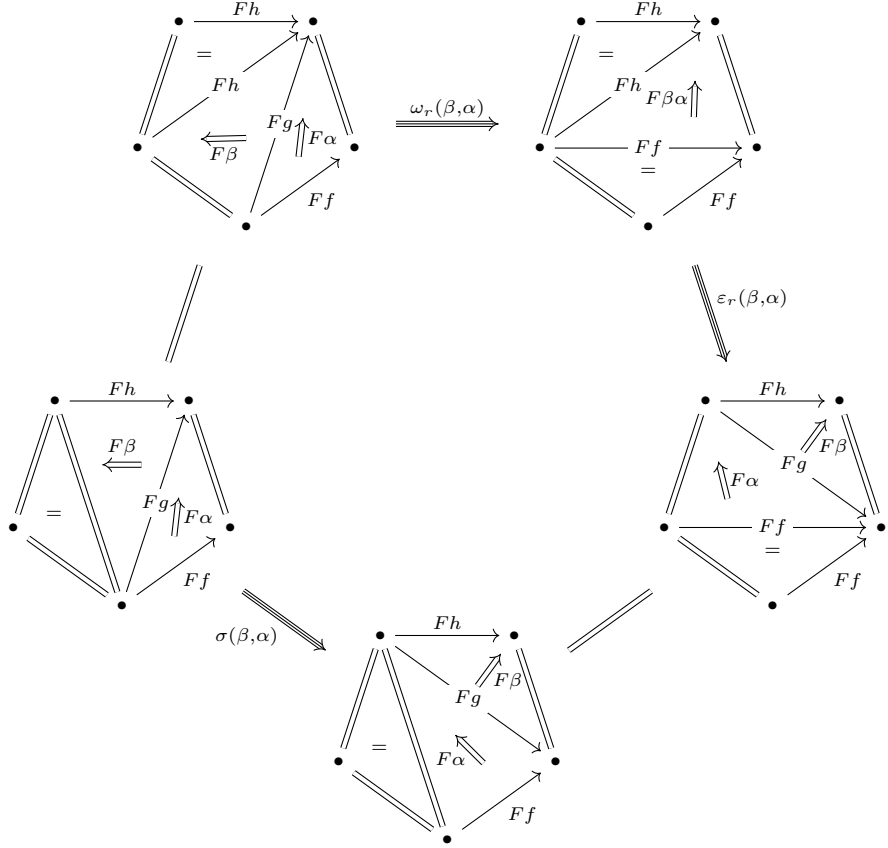
\begin{figure}
    \begin{tikzpicture}[font=\scriptsize, scale=1.5]
    \pentagon{%
    /pentagon/label/.cd,    
    23=$Fh$, 04=$Ff$,
    03=$Fg$, 13=$Fh$, 14=$Ff$, 24=$Fg$,
    012={$=$}, 034=$F\alpha$, 023=$F\beta$, 123={$=$}, 134=$F\beta\alpha$,
    014={$=$}, 024=$F\alpha$, 234=$F\beta$, 013=$F\beta$, 124=$F\alpha$,
    0134=${\omega_r(\beta, \alpha)}$, 0234=${\sigma(\beta, \alpha)}$, 1234=${\eps_r(\beta, \alpha)}$,
    /pentagon/arrowstyle/.cd,
    01={equal}, 12={equal}, 34={equal}, 02={equal}, 0123={equal}, 0124={equal},
    012={phantom, description}, 123={phantom, description}, 014={phantom, description},
    01234={phantom, description},
    /pentagon/labelstyle/.cd,
    012={anchor=center}, 123={anchor=center}, 014={anchor=center},
    01234={anchor=center}
    }
    \end{tikzpicture}
    \caption{The $4$-simplex $x_\omega$, showing that $\omega_{\text{r}}(\beta, \alpha)$
    is invertible.}
    \label{fig:omega_r}
    \end{figure}
    \end{paragr}

\begin{paragr}\label{paragr:cond-iii}
	Let $F \colon A \to B$ be a normalised oplax $3$-functor
	and consider the normalised oplax $3$-functor $T \colon \Deltan{3} \to A$
	given by
	\begin{center}
		\begin{tikzpicture}[scale=1.3]
		\squares{%
			/squares/label/.cd,
			01={$h$}, 12={$i$}, 23={}, 03={$f$}, 02={$g$}, 13={$i$},
			012={$\beta$}, 023={$\alpha$}, 013={$\beta\comp_1\alpha$}, 123={$=$},
			0123={$1_{\beta \comp_1 \alpha}$},
			/squares/arrowstyle/.cd,
			23={equal}, 013={phantom, description}, 123={phantom, description},
			/squares/labelstyle/.cd,
			013={anchor=center}, 123={anchor=center}
		}
		\end{tikzpicture}\ .
	\end{center}
	Following the definition given in paragraph~\ref{paragr:def_cellular_to_simplicial},
	the conditions of normalisations impose that the image under $F$ of $T$ is
	\begin{center}
		\begin{tikzpicture}[scale=1.5]
		\squares{%
			/squares/label/.cd,
			01={$F_{\treeL}(h)$}, 12={$F_{\treeL}(i)$}, 23={}, 03={$F_{\treeL}(f)$},
			02={$F_{\treeL}(g)$}, 13={$F_{\treeL}(i)$},
			012={$F(\bar\beta)$}, 023={$F_{\treeLL}(\alpha)$}, 013={$\phantom{O}F(\overline\beta\comp_1\alpha)$},
			123={},
			0123={$\Gamma$},
			/squares/arrowstyle/.cd,
			23={equal}, 013={phantom, description}, 123={phantom, description},
			/squares/labelstyle/.cd,
			013={anchor=center}, 123={anchor=center}
		}
		\end{tikzpicture}\,,
	\end{center}
	where paragraph~\ref{paragr:cond-ii} and Remark~\ref{rem:treeY} give
	\[F(\bar\alpha) = F_{\treeV(g, f)} \comp_1 F_{\treeLL}(\alpha)\]
	as well as
	\[
	F(\overline\beta\comp_1\alpha) = F_{\treeV(g, f)} \comp_1 F_{\treeLL}(\beta) \comp_1 F_{\treeLL}(\alpha)\,.
	\]
	Moreover, the four main $3$-cells of $\Gamma$ are by definition
	\[
	F_{\treeVLeft}(1_{i}, h) = 1_{F_{\treeV}(i, h)}\,,
	\]
	\[
	F_{\treeW}(1_{a''}, i, h) = 1_{F_{\treeV}(i, h)}\,,
	\]
	\[
	F_{\treeLLL}(1_{\beta\comp_1\alpha}) = 1_{F_{\treeLL}(\beta)\comp_1 F_{\treeLL}(\alpha)}
	\]
	and
	\[
	F_{\treeVRight}(1_{a''}, \beta) = 1_{F_{\treeLL(\beta)}}\,.
	\]
	Hence the $3$-cell $\Gamma$ is trivial, which implies
	that the $3$-cells $\eps_{\text{l}}(\beta, \alpha)$,
	$\eps_{\text{r}}(\beta, \alpha)$, $\omega_{\text{l}}(\beta, \alpha)$ and $\omega_{\text{r}}(\beta, \alpha)$
	associated to the morphism of simplicial sets
	$\Nl(F)\colon \SN(A) \to \SN(B)$ are all trivial.
\end{paragr}

    \begin{paragr}\label{paragr:encode_3cells}
     The last piece of information we want to analyse in this section is
     the behaviour of the morphism $F \colon \SN(A) \to \SN(B)$ with respect
     to the $3$\nbd-cells of $A$. Let
     \[
	   \begin{tikzcd}[column sep=4.5em]
	   \bullet
	   \ar[r, bend left=60, "f", ""'{name=f}]
	   \ar[r, bend right=60, "g"', ""{name=g}]
	   \ar[Rightarrow, from=f, to=g, shift right=0.5em, bend right, shorten <=1mm, shorten >=1mm, "\alpha"', ""{name=al}]
	   \ar[Rightarrow, from=f, to=g, shift left=0.5em, bend left, shorten <=1mm, shorten >=1mm, "\beta", ""'{name=ar}]
	   \arrow[triple, from=al, to=ar, "\Gamma"]{}
	   & \bullet
	   \end{tikzcd}
     \]
     be a $3$-cell of $A$. We have several ways of encoding $\Gamma$ as a $3$-simplex of $\SN(A)$.
     In particular, we have the following $3$-simplices:
     
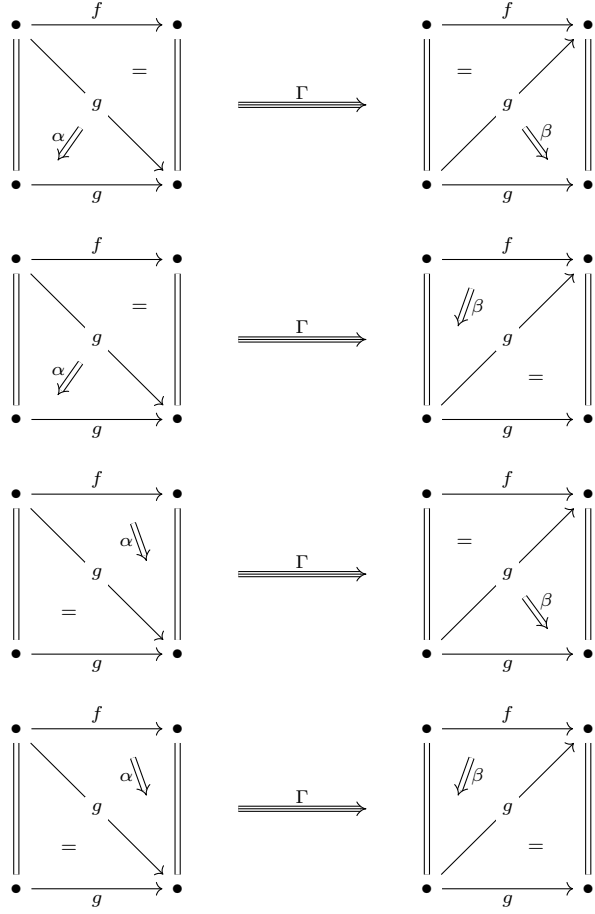
\begin{center}
     	\begin{tikzpicture}[scale=1.5,  font=\footnotesize]
     	\squares{%
     		/squares/label/.cd,
     		01={}, 12={$g$}, 23={}, 03={$f$}, 02={$g$}, 13={$g$},
     		012={$\alpha$}, 023={$=$}, 013={$=$}, 123={$\beta$},
     		0123={$\Gamma$},
     		/squares/arrowstyle/.cd,
     		01={equal}, 23={equal},
     		023={phantom, description}, 013={phantom, description},
     		/squares/labelstyle/.cd,
     		023={anchor=center}, 013={anchor=center}
     	}
     	\end{tikzpicture}
     \end{center}
     \begin{center}
     	\begin{tikzpicture}[scale=1.5,  font=\footnotesize]
     	\squares{%
     		/squares/label/.cd,
     		01={}, 12={$g$}, 23={}, 03={$f$}, 02={$g$}, 13={$g$},
     		012={$\alpha$}, 023={$=$}, 013={$\beta$}, 123={$=$},
     		0123={$\Gamma$},
     		/squares/arrowstyle/.cd,
     		01={equal}, 23={equal},
     		023={phantom, description}, 123={phantom, description},
     		/squares/labelstyle/.cd,
     		023={anchor=center}, 123={anchor=center}
     	}
     	\end{tikzpicture}
     \end{center}
     \begin{center}
     	\begin{tikzpicture}[scale=1.5,  font=\footnotesize]
     	\squares{%
     		/squares/label/.cd,
     		01={}, 12={$g$}, 23={}, 03={$f$}, 02={$g$}, 13={$g$},
     		012={$=$}, 023={$\alpha$}, 013={$=$}, 123={$\beta$},
     		0123={$\Gamma$},
     		/squares/arrowstyle/.cd,
     		01={equal}, 23={equal},
     		012={phantom, description}, 013={phantom, description},
     		/squares/labelstyle/.cd,
     		012={anchor=center}, 013={anchor=center}
     	}
     	\end{tikzpicture}
     \end{center}
     \begin{center}
     	\begin{tikzpicture}[scale=1.5,  font=\footnotesize]
     	\squares{%
     		/squares/label/.cd,
     		01={}, 12={$g$}, 23={}, 03={$f$}, 02={$g$}, 13={$g$},
     		012={$=$}, 023={$\alpha$}, 013={$\beta$}, 123={$=$},
     		0123={$\Gamma$},
     		/squares/arrowstyle/.cd,
     		01={equal}, 23={equal},
     		012={phantom, description}, 123={phantom, description},
     		/squares/labelstyle/.cd,
     		012={anchor=center}, 123={anchor=center}
     	}
     	\end{tikzpicture}
     \end{center}
     
     We claim that all these $3$-simplices of $\SN(A)$ are sent under $F$ to $3$-simplices of $\SN(B)$
     such that their principal $3$-cell is invertible,
     principal cells that we shall call $F\Gamma_i$, $i=1, 2, 3, 4$, respectively.
     The $4$-simplex $z_\Gamma$ of $\SN(B)$ depicted
     in~\eqref{fig:Gamma1}
     shows this claim for the last two $3$-simplices, while the $4$-simplex~$y_\Gamma$
     depicted in figure~\eqref{fig:Gamma2}
     proves the claim for the middle two. We leave to the reader the easy
     assignment of describing a $4$-simplex of $\SN(B)$ showing the claim
     for the first two $3$-simplices above.
    \end{paragr}
    
    \begin{figure}
    	\centering
    	\subfloat[][The $4$-simplex $z_\Gamma$.]{%
    		\begin{tikzpicture}[scale=1.25, font=\footnotesize]\label{fig:Gamma1}
    		\pentagon{%
    			/pentagon/label/.cd,
    			01={}, 12={$Fg$}, 23={}, 34={}, 04={$Ff$},
    			02={$Fg$}, 03={$Ff$}, 13={$Fg$}, 14={$Ff$}, 24={},
    			012={$=$}, 023={$F\alpha_r$}, 034={$\ =$},
    			013={$F\beta_l$}, 123={$=$}, 014={$=$},
    			134={$F\beta_r$}, 124={$F\beta_r$},
    			024={$F\alpha_r$}, 234={$=$},
    			0123={$F\Gamma_4$},
    			0134={$\tau_\text{d}(\beta)$},
    			1234={},
    			0234={},
    			0124={$F\Gamma_3$},
    			01234={},
    			/pentagon/arrowstyle/.cd,
    			01={equal}, 23={equal}, 34={equal}, 24={equal},
    			012={phantom, description}, 034={phantom, description},
    			123={phantom, description},
    			014={phantom, description}, 234={phantom, description},
    			1234={equal}, 0234={equal},
    			01234={phantom, description},
    			/pentagon/labelstyle/.cd,
    			24={}, 024={below = 5pt},
    			012={anchor=center}, 034={anchor=center},
    			123={anchor=center},
    			014={anchor=center}, 234={anchor=center}
    		}
    		\end{tikzpicture}
    	} \\
    	\subfloat[][The $4$-simplex $y_\Gamma$.]{%
    		\begin{tikzpicture}[scale=1.25, font=\footnotesize]\label{fig:Gamma2}
    		\pentagon{%
    			/pentagon/label/.cd,
    			01={}, 12={}, 23={$Fg$}, 34={}, 04={$Ff$},
    			02={}, 03={$Ff$}, 13={$Fg$}, 14={$Ff$}, 24={$Fg$},
    			012={$=$}, 023={$F\alpha_l$}, 034={$=$},
    			013={$F\alpha_l$}, 123={$=$}, 014={$=$},
    			134={$F\alpha_r$}, 124={$F\beta_r$},
    			024={$F\beta_l$}, 234={$=$},
    			0123={},
    			0134={$\tau_\text{l}(\alpha)$},
    			1234={$F\Gamma_2$},
    			0234={$F\Gamma_3$},
    			0124={},
    			01234={},
    			/pentagon/arrowstyle/.cd,
    			01={equal}, 12={equal}, 34={equal}, 02={equal},
    			012={phantom, description}, 034={phantom, description},
				123={phantom, description},
    			014={phantom, description}, 234={phantom, description},
    			0124={equal}, 0123={equal},
    			01234={phantom, description},
    			/pentagon/labelstyle/.cd,
    			02={}, 024={below = 5pt},
    			012={anchor=center}, 034={anchor=center},
    			123={anchor=center},
    			014={anchor=center}, 234={anchor=center},
    		}
    		\end{tikzpicture}
    	}
    	\caption{The $4$-simplices governing the images of $\Gamma$.}
    	\label{fig:Gamma}
    \end{figure}
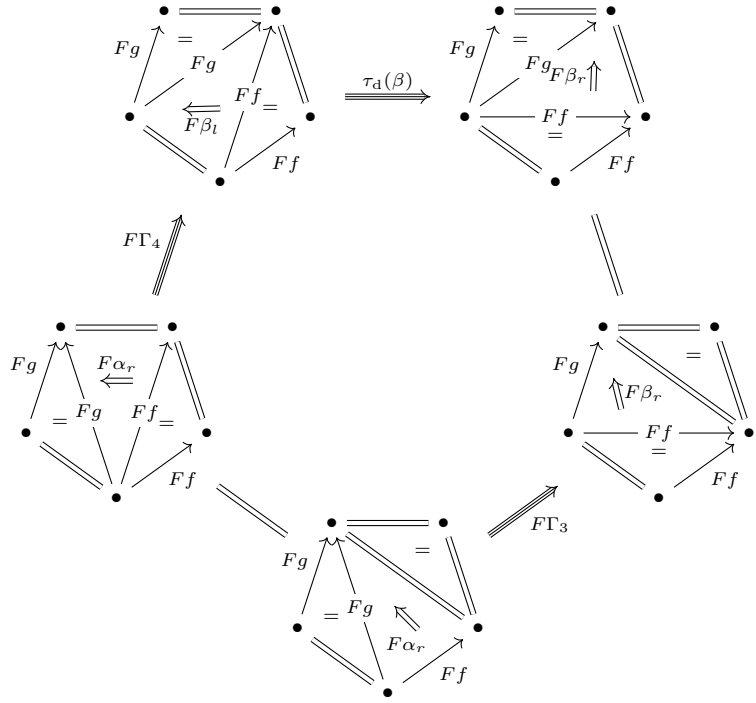
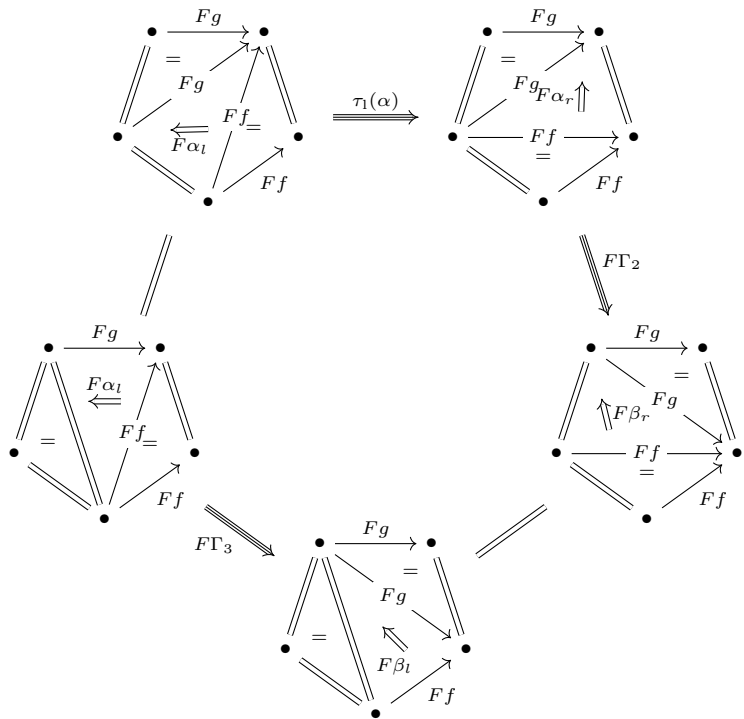

    \begin{rem}
    	Notice that the $3$-cells $F\Gamma_i$, $i=1, 2, 3, 4$,of the preceding paragraph are linked together by
    	whiskering with invertible $2$-cells studied
    	in this subsection, such as $\tau_{\text{d}}$ and
    	$\tau_{\text{u}}$. If $F$ is actually the image
    	of a normalised oplax $3$-functor $G$, then we have
    	shown that these $2$-cells are trivial and therefore
    	in this case the four $3$-cells $\Nl(G)\Gamma_i$
    	are in fact all equal to $G_{\treeLLL}(\Gamma)$.
    \end{rem}

    \section{Simplicial oplax 3-morphisms}
    The previous subsection shows that if we consider
    a normalised oplax $3$-functor $F \colon A \to B$, then its nerve
    $\Nl(F)$ has the property that
    some particular (non-degenerate) $3$-simplices of $\SN(A)$ with trivial
    principal $3$-cell are sent to $3$-simplices of $\SN(3)$
    where the principal $3$-cell is also trivial.
    
    In this section we consider the class of morphisms
    of simplicial sets between Street nerve of $3$-categories
    having precisely this property and we show that
    they form a subcategory of simplicial sets.
    In fact, we will prove that they are canonically equivalent to normalised oplax $3$-functors.
    As a first main step towards this correspondence, in the
    following subsection we shall show how to associate a normalised
    oplax $3$-functor to a simplicial morphism between nerves
    of $3$-categories satisfying these trivialising properties.
    
    Throughout this subsection, we shall make heavy use
    of the notations introduced in the preceding subsection.
    
    \begin{definition}\label{def:simpl_oplax}
    	Let $A$ and $B$ be two small $3$-categories.
    	We say that a morphism $F \colon \SN(A) \to \SN(B)$
    	is a \ndef{simplicial oplax $3$-morphism}
    	\index{simplicial oplax $3$-morphism} if the following conditions are satisfied:
    	\begin{enumerate}
    		\item\label{cond:simpl_oplax-i} for any $2$-cell $\alpha$ of $A$, the $3$-cell
    		$\tau_{\text{d}}(\alpha)$ of $B$
    		is trivial;

    		\item\label{cond:simpl_oplax-ii} for any $2$-cell $\alpha \colon f \to h\comp_0 g$ of $A$,
    		the $3$-cell $\gamma_{\text{l}}(\alpha)$
    		of $B$ is trivial;
    		
            \item\label{cond:simpl_oplax-iii} for any pair of $1$-composable $2$-cells
            $\alpha$ and $\beta$ of $A$ as in~\ref{paragr:encode_horizontal_comp},
    		the $3$-cell $\eps_{\text{l}}(\beta, \alpha)$
    		of $B$ is trivial.
    		
    	\end{enumerate}
    \end{definition}
    
    \begin{rem}
    	It is clear that the definition above can be
    	framed within stratified simplicial sets.
    	However, the author sees little or no advantage
    	in pursuing this point of view,
    	since the $3$-simplices involved are very particular
    	and no lifting property is present.
    \end{rem}
    
    \begin{rem}\label{rem:trivial_cells}
     The relations described in paragraphs~\ref{paragr:tau_invertible}
     and~\ref{paragr:sigma} tell us
     that under condition~\ref{cond:simpl_oplax-i} above also the
     $3$-cells $\tau_{\text{u}}(\gamma)$ and $\sigma(\beta, \alpha)$
     of $B$ are also trivial, for any choice of $2$-cells
     $\alpha$, $\beta$ and $\gamma$ of $A$, such that the first two are $1$-composable.
     
     Assuming conditions~\ref{cond:simpl_oplax-i} and~\ref{cond:simpl_oplax-ii}
     and using what we just observed in the relations of paragraph~\ref{paragr:gamma_invertible}
     we get immediately that the $3$-cell $\gamma_{\text{r}}(\alpha)$ is trivial.
     
     If $F \colon \SN(A) \to \SN(B)$ is a simplicial oplax $3$-morphism,
     putting together all we have said right above
     and the relations analysed in paragraphs~\ref{paragr:eps_invertible}
     and~\ref{paragr:encode_simple_horizontal_comp} gives us that
     the $3$-cells $\eps_{\text{r}}(\beta, \alpha)$, $\omega_{\text{r}}(\beta, \alpha)$
     and $\omega_{\text{l}}(\beta, \alpha)$ are trivial, for any appropriate
     choice of $3$-cells $\alpha$ and $\beta$ of $A$.
     
     In short, all the invertible $3$-cells of $B$ we described in the previous
     subsection are actually trivial whenever $F$ is a simplicial oplax $3$-morphism.
    \end{rem}

    \begin{paragr}\label{paragr:2-cell_B}
     Let $F \colon \SN(A) \to \SN(B)$ be a morphism of simplicial sets
     satisfying condition~\ref{cond:simpl_oplax-i}. It follows from the previous
     remark that for any $2$-cell $\alpha$ of $A$, the $2$-cells $F(\alpha_l)$
     and $F(\alpha_r)$ of $B$ coincide. Whenever this happens we shall then
     simply write $F(\alpha)$, or more often just $F\alpha$ for this $2$-cell
     of $B$.
     Furthermore, the relations observed in paragraph~\ref{paragr:encode_3cells}
     give us that for any $3$-cell $\Gamma \colon\alpha \to \beta$ of $A$,
     the four ways we described in that paragraph to encode
     the image of $\Gamma$ via $F$ are all the same $3$-cell of $B$, that
     we shall then call $F(\Gamma)$ or simply $F\Gamma$.
    \end{paragr}

	\begin{exem}
		For any normalised oplax $3$-functor $F \colon A \to B$,the previous subsection shows that its nerve
		$\Nl(F)$ is a simplicial oplax $3$-morphism.
	\end{exem}

    We now check that simplicial oplax $3$-morphisms are closed under composition.
    Let $F \colon A \to B$ and $G \colon B \to C$ be two simplicial oplax $3$-morphisms.
    
    \begin{paragr}\label{paragr:simpl_oplax_i}
     Let $\alpha \colon f \to g$ be a $2$-cell of $A$.
     By assumption, we have the $3$-simplex
     \begin{center}
      \begin{tikzpicture}[scale=1.5, font=\footnotesize]
       \squares{%
        /squares/label/.cd,
        01={}, 12={$F(g)$}, 23={}, 03={$F(f)$}, 02={$F(f)$}, 13={$F(f)$},
        012={$F(\alpha_l)$}, 023={$=$}, 013={$=$}, 123={$F(\alpha_r)$},
        0123={$=$},
        /squares/arrowstyle/.cd,
        01={equal}, 23={equal},
        023={phantom, description}, 013={phantom, description},
        0123={phantom, description},
        /squares/labelstyle/.cd,
        012={below right= 1pt and 1pt}, 123={below left = 1pt and 1pt},
        023={anchor=center}, 013={anchor=center},
        0123={anchor=center}
       }
      \end{tikzpicture}
     \end{center}
    of $\SN(B)$ . Setting $\beta = F(\alpha_l) = F(\alpha_r)$
    and applying $G$ to the two $3$-simplices above, we get
    the following $3$-simplex of $\SN(C)$:
    \begin{center}
      \begin{tikzpicture}[scale=1.5, font=\footnotesize]
       \squares{%
        /squares/label/.cd,
        01={}, 12={$GFg$}, 23={}, 03={$GFf$}, 02={$GFf$}, 13={$GFf$},
        012={$G(\beta_l)$}, 023={$=$}, 013={$=$}, 123={$G(\beta_r)$},
        0123={$=$},
        /squares/arrowstyle/.cd,
        01={equal}, 23={equal},
        023={phantom, description}, 013={phantom, description},
        0123={phantom, description},
        /squares/labelstyle/.cd,
        012={below right= 1pt and 1pt}, 123={below left = 1pt and 1pt},
        023={anchor=center}, 013={anchor=center},
        0123={anchor=center}
       }
      \end{tikzpicture}\ .
     \end{center}
    Since $\beta_l = F(\alpha_l)$ and $\beta_r = F(\alpha_r)$ by definition,
    we have that $G(\beta_l) = GF(\alpha_l)$ and
    $G(\beta_r) = GF(\alpha_r)$. Thus the morphism $GF \colon \SN(A) \to \SN(C)$
    of simplicial sets satisfies condition~\ref{cond:simpl_oplax-i} of
    the definition of simplicial oplax $3$-morphisms.
    \end{paragr}
    
    \begin{paragr}\label{paragr:simpl_oplax_ii}
     Let $\alpha \colon f \to h\comp_0 g$ be a $2$-cell of $A$.
     By assumption, we have a $3$-simplex
     \begin{center}
      \begin{tikzpicture}[scale=1.5, font=\footnotesize]
       \squares{%
        /squares/label/.cd,
        01={}, 12={$F(g)$}, 23={$F(h)$}, 03={$F(f)$}, 02={$F(g)$}, 13={$F(hg)$},
        012={$=$}, 023={$F(\bar\alpha)$}, 013={$F(\alpha)$}, 123={$F_{g, f}$},
        0123={$\gamma_{\text{l}}(\alpha)$},
        /squares/arrowstyle/.cd,
        01={equal},
        012={phantom, description},
        0123={equal},
        /squares/labelstyle/.cd,
        012={below right= 1pt and 1pt}, 123={below left = 1pt and 1pt},
        012={anchor=center}
       }
      \end{tikzpicture}
     \end{center}
     of $\SN(B)$, where the main $3$-cell is trivial as the morphism $F$
     verifies condition~\ref{cond:simpl_oplax-ii}.
     We have to show that the $3$-cell of the $3$-simplex
     \begin{center}
      \begin{tikzpicture}[scale=1.5, font=\footnotesize]
       \squares{%
        /squares/label/.cd,
        01={}, 12={$GFg$}, 23={$GFh$}, 03={$GFf$}, 02={$GFg$}, 13={$GFhg$},
        012={$=$}, 023={$GF(\bar\alpha)$}, 013={$GF(\alpha)$}, 123={$GF_{g, f}$},
        0123={$\gamma_{\text{l}}(\alpha)$},
        /squares/arrowstyle/.cd,
        01={equal},
        012={phantom, description},
        /squares/labelstyle/.cd,
        012={below right= 1pt and 1pt}, 123={below left = 1pt and 1pt},
        023={above left = -1pt and 1pt}, 013={swap, above right = -1pt and 1pt},
        012={anchor=center}
       }
      \end{tikzpicture}
     \end{center}
    of $\SN(C)$ is trivial. The $2$-cell $\beta = F_{g, f} \comp_1 F\alpha$
    of $B$ has $Ff$ as source and $Fh \comp_0 Fg$ as target.
    So applying the morphism $G \colon \SN(B) \to \SN(C)$ we get
    a $3$-simplex
    \begin{center}
      \begin{tikzpicture}[scale=1.5, font=\footnotesize]
       \squares{%
        /squares/label/.cd,
        01={}, 12={$GFg$}, 23={$GFh$}, 03={$GFf$}, 02={$GFg$}, 13={$GFhg$},
        012={$=$}, 023={$G\bar\beta$}, 013={$G(\beta)$}, 123={$G_{Fg, Ff}$},
        0123={$\gamma_{\text{l}}(\beta)$},
        /squares/arrowstyle/.cd,
        01={equal},
        012={phantom, description},
        0123={equal},
        /squares/labelstyle/.cd,
        123={below left = 1pt and 1pt},
        012={anchor=center}
       }
      \end{tikzpicture}
     \end{center}
    of $\SN(C)$ where the main $3$-cell is trivial by condition~\ref{cond:simpl_oplax-ii}.
    Notice that by definition $G\bar\beta = GF\bar\alpha$ and
    $G\beta = G(F_{h, g} \comp_1 F\alpha)$; this latter $2$-cell of~$C$
    is equal to $G(F_{h, g}) \comp_1 GF\alpha$ by condition~\ref{cond:simpl_oplax-iii}.
    Moreover, condition~\ref{cond:simpl_oplax-ii}
    also entails that the following $3$-simplex
    \begin{center}
      \begin{tikzpicture}[scale=1.5, font=\footnotesize]
       \squares{%
        /squares/label/.cd,
        01={$GFg$}, 12={$GFh$}, 23={}, 03={$GFhg$}, 02={$G(FhFg)$}, 13={$GFh$},
        012={$G_{Fh, Fg}$}, 023={$G(F_{h, g})$}, 013={$G(\overline{F_{h, g}})$},
        123={$=$},
        0123={$\gamma_{\text{r}}(F_{h, g})$},
        /squares/arrowstyle/.cd,
        23={equal},
        123={phantom, description},
        0123={equal},
        /squares/labelstyle/.cd,
        012={below right= 0pt and 1pt}, 123={anchor=center},
        023={above left}, 
        013={swap, above right = 1pt and 1pt}
       }
      \end{tikzpicture}
     \end{center}
    of $\SN(C)$ is trivial, as we have observed in Remark~\ref{rem:trivial_cells}.
    The $4$-simplex of $\SN(C)$ displayed
    in figure~\ref{fig:simpl_oplax_ii} allows to conclude.
    
    \begin{figure}
     \begin{tikzpicture}[font=\footnotesize, scale=1.6]
    \pentagon{%
    /pentagon/label/.cd,
    12=$GFg$, 23=$GFh$, 04=$GFf$,
    012=${=}$, 034=${=}$, 023=$GF\bar\alpha$, 123=${G_{Fh, Fg}}$, 134=${G(F_{h, g})}$,
    014=$GF\alpha$, 024=$GF\bar\alpha$, 234=${=}$, 013=$G\beta$, 124=${(GF)_{h, g}}$,
    0123={\footnotesize{\ref{cond:simpl_oplax-ii}}}, 0124=${\gamma_l(\alpha)}$,
    0134=${\omega_l(F_{h, g}, F\alpha)}$, 1234={\footnotesize{\ref{cond:simpl_oplax-ii}}},
    /pentagon/arrowstyle/.cd,
    01={equal}, 34={equal},
    0234={equal}, 0123={equal}, 0134={equal}, 1234={equal},
    012={phantom, description}, 013={phantom, description},
    014={phantom, description}, 023={phantom, description},
    024={phantom, description}, 034={phantom, description},
    123={phantom, description}, 124={phantom, description},
    134={phantom, description}, 234={phantom, description},
    01234={phantom},
    /pentagon/labelstyle/.cd,
    012={anchor=center}, 013={anchor=center}, 014={anchor=center},
    023={anchor=center}, 024={anchor=center}, 034={anchor=center},
    123={anchor=center}, 124={anchor=center}, 134={anchor=center},
    234={anchor=center}
    }
    \end{tikzpicture}
    \caption{The $3$-cell $\gamma_\text{l}(\alpha)$ is trivial.}
    \label{fig:simpl_oplax_ii}
    \end{figure}
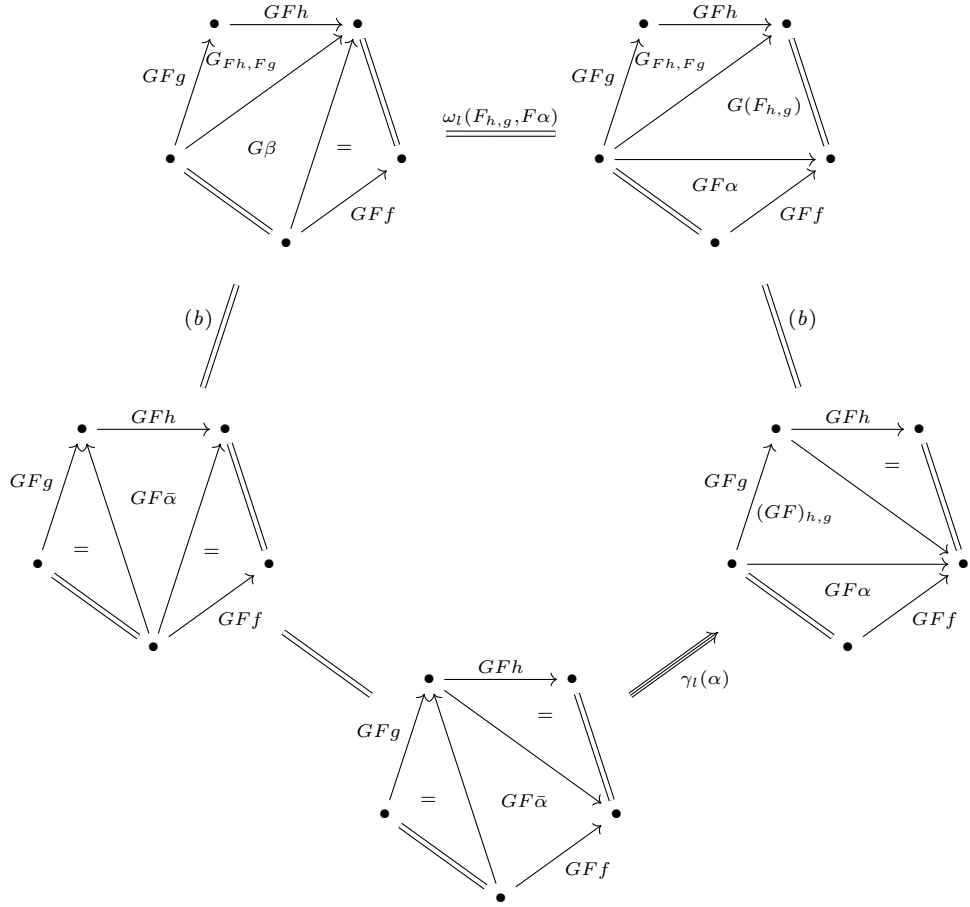

    \end{paragr}
    
    \begin{paragr}\label{paragr:simpl_oplax_iii}
     Consider two $1$-composable $2$-cells
     
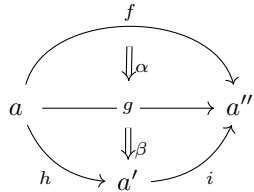
\[
      \begin{tikzcd}[row sep=1.35em]
       a^{\phantom\prime}
       \ar[rr, bend left=75, "f", ""'{name=f}]
       \ar[rr, "g"{description, name=g}]
       \ar[rd, bend right, "h"']
       &&
         a''
       \\
       & a'
         \ar[ru, bend right, "i"']
       &
       \ar[Rightarrow, from=f, to=g, shorten <=1mm, shorten >=2mm, "\alpha"]
       \ar[Rightarrow, from=g, to=2-2, shorten <=1mm, pos=0.4, "\beta" near end]
      \end{tikzcd}
     \]
     of $A$. The $3$-simplex
     
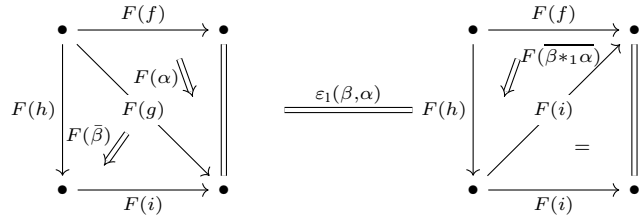
\begin{center}
      \begin{tikzpicture}[scale=1.5, font=\footnotesize]
       \squares{%
        /squares/label/.cd,
        01={$F(h)$}, 12={$F(i)$}, 23={}, 03={$F(f)$}, 02={$F(g)$}, 13={$F(i)$},
        012={$F(\bar\beta)$}, 023={$F(\alpha)$}, 013={$F(\overline{\beta\comp_1\alpha})$},
        123={$=$},
        0123={$\eps_{\text{l}}(\beta, \alpha)$},
        /squares/arrowstyle/.cd,
        23={equal},
        123={phantom, description},
        0123={equal},
        /squares/labelstyle/.cd,
        123={anchor=center},
        013={swap, above right = 2pt and 1pt}
       }
      \end{tikzpicture}
     \end{center}
     of $\SN(B)$ has trivial main $3$-cell $\eps_\text{l}(\beta, \alpha)$
     by condition~\ref{cond:simpl_oplax-iii}.
     Thus $F\overline{\beta \comp_1 \alpha}$ is equal to $F\bar \beta \comp_1 F\alpha$
     (which by the preceding paragraph we also know to be equal to the $3$\nbd-cell
     $F_{h, i} \comp_1 F\beta \comp_1 F\alpha$).
     Now, applying the morphism $G$ with condition~\ref{cond:simpl_oplax-iii} at hand
     we get the $3$-simplex
     \begin{center}
      \begin{tikzpicture}[scale=1.5, font=\footnotesize]
       \squares{%
        /squares/label/.cd,
        01={$GFh$}, 12={$GFi$}, 23={}, 03={$GFf$}, 02={$GFg$}, 13={$GFi$},
        012={$G\overline{F\bar\beta}$}, 023={$GF\alpha$},
        013={$\;G\overline{F\overline{\beta\comp_1\alpha}}$},
        123={$=$},
        0123={$\eps_{\text{l}}(F\beta, F\alpha)$},
        /squares/arrowstyle/.cd,
        23={equal},
        123={phantom, description},
        013={phantom, description},
        0123={equal},
        /squares/labelstyle/.cd,
        123={anchor=center}, 013={anchor=center},
        012={below right}
       }
      \end{tikzpicture}
     \end{center}
     with trivial main $3$-cell $\eps_\text{l}(F\beta, F\alpha)$.
     Noticing that the $2$-simplex $G\overline{F\bar\beta}$ of $\SN(C)$
     is precisely $GF\bar \beta$
     and that the $2$-simplex $G\overline{F\bar\beta\comp_1\alpha}$ of $\SN(C)$
     is equal to $GF\overline{\beta \comp_1 \alpha}$, we conclude
     that the morphism $GF \colon \SN(A) \to \SN(B)$ satisfies condition~\ref{cond:simpl_oplax-iii}.
    \end{paragr}
    
    \begin{thm}\label{thm:simpl_oplax_category}
     The class of simplicial oplax $3$-morphisms is stable under composition.
    \end{thm}
    
    \begin{proof}
     This follows immediately from paragraphs~\ref{paragr:simpl_oplax_i},
     \ref{paragr:simpl_oplax_ii} and~\ref{paragr:simpl_oplax_iii}.
    \end{proof}
    
    \begin{definition}\label{def:simpl_oplax_category}
     We shall denote by $\dCat$ the subcategory of $\EnsSimp$
     whose objects are the nerves of the
     small $3$-categories and whose morphisms are
     simplicial oplax $3$-functors. This shall be called
     the category of \ndef{small $3$-categories and simplicial
     oplax $3$-morphisms}.
    \end{definition}
    
    \begin{rem}
     The nerve of any $3$-functor $u \colon A \to B$
     is clearly a simplicial oplax $3$-morphism,
     since the Street nerve of a $3$-functor sends
     $3$-simplices of $A$ with trivial principal $3$-cell to
     $3$-simplices of $B$ with trivial principal $3$-cell;
     hence, the nerve induces a faithful functor
     $\nCat{3} \hookrightarrow \dCat$.
    \end{rem}

\subsection{Simplicial to cellular}\label{section:simplicial-to-cellular}
Let $F \colon A \to B$ be a simplicial oplax $3$-morphism.

\begin{paragr}[Data]\label{paragr:simplicial_to_cellular-data}
	We now associate to $F$ the data of a normalised oplax $3$-functor. 

	\begin{description}
	
		\item[%
		\scalebox{0.3}{
			\begin{forest}
				for tree={%
					label/.option=content,
					grow'=north,
					content=,
					circle,
					fill,
					minimum size=3pt,
					inner sep=0pt,
					s sep+=15,
				}
				[ ]
			\end{forest}
		}
		] The map $F_{\treeDot}$ that
		to each object $a$ of $A$, \ie any $0$-simplex of $\SN(A)_0$,
		assigns an object $F_{\treeDot}(a)$ of $B$, \ie a $0$-simplex
		of $\SN(B)$, is simply defined to be $F_0$.
		
		\item[
		\scalebox{0.3}{
			\begin{forest}
				for tree={%
					label/.option=content,
					grow'=north,
					content=,
					circle,
					fill,
					minimum size=3pt,
					inner sep=0pt,
					s sep+=15,
				}
				[ [] ]
			\end{forest}
		}
		] The map $F_{\treeLog}$ that to each $1$-cell $f \colon a \to a'$
		of $A$, \ie any $1$-simplex of $\SN(A)_1$, assigns a $1$-cell $F_{\treeLog}(f) \colon F_{\treeDot}(a) \to F_{\treeDot}(a')$, \ie a $1$-simplex of $\SN(B)$,
		is simply defined to be $F_1 \colon \SN(A)_1 \to \SN(B)_1$.
		
		\item[%
		\scalebox{0.3}{
			\begin{forest}
				for tree={%
					label/.option=content,
					grow'=north,
					content=,
					circle,
					fill,
					minimum size=3pt,
					inner sep=0pt,
					s sep+=15,
				}
				[ [][] ]
			\end{forest}
		}
		] The map $F_{\treeV}$ that 
			to each pair of $0$\hyp{}composable $1$\hyp{}cells
			\[
			 \begin{tikzcd}
			  a \ar[r, "f"] & a' \ar[r, "g"] &  a''
			 \end{tikzcd}
			\]
			of $A$ assigns a $2$-cell $F_{\treeV}(g, f)$
			\[
			 \begin{tikzcd}[column sep=small]
				 F_{\treeDot}(a)
				 \ar[rd, "F_{\treeLog}(f)"']
				 \ar[rr, "F_{\treeLog}(g\comp_0 f)"{name=gf}] &&
				 F_{\treeDot}(a'') \\
				 & F_{\treeDot}(a') \ar[ru, "F_{\treeV}(g)"'] & 
				 \ar[Rightarrow, shorten <=1.5mm, from=gf, to=2-2]
			 \end{tikzcd}
			\]
			of $B$, that is
			\[
			 F_{\treeV}(g, f) \colon F_{\treeLog}(g \comp_0 f) \to
			 F_{\treeLog}(g) \comp_0 F_{\treeLog}(f)
			\]
			is defined as $F_{\treeV}(g, f) := F_{g, f}$
			(see paragraph~\ref{paragr:encode_comp_1cells}).

		\item[%
		\scalebox{0.3}{
			\begin{forest}
				for tree={%
					label/.option=content,
					grow'=north,
					content=,
					circle,
					fill,
					minimum size=3pt,
					inner sep=0pt,
					s sep+=15,
				}
				[ [[]] ]
			\end{forest}
		}
		] The map $F_{\treeLL}$
		that to each $2$-cell $\alpha \colon f \to g$ of $A$ associates
		a $2$-cell
		\[F_{\treeLL}(\alpha) \colon F_{\treeLog}(f) \to F_{\treeLog}(g)\]
		of $B$ is defined to be $F_{\treeLL}(\alpha) = F(\alpha)$,
		with the notation of paragraph~\ref{paragr:2-cell_B}.
		
		\item[%
		\scalebox{0.3}{
			\begin{forest}
				for tree={%
					label/.option=content,
					grow'=north,
					content=,
					circle,
					fill,
					minimum size=3pt,
					inner sep=0pt,
					s sep+=15,
				}
				[ [] [] [] ]
			\end{forest}
		}
		] We define the map $F_{\treeW}$
        by mapping any triple of $0$-composable $1$ cells
		\[
		 \begin{tikzcd}[column sep=small]
		  a \ar[r, "f"] & a' \ar[r, "g"] & a'' \ar[r, "h"] & a'''
		 \end{tikzcd}
		\]
		of $A$, \ie any $3$-simplex of $\SN(A)$ of the form
		\begin{center}
		\begin{tikzpicture}[scale=1.5]
            \squares{%
                /squares/label/.cd,
                0=$a$, 1=$a'$, 2=$a''$, 3=$a'''$,
                01=$f$, 12=$g$, 23=$h$, 02=$gf$, 03=$hgf$, 13=$hg$,
                012=${=}$, 023=${=}$, 123=${=}$, 013=${=}$,
                /squares/arrowstyle/.cd,
                012={phantom, description}, 023={phantom, description},
                123={phantom, description}, 013={phantom, description},
                0123={equal},
                /squares/labelstyle/.cd,
                012={anchor=center}, 023={anchor=center},
                123={anchor=center}, 013={anchor=center}
            }
        \end{tikzpicture}\ ,
        \end{center}
        to the main $3$-cell $F_{\treeW}(h, g, f)$ of 
        the $3$-simplex of $B$
        image of the above $3$-simplex of $A$ by the morphism $F$:
        \begin{center}
		\begin{tikzpicture}[scale=1.5]
            \squares{%
                /squares/label/.cd,
                0=$Fa$, 1=$Fa'$, 2=$Fa''$, 3=$Fa'''$,
                01=$Ff$, 12=$Fg$, 23=$Fh$, 02=$F(gf)$, 03=$F(hgf)$, 13=$F(hg)$,
                012=${F_{g, f}}$, 023=${F_{h, gf}}$, 123=${F_{h, g}}$, 013=${F_{hg, f}}$,
                0123=${F(h, g, f)}$,
                /squares/arrowstyle/.cd,
                012={phantom, description}, 023={phantom, description},
                123={phantom, description}, 013={phantom, description},
                /squares/labelstyle/.cd,
                012={anchor=center}, 023={anchor=center},
                123={anchor=center}, 013={anchor=center}
            }
        \end{tikzpicture}\ ,
        \end{center}
		where $F_{h, gf} = F_{\treeV}(h, g \comp_0 f)$ and $F_{hg, f} = F_{\treeV}(h \comp_0 g, f)$,
		so that the $3$-cell $F_{\treeW}(h, g, f)$ has
		\[
            F_{\treeLog}(h) \comp_0 F_{\treeV}(g, f) \comp_1 F_{\treeV}(h, g \comp_0 f)
		\]
		as source and
		\[
            F_{\treeV}(h, g) \comp_0 F_{\treeLog}(f) \comp_1 F_{\treeV}(h \comp_0 g, f)
		\]
		as target.
		
		\item[%
		\scalebox{0.3}{
			\begin{forest}
				for tree={%
					label/.option=content,
					grow'=north,
					content=,
					circle,
					fill,
					minimum size=3pt,
					inner sep=0pt,
					s sep+=15,
				}
				[ [] [[]] ]
			\end{forest}
		}
		] Consider a whiskering
		\[
		 \begin{tikzcd}[column sep=4.5em]
		  a
		  \ar[r, bend left, "f", ""{below, name=f}]
		  \ar[r, bend right, "f'"', ""{name=fp}]
		  \ar[Rightarrow, from=f, to=fp, "\alpha"]
		  &
		  a' \ar[r,"g"] &
		  a''
		 \end{tikzcd}
		\]
		of $A$ and the following associated $3$-simplex
		\begin{center}
		 \begin{tikzpicture}[scale=1.5]
		  \squares{%
                /squares/label/.cd,
                0=$a$, 1=$a$, 2=$a'$, 3=$a''$,
                12=$f'$, 23=$g$, 02=$f$, 03=$gf$, 13=$gf'$,
                012=$\alpha$, 023=${=}$, 123=${=}$, 013={$g\alpha$},
                /squares/arrowstyle/.cd,
                01={equal},
                023={phantom, description},
                123={phantom, description},
                0123={equal},
                /squares/labelstyle/.cd,
                023={anchor=center},
                123={anchor=center}
		  }
		 \end{tikzpicture}
		\end{center}
        of $\SN(A)$, where we wrote $g\alpha$ for $g \comp_0 \alpha$.
        We define $F_{\treeVRight}(g, \alpha)$
        to be the main $3$-cell of the image under $F$
		of the above $3$-simplex:
		\begin{center}
		 \begin{tikzpicture}[scale=1.5]
		  \squares{%
                /squares/label/.cd,
                0=$Fa$, 1=$Fa$, 2=$Fa'$, 3=$Fa''$,
                12=$Ff'$, 23=$Fg$, 02=$Ff$, 03=$F(gf)$, 13=$F(gf')$,
                012=$F\alpha$, 023=$F_{g, f}$, 123={$F_{g, f'}$},
                013={$Fg\alpha$},
                0123=${F(g, \alpha)}$,
                /squares/arrowstyle/.cd,
                01={equal},
                012={phantom, description}, 023={phantom, description},
                123={phantom, description}, 013={phantom, description},
                /squares/labelstyle/.cd,
                012={anchor=center}, 023={anchor=center},
                123={anchor=center}, 013={anchor=center}
		  }
		 \end{tikzpicture}\ ;
		\end{center}
		so that
		\[
            F_{\treeVRight}(\alpha, g) \colon
            F_{\treeLog}(g) \comp_0 F_{\treeLL}(\alpha) \comp_1 F_{\treeV}(g, f)
            \to
            F_{\treeV}(g, f') \comp_1 F_{\treeLL}(g \comp_0 \alpha)
        \]
        as a $3$-cell of $B$.

		\item[%
		\scalebox{0.3}{
			\begin{forest}
				for tree={%
					label/.option=content,
					grow'=north,
					content=,
					circle,
					fill,
					minimum size=3pt,
					inner sep=0pt,
					s sep+=15,
				}
				[ [[]] [] ]
			\end{forest}
		}
		] Consider a whiskering
		\[
            \begin{tikzcd}[column sep=4.5em]
             a \ar[r, "f"] &
             a'
             \ar[r, bend left, "g", ""{below, name=g}]
		  \ar[r, bend right, "g'"', ""{name=gp}]&
		  a''
		  \ar[Rightarrow, from=g, to=gp, "\beta"]
            \end{tikzcd}
		\]
		of $A$ and the following associated $3$-simplex
		\begin{center}
		 \begin{tikzpicture}[font=\footnotesize, scale=1.5]
		  \squares{%
                /squares/label/.cd,
                0=$a$, 1=$a'$, 2=$a''$, 3=$a''$,
                01=$f$, 12=$g'$, 02=$gf$, 03=$gf$, 13=$g$,
                012={$=$}, 023={$\beta f$}, 123=${\beta}$, 013=${=}$,
                /squares/arrowstyle/.cd,
                23={equal},
                012={phantom, description},
                013={phantom, description},
                0123={equal},
                /squares/labelstyle/.cd,
                012={anchor=center},
                013={anchor=center}
		  }
		 \end{tikzpicture}
		\end{center}
		of $SN(A)$, where we wrote $\beta f$ for $\beta \comp_0 f$.
		We define $F_{\treeVLeft}(\beta, f)$ as the main $3$-cell
		of the image under $F$ of the above $3$-simplex:
		\begin{center}
		 \begin{tikzpicture}[font=\footnotesize, scale=1.6]
		  \squares{%
                /squares/label/.cd,
                0=$Fa$, 1=$Fa'$, 2=$Fa''$, 3=$Fa''$,
                01=$Ff$, 12=$Fg'$, 02=$F(gf)$, 03=$F(gf)$, 13=$Fg$,
                012={$F_{g', f}$}, 023={$F(\beta f)$}, 123=${F\beta}$, 013=${F_{g, f}}$,
                0123=${F(\beta, f)}$,
                /squares/arrowstyle/.cd,
                23={equal}
		  }
		 \end{tikzpicture}\ ;
		\end{center}
		so that
		\[
		 F_{\treeVLeft}(\beta, f) \colon
		 F_{\treeV}(g', f) \comp_1 F_{\treeLL}(\beta \comp_0 f)
		 \to
		 F_{\treeLL}(\beta) \comp_0 F_{\treeLog}(f) \comp_1 F_{\treeV}(g, f)
		\]
		as a $3$-cell of $B$.

		\item[%
		\scalebox{0.3}{
			\begin{forest}
				for tree={%
					label/.option=content,
					grow'=north,
					content=,
					circle,
					fill,
					minimum size=3pt,
					inner sep=0pt,
					s sep+=15,
				}
				[ [[[]]] ]
			\end{forest}
		}] Consider a $3$-cell $\gamma \colon \alpha \to \alpha'$ of $A$,
		and the following associated $3$-simplex
		\begin{center}
		 \begin{tikzpicture}[scale=1.6]
		  \squares{%
		   /squares/label/.cd,
            12=$f'$, 02=$f'$, 03=$f$, 13=$f$,
            012=${=}$, 023=$\alpha$, 123=$\alpha'$, 013=${=}$,
            0123=$\gamma$,
            /squares/arrowstyle/.cd,
            01={equal}, 23={equal},
            012={phantom, description},
            013={phantom, description},
            /squares/labelstyle/.cd,
            012={anchor=center},
            013={anchor=center}
          }
		 \end{tikzpicture}
		\end{center}
		of $\SN(A)$.
		The $3$-cell
		\[F_{\treeLLL}(\gamma) \colon F_{\treeLL}(\alpha) \to F_{\treeLL}(\alpha')\]
		of $B$ is defined to be the main $3$-cell of the image under $F$ of the above
		$3$-simplex:
		\begin{center}
		 \begin{tikzpicture}[scale=1.6]
		  \squares{%
		   /squares/label/.cd,
            12=$Ff'$, 02=$Ff'$, 03=$Ff$, 13=$Ff$,
            012=${=}$, 023=$F\alpha$, 123=$F\alpha'$, 013=${=}$,
            0123=$F\gamma$,
            /squares/arrowstyle/.cd,
            012={phantom, description},
            013={phantom, description},
            /squares/labelstyle/.cd,
            012={anchor=center},
            013={anchor=center}
          }
		 \end{tikzpicture}\ .
		\end{center}
	\end{description}
	\end{paragr}

	These data satisfy the normalisation conditions.

		\begin{paragr}[Normalisation]\label{paragr:simplicial_to_cellular-norm}
		In this paragraph we shall commit the abuse of denoting the main
		cell of a simplex of $\SN(B)$ by the simplex itself.
		The normalisation is an immediate consequence of the degeneracies
		of $F$. 
		\begin{description}
			\item[%
			\scalebox{0.3}{
				\begin{forest}
					for tree={%
						label/.option=content,
						grow'=north,
						content=,
						circle,
						fill,
						minimum size=3pt,
						inner sep=0pt,
						s sep+=15,
					}
					[]
				\end{forest}
			}
			] for any object $a$ of $A$, we have
			\[
			F_{\treeLog}(1_a) = s^0_1(F_{\treeDot} a) = F_{\dgn{0}{1}}(a) =  1_{F_{\treeDot}(a)}\,;
			\]
			
			\item[%
			\scalebox{0.3}{
				\begin{forest}
					for tree={%
						label/.option=content,
						grow'=north,
						content=,
						circle,
						fill,
						minimum size=3pt,
						inner sep=0pt,
						s sep+=15,
					}
					[ [] ]
				\end{forest}
			}
			] for any $1$-cell $f$ of $A$, we have
			\[
			F_{\treeLL}(1_f) = s^0_2(F_{\treeL} f) = F_{\dgn{0}{2}}(f) = 1_{F_{\treeLog}(f)}\,;
			\]
			
			\item[%
			\scalebox{0.3}{
				\begin{forest}
					for tree={%
						label/.option=content,
						grow'=north,
						content=,
						circle,
						fill,
						minimum size=3pt,
						inner sep=0pt,
						s sep+=15,
					}
					[ [][] ]
				\end{forest}
			}
			] for any $2$-cell $f \colon a \to a'$ of $A$, we have
			\[
			F_{\treeV}(1_{a'}, f) = s^1_2(F_{\treeL} f) = 1_{F_{\treeL}(f)} =
			s^0_2(F_{\treeL}(f)) = F_{\treeV}(f, 1_a)\,;
			\]
			
			\item[%
			\scalebox{0.3}{
				\begin{forest}
					for tree={%
						label/.option=content,
						grow'=north,
						content=,
						circle,
						fill,
						minimum size=3pt,
						inner sep=0pt,
						s sep+=15,
					}
					[ [ [] ] ]
				\end{forest}
			}
			] for any $2$-cell $\alpha$ of $A$ we have
			\[
			F_{\treeLLL}(1_\alpha) = s^0_3(\alpha_r) = F_{\dgn{0}{3}}(\alpha_r)
			= 1_{F_{\treeLL}(\alpha)}\,;
			\]
			
			\item[%
			\scalebox{0.3}{
				\begin{forest}
					for tree={%
						label/.option=content,
						grow'=north,
						content=,
						circle,
						fill,
						minimum size=3pt,
						inner sep=0pt,
						s sep+=15,
					}
					[ [][][] ]
				\end{forest}
			}
			] for any pair $a \xto{f} a' \xto{g} a''$ of $1$-cells of $A$, we have
			\begin{align*}
			F_{\treeW}(g, f, 1_a) &= s^0_3(F_{\treeV}(g, f)) = 1_{F_{\treeV}(g, f)}\,,\\
			F_{\treeW}(g, 1_{a'}, f) &= s^1_3(F_{\treeV}(g, f)) = 1_{F_{\treeV}(g, f)}\,,\\
			F_{\treeW}(1_{a''}, g, f) &= s^2_3(F_{\treeV}(g, f)) = 1_{F_{\treeV}(g, f)}\,;
			\end{align*}
			
			\item[%
			\scalebox{0.3}{
				\begin{forest}
					for tree={%
						label/.option=content,
						grow'=north,
						content=,
						circle,
						fill,
						minimum size=3pt,
						inner sep=0pt,
						s sep+=15,
					}
					[ [][[]] ]
				\end{forest}
			}
			] for any pair $a \xto{f} a' \xto{g} a''$ of $1$-cells of $A$, we have
			\[
			F_{\treeVRight}(g, 1_f) = s^0_3(F_{\treeV}(g, f)) = 1_{F_{\treeV}(g, f)}\,,
			\]
			and for any $2$-cell $\alpha \colon f \to f'$ of $A$, we have
			\[
			F_{\treeVRight}(1_a', \alpha) = s^2_3(F_{\treeLL}(\alpha)) = 1_{F_{\treeLL}(\alpha)}\,;
			\]
			
			\item[%
			\scalebox{0.3}{
				\begin{forest}
					for tree={%
						label/.option=content,
						grow'=north,
						content=,
						circle,
						fill,
						minimum size=3pt,
						inner sep=0pt,
						s sep+=15,
					}
					[ [[]][] ]
				\end{forest}
			}
			] this normalisation is dual to the previous one;
		\end{description}
	\end{paragr}
	
	Checking the coherences is much strenuous.
	The next subsection is devoted to establish that
	they hold, thus proving the following result.
	
	\begin{thm}
	 Let $F \colon A \to B$ be a simplicial oplax $3$-morphism.
	 With the data defined right above, $F$ is a normalised oplax $3$-functor
	 from $A$ to $B$.
	\end{thm}

	\subsection{Coherences}
	Let $F \colon A \to B$ be a simplicial oplax $3$-morphism.
	In this subsection we are going to check that the data of~$F_{\treeDot}$,
	$F_{\treeLog}$, $F_{\treeV}$, $F_{\treeLL}$, $F_{\treeW}$ $F_{\treeVLeft}$, $F_{\treeVRight}$
	and~$F_{\treeLLL}$ defined above satisfy the set of coherences for a 
	normalised oplax $3$-functor defined in paragraph~\ref{paragr:lax_3functor_cellular_coherences}.
	We shall do so by showing that every coherence can be encoded in a
	particular $4$-simplex of~$\SN(B)$, image of a $4$\hyp{}simplex of~$\SN(A)$.
	We shall only draw the former, the latter being
	clear. Moreover, we shall omit the various
	whiskerings when denoting the $3$-cells of these $4$-simplices, which are
	nevertheless clear from the picture.
	
	\begin{description}
	 \item[%
		\scalebox{0.3}{
			\begin{forest}
				for tree={%
					label/.option=content,
					grow'=north,
					content=,
					circle,
					fill,
					minimum size=3pt,
					inner sep=0pt,
					s sep+=15,
				}
				[ [[][]] ]
			\end{forest}
		}
		] For any pair of $1$-composable $2$-cells
		
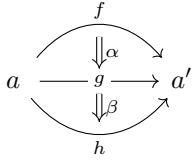
\[
		 \begin{tikzcd}[column sep=4.5em]
		  a\phantom{'}
		  \ar[r, bend left=50, looseness=1.2, "f", ""{below, name=f}]
		  \ar[r, "g" description, ""{name=gu}, ""{below, name=gd}]
		  \ar[r, bend right=50, looseness=1.2, "h"', ""{name=h}]
		  \ar[Rightarrow, from=f, to=gu, "\alpha"]
		  \ar[Rightarrow, from=gd, to=h, "\beta"]&
		  a'
		 \end{tikzcd}
		\]
		of $A$, all the images by $F$ of the related $3$-simplices
		have a trivial main $3$-cell of $B$; indeed, this is
		the case for the $3$-cells $\eps_{\text{l}}(\beta, \alpha)$, $\eps_{\text{r}}(\beta, \alpha)$,
		$\omega_{\text{l}}(\beta, \alpha)$ and $\omega_{\text{l}}(\beta, \alpha)$ of $B$.
		Thus we have an \emph{identity} $3$-cell
		\[
		 F_{\treeLL}(\beta) \comp_1 F_{\treeLL}(\alpha)
		 \to 
		 F_{\treeLL}(\beta \comp_1 \alpha)
		\]
		of $B$, establishing the coherence for $\treeY$.
		
	 \item[%
        \scalebox{0.3}{
            \begin{forest}
                for tree={%
                    label/.option=content,
                    grow'=north,
                    content=,
                    circle,
                    fill,
                    minimum size=3pt,
                    inner sep=0pt,
                    s sep+=15,
                }
                [
                    [] [] [] []
                ]
            \end{forest}
        }
     ] Consider a quadruple
     \[
      \begin{tikzcd}
       \bullet \ar[r, "f"] & \bullet \ar[r, "g"] & \bullet \ar[r, "h"] & \bullet \ar[r, "i"] & \bullet
      \end{tikzcd}
     \]
     of $0$-composable $1$-cells of $A$. The $4$-simplex of $\SN(B)$ in figure~\ref{fig:coherence_ifgh},
     \begin{figure}
     \centering
      \begin{tikzpicture}[scale=1.8]
        \pentagon{%
            /pentagon/label/.cd,
            01=$Ff$, 12=$Fg$, 23=$Fh$, 34=$Fi$, 04=$F(ihgf)$,
            02=$F(gf)$, 03=$F(hgf)$, 13=$F(hg)$, 14=$F(ihg)$, 24=$F(ih)$,
            012=${F_{\treeV}(g, f)}$, 034=${F_{i, hgf}}$,
            023=${F_{\treeV}(h, gf)}$, 123=${F_{\treeV}(h, g)}$, 134=${F_{\treeV}(i, hg)}$,
            014=${F_{\treeV}(ihg, f)}$, 024=${F_{\treeV}(ih, gf)}$, 234=${F_{\treeV}(i, h)}$,
            013=${F_{\treeV}(hg, f)}$, 124=${F_{\treeV}(ih, g)}$,
            0123=${F_{\treeW}(h, g, f)}$, 0124=${F_{\treeW}(ih, g, f)}$,
            0134=${F_{\treeW}(i, hg, f)}$, 0234=${F_{\treeW}(i, h, gf)}$,
            1234=${F_{\treeW}(i, h, g)}$,
            /pentagon/arrowstyle/.cd,
            03={pos=0.55}, 02={pos=0.55},
            012={phantom, description}, 013={phantom, description},
            014={phantom, description}, 023={phantom, description},
            024={phantom, description}, 034={phantom, description},
            123={phantom, description}, 124={phantom, description},
            134={phantom, description}, 234={phantom, description},
            01234={phantom, description},
            /pentagon/labelstyle/.cd,
            012={anchor=center}, 013={anchor=center}, 014={anchor=center},
            023={anchor=center}, 024={anchor=center}, 034={anchor=center},
            123={anchor=center}, 124={anchor=center}, 134={anchor=center},
            234={anchor=center}        
        }
      \end{tikzpicture}
      \caption{Establishing the coherence $ihgf$.}
        \label{fig:coherence_ifgh}
     \end{figure}
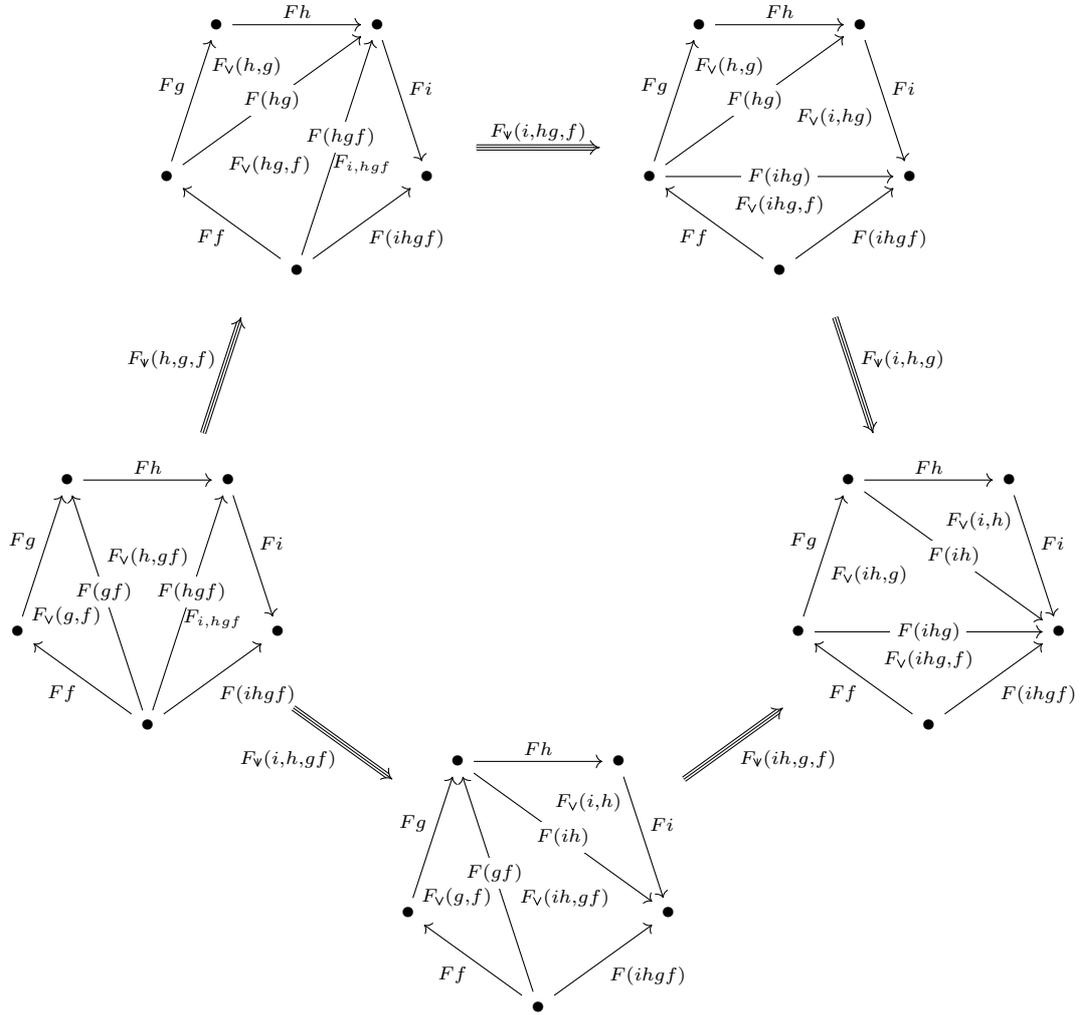
     where $F_{i, hgf} = F_{\treeV}(i, ghf)$ by definition,
     shows that the $3$-cells
     \begin{gather*}
       F_{\treeV}(i, h) \comp_0 F_{\treeLog}(g) \comp_0 F_{\treeLog}(f)
       \comp_1 F_{\treeW}(ih, g, h) \\
       \comp_2\\
       F_{\treeLog}(i) \comp_0 F_{\treeLog}(h) \comp_0 F_{\treeV}(g, f)
       \comp_1 F_{\treeW}(i, h, gf)
     \end{gather*}
     and
     \begin{gather*}
       F_{\treeW}(i, h, g) \comp_0 F_{\treeLog}(f) \comp_1 F_{\treeV}(i\comp_0 h \comp_0 g, f)\\
       \comp_2\\
       F_{\treeLog}(i) \comp_0 F_{\treeV}(h, g) \comp_0 F_{\treeLog}(f) \comp_1
       F_{\treeW}(i, h \comp_0 g, f)\\
       \comp_2 \\
       F_{\treeLog}(i) \comp_0 F_{\treeW}(h, g, f) \comp_1 F_{\treeV}(i, h\comp_0 g \comp_0 f)
     \end{gather*}
     of $B$ are equal. This establishes the coherence
     \scalebox{0.3}{
            \begin{forest}
                for tree={%
                    label/.option=content,
                    grow'=north,
                    content=,
                    circle,
                    fill,
                    minimum size=3pt,
                    inner sep=0pt,
                    s sep+=15,
                }
                [
                    [] [] [] []
                ]
            \end{forest}
        }.
     
     \item[%
        \scalebox{0.3}{
            \begin{forest}
                for tree={%
                    label/.option=content,
                    grow'=north,
                    content=,
                    circle,
                    fill,
                    minimum size=3pt,
                    inner sep=0pt,
                    s sep+=15,
                }
                [
                    [ [] ] [] []
                ]
            \end{forest}
        }
     ] Consider a triple
     
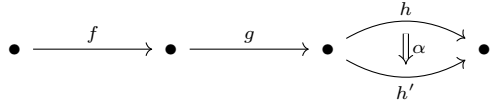
\[
      \begin{tikzcd}[column sep=4.5em]
       \bullet \ar[r, "f"] & \bullet \ar[r, "g"] &
       \bullet
       \ar[r, bend left, "h", ""{below, name=h}]
       \ar[r, bend right, "h'"', ""{name=h2}]
       \ar[Rightarrow, from=h, to=h2, "\alpha"]
       & \bullet
      \end{tikzcd}
     \]
     of $0$-composable cells $f$, $g$ and $\alpha$ of $A$.
     The $4$\hyp{}simplex of $\SN(B)$ depicted in figure~\ref{fig:coherence_alpha-g-f}
     ensures that the $3$-cells
     \begin{gather*}
        F_{\treeLL}(\alpha) \comp_0 F_{\treeLog}(g) \comp_0 F_{\treeLog}(f)
        \comp_1 F_{\treeW}(h, g, f)\\
        \comp_2\\
        F_{\treeLog}(h') \comp_0 F_{\treeV}(g, f) \comp_1 F_{\treeVLeft}(\alpha, g \comp_0 f)
     \end{gather*}
     and
     \begin{gather*}
        F_{\treeVLeft}(\alpha, g) \comp_0 F_{\treeLog}(f) \comp_1 F_{\treeV}(h\comp_0 g, f)\\
        \comp_2\\
        F_{\treeV}(h', g) \comp_0 F_{\treeLog}(f) \comp_1 F_{\treeVLeft}(\alpha \comp_0 g, f)\\
        \comp_2\\
        F_{\treeW}(h', g, f) \comp_1 F_{\treeLL}(\alpha \comp_0 g \comp_0 f)
     \end{gather*}
     of $B$ are equal. This establishes the coherence
     \scalebox{0.3}{
            \begin{forest}
                for tree={%
                    label/.option=content,
                    grow'=north,
                    content=,
                    circle,
                    fill,
                    minimum size=3pt,
                    inner sep=0pt,
                    s sep+=15,
                }
                [
                    [ [] ] [] []
                ]
            \end{forest}
        }.
        \begin{figure}
        \centering
         \begin{tikzpicture}[scale=1.8]
            \pentagon{%
                /pentagon/label/.cd,
                01=$Ff$, 12=$Fg$, 23=$Fh'$, 
                04=$F(hgf)$,
                02=$F(gf)$, 03=$F(h'gf)$, 13=$F(h'g)$, 14=$F(h'g)$, 24=$Fh$,
                012=${F_{\treeV}(g, f)}$, 034=${\phantom{O}F_{\treeLL}(\alpha gf)}$,
                023=${F_{\treeV}(h', gf)}$, 123=${F_{\treeV}(h', g)}$, 134=${F_{\treeLL}(\alpha g)}$,
                014=${F_{\treeV}(ihg, f)}$, 024=${F_{\treeV}(ih, gf)}$, 234=${F_{\treeV}(i, h)}$,
                013=${F_{\treeV}(h'g, f)}$, 124=${F_{\treeV}(h, g)}$,
                0123=${F_{\treeW}(h', g, f)}$, 0124=${F_{\treeW}(h, g, f)}$,
                0134=${F_{\treeVLeft}(\alpha g, f)}$, 0234=${F_{\treeVLeft}(\alpha, gf)}$,
                1234=${F_{\treeVLeft}(\alpha, g)}$,
                /pentagon/arrowstyle/.cd,
                03={pos=0.55}, 02={pos=0.55}, 34={equal},
                012={phantom, description}, 013={phantom, description},
                014={phantom, description}, 023={phantom, description},
                024={phantom, description}, 034={phantom, description},
                123={phantom, description}, 124={phantom, description},
                134={phantom, description}, 234={phantom, description},
                01234={phantom, description},
                /pentagon/labelstyle/.cd,
                012={anchor=center}, 013={anchor=center}, 014={anchor=center},
                023={anchor=center}, 024={anchor=center}, 034={anchor=center},
                123={anchor=center}, 124={anchor=center}, 134={anchor=center},
                234={anchor=center}        
            }
         \end{tikzpicture}
         \caption{Establishing the coherence $\alpha g f$.}
         \label{fig:coherence_alpha-g-f}
        \end{figure}

     \item[%
        \scalebox{0.3}{
            \begin{forest}
                for tree={%
                    label/.option=content,
                    grow'=north,
                    content=,
                    circle,
                    fill,
                    minimum size=3pt,
                    inner sep=0pt,
                    s sep+=15,
                }
                [
                    [] [ [] ] []
                ]
            \end{forest}
        }
     ] Consider a triple
     
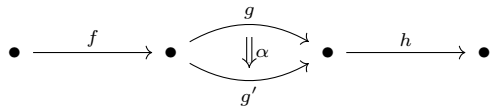
\[
      \begin{tikzcd}[column sep=4.5em]
       \bullet \ar[r, "f"] &
       \bullet
       \ar[r, bend left, "g", ""{below, name=g}]
       \ar[r, bend right, "g'"', ""{name=g2}]
       \ar[Rightarrow, from=g, to=g2, "\alpha"] &
       \bullet \ar[r, "h"] & \bullet
      \end{tikzcd}
     \]
     of $0$-composable cells $f$, $\alpha$ and $h$ of $A$.
     The proof of the coherence for this cellular pasting diagram
     is quite involved and relies on the construction and analysis
     of four $4$-simplices of $\SN(B)$.
     The main such $4$-simplex is depicted in figure~\ref{fig:coherence_h-alpha-f},
     where
     \[
      \beta = F_{\treeV}(g', f) \comp_1 F_{\treeLL}(\alpha \comp_0 f)
      \quadet
      \gamma = F_{\treeV}(h\comp_0 g, f) \comp_1 F_{\treeLL}(h \comp_0 \alpha \comp_0 f)\,,
     \]
     and shows that the $3$-cell
     \begin{gather}\label{eq:coherence_h-alpha-f_up}
        F_{\treeVRight}(h, \alpha) \comp_0 F_{\treeLog}(f)
        \comp_1 F_{\treeV}(h \comp_0 g, f) \notag\\
        \comp_2 \notag\\
        F_{\treeLog}(h) \comp_0 F_{\treeLL}(\alpha) \comp_0 F_{\treeLog}(f)
        \comp_1 F_{\treeW}(h, g, f)\\
        \comp_2\notag\\
        F_{\treeLog}(h) \comp_0 \Omega
        \comp_1 F_{\treeV}(h, g\comp_0 f) \notag
     \end{gather}
     of $B$ is equal the following $3$-cell
     \begin{equation*}
        F_{\treeV}(h, g') \comp_0 F_{\treeLog}(f)
        \comp_1 \Phi\:  \comp_2\:  \Psi
     \end{equation*}
     of $B$. Now, the $4$-simplices depicted
     in figures~\ref{fig:Omega}, \ref{fig:Psi} and~\ref{fig:Phi}
     entail the equalities
     \[
      \Omega = F_{\treeVLeft}(\alpha, f)\,,
     \]
     \[
      \Phi = F_{\treeVLeft}(h\comp_0 \alpha , f)
     \]
     and
     \[
      \Psi = F_{\treeW}(h, g', f) \comp_1 F_{\treeLL}(h\comp_0 \alpha \comp_0 f)\:
      \comp_2 \: F_{\treeLog}(h) \comp_0 F_{\treeV}(g', f)
      \comp_1 F_{\treeVRight}(h, \alpha \comp_0 f)\,.
     \]
     We can then conclude that the $3$-cell~\eqref{eq:coherence_h-alpha-f_up}
     of $B$ is equal to
     \begin{gather*}
        F_{\treeV}(h, g') \comp_0 F_{\treeLog}(f)
        \comp_1 F_{\treeVLeft}(h\comp_0 \alpha, f)\\
        \comp_2\\
        F_{\treeW}(h, g', f) \comp_1 F_{\treeLL}(h\comp_0 \alpha \comp_0 f)\\
        \comp_2\\
        F_{\treeLog}(h) \comp_0 F_{\treeV}(g', f) \comp_1
        F_{\treeVRight}(h, \alpha \comp_0 f)\,,
     \end{gather*}
     thereby establishing the coherence of
     \scalebox{0.3}{
            \begin{forest}
                for tree={%
                    label/.option=content,
                    grow'=north,
                    content=,
                    circle,
                    fill,
                    minimum size=3pt,
                    inner sep=0pt,
                    s sep+=15,
                }
                [
                    [] [ [] ] []
                ]
            \end{forest}
        }.
    \begin{figure}
    \centering
     \begin{tikzpicture}[scale=1.8]
        \pentagon{%
            /pentagon/label/.cd,
            01=$Ff$, 
            23=$Fg'$, 34=$Fh$, 04=$F(hgf)$,
            02=$Ff$, 03=$F(gf)$, 13=$Fg$, 14=$F(gf)$, 24=$F(g'f)$,
            012=${=}$, 034=${\phantom{O}F_{\treeV}(h, gf)}$,
            023=$\beta$, 123=${F_{\treeLL}(\alpha)}$, 134=${F_{\treeV}(h, g)}$,
            014=${F_{\treeV}(hg, f)}$, 024=$\gamma$, 234=${F_{\treeV}(h, g')}$,
            013=${F_{\treeV}(g, f)}$, 124=${F_{\treeLL}(h\alpha)}$,
            0123=$\Omega$,
            0124=${\Phi}$,
            0134=${F_{\treeW}(h, g, f)}$, 0234=${\Psi}$,
            1234=${F_{\treeVRight}(h, \alpha)}$,
            /pentagon/arrowstyle/.cd,
            12={equal}, 03={pos=0.55}, 02={pos=0.55},
            012={phantom, description}, 013={phantom, description},
            014={phantom, description}, 023={phantom, description},
            024={phantom, description}, 034={phantom, description},
            123={phantom, description}, 124={phantom, description},
            134={phantom, description}, 234={phantom, description},
            01234={phantom, description},
            /pentagon/labelstyle/.cd,
            012={anchor=center}, 013={anchor=center}, 014={anchor=center},
            023={anchor=center}, 
            024={anchor=center}, 034={anchor=center},
            123={anchor=center}, 124={anchor=center}, 134={anchor=center},
            234={anchor=center}        
        }
     \end{tikzpicture}
     \caption{Establishing the coherence $h\alpha f$.}
     \label{fig:coherence_h-alpha-f}
    \end{figure}
    
        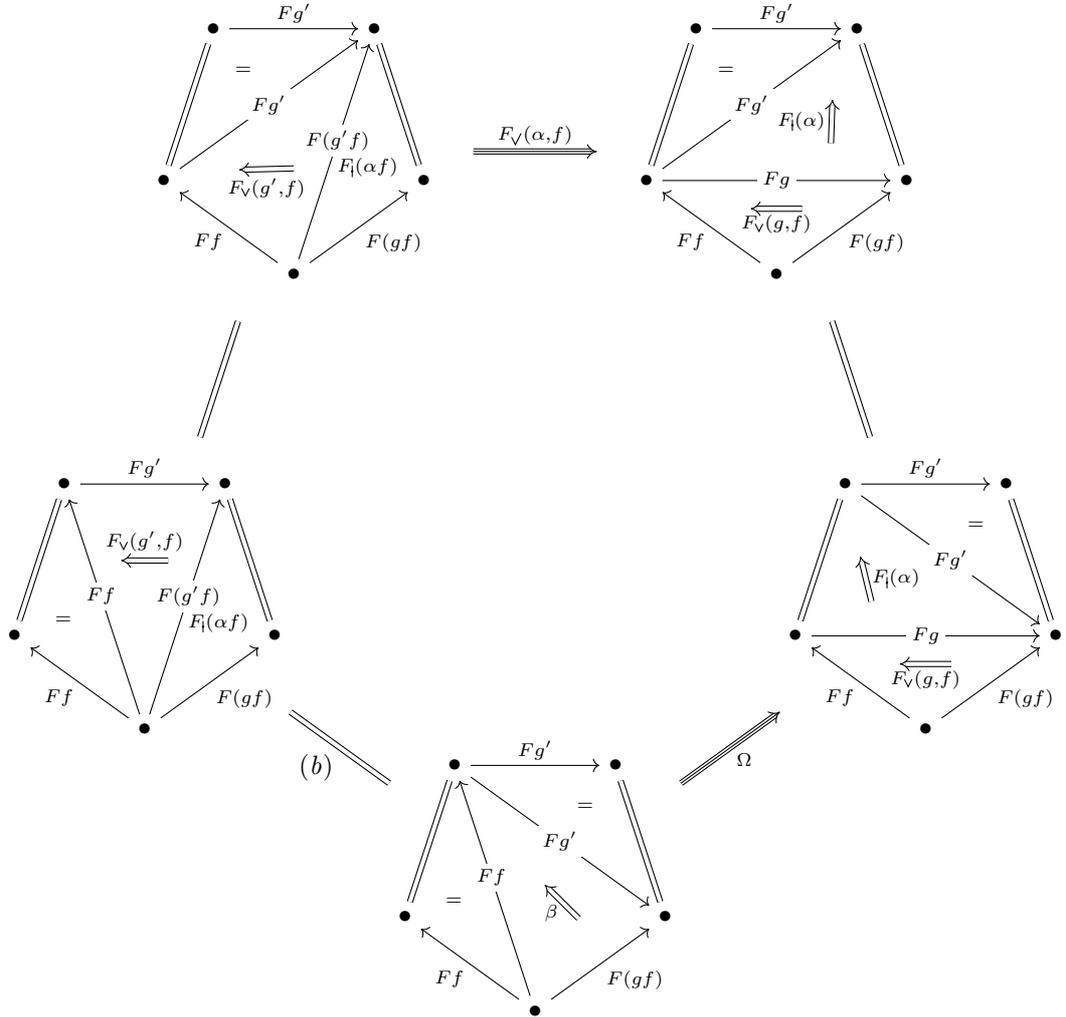
\begin{figure}
        \centering
        	\begin{tikzpicture}[scale=1.8]
        	\pentagon{%
        		/pentagon/label/.cd,
        		01=$Ff$, 12={}, 
        		23=$Fg'$, 34={}, 04=$F(gf)$,
        		02=$Ff$, 03=$F(g'f)$, 13=$Fg'$, 14=$Fg$, 24=$Fg'$,
        		012={$=$}, 034=${\phantom{O}F_{\treeLL}(\alpha f)}$,
        		023={$F_{\treeV}(g', f)$}, 123={$=$}, 134={$F_{\treeLL}(\alpha)$},
        		014={$F_{\treeV}(g, f)$}, 024={$\beta$}, 234={$=$},
        		013={$F_{\treeV}(g', f)$}, 124={$F_{\treeLL}(\alpha)$},
        		0123={},
        		0124=$\Omega$,
        		0134={$F_{\treeVRight}(\alpha, f)$},
        		0234={\ref{cond:simpl_oplax-ii}},
        		1234={},
        		/pentagon/arrowstyle/.cd,
        		12={equal}, 03={pos=0.55}, 02={pos=0.55}, 34={equal},
        		012={phantom, description}, 123={phantom, description},
        		234={phantom, description}, 034={phantom, description},
        		0123={equal}, 0234={equal}, 1234={equal},
        		01234={phantom, description},
        		/pentagon/labelstyle/.cd,
        		012={anchor=center}, 123={anchor=center}, 234={anchor=center} ,
        		034={anchor=center}     
        	}
        	\end{tikzpicture}
        	\caption{The $3$-cell $\Omega$ of $B$.}
        	\label{fig:Omega}
        \end{figure}
    
    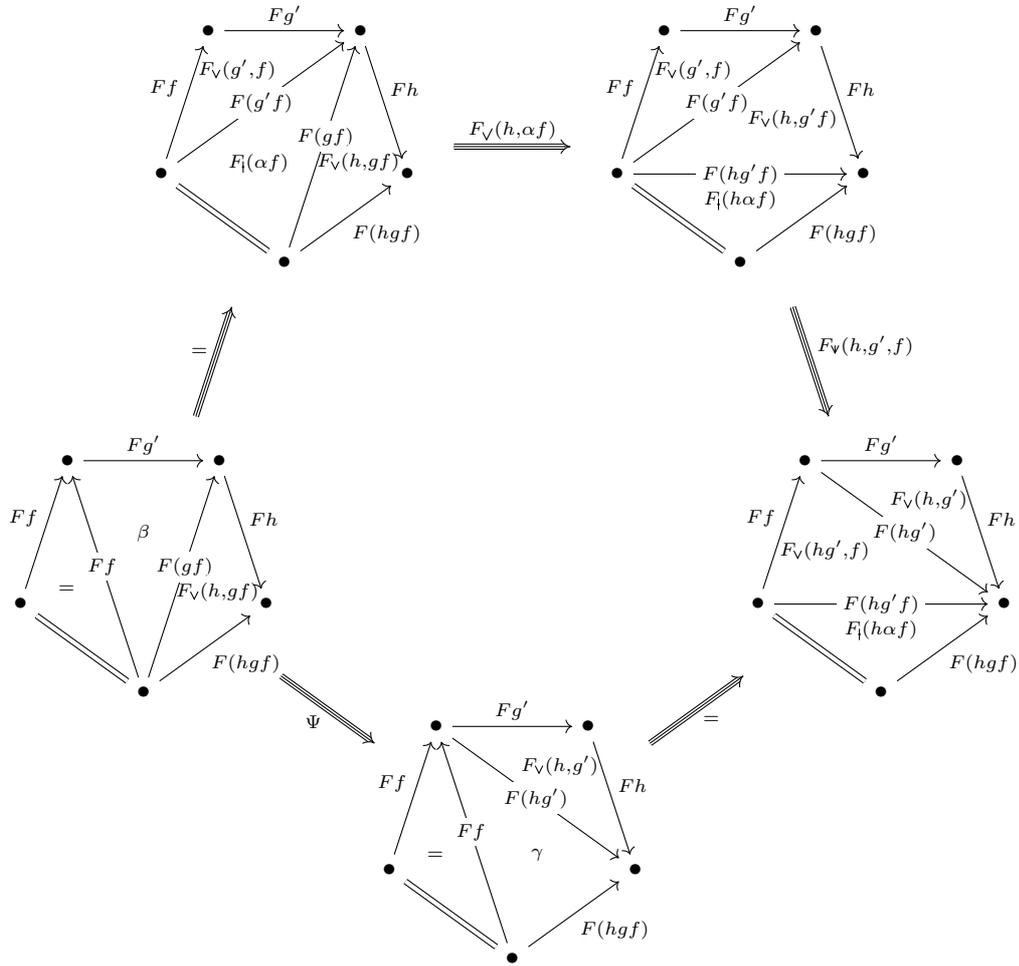
\begin{figure}
     \centering
      \begin{tikzpicture}[scale=1.7]
        \pentagon{%
            /pentagon/label/.cd,
            12=$Ff$, 23=$Fg'$, 34=$Fh$, 04=$F(hgf)$,
            02=$Ff$, 03=$F(gf)$, 13=$F(g'f)$, 14=$F(hg'f)$, 24=$F(hg')$,
            012=${=}$, 034=${\phantom{Oi}F_{\treeV}(h, gf)}$,
            023=$\beta$, 123=${F_{\treeV}(g', f)}$, 134=${F_{\treeV}(h, g'f)}$,
            014=${F_{\treeLL}(h\alpha f)}$, 024=$\gamma$, 234=${F_{\treeV}(h, g')}$,
            013=${F_{\treeLL}(\alpha f)}$, 124=${F_{\treeV}(hg', f)}$,
            0123=${=}$,
            0124=${=}$,
            0134=${F_{\treeVRight}(h, \alpha f)}$, 0234=${\Psi}$,
            1234=${F_{\treeW}(h, g', f)}$,
            /pentagon/arrowstyle/.cd,
            01={equal}, 03={pos=0.55}, 02={pos=0.55},
            012={phantom, description}, 013={phantom, description},
            014={phantom, description}, 023={phantom, description},
            024={phantom, description}, 034={phantom, description},
            123={phantom, description}, 124={phantom, description},
            134={phantom, description}, 234={phantom, description},
            01234={phantom, description},
            /pentagon/labelstyle/.cd,
            012={anchor=center}, 013={anchor=center}, 014={anchor=center},
            023={anchor=center}, 
            024={anchor=center}, 034={anchor=center},
            123={anchor=center}, 124={anchor=center}, 134={anchor=center},
            234={anchor=center}        
        }
      \end{tikzpicture}
      \caption{The $3$-cell $\Psi$ of $B$.}
      \label{fig:Psi}
     \end{figure}
     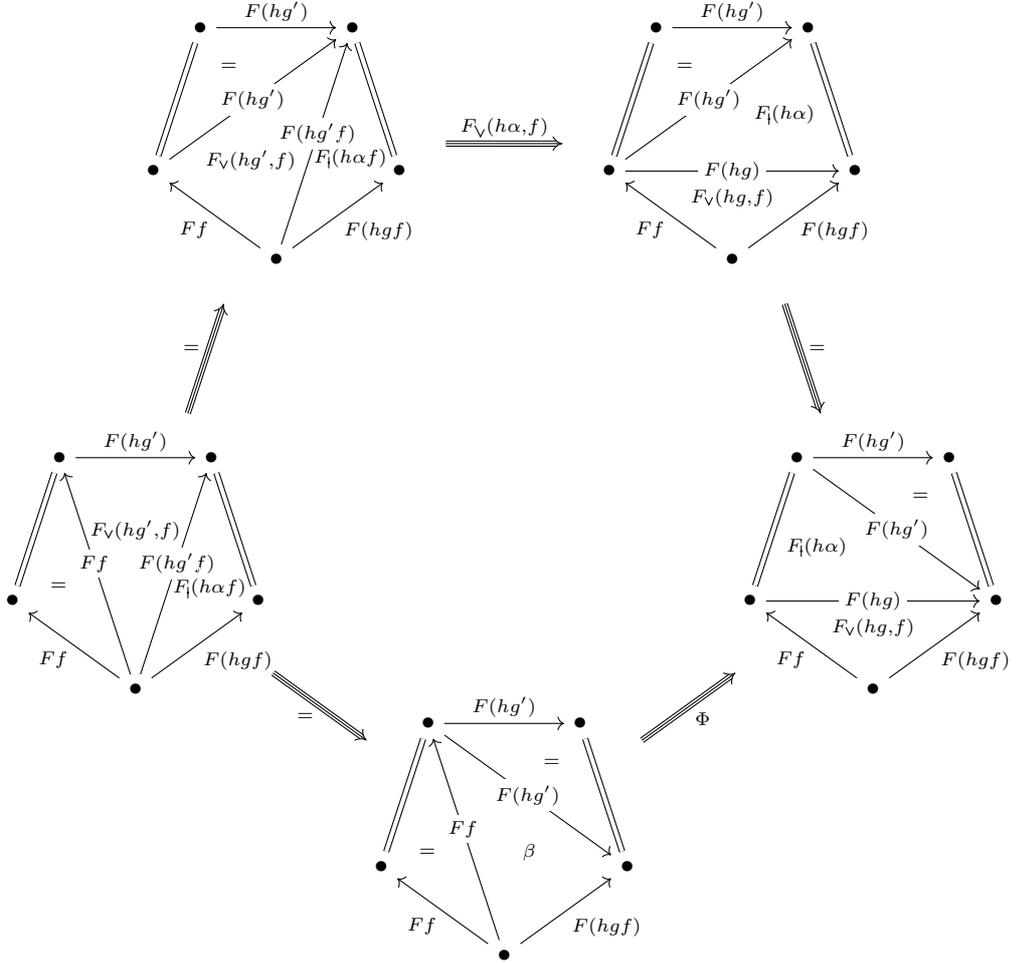
\begin{figure}
     \centering
      \begin{tikzpicture}[scale=1.7]
        \pentagon{%
            /pentagon/label/.cd,
            01=$Ff$, 
            23=$F(hg')$, 
            04=$F(hgf)$,
            02=$Ff$, 03=$F(hg'f)$, 13=$F(hg')$, 14=$F(hg)$, 24=$F(hg')$,
            012=${=}$, 034=${\phantom{Oi}F_{\treeLL}(h\alpha f)}$,
            023=${F_{\treeV}(hg', f)}$, 123=${=}$, 134=${F_{\treeLL}(h\alpha)}$,
            014=${F_{\treeV}(hg, f)}$, 024=$\beta$, 234=${=}$,
            013=${F_{\treeV}(hg', f)}$, 124=${F_{\treeLL}(h\alpha)}$,
            0123=${=}$,
            0124=${\Phi}$,
            0134=${F_{\treeVLeft}(h\alpha, f)}$, 0234=${=}$,
            1234=${=}$,
            /pentagon/arrowstyle/.cd,
            12={equal}, 34={equal},
            03={pos=0.55}, 02={pos=0.55},
            012={phantom, description}, 013={phantom, description},
            014={phantom, description}, 023={phantom, description},
            024={phantom, description}, 034={phantom, description},
            123={phantom, description}, 124={phantom, description},
            134={phantom, description}, 234={phantom, description},
            01234={phantom, description},
            /pentagon/labelstyle/.cd,
            012={anchor=center}, 013={anchor=center}, 014={anchor=center},
            023={anchor=center}, 
            024={anchor=center}, 034={anchor=center},
            123={anchor=center}, 124={anchor=center}, 134={anchor=center},
            234={anchor=center}        
        }
      \end{tikzpicture}
     \caption{The $3$-cell $\Phi$ of $B$.}
     \label{fig:Phi}
    \end{figure}

     \item[%
        \scalebox{0.3}{
            \begin{forest}
                for tree={%
                    label/.option=content,
                    grow'=north,
                    content=,
                    circle,
                    fill,
                    minimum size=3pt,
                    inner sep=0pt,
                    s sep+=15,
                }
                [
                    [] [] [ [] ]
                ]
            \end{forest}
        }
     ] Consider a triple
     \[
      \begin{tikzcd}[column sep=4.5em]
       \bullet
       \ar[r, bend left, "f", ""{below, name=f}]
       \ar[r, bend right, "f'"', ""{name=f2}]
       \ar[Rightarrow, from=f, to=f2, "\alpha"] &
       \bullet \ar[r, "g"] &
       \bullet \ar[r, "h"] & \bullet
      \end{tikzcd}
     \]
     of $0$-composable cells $\alpha$, $g$ and $h$ of $A$.
     The $4$-simplex of $\SN(B)$ depicted in
     figure~\ref{fig:coherence_h-g-alpha}, totally
     symmetric to~\ref{fig:coherence_alpha-g-f},
     shows that the $3$-cells
     \begin{gather*}
        F_{\treeW}(h, g, f') \comp_1 F_{\treeLL}(h \comp_0 g \comp_0 \alpha)\\
        \comp_2\\
        F_{\treeLog}(h) \comp_0 F_{\treeV}(g, f')
        \comp_1 F_{\treeVRight}(h, g \comp_0 \alpha)\\
        \comp_2\\
        F_{\treeLog}(h) \comp_0 F_{\treeVRight}(g, \alpha)
        \comp_1 F_{\treeV}(h, g \comp_0 f)
     \end{gather*}
     and
     \begin{gather*}
        F_{\treeV}(h, g) \comp_0 F_{\treeLog}(f')
        \comp_1 F_{\treeVRight}(h\comp_0 g, \alpha)\\
        \comp_2\\
        F_{\treeLog}(h) \comp_0 F_{\treeLog}(g) \comp_0 F_{\treeLL}(\alpha)
        \comp_1 F_{\treeV}(h, g \comp_0 f)
     \end{gather*}
     of $B$ are equal, therefore establishing the coherence 
     \scalebox{0.3}{
            \begin{forest}
                for tree={%
                    label/.option=content,
                    grow'=north,
                    content=,
                    circle,
                    fill,
                    minimum size=3pt,
                    inner sep=0pt,
                    s sep+=15,
                }
                [
                    [] [] [ [] ]
                ]
            \end{forest}
        }.
    \begin{figure}
     \centering
     \begin{tikzpicture}[scale=1.7]
        \pentagon{%
            /pentagon/label/.cd,
            12=$Ff'$, 23=$Fg$, 34=$Fh$, 
            04=$F(hgf)$,
            02=$Ff$, 03=$F(gf)$, 13=$F(hg')$, 14=$F(hgf')$, 24=$F(hg)$,
            012=${F\alpha}$, 034=${\phantom{Oi}F_{\treeV}(h, gf)}$,
            023=${F_{\treeV}(g, f)}$, 123=${F_{\treeV}(g, f')}$, 134=${F_{\treeV}(h, gf')}$,
            014=${F_{\treeLL}(hg\alpha )}$, 024=${F_{\treeV}(hg, f)}$, 234=${F_{\treeV}(h, g)}$,
            013=${F_{\treeLL}(g\alpha)}$, 124=${F_{\treeV}(hg, f')}$,
            0123=${F_{\treeVRight}(g, \alpha)}$,
            0124=${F_{\treeVRight}(hg, \alpha)}$,
            0134=${F_{\treeVRight}(h, g\alpha)}$,
            0234=${F_{\treeW}(h, g, f)}$,
            1234=${F_{\treeW}(h, g, f')}$,
            /pentagon/arrowstyle/.cd,
            01={equal}, 03={pos=0.55}, 02={pos=0.55},
            012={phantom, description}, 013={phantom, description},
            014={phantom, description}, 023={phantom, description},
            024={phantom, description}, 034={phantom, description},
            123={phantom, description}, 124={phantom, description},
            134={phantom, description}, 234={phantom, description},
            01234={phantom, description},
            /pentagon/labelstyle/.cd,
            012={anchor=center}, 013={anchor=center}, 014={anchor=center},
            023={anchor=center}, 
            024={anchor=center}, 034={anchor=center},
            123={anchor=center}, 124={anchor=center}, 134={anchor=center},
            234={anchor=center}        
        }
     \end{tikzpicture}
     \caption{Establishing the coherence $hg\alpha$.}
     \label{fig:coherence_h-g-alpha}
    \end{figure}

    \begin{figure}[t]
      \centering
      \begin{tikzpicture}[scale=1.7]
        \pentagon{%
            /pentagon/label/.cd,
            23=$Ff''$, 34=$Fg$, 
            04=$F(gf)$,
            03=$Ff$, 13=$Ff'$, 14=$F(gf')$, 24=$F(gf'')$,
            012=${=}$, 023=${F_{\treeLL}(\beta\comp_1 \alpha)}$, 
            034=${\phantom{Oi}F_{\treeV}(g, f)}$,
            013=${F_{\treeLL}(\alpha)}$, 123=${F_{\treeLL}(\beta)}$,
            134=${F_{\treeV}(g, f')}$, 014=${F_{\treeLL}(g\alpha )}$,
            024=${F_{\treeLL}(g(\beta\comp_1\alpha))}$, 234=${F_{\treeV}(g, f'')}$,
            124=${F_{\treeLL}(g\beta)}$,
            0123=${F_{\treeY}(\beta, \alpha)}$,
            0134=${F_{\treeVRight}(g, \alpha)}$,
            1234=${F_{\treeVRight}(g, \beta)}$,
            0234=${F_{\treeVRight}(g, \beta \circ \alpha)}$,
            0124=${F_{\treeY}(g\beta, g\alpha)}$,
            /pentagon/arrowstyle/.cd,
            01={equal}, 12={equal}, 03={pos=0.55}, 02={pos=0.55},
            012={phantom, description}, 013={phantom, description},
            014={phantom, description}, 023={phantom, description},
            024={phantom, description}, 034={phantom, description},
            123={phantom, description}, 124={phantom, description},
            134={phantom, description}, 234={phantom, description},
            01234={phantom, description},
            /pentagon/labelstyle/.cd,
            012={anchor=center}, 013={anchor=center}, 014={anchor=center},
            023={anchor=center}, 
            024={anchor=center}, 034={anchor=center},
            123={anchor=center}, 124={anchor=center}, 134={anchor=center},
            234={anchor=center}        
        }
      \end{tikzpicture}
      \caption{Establishing the coherence $g\beta\alpha$.}
      \label{fig:coherence_g-beta-alpha}
     \end{figure}
    \item[%
        \scalebox{0.3}{
            \begin{forest}
                for tree={%
                    label/.option=content,
                    grow=north,
                    content=,
                    circle,
                    fill,
                    minimum size=3pt,
                    inner sep=0pt,
                    s sep+=15,
                }
                [
                    [
                        [] []
                    ]
                    []
                ]
            \end{forest}
        }
     ] Consider a triple
     
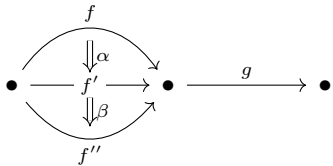
\[
      \begin{tikzcd}[column sep=4.5em]
       \bullet
       \ar[r, bend left=55, looseness=1.3, "f", ""{below, name=f1}]
       \ar[r, "f'"{description}, ""{name=f2u}, ""{below, name=f2d}]
       \ar[r, bend right=50, looseness=1.3, "f''"', ""{name=f3}]
       \ar[Rightarrow, from=f1, to=f2u, "\alpha"]
       \ar[Rightarrow, from=f2d, to=f3, "\beta"]
       &
       \bullet \ar[r, "g"] & \bullet
      \end{tikzcd}
     \]
     of cells $\alpha$, $\beta$ and $g$ of $A$ as in the drawing.
     The $4$-simplex of $\SN(B)$ depicted in figure~\ref{fig:coherence_g-beta-alpha}
     shows that the $3$-cells
     \begin{gather*}
        F_{\treeVRight}(g, \beta) \comp_1 F_{\treeLL}(g \comp_0 \alpha)\ 
        \comp_2\\
        F_{\treeLog}(g) \comp_0 F_{\treeLL}(\beta)
        \comp_1 F_{\treeVRight}(g, \alpha)\ 
        \comp_2\\
        F_{\treeLog}(g) \comp_0 F_{\treeY}(\beta, \alpha)
        \comp_1 F_{\treeV}(g, f)
     \end{gather*}
     and
     \[
        F_{\treeV}(g, f'') \comp_1 F_{\treeY}( g\comp_0 \beta, g \comp_0 \alpha)\ 
        \comp_2\ 
        F_{\treeVRight}(g, \beta \comp_1 \alpha)
     \]
     of $B$ are equal. Since $F_{\treeY}(\beta, \alpha)$
     and $F_{\treeY}( g\comp_0 \beta, g \comp_0 \alpha)$ are trivial
     by condition~\ref{cond:simpl_oplax-iii},
     the $4$-simplex actually exhibits the equality
     of the $3$-cells
     \[
      F_{\treeVRight}(g, \beta) \comp_1 F_{\treeLL}(g \comp_0 \alpha)\ 
        \comp_2\ 
        F_{\treeLog}(g) \comp_0 F_{\treeLL}(\beta)
        \comp_1 F_{\treeVRight}(g, \alpha)
     \]
     and
     \[
      F_{\treeV}(g, f'') \comp_1 F_{\treeY}(\beta, \alpha)
     \]
     of $B$.

     \item[%
        \scalebox{0.3}{
            \begin{forest}
                for tree={%
                    label/.option=content,
                    grow'=north,
                    content=,
                    circle,
                    fill,
                    minimum size=3pt,
                    inner sep=0pt,
                    s sep+=15,
                }
                [
                    [
                        [] []
                    ]
                    []
                ]
            \end{forest}
        }
     ] Consider a triple
     
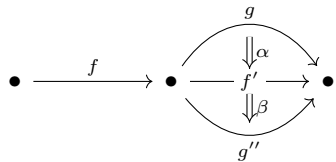
\[
      \begin{tikzcd}[column sep=4.5em]
        \bullet \ar[r, "f"] &
        \bullet
        \ar[r, bend left=55, looseness=1.3, "g", ""{below, name=g1}]
       \ar[r, "f'"{description}, ""{name=g2u}, ""{below, name=g2d}]
       \ar[r, bend right=50, looseness=1.3, "g''"', ""{name=g3}]
       \ar[Rightarrow, from=g1, to=g2u, "\alpha"]
       \ar[Rightarrow, from=g2d, to=g3, "\beta"]
       &
       \bullet
      \end{tikzcd}
     \]
     of cells $\alpha$, $\beta$ and $g$ of $A$ as in the drawing.
     The $4$-simplex of $\SN(B)$ displayed in figure~\ref{fig:coherence_beta-alpha-f},
     completely dual to the $4$-simplex~\ref{fig:coherence_g-beta-alpha},
     shows that the $3$-cells
     \begin{gather*}
        \bigl(F_{\treeY}(\beta, \alpha) \comp_0 F_{\treeLog}(f)
        \comp_1 F_{\treeV}(g, f)\bigr) \comp_2\\
        \bigl(F_{\treeLL}(\beta) \comp_0 F_{\treeLog}(f)
        \comp_1 F_{\treeVLeft}(\alpha, f) \bigr)\comp_2\\ 
        \bigl(F_{\treeVLeft}(\beta, f) \comp_1 F_{\treeLL}(\alpha \comp_0 f) \bigr)
     \end{gather*}
     and
     \[
        F_{\treeVLeft}(\beta \comp_1 \alpha, f)\ 
        \comp_2\ 
        F_{\treeV}(g'', f) \comp_1 F_{\treeY}(\beta \comp_0 f, \alpha \comp_0 f)
     \]
     of $B$ are equal. Since the $3$-cells
     $F_{\treeY}(\beta, \alpha)$ and $F_{\treeY}(\beta \comp_0 f, \alpha \comp_0 f)$
     are trivial by condition~\ref{cond:simpl_oplax-iii},
     the $4$-simplex is actually imposing the equality of the $3$-cells
     \[
        F_{\treeLL}(\beta) \comp_0 F_{\treeLog}(f)
        \comp_1 F_{\treeVLeft}(\alpha, f)\ 
        \comp_2\ 
        F_{\treeVLeft}(\beta, f) \comp_1 F_{\treeLL}(\alpha \comp_0 f)
     \]
     and
     \[
        F_{\treeVLeft}(\beta \comp_1 \alpha, f)
     \]
     of $B$.
     \begin{figure}
      \centering
      \begin{tikzpicture}[scale=1.7]
        \pentagon{%
            /pentagon/label/.cd,
            01=$Ff$, 12=$Fg''$,
            04=$F(gf)$,
            02=$F(gf'')$, 03=$F(g'f)$,
            13=$Fg'$, 14=$Fg$, 
            012=${F_{\treeV}(g'', f)}$, 023=${F_{\treeLL}(\beta f)}$, 
            034=${\phantom{Oi}F_{\treeLL}(\alpha f)}$,
            013=${F_{\treeV}(g', f)}$, 123=${F_{\treeLL}(\beta)}$,
            134=${F_{\treeLL}(\alpha)}$, 014=${F_{\treeV}(g, f)}$,
            024=${F_{\treeLL}((\beta\comp_1\alpha)f)}$, 234=${=}$,
            124=${F_{\treeLL}(\beta\circ \alpha)}$,
            0123=${F_{\treeVLeft}(\beta, f)}$,
            0134=${F_{\treeVLeft}(\alpha, f)}$,
            1234=${F_{\treeY}(\beta, \alpha)}$,
            0234=${F_{\treeY}(\beta f, \alpha f)}$,
            0124=${F_{\treeVLeft}(\beta\comp_1 \alpha, f)}$,
            /pentagon/arrowstyle/.cd,
            23={equal}, 34={equal},
            03={pos=0.55}, 02={pos=0.55},
            012={phantom, description}, 013={phantom, description},
            014={phantom, description}, 023={phantom, description},
            024={phantom, description}, 034={phantom, description},
            123={phantom, description}, 124={phantom, description},
            134={phantom, description}, 234={phantom, description},
            01234={phantom, description},
            /pentagon/labelstyle/.cd,
            012={anchor=center}, 013={anchor=center}, 014={anchor=center},
            023={anchor=center}, 
            024={anchor=center}, 034={anchor=center},
            123={anchor=center}, 124={anchor=center}, 134={anchor=center},
            234={anchor=center}        
        }
      \end{tikzpicture}
      \caption{Establishing the coherence $\beta\alpha f$.}
      \label{fig:coherence_beta-alpha-f}
     \end{figure}

     \item[%
        \scalebox{0.3}{
            \begin{forest}
                for tree={%
                    label/.option=content,
                    grow'=north,
                    content=,
                    circle,
                    fill,
                    minimum size=3pt,
                    inner sep=0pt,
                    s sep+=15,
                }
                [
                    [ [] ]
                    [ [] ]
                ]
            \end{forest}
        }
     ]
         \begin{figure}
     \centering
     \begin{tikzpicture}[scale=1.7]
        \pentagon{%
            /pentagon/label/.cd,
            12=$Ff'$,
            23=$Fg'$, 
            04=$F(gf)$,
            02=$Ff$, 03=$F(g'f)$,
            13=$F(g'f')$, 14=$F(gf')$, 
            012=${F_{\treeLL}(\alpha)}$, 023=${F_{\treeV}(g', f)}$, 
            034=${\phantom{Oi}F_{\treeLL}(\beta f)}$,
            013=${F_{\treeLL}(g'\alpha)}$, 123=${\phantom{o}F_{\treeV}(g', f')}$,
            134=${F_{\treeLL}(\beta f')}$, 014=${F_{\treeLL}(g\alpha)}$,
            024=${F_{\treeV}(g, f)}$, 234=${F_{\treeLL}(\beta)}$,
            124=${F_{\treeV}(g, f')}$,
            0123=${F_{\treeVRight}(g', \alpha)}$,
            0134=${F_{\text{ex}}(\beta, \alpha)}$,
            1234=${F_{\treeVLeft}(\beta, f')}$,
            0234=${F_{\treeVLeft}(\beta, f)}$,
            0124=${F_{\treeVRight}(g, \alpha)}$,
            /pentagon/arrowstyle/.cd,
            01={equal}, 34={equal},
            03={pos=0.55}, 02={pos=0.55},
            012={phantom, description}, 013={phantom, description},
            014={phantom, description}, 023={phantom, description},
            024={phantom, description}, 034={phantom, description},
            123={phantom, description}, 124={phantom, description},
            134={phantom, description}, 234={phantom, description},
            0134={equal},
            01234={phantom, description},
            /pentagon/labelstyle/.cd,
            012={anchor=center}, 013={anchor=center}, 014={anchor=center},
            023={anchor=center}, 
            024={anchor=center}, 034={anchor=center},
            123={anchor=center}, 124={anchor=center}, 134={anchor=center},
            234={anchor=center}        
        }
     \end{tikzpicture}
     \caption{Establishing the coherence $\beta \alpha$.}
     \label{fig:coherence_beta-alpha}
    \end{figure}
     Consider a pair
     
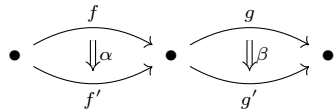
\[
      \begin{tikzcd}[column sep=4.5em]
        \bullet
        \ar[r, bend left, "f", ""{below, name=f1}]
        \ar[r, bend right, "f'"', ""{name=f2}]
        \ar[Rightarrow, from=f1, to=f2, "\alpha"]
        &
        \bullet
        \ar[r, bend left, "g", ""{below, name=g1}]
        \ar[r, bend right, "g'"', ""{name=g2}]
        \ar[Rightarrow, from=g1, to=g2, "\beta"]
        & \bullet
      \end{tikzcd}
     \]
     of $0$-composable $2$-cells $\alpha$ and $\beta$ of $A$.
     The $4$-simplex of $\SN(B)$ depicted in figure~\ref{fig:coherence_beta-alpha}
     shows that the $3$-cells
     \[
        F_{\treeVLeft}(\beta, f') \comp_1 F_{\treeLL}(g\comp_0 \alpha)\ 
        \comp_2\ 
        F_{\treeV}(g', f') \comp_1 F_{\text{ex}}(\beta, \alpha)\ 
        \comp_2\ 
        F_{\treeVRight}(g', \alpha) \comp_1 F_{\treeLL}(\beta \comp_0 f)
     \]
     and
     \[
        F_{\treeLL}(\beta) \comp_0 F_{\treeLog}(f')
        \comp_1 F_{\treeVRight}(g, \alpha)\ 
        \comp_2\ 
        F_{\treeLog}(g') \comp_0 F_{\treeLL}(\alpha)
        \comp_1 F_{\treeVRight}(\beta, f)
     \]
     of $B$ are equal. Here we denote by
     $F_{\text{ex}}(\beta, \alpha)$ the identity $3$-cell
     going from $F_{\treeLL}(g' \comp_0 \alpha) \comp_1 F_{\treeLL}(\beta \comp_0 f)$
     to $F_{\treeLL}(\beta \comp_0 f') \comp_1 F_{\treeY}(g \comp_0 \alpha)$
     that we get from the composition of the following pair of trivial $3$-cells
     \[
      \begin{tikzcd}[column sep=-6em]
       \null &
       F_{\treeLL}(g' \comp_0 \alpha \comp_1  \beta \comp_0 f)
       \arrow[triple, swap, "{F_{\treeY}(g'\comp_0 \alpha, \beta \comp_0 f)}"]{ldd}
       = F_{\treeLL}(\beta \comp_0 f' \comp_1 g \comp_0 \alpha)
       \arrow[triple, "{F_{\treeY}(\beta \comp_0 f', g \comp_0 \alpha)}"]{rdd}
       & \null \\ \\
       F_{\treeLL}(g' \comp_0 \alpha) \comp_1 F_{\treeLL}(\beta \comp_0 f)
       & \null &
       F_{\treeLL}(\beta \comp_0 f') \comp_1 F_{\treeY}(g \comp_0 \alpha)
      \end{tikzcd}
     \]
     of $B$, where the equality in the upper row is just the exchange law.

     Since $F_{\text{ex}}(\beta, \alpha)$ is a trivial
     $3$-cell, the $4$-simplex is actually imposing the equality between
     the $3$-cells
     \[
        F_{\treeVLeft}(\beta, f') \comp_1 F_{\treeLL}(g\comp_0 \alpha)\ 
        \comp_2\ 
        F_{\treeVRight}(g', \alpha) \comp_1 F_{\treeLL}(\beta \comp_0 f)
     \]
     and
     \[
      F_{\treeLL}(\beta) \comp_0 F_{\treeLog}(f')
        \comp_1 F_{\treeVRight}(g, \alpha)\ 
        \comp_2\ 
        F_{\treeLog}(g') \comp_0 F_{\treeLL}(\alpha)
        \comp_1 F_{\treeVRight}(\beta, f)
     \]
     of $B$, thereby establishing the coherence
     \scalebox{0.3}{
            \begin{forest}
                for tree={%
                    label/.option=content,
                    grow'=north,
                    content=,
                    circle,
                    fill,
                    minimum size=3pt,
                    inner sep=0pt,
                    s sep+=15,
                }
                [
                    [ [] ]
                    [ [] ]
                ]
            \end{forest}
        }.

     \item[%
        \scalebox{0.3}{
            \begin{forest}
                for tree={%
                    label/.option=content,
                    grow'=north,
                    content=,
                    circle,
                    fill,
                    minimum size=3pt,
                    inner sep=0pt,
                    s sep+=15,
                }
                [
                    [
                        [] [] []
                    ]
                ]
            \end{forest}
        }
     ]
     Consider a triple
     \[
      \begin{tikzcd}[column sep=4.7em]
       \bullet
       \ar[r, bend left=80, looseness=1.6, ""{below, name=1}]
       \ar[r, bend left, ""{name=2u}, ""{below, name=2d}]
       \ar[r, bend right, ""{name=3u}, ""{below, name=3d}]
       \ar[r, bend right=80, looseness=1.6, ""{name=4}]
       \ar[Rightarrow, from=1, to=2u, "\alpha"]
       \ar[Rightarrow, from=2d, to=3u, "\beta"]
       \ar[Rightarrow, from=3d, to=4, "\gamma"] &
       \bullet
      \end{tikzcd}
     \]
     of $1$-composable $2$-cells $\alpha$, $\beta$ and $\gamma$
     of $A$. The simplicial oplax $3$-morphism $F$ trivially
     satisfies the coherence associated to this tree, which is
     the trivial equality between the following identity $3$-cell
     \[
      F_{\treeY}(\gamma\comp_1 \beta, \alpha) \comp_2
      F_{\treeY}(\gamma, \beta) \comp_1 F_{\treeLL}(\alpha)
     \]
     and
     \[
      F_{\treeY}(\gamma, \beta\comp_1 \alpha) \comp_2
      F_{\treeLL}(\gamma) \comp_1 F_{\treeY}(\beta, \alpha)
     \]
     of $B$. This coherence is encoded in the $4$-simplex
     of $\SN(B)$ depicted in figure~\ref{fig:coherence_gamma-beta-alpha}.
     \begin{figure}
      \centering
      \begin{tikzpicture}[scale=1.7]
        \pentagon{%
            /pentagon/label/.cd,
            01=$Fi$,
            04=$Ff$,
            02=$Fh$, 03=$Fg$,
            012=${F_{\treeLL}(\gamma)}$, 023=${F_{\treeLL}(\beta)}$, 
            034=${\phantom{Oi}F_{\treeLL}(\alpha)}$,
            013=${F_{\treeLL}(\gamma\comp_1\beta)}$, 123=${=}$,
            134=${=}$, 014=${F_{\treeLL}(\gamma\comp_1 \beta\comp_1\alpha)}$,
            024=${F_{\treeLL}(\beta\comp_1\alpha)}$, 234=${=}$,
            124=${=}$,
            0123=${F_{\treeY}(\gamma, \beta)}$,
            0134=${F_{\treeY}(\gamma\comp_1\beta, \alpha)}$,
            0234=${F_{\treeY}(\beta, \alpha)}$,
            0124=${F_{\treeY}(\gamma, \beta\comp_1\alpha)}$,
            /pentagon/arrowstyle/.cd,
            12={equal}, 23={equal}, 34={equal}, 
            03={pos=0.55}, 02={pos=0.55},
            012={phantom, description}, 013={phantom, description},
            014={phantom, description}, 023={phantom, description},
            024={phantom, description}, 034={phantom, description},
            123={phantom, description}, 124={phantom, description},
            134={phantom, description}, 234={phantom, description},
            0123={equal}, 0134={equal}, 1234={equal},
            0234={equal}, 0124={equal}, 
            01234={phantom, description},
            /pentagon/labelstyle/.cd,
            012={anchor=center}, 013={anchor=center}, 014={anchor=center},
            023={anchor=center}, 
            024={anchor=center}, 034={anchor=center},
            123={anchor=center}, 124={anchor=center}, 134={anchor=center},
            234={anchor=center}        
        }
      \end{tikzpicture}
      \caption{Representing the trivial coherence $\gamma\beta\alpha$}
      \label{fig:coherence_gamma-beta-alpha}
     \end{figure}
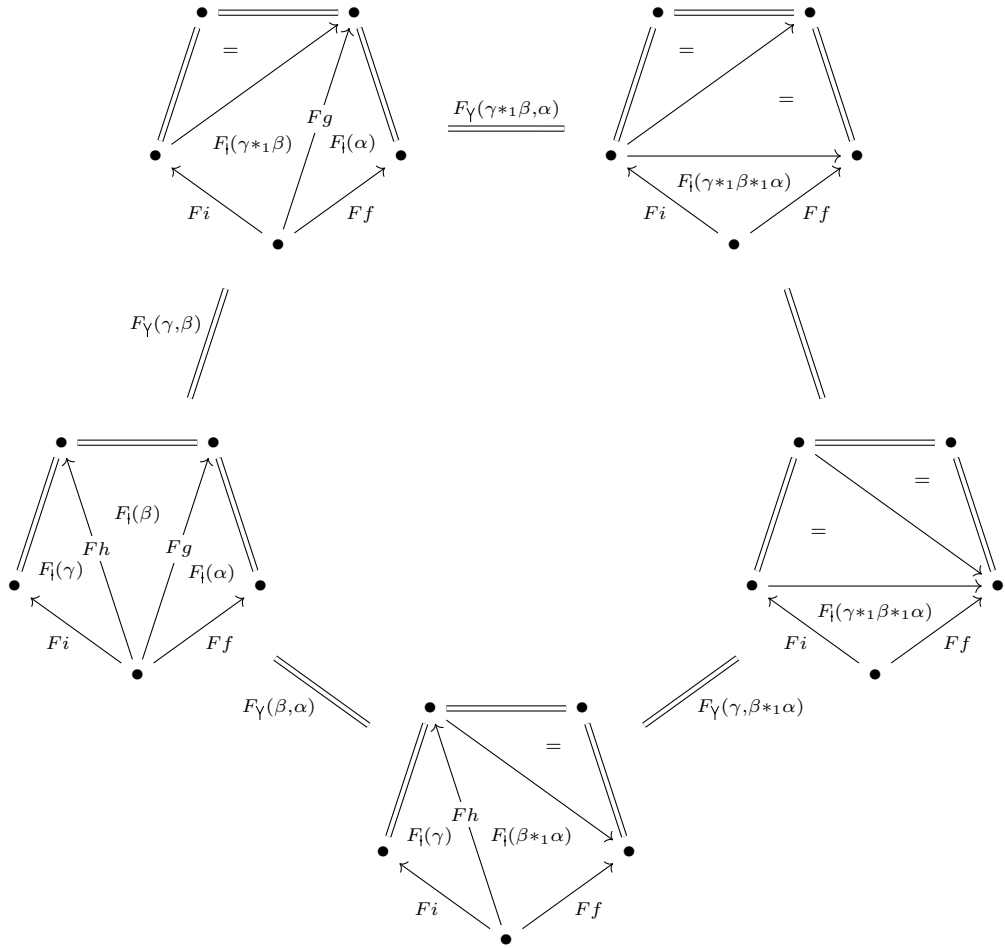

     \item[%
     \scalebox{0.3}{
     	\begin{forest}
     		for tree={%
     			label/.option=content,
     			grow=north,
     			content=,
     			circle,
     			fill,
     			minimum size=3pt,
     			inner sep=0pt,
     			s sep+=15,
     		}
     		[
     		[
     		[] [ [] ]
     		]
     		]
     	\end{forest}
     }
     ] Consider a pair
     
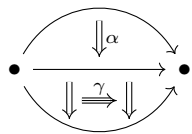
\[
      \begin{tikzcd}[column sep=5em]
	      \bullet
	      \ar[r, bend left=60, looseness=1.2, ""{below, name=1}]
	      \ar[r, ""{name=2u}, ""{below, name=2d}]
	      \ar[r, bend right=60, looseness=1.2, ""{name=3}]
	      \ar[Rightarrow, from=1, to=2u, "\alpha"]
	      \ar[Rightarrow, from=2d, to=3, shift right=2.6ex, ""{name=beta1}]
	      \ar[Rightarrow, from=2d, to=3, shift left=2.6ex, ""'{name=beta2}]
	      \arrow[triple, from=beta1, to=beta2, "\gamma"]{}
	      &
	      \bullet
      \end{tikzcd}
     \]
     of $1$-composable cells $\alpha$ and $\gamma$ of $A$.
     The $4$-simplex of $\SN(B)$ depicted in figure~\ref{fig:coherence_gamma-alpha}
     shows that the $3$-cells
     \[
	     F_{\treeLLL}(\gamma) \comp_1 F_{\treeLL}(\alpha)
     \]
     and
     \[
	     F_{\treeLLL}(\gamma \comp_1 \alpha)
     \]
     of $B$ are equal, which establishes the coherence for this tree.
     \begin{figure}
      \centering
      \begin{tikzpicture}[scale=1.7]
        \pentagon{%
            /pentagon/label/.cd,
            12=$Fh$,
            04=$Ff$,
            02=$Fg$, 03=$Fh$,
            13=$Fh$, 14=$Ff$, 
            012=${=}$, 023=${F_{\treeLL}(\beta)}$, 
            034=${\phantom{Oi}F_{\treeLL}(\alpha)}$,
            013=${F_{\treeLL}(\beta')}$, 123=${=}$,
            134=${F_{\treeLL}(\beta'\comp_1 \alpha)}$, 014=${=}$,
            024=${F_{\treeLL}(\beta\comp_1\alpha)}$, 234=${=}$,
            124=${F_{\treeLL}(\beta'\comp_1 \alpha)}$,
            0123=${F_{\treeLLL}(\gamma)}$,
            0134=${F_{\treeY}(\beta', \alpha)}$,
            0234=${F_{\treeY}(\beta, \alpha)}$,
            0124=${F_{\treeLLL}(\gamma\comp_1\alpha)}$,
            /pentagon/arrowstyle/.cd,
            01={equal}, 23={equal}, 34={equal}, 
            03={pos=0.55}, 02={pos=0.55},
            24={equal},
            012={phantom, description}, 013={phantom, description},
            014={phantom, description}, 023={phantom, description},
            024={phantom, description}, 034={phantom, description},
            123={phantom, description}, 124={phantom, description},
            134={phantom, description}, 234={phantom, description},
            0134={equal}, 1234={equal},
            0234={equal},
            01234={phantom, description},
            /pentagon/labelstyle/.cd,
            012={anchor=center}, 013={anchor=center}, 014={anchor=center},
            023={anchor=center}, 
            024={anchor=center}, 034={anchor=center},
            123={anchor=center}, 124={anchor=center}, 134={anchor=center},
            234={anchor=center}        
        }
      \end{tikzpicture}
      \caption{Establishing the coherence $\gamma \alpha$.}
      \label{fig:coherence_gamma-alpha}
     \end{figure}

     \item[%
     \scalebox{0.3}{
     	\begin{forest}
     		for tree={%
     			label/.option=content,
     			grow=north,
     			content=,
     			circle,
     			fill,
     			minimum size=3pt,
     			inner sep=0pt,
     			s sep+=15,
     		}
     		[
	     		[
		     		[ [] ] []
	     		]
     		]
     	\end{forest}
     }
     ] For any pair
     
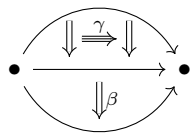
\[
     \begin{tikzcd}[column sep=5em]
     \bullet
     \ar[r, bend left=60, looseness=1.2, ""{below, name=1}]
     \ar[r, ""{name=2u}, ""{below, name=2d}]
     \ar[r, bend right=60, looseness=1.2, ""{name=3}]
     \ar[Rightarrow, from=2d, to=3, "\beta"]
     \ar[Rightarrow, from=1, to=2u, shift right=2.6ex, ""{name=beta1}]
     \ar[Rightarrow, from=1, to=2u, shift left=2.6ex, ""'{name=beta2}]
     \arrow[triple, from=beta1, to=beta2, "\gamma"]{}
     &
     \bullet
     \end{tikzcd}
     \]
     of $1$-composable cells $\gamma$ and $\beta$ of $A$,
     there is a $4$-simplex of $\SN(B)$ dual to
     the one depicted in~\ref{fig:coherence_gamma-alpha}
     showing the equality between the $3$-cells
     \[
     F_{\treeLL}(\beta) \comp_1 F_{\treeLLL}(\gamma)
     \]
     and
     \[
     F_{\treeLLL}(\beta \comp_1 \gamma)
     \]
     of $B$ and thus establishing the coherence for the tree
     \scalebox{0.3}{
     	\begin{forest}
     		for tree={%
     			label/.option=content,
     			grow=north,
     			content=,
     			circle,
     			fill,
     			minimum size=3pt,
     			inner sep=0pt,
     			s sep+=15,
     		}
     		[
	     		[
		     		[ [] ] []
	     		]
     		]
     	\end{forest}
     }.

     \begin{figure}[ht]
      \centering
      \begin{tikzpicture}[scale=1.7]
        \pentagon{%
            /pentagon/label/.cd,
            12=$Fg$,
            04=$Ff$,
            02=$Ff$, 03=$Fg$,
            13=$Fg$, 14=$Ff$, 
            012=${=}$, 023=${F_{\treeLL}(\alpha)}$, 
            034=${=}$,
            013=${F_{\treeLL}(\beta)}$, 123=${=}$,
            134=${F_{\treeLL}(\gamma)}$, 014=${=}$,
            024=${F_{\treeLL}(\gamma)}$, 234=${=}$,
            124=${F_{\treeLL}(\gamma)}$,
            0123=${F_{\treeLLL}(\Gamma)}$,
            0134=${F_{\treeLLL}(\Gamma')}$,
            0234=${F_{\treeLLL}(\Gamma'\comp_2\Gamma)}$,
            /pentagon/arrowstyle/.cd,
            01={equal}, 23={equal}, 34={equal}, 
            03={pos=0.55}, 02={pos=0.55},
            24={equal},
            012={phantom, description}, 013={phantom, description},
            014={phantom, description}, 023={phantom, description},
            024={phantom, description}, 034={phantom, description},
            123={phantom, description}, 124={phantom, description},
            134={phantom, description}, 234={phantom, description},
            1234={equal},
            0124={equal},
            01234={phantom, description},
            /pentagon/labelstyle/.cd,
            012={anchor=center}, 013={anchor=center}, 014={anchor=center},
            023={anchor=center}, 
            024={anchor=center}, 034={anchor=center},
            123={anchor=center}, 124={anchor=center}, 134={anchor=center},
            234={anchor=center}        
        }
      \end{tikzpicture}
      \caption{Establishing the coherence $\Gamma'\Gamma$.}
      \label{fig:coherence_Gamma-Gamma}
     \end{figure}

     \item[%
     \scalebox{0.3}{
     	\begin{forest}
     		for tree={%
     			label/.option=content,
     			grow=north,
     			content=,
     			circle,
     			fill,
     			minimum size=3pt,
     			inner sep=0pt,
     			s sep+=15,
     		}
     		[
	     		[
		     		[
			     		[] []
		     		]
	     		]
     		]
     	\end{forest}
     }
     ] Consider a pair
     
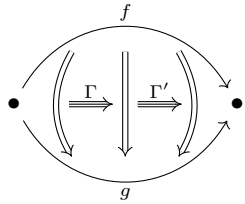
\[
	     \begin{tikzcd}[column sep=7em]
		     \bullet
		     \ar[r, bend left=60, looseness=1.2, "f", "\phantom{bullet}"'{name=1}]
		     \ar[r, bend right=60, looseness=1.2, "g"', "\phantom{bullet}"{name=3}]
		     \ar[Rightarrow, from=1, to=3, shift right=4ex, bend right, ""{name=beta1}]
		     \ar[Rightarrow, from=1, to=3, ""'{name=beta2d}, ""{name=beta2u}]
		     \ar[Rightarrow, from=1, to=3, shift left=4ex, bend left, ""'{name=beta3}]
		     \arrow[triple, from=beta1, to=beta2d, "\Gamma"]{}
		     \arrow[triple, from=beta2u, to=beta3, "\Gamma'"]{}
		     &
		     \bullet
	     \end{tikzcd}
     \]
     of $2$-composable $3$-cells $\Gamma\colon \alpha \to \beta$ and $\Gamma'\colon \beta \to \delta$ of $A$.
     The $4$-simplex of $\SN(B)$ displayed in figure~\ref{fig:coherence_Gamma-Gamma}
     shows that we have the equality
     \[
	     F_{\treeLLL}(\delta \comp_2 \gamma) = F_{\treeLLL}(\delta) \comp_2 F_{\treeLLL}(\gamma)
     \]
     between these two $3$-cells of $B$.

     \begin{figure}
      \centering
      \begin{tikzpicture}[scale=1.7]
        \pentagon{%
            /pentagon/label/.cd,
            01=$Ff$,
            12=$Fg'$,
            04=$F(gf)$,
            02=$F(g'f)$, 03=$F(gf)$,
            13=$Fg$, 14=$Fg$, 
            012=${F_{\treeV}(g', f)}$, 023=${F_{\treeLL}(\alpha f)}$, 
            034=${\phantom{O}=}$,
            013=${F_{\treeV}(g, f)}$, 123=${F_{\treeLL}(\alpha)}$,
            134=${=}$, 014=${F_{\treeV}(g, f)}$,
            024=${F_{\treeLL}(\beta f)}$, 234=${=}$,
            124=${F_{\treeLL}(\beta)}$,
            0123=${F_{\treeVLeft}(\alpha, f)}$,
            1234=${F_{\treeLLL}(\Gamma)}$,
            0234=${F_{\treeLLL}(\Gamma'\comp_0 f)}$,
            0124=${F_{\treeVLeft}(\beta, f)}$,
            /pentagon/arrowstyle/.cd,
            23={equal}, 34={equal}, 
            03={pos=0.55}, 02={pos=0.55},
            24={equal},
            012={phantom, description}, 013={phantom, description},
            014={phantom, description}, 023={phantom, description},
            024={phantom, description}, 034={phantom, description},
            123={phantom, description}, 124={phantom, description},
            134={phantom, description}, 234={phantom, description},
            0123={equal}, 0134={equal},
            0124={equal},
            01234={phantom, description},
            /pentagon/labelstyle/.cd,
            012={anchor=center}, 013={anchor=center}, 014={anchor=center},
            023={anchor=center}, 
            024={anchor=center}, 034={anchor=center},
            123={anchor=center}, 124={anchor=center}, 134={anchor=center},
            234={anchor=center}        
        }
       \end{tikzpicture}
      \caption{Establishing the coherence $\Gamma f$.}
      \label{fig:coherence_Gamma-f}
     \end{figure}

         \begin{figure}
         	\centering
         	\begin{tikzpicture}[scale=1.7]
         	\pentagon{%
         		/pentagon/label/.cd,
         		01={},
         		12={$Fg'$},
         		04=$Fg$,
         		02={$Fg'$}, 03=$Fg$,
         		13=$Fg'$, 14=$Fg'$, 
         		012={$=$}, 023={$F_{\treeLL}(\alpha)$}, 034={$=$},
         		123={$=$}, 013={$F_{\treeLL}(\beta)$},
         		014={$F_{\treeLL}(\beta)$}, 134={$=$},
         		124={$=$}, 234={$=$},
         		024={$F_{\treeLL}(\beta)$},
         		0123={$F_{\treeLLL}(\Gamma)$},
         		0124={}, 0134={}, 0234={$\Delta$}, 1234={},
         		01234={},
         		/pentagon/arrowstyle/.cd,
         		01={equal}, 23={equal}, 34={equal}, 
         		03={pos=0.55}, 02={pos=0.55},
         		24={equal},
         		012={phantom, description}, 034={phantom, description},
         		123={phantom, description}, 124={phantom, description},
         		134={phantom, description}, 234={phantom, description},
         		1234={equal}, 0134={equal},
         		0124={equal},
         		01234={phantom, description},
         		/pentagon/labelstyle/.cd,
         		012={anchor=center}, 034={anchor=center},
         		123={anchor=center}, 124={anchor=center}, 134={anchor=center},
         		234={anchor=center}        
         	}
         	\end{tikzpicture}
         	\caption{the $3$-cell $\Delta$ of $B$.}
         	\label{fig:Delta}
         \end{figure}
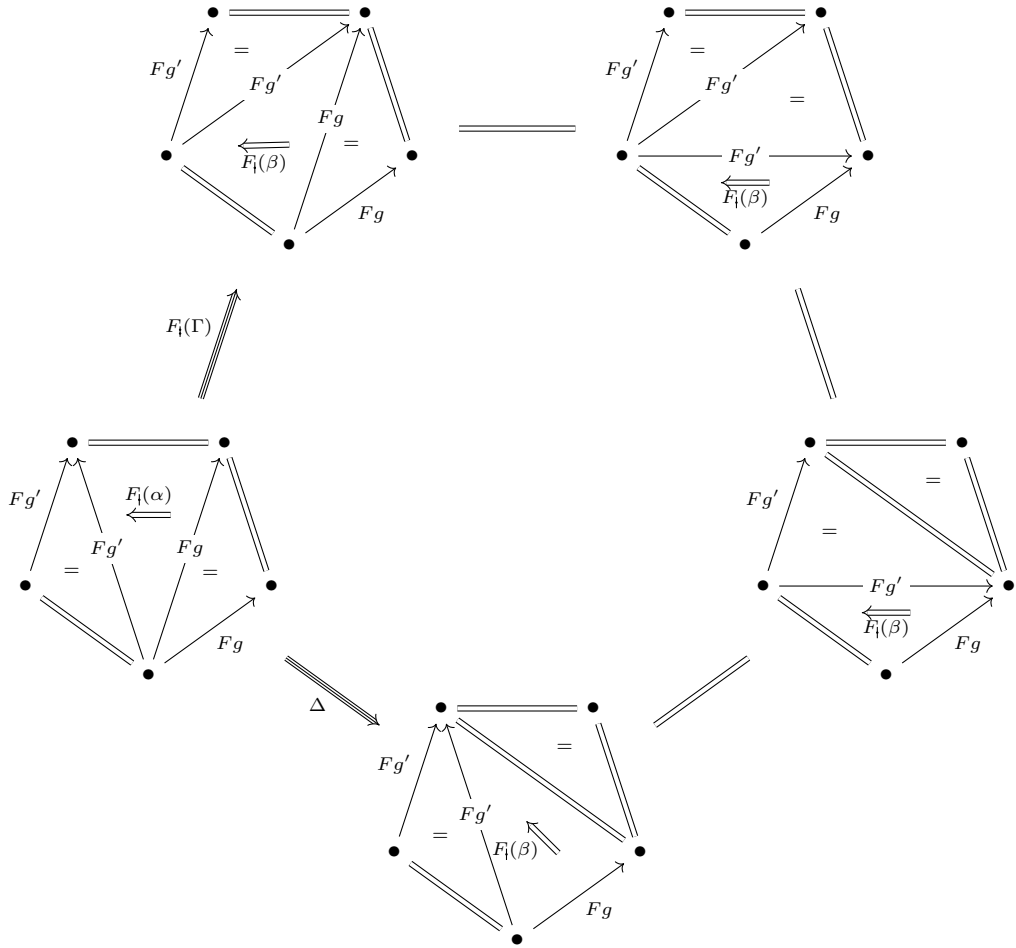
     
          \item[%
     \scalebox{0.3}{
     	\begin{forest}
     		for tree={%
     			label/.option=content,
     			grow=north,
     			content=,
     			circle,
     			fill,
     			minimum size=3pt,
     			inner sep=0pt,
     			s sep+=15,
     		}
     		[
	     		[]
     			[ [ [] ] ]
     		]
     	\end{forest}
     }
     ] Consider a pair
     
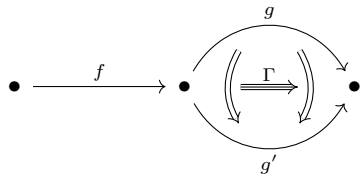
\[
     \begin{tikzcd}[column sep=5em]
     \bullet
     \ar[r, "f"]
     &
     \bullet
     \ar[r, bend left=60, looseness=1.2, "g", "\phantom{bullet}"'{ name=1}]
     \ar[r, bend right=60, looseness=1.2, "g'"', "\phantom{bullet}"{name=3}]
     \ar[Rightarrow, from=1, to=3, shift right=2ex, bend right, ""{name=beta1}]
     \ar[Rightarrow, from=1, to=3, shift left=2ex, bend left, ""'{name=beta3}]
     \arrow[triple, from=beta1, to=beta3, "\Gamma"]{}
     &
     \bullet
     \end{tikzcd}
     \]
     of $0$-composable cells $f$ and $\Gamma \colon \alpha \to \beta$ of $A$.
     The $4$-simplex of $\SN(B)$ depicted in figure~\ref{fig:coherence_Gamma-f}
     shows that the equality
     \[
     F_{\treeVLeft}(\beta, f) \comp_1 F_{\treeLLL}(\Gamma \comp_0 f) = \Delta \comp_0 F_{\treeLog}(f) \comp_1 F_{\treeVLeft}(\alpha, f)\,.
     \]
     The $3$-cell $\Delta$ is in fact equal to $F_{\treeLLL}(\Gamma)$, as the $4$-simplex of~$\SN(B)$
     depicted in figure~\ref{fig:Delta} shows. Hence, the coherence for this tree is verified.

     \item[%
     \scalebox{0.3}{
     	\begin{forest}
     		for tree={%
     			label/.option=content,
     			grow=north,
     			content=,
     			circle,
     			fill,
     			minimum size=3pt,
     			inner sep=0pt,
     			s sep+=15,
     		}
     		[
	     		[
		     		[ [] ]
	     		]
	     		[]
     		]
     	\end{forest}
     }
     ] Consider a pair
     
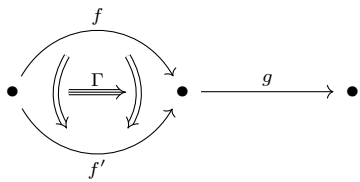
\[
     \begin{tikzcd}[column sep=5em]
     \bullet
     \ar[r, bend left=60, looseness=1.2, "f", "\phantom{bullet}"'{name=1}]
     \ar[r, bend right=60, looseness=1.2, "f'"', "\phantom{bullet}"{name=3}]
     \ar[Rightarrow, from=1, to=3, shift right=2ex, bend right, ""{name=beta1}]
     \ar[Rightarrow, from=1, to=3, shift left=2ex, bend left, ""'{name=beta3}]
     \arrow[triple, from=beta1, to=beta3, "\Gamma"]{}
     &
     \bullet
     \ar[r, "g"]
     &
     \bullet
     \end{tikzcd}
     \]
     of $0$-composable cells $\Gamma\colon \alpha \to \beta$ and $g$ of $A$.
     There is a $4$-simplex of $\SN(B)$ dual to the one depicted in figure~\ref{fig:coherence_Gamma-f}
     showing that the following equality
     \[
	     F_{\treeVRight}(g, \beta) \comp_1 F_{\treeLLL}(g \comp_0 \Gamma) = F_{\treeLog}(g) \comp_0 F_{\treeLLL}(\Gamma) \comp_1 F_{\treeVRight}(g, \alpha)
     \]
     of $3$-cells of $B$ holds and thus establishing the coherence
     \scalebox{0.3}{
     	\begin{forest}
     		for tree={%
     			label/.option=content,
     			grow=north,
     			content=,
     			circle,
     			fill,
     			minimum size=3pt,
     			inner sep=0pt,
     			s sep+=15,
     		}
     		[
	     		[
		     		[ [] ]
	     		]
	     		[]
     		]
     	\end{forest}
     }.

	\end{description}

\section{Correspondence}
 
The results of the preceding sections give us an application
that associates to a simplicial oplax $3$-morphism $F$ a normalised oplax $3$-functor,
that we shall denote by $\cCl(F)$
as well as an application in the opposite direction assigning to each
normalised oplax $3$-functor $G$ a simplicial oplax $3$-morphism $\Nl(G)$.
These two applications are actually inverses
of one another, so that we have a precise bijective correspondence
between oplax $3$-functors and simplicial oplax $3$-morphisms.
The aim of this section is to check this statement.

\begin{paragr}
 It results immediately from the definitions and from conditions~\ref{cond:simpl_oplax-i},
 \ref{cond:simpl_oplax-ii} and~\ref{cond:simpl_oplax-iii} that given any
 normalised oplax $3$-functor $F \colon A \to B$, we have an equality
 $F = \cCl\Nl(F)$. For instance, for the tree
 \scalebox{0.3}{
 	\begin{forest}
 		for tree={%
 			label/.option=content,
 			grow'=north,
 			content=,
 			circle,
 			fill,
 			minimum size=3pt,
 			inner sep=0pt,
 			s sep+=15,
 		}
 		[ 
 		 [][][]
 		]
 	\end{forest}
 }
 we have that to any triple $(h, g, f)$ of composable $1$-cells of $A$,
 the simplicial oplax $3$-morphism $\Nl(F) \colon \SN(A) \to \SN(B)$ associates
 the $3$-simplex
 
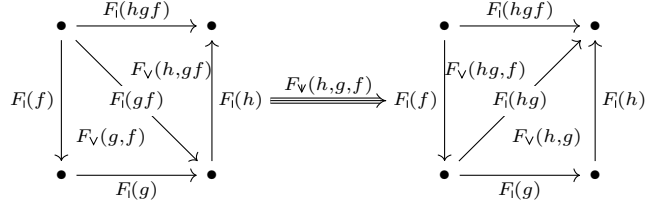
\begin{center}
  \begin{tikzpicture}[scale=1.4, font=\footnotesize]
  \squares{%
  	/squares/label/.cd,
  	0=$\bullet$, 1=$\bullet$, 2=$\bullet$, 3=$\bullet$,
  	01=${F_{\treeL}(f)}$, 12=${F_{\treeL}(g)}$, 23=${F_{\treeL}(h)}$,
  	02=${F_{\treeL}(gf)}$, 03=${F_{\treeL}(hgf)}$, 13=${F_{\treeL}(hg)}$,
  	012=${F_{\treeV}(g, f)}$, 023=${F_{\treeV}(h, gf)\phantom{o}}$, 123=${F_{\treeV}(h, g)}$, 013=${\phantom{o}F_{\treeV}(hg, f)}$,
  	0123=${F_{\treeW}(h, g, f)}$,
  	/squares/arrowstyle/.cd,
  	012={phantom, description}, 023={phantom, description}, 123={phantom, description},
  	013={phantom, description},
  	/squares/labelstyle/.cd,
  	012={anchor=center}, 023={anchor=center}, 123={anchor=center},
  	013={anchor=center}
  }
  \end{tikzpicture} 
 \end{center}
 of $\SN(A)$; by definition, we set $\cCl\Nl(F)_{\treeW}(h, g, f)$ to be
 the main $3$-cell of this $3$-simplex, \ie $F_{\treeW}(h, g, f)$.
\end{paragr}

\begin{paragr}
 Let $F \colon \SN(A) \to \SN(B)$ be a simplicial oplax
 $3$-morphism. We have seen in subsection~\ref{section:simplicial-to-cellular}
 that there is a canonically associated
 normalised oplax $3$-functor $\cCl(F) \colon A \to B$.
 Moreover, in subsection~\ref{section:cellular-to-simplicial}
 we have shown that given any $n$-simplex of $\SN(A)$ in the form of
 a normalised oplax $3$-functor $x \colon \Deltan{n} \to A$,
 for $n \ge 0$,
 we get an $n$-simplex $\Nl\cCl F(x)$ of $\SN(B)$ given by the
 composition $F\circ x$. We want to check that $F(x) = \Nl\cCl F(x)$.
 Since $\SN(B)$ is $4$-coskeletal (see~\cite[Theorem 5.2]{Street}), it is
 enough to check that $F(x) = \Nl\cCl F(x)$ for all $n$-simplices $x$
 of $\SN(A)$ with $0 \le n \le 4$.
 
 The result is trivially verified for $0$-simplices and $1$-simplices.
 Consider a $2$\nbd-sim\-plex~$x$
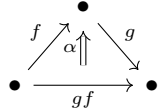
 \[
  \begin{tikzcd}[column sep = small]
   & \bullet \ar[dr, "g"]& \\
   \bullet \ar[ur, "f"] \ar[rr, "gf"', ""{name=gf}]
   && \bullet
   \ar[Rightarrow, from=gf, to=1-2, shorten <=1mm, shorten >=1mm, "\alpha"]
  \end{tikzcd}
 \]
 of $\SN(A)$.
 The $1$-skeletons of $F(x)$ and $\Nl\cCl F(x)$ coincide and
 the $2$-cell of $B$ filling the $2$-simplex $\Nl\cCl F(x)$
 is defined by $\cCl F_{\treeV}(g, f) \comp_1 \cCl F_{\treeLL}(\alpha)$.
 Condition~\ref{cond:simpl_oplax-ii} states precisely
 that this is the $2$-cell of $B$ filling the
 $2$-simplex $F(x)$.
 
 \begin{figure}
 \centering
  \begin{tikzpicture}[scale=1.3]
   \squares{%
     /squares/label/.cd,
     0=$\bullet$, 1=$\bullet$, 2=$\bullet$, 3=$\bullet$,
     01=$f$, 12=$g$, 23=$h$, 02=$i$, 03=$j$, 13=$k$,
     012=$\beta$, 023=$\alpha$, 123=$\delta$, 013=$\gamma$,
     0123=$\Gamma$
     }
  \end{tikzpicture}
  \caption{The $3$-simplex $x$ of $\SN(A)$.}
  \label{fig:Gamma-corr}
 \end{figure}
 Consider a $3$-simplex $x$ of $\SN(A)$ as depicted in figure~\ref{fig:Gamma-corr}.
 By the definition given in paragraph~\ref{paragr:def_cellular_to_simplicial},
 the main $3$-cell of the $3$-simplex $\Nl\cCl F(x)$ of $\SN(B)$
 is defined as
 \begin{gather}\label{eq:3cell-3simplex}
	\cCl F_{\treeV}(h, g) \comp_0 \cCl F(f) \comp_1
	\cCl F_{\treeVLeft}(\delta, f) \comp_1
	\cCl F_{\treeLL}(\gamma) \notag\\
	\comp_2\notag\\
	\cCl F_{\treeW}(h, g, f) \comp_1
	\cCl F_{\treeLLL}(\Gamma) \\
	\comp_2 \notag\\
	\cCl F(h) \comp_0 F_{\treeV}(g, f) \comp_1
	\cCl F_{\treeVRight}(h, \beta) \comp_1
	\cCl F_{\treeLL}(\alpha)\,. \notag
 \end{gather}
 We already know that the $2$-skeleton of $F(x)$ and $\Nl\cCl F(x)$ coincide,
 so we are left with showing that the main $3$-cell $\Psi$ of the $3$-simplex
 $F(x)$ of $\SN(B)$ corresponds to the above $3$-cell.
 This demands a careful analysis of some $4$-simplices
 encoding valuable information for an explicit description
 of $F(x)$, which will be carried over in the next paragraph.
 We end this paragraph by remarking that once the equivalence of translation
 of $3$-simplices is verified, then the correspondence for $4$-simplices
 follows easily, since it is completely characterised by the $2$-composition
 of $3$-cells in $B$.
\end{paragr}

\begin{paragr}\label{paragr:image_3-simplex}
 \begin{figure}
 \centering
 \begin{tikzpicture}[scale=1.3, font=\footnotesize]
  \pentagon{%
    /pentagon/label/.cd,
    0=$\bullet$, 1=$\bullet$, 2=$\bullet$, 3=$\bullet$, 4=$\bullet$,
    01=$f$, 12=$g$, 23=$h$, 34={}, 04=$i$,
    02=$j$, 03=$hj$, 13=$hg$, 14=$k$, 24=$h$,
    012=$\beta$, 034=$\alpha$, 023={}, 123={}, 134=$\delta$,
    014=$\gamma$, 024=$\alpha$, 234={}, 013=$h\beta$, 124=$\delta$,
    0123=$(\bigstar)$, 0134=$\Gamma$, 1234={\ref{cond:simpl_oplax-ii}},
    0234={\ref{cond:simpl_oplax-ii}}, 0124=$\Gamma$, 
    01234={},
    /pentagon/arrowstyle/.cd,
    01234={phantom, description}, 34={equal},
    1234={equal}, 0234={equal}
    }
\end{tikzpicture}
\caption{The $4$-simplex $y$ of $\SN(A)$}
\label{fig:Gamma1-corr}
\end{figure}

\begin{figure}
 \centering
 \subfloat[][The $3$-arrow $(\bigstar)$.]
 {%
  \begin{tikzpicture}[scale=1.2, font=\footnotesize]\label{fig:Star}
   \pentagon{%
    /pentagon/label/.cd,
    0=$\bullet$, 1=$\bullet$, 2=$\bullet$, 3=$\bullet$, 4=$\bullet$,
    01={}, 12=$f$, 23=$g$, 34=$h$, 04=$hj$,
    02=$f$, 03=$j$, 13=$gf$, 14=$hj$, 24=$hg$,
    012={$=$}, 034={}, 023=$\beta$, 123={}, 134=$h\beta$,
    014={$=$}, 024=$h\beta$, 234={}, 013=$\beta$, 124=$h\beta$,
    0123={}, 0134={$\TreeVRight$}, 1234={$(\spadesuit)$},
    0234={$(\bigstar)$}, 0124={}, 
    01234={},
    /pentagon/arrowstyle/.cd,
    01234={phantom, description}, 01={equal},
    012={phantom, description}, 014={phantom, description},
    0123={equal}, 0124={equal},
    /pentagon/labelstyle/.cd,
    012={anchor=center}, 014={anchor=center}
   }
  \end{tikzpicture}
 }
 \\
 \subfloat[][The $3$-arrow $(\spadesuit)$.]
 {%
  \begin{tikzpicture}[scale=1.2, font=\footnotesize]\label{fig:Spadesuit}
   \pentagon{%
    /pentagon/label/.cd,
    0=$\bullet$, 1=$\bullet$, 2=$\bullet$, 3=$\bullet$, 4=$\bullet$,
    01=$f$, 12=$g$, 23=$h$, 34={}, 04=$hj$,
    02=$gf$, 13=$hg$, 14=$hg$, 24=$h$,
    012={}, 034=$h\beta$, 023={}, 123={}, 134={$=$},
    014=$h\beta$, 024=$h\beta$, 234={$=$}, 013={}, 124={},
    0123=$\TreeW$, 0134={\ref{cond:simpl_oplax-ii}}, 1234={},
    0234={\ref{cond:simpl_oplax-ii}}, 0124=$(\spadesuit)$, 
    01234={},
    /pentagon/arrowstyle/.cd,
    01234={phantom, description}, 34={equal},
    134={phantom, description}, 234={phantom, description}, 1234={equal},
    0134={equal}, 0234={equal},
    /pentagon/labelstyle/.cd,
    134={anchor=center}, 234={anchor=center}
   }
  \end{tikzpicture}
 }
 \caption{The $3$-cells $(\bigstar)$ and $(\spadesuit)$.}
\end{figure}
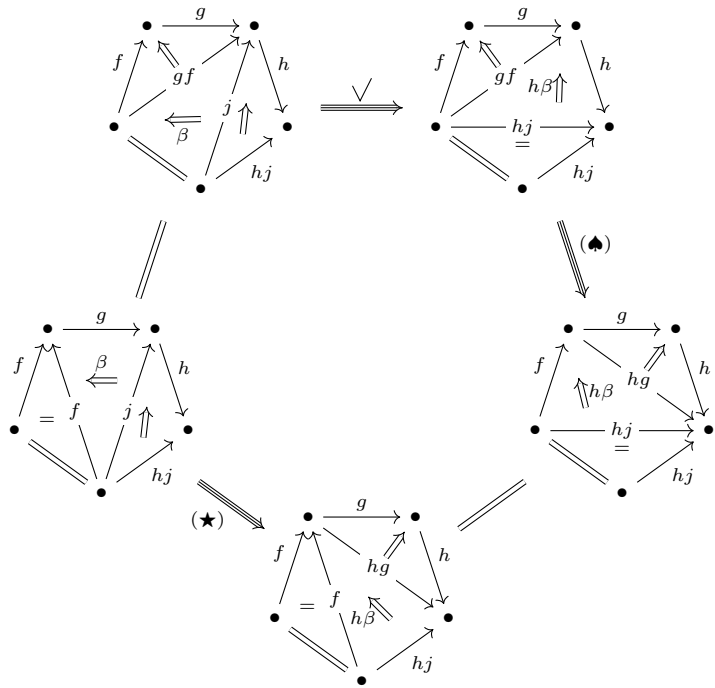
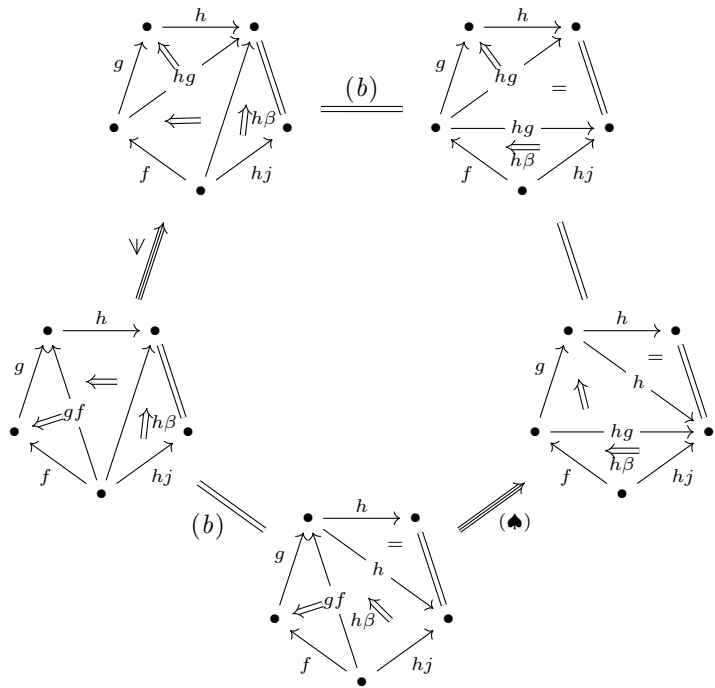

\begin{figure}
 \centering
 \begin{tikzpicture}[scale=1.3, font=\footnotesize]
   \pentagon{%
    /pentagon/label/.cd,
    0=$\bullet$, 1=$\bullet$, 2=$\bullet$, 3=$\bullet$, 4=$\bullet$,
    01={}, 12=$f$, 23=$hg$, 34={}, 04=$i$,
    02=$f$, 03=$hj$, 13=$hgf$, 14=$kf$, 24=$k$,
    012={}, 034=$\alpha$, 023=$h\beta$, 123={}, 134=$\delta f$,
    014=$\gamma$, 024=$\gamma$, 234=$\delta$, 013=$h\beta$, 124={},
    0123={\ref{cond:simpl_oplax-ii}}, 0134=$\Gamma$, 1234=$\TreeVLeft$,
    0234=$\Gamma$, 0124={\ref{cond:simpl_oplax-ii}}, 
    01234={},
    /pentagon/arrowstyle/.cd,
    01234={phantom, description}, 01={equal}, 34={equal},
    012={phantom, description},
    0123={equal}, 0124={equal},
    /pentagon/labelstyle/.cd,
    012={anchor=center}
   }
  \end{tikzpicture}
  \caption{The $4$-simplex $z$ of $\SN(A)$.}
  \label{fig:Gamma3}
\end{figure}

\begin{figure}
 \centering
 \begin{tikzpicture}[scale=1.3, font=\footnotesize]
   \pentagon{%
    /pentagon/label/.cd,
    0=$\bullet$, 1=$\bullet$, 2=$\bullet$, 3=$\bullet$, 4=$\bullet$,
    01={}, 12=$hgf$, 34={}, 04=$i$,
    02=$hj$, 03=$i$, 13=$hgf$, 14=$hgf$,
    012=$h\beta$, 023=$\alpha$, 
    014=$\delta + \gamma$, 024=$\alpha$, 013=$\beta+\alpha$,
    034={$=$}, 123={$=$}, 134={$=$},
    124={$=$}, 234={$=$},
    0123={\ref{cond:simpl_oplax-iii}}, 0134=$\Gamma$,
    0124=$\Gamma$, 0234={}, 1234={},
    01234={},
    /pentagon/arrowstyle/.cd,
    01234={phantom, description},
    01={equal}, 23={equal}, 34={equal}, 24={equal},
    034={phantom, description}, 123={phantom, description}, 134={phantom, description},
    124={phantom, description}, 234={phantom, description},
    0123={equal},
    /pentagon/labelstyle/.cd,
     034={anchor=center}, 123={anchor=center}, 134={anchor=center},
     124={anchor=center}, 234={anchor=center}
   }
  \end{tikzpicture}
  \caption{The $4$-simplex $z'$ of $\SN(A)$.}
  \label{fig:Gamma4}
\end{figure}

Consider a $3$-simplex $x$ of $\SN(A)$ as depicted in figure~\ref{fig:Gamma-corr}
and call $\Psi$ the main $3$-cell of the $3$-simplex $F(x)$ of $\SN(B)$.
The $4$-simplex of $\SN(A)$ depicted in figure~\ref{fig:Gamma1-corr} shows
that $\Psi$ is given by the $2$-composition
of the image of the $3$-cell denoted by $(\bigstar)$, whiskered
with $\cCl F(\alpha)$, followed by the image of the main $3$-cell of the
following $3$-simplex
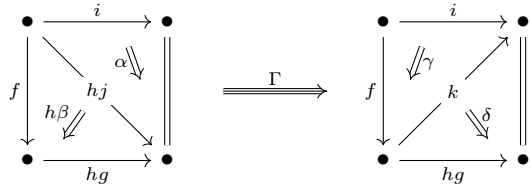
\begin{center}
 \begin{tikzpicture}[scale=1.3]
   \squares{%
     /squares/label/.cd,
     0=$\bullet$, 1=$\bullet$, 2=$\bullet$, 3=$\bullet$,
     01=$f$, 12=$hg$, 02=$hj$, 03=$i$, 13=$k$, 23={},
     012=$h\beta$, 023=$\alpha$, 123=$\delta$, 013=$\gamma$,
     0123=$\Gamma$,
     /squares/arrowstyle/.cd,
     23={equal}
     }
 \end{tikzpicture}
\end{center}
of $\SN(A)$, which is a $3$-cell of $\SN(B)$ that we shall call $\Phi$,
whiskered by $\cCl F_{\treeV}(h, g)$.
The $4$-simplices depicted in figure~\ref{fig:Star}
show that the image under $F$ of the $3$-cell $(\bigstar)$
is given by
\[
 \cCl F_{\treeLL}(\spadesuit) \comp_2 \cCl F_{\treeV}(g, f) \comp_1 \cCl F_{\TreeVRight}(h, \beta)
\]
and in turn the $4$-simplex of $\SN(A)$ depicted in figure~\ref{fig:Spadesuit}
entails that
\[
 \cCl F_{\treeLL}(\spadesuit) = \cCl F_{\treeW}(h, g, f) \comp_1 \cCl F_{\treeLL}(h\beta)\,.
\]
Therefore we get that $\Psi$ is given by the composition
\begin{gather*}
    \cCl F_{\treeV}(h, g) \comp_0 \cCl F_{\treeL}(f) \comp_1 \Phi \\
	\comp_2\\
	\cCl F_{\treeW}(h, g, f) \comp_1
	\cCl F_{\treeLL}(h\comp_0 \beta) \comp_1
	\cCl F_{\treeLL}(\alpha) \\
	\comp_2\\
	\cCl F(h) \comp_0 F_{\treeV}(g, f) \comp_1
	\cCl F_{\treeVRight}(h, \beta) \comp_1
	\cCl F_{\treeLL}(\alpha)\,.
\end{gather*}
Finally, the $4$-simplex~$z$ of~$\SN(A)$ depicted in figure~\ref{fig:Gamma3},
joint with the $4$-simplex~$z'$ of $\SN(A)$ depicted in figure~\ref{fig:Gamma4},
where we have denoted by
$\beta + \alpha$ and $\delta + \gamma$ the evident whiskered compositions,
and another one totally similar but dual
show that the $3$-cell $\Phi$
of $B$ is in fact
\begin{gather*}
 \cCl F_{\treeV}(h, g) \comp_0 \cCl F_{\treeL}(f) \comp_1
 \cCl F_{\treeVLeft}(\delta, f) \comp_1 \cCl F_{\treeLL}(\gamma)
 \\ \comp_2 \\
 \cCl F_{\treeV}(h, g) \comp_0 \cCl F_{\treeL}(f) \comp_1
 \cCl F_{\treeV}(h\comp_0 g, f) \comp_1 \cCl F_{\treeLLL}(\Gamma)\,.
\end{gather*}
Using the interchange law we immediately deduce from the above
that the $3$-cell $\Psi$ of $B$ is precisely the $3$-cell
detailed in~\eqref{eq:3cell-3simplex}.
\end{paragr}

\begin{paragr}
Given two oplax $3$-functors $F \colon A \to B$
and $G \colon B \to C$, we now check that the ``obvious'' candidate for the
composite oplax $3$-functor $G\circ F$ corresponds via the bijections
established above to the composite of the simplicial oplax $3$-morphisms
associated to $F$ and $G$.
That is, we shall show that $G \circ F = \cCl(\Nl G \circ \Nl F)$.
Hence, we will deduce that oplax $3$-functors admit a composition
operation and that the category of $3$-categories and oplax $3$-functors
and $3$-categories and simplicial oplax $3$-morphisms are isomorphic.

We already know from section~\ref{section:cellular-to-simplicial}
that:
\begin{description}
		\item[\scalebox{0.3}{
			\begin{forest}
				for tree={%
					label/.option=content,
					grow'=north,
					content=,
					circle,
					fill,
					minimum size=3pt,
					inner sep=0pt,
					s sep+=15,
				}
				[ ]
			\end{forest}
		}] for any object $a$ of $A$,
		we have $GF_{\treeDot}(a) = G_{\treeDot}\bigl(F_{\treeDot}(a)\bigr)$;
		
		\item[\scalebox{0.3}{
			\begin{forest}
				for tree={%
					label/.option=content,
					grow'=north,
					content=,
					circle,
					fill,
					minimum size=3pt,
					inner sep=0pt,
					s sep+=15,
				}
				[ [] ]
			\end{forest}
		}] for any $1$-cell $f \colon a \to a'$
		of $A$, we have $GF_{\treeL}(f) = G_{\treeL}\bigl(F_{\treeL}(f)\bigr)$;
		
		\item[\scalebox{0.3}{
			\begin{forest}
				for tree={%
					label/.option=content,
					grow'=north,
					content=,
					circle,
					fill,
					minimum size=3pt,
					inner sep=0pt,
					s sep+=15,
				}
				[ 
				[][]
				]
			\end{forest}
		}] for any pair of $0$\hyp{}composable $1$\hyp{}cells
			\[
			 \begin{tikzcd}
			  a \ar[r, "f"] & a' \ar[r, "g"] &  a''
			 \end{tikzcd}
			\]
			of $A$, we have
			$GF_{\treeV}(g, f) = G_{\treeV}\bigl(F_{\treeL}(g), F_{\treeL}(f)\bigr)
			 \comp_1 G_{\treeLL}\bigl(F_{\treeV}(g, f)\bigr)$;

		\item[\scalebox{0.3}{
			\begin{forest}
				for tree={%
					label/.option=content,
					grow'=north,
					content=,
					circle,
					fill,
					minimum size=3pt,
					inner sep=0pt,
					s sep+=15,
				}
				[ 
				[[]]
				]
			\end{forest}
		}] for any $2$-cell $\alpha \colon f \to g$ of $A$,
		we have $GF_{\treeLL}(\alpha) = G_{\treeLL}\bigl(F_{\treeLL}(\alpha)\bigr)$;
		
		\item[\scalebox{0.3}{
			\begin{forest}
				for tree={%
					label/.option=content,
					grow'=north,
					content=,
					circle,
					fill,
					minimum size=3pt,
					inner sep=0pt,
					s sep+=15,
				}
				[ 
				[][][]
				]
			\end{forest}
		}] for any triple of $0$-composable $1$-cells
		\[
		 \begin{tikzcd}[column sep=small]
		  a \ar[r, "f"] & a' \ar[r, "g"] & a'' \ar[r, "h"] & a'''
		 \end{tikzcd}
		\]
		of $A$, we have that $GF_{\treeW}(h, g, f)$ is the $3$-cell
		\begin{gather*}
			G_{\treeV}(F_{\treeLog}h, F_{\treeLog}g) \comp_0 GF_{\treeLog}(F_{\treeLog}f) \comp_1
			G_{\treeVLeft}(F_{\treeV}(h, g), F_{\treeLog}f)\\
			\comp_2\\
			G_{\treeW}(F_{\treeLog}h, F_{\treeLog}g, F_{\treeLog}f) \comp_1
			G_{\treeLLL}(F_{\treeW}(h, g, f)) \\
			\comp_2\\
			GF_{\treeLog}(h) \comp_0 G_{\treeV}(F_{\treeLog}g, F_{\treeLog}f) \comp_1
			G_{\treeVRight}(F_{\treeLog}h, F_{\treeV}(g, f)) \comp_1
			G_{\treeLL}(F_{\treeV}(h, gf))\,.
		\end{gather*}
    \end{description}
    
    As for the remaining trees, we have:
    \begin{description}
		\item[\scalebox{0.3}{
			\begin{forest}
				for tree={%
					label/.option=content,
					grow'=north,
					content=,
					circle,
					fill,
					minimum size=3pt,
					inner sep=0pt,
					s sep+=15,
				}
				[ 
				[][[]]
				]
			\end{forest}
		}] for any whiskering
		
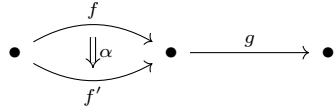
\[
		 \begin{tikzcd}[column sep=4.5em]
		  \bullet
		  \ar[r, bend left, "f", ""{below, name=f}]
		  \ar[r, bend right, "f'"', ""{name=fp}]
		  \ar[Rightarrow, from=f, to=fp, "\alpha"]
		  &
		  \bullet \ar[r,"g"] &
		  \bullet
		 \end{tikzcd}
		\]
		of $A$, we define $GF_{\treeVRight}(g, \alpha)$ to be the $3$-cell
		\begin{gather*}
		G_{\treeV}\big(F_{\treeL}(g), F_{\treeL}(f')\bigr) \comp_1
		G_{\treeLLL}\bigl(F_{\treeVRight}(g, \alpha)\bigr) \\
		\comp_2 \\
		G_{\treeVRight}\bigl(F_{\treeL}(g), F_{\treeLL}(\alpha)\bigr) \comp_1
		G_{\treeLL}\bigl(F_{\treeV}(g, f)\bigr)
		\end{gather*}
		of $C$. 

		\item[\scalebox{0.3}{
			\begin{forest}
				for tree={%
					label/.option=content,
					grow'=north,
					content=,
					circle,
					fill,
					minimum size=3pt,
					inner sep=0pt,
					s sep+=15,
				}
				[ 
				[[]][]
				]
			\end{forest}
		}] for any whiskering
		
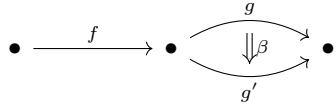
\[
            \begin{tikzcd}[column sep=4.5em]
             \bullet \ar[r, "f"] &
             \bullet
             \ar[r, bend left, "g", ""{below, name=g}]
		  \ar[r, bend right, "g'"', ""{name=gp}]&
		   \bullet
		  \ar[Rightarrow, from=g, to=gp, "\beta"]
            \end{tikzcd}
		\]
		of $A$, the $3$-cell $GF_{\treeVLeft}(\beta, f)$ is defined to be
		\begin{gather*}
         G_{\treeVLeft}\bigl(F_{\treeLL}(\beta), F_{\treeL}(f)\bigr) \comp_1
         G_{\treeLL}\bigl(F_{\treeV}(g, f)\bigr) \\
         \comp_2 \\
         G_{\treeV}\big(F_{\treeL}(g'), F_{\treeL}(f)\bigr) \comp_1
         G_{\treeLLL}\bigl(F_{\treeVLeft}(\beta, f)\bigr)\,.
		\end{gather*}

		\item[\scalebox{0.3}{
			\begin{forest}
				for tree={%
					label/.option=content,
					grow'=north,
					content=,
					circle,
					fill,
					minimum size=3pt,
					inner sep=0pt,
					s sep+=15,
				}
				[ 
				[[[]]]
				]
			\end{forest}
		}] for any $3$-cell $\gamma \colon \alpha \to \alpha'$ of $A$,
		we set
		\[
		 GF_{\treeLLL}(\gamma) = G_{\treeLLL}\bigl(F_{\treeLLL}(\gamma)\bigr)\,.
		\]
	\end{description}
\end{paragr}

\begin{paragr}
	We have to show that the definition we have given
	for the composition for the trees $\treeVLeft$, $\treeVRight$
	and $\treeLLL$ agree with the data encoded by the composition
	of the associated simplicial morphisms.
	
	\begin{description}
		\item[\scalebox{0.3}{
			\begin{forest}
				for tree={%
					label/.option=content,
					grow'=north,
					content=,
					circle,
					fill,
					minimum size=3pt,
					inner sep=0pt,
					s sep+=15,
				}
				[ 
				[][[]]
				]
			\end{forest}
		}] given a whiskering
		
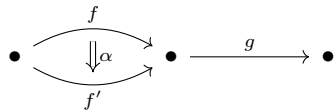
\[
		\begin{tikzcd}[column sep=4.5em]
		\bullet
		\ar[r, bend left, "f", ""{below, name=f}]
		\ar[r, bend right, "f'"', ""{name=fp}]
		\ar[Rightarrow, from=f, to=fp, "\alpha"]
		&
		\bullet \ar[r,"g"] &
		\bullet
		\end{tikzcd}
		\]
		of $A$, consider the $3$-simplex~$x$
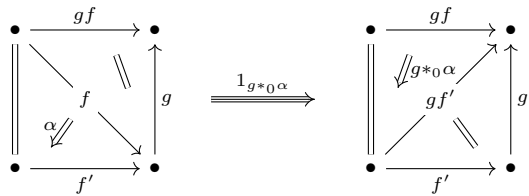
	\begin{center}
		\begin{tikzpicture}[scale=1.3, font=\footnotesize]
		\squares{%
			/squares/label/.cd,
			0=$\bullet$, 1=$\bullet$, 2=$\bullet$, 3=$\bullet$,
			01={}, 12=$f'$, 23=$g$, 
			02=$f$, 03=$gf$, 13=$gf'$,
			012=$\alpha$, 023={}, 123={}, 013=$g\comp_0 \alpha$,
			0123=$1_{g\comp_0 \alpha}$,
			/squares/arrowstyle/.cd,
			01={equal}, 023={equal}, 123={equal}
		}
		\end{tikzpicture}
	\end{center}
	of $\SN(A)$. This is sent by $\Nl(G)\Nl(F)$ to the
	$3$-simplex
	\begin{center}
		\begin{tikzpicture}[scale=1.8, font=\footnotesize]
		\squares{%
			/squares/label/.cd,
			0=$\bullet$, 1=$\bullet$, 2=$\bullet$, 3=$\bullet$,
			01={}, 12=${GF_{\treeL}(f')}$, 23=${GF_{\treeL}(g)}$, 
			02=${GF_{\treeL}(f)}$, 03=${GF_{\treeL}(gf)}$, 13=${GF_{\treeL}(gf')}$,
			012=${GF_{\treeLL}(\alpha)}$, 023=${GF_{\treeV}(g, f)\phantom{O}}$,
			123=${GF_{\treeV}(g, f')}$, 013=${\phantom{OO}GF_{\treeLL}(g\comp_0 \alpha)}$,
			0123=${\Nl(G)\bigl(F_{\treeVRight}(g, \alpha)\bigr)}$,
			/squares/arrowstyle/.cd,
			01={equal}, 012={phantom, description},
			023={phantom, description}, 123={phantom, description},
			013={phantom, description},
			/squares/labelstyle/.cd,
			012={anchor=center},
			023={anchor=center}, 123={anchor=center},
			013={anchor=center}
		}
		\end{tikzpicture}
	\end{center}
	of $\SN(C)$, where we denoted by $\Nl(G)\bigl(F_{\treeVRight}(g, \alpha)\bigr)$
	the image under $Nl(G)$ of the main $3$-cell $F_{\treeVRight}(g, \alpha)$
	of the $3$-simplex $\Nl(F)(x)$.
	By paragraph~\ref{paragr:image_3-simplex}, we find that
	the $3$-cell $\SN(G)\bigl(F_{\treeVRight}(g, \alpha)\bigr)$
	of $C$ is precisely
	\begin{gather*}
	G_{\treeV}\big(F_{\treeL}(g), F_{\treeL}(f')\bigr) \comp_1
	G_{\treeLLL}\bigl(F_{\treeVRight}(g, \alpha)\bigr) \\
	\comp_2 \\
	G_{\treeVRight}\bigl(F_{\treeL}(g), F_{\treeLL}(\alpha)\bigr) \comp_1
	G_{\treeLL}\bigl(F_{\treeV}(g, f)\bigr)\,;
	\end{gather*}

	\item[\scalebox{0.3}{
		\begin{forest}
			for tree={%
				label/.option=content,
				grow'=north,
				content=,
				circle,
				fill,
				minimum size=3pt,
				inner sep=0pt,
				s sep+=15,
			}
			[ 
			[[]][]
			]
		\end{forest}
	}] given a whiskering
	\[
	\begin{tikzcd}[column sep=4.5em]
	\bullet \ar[r, "f"] &
	\bullet
	\ar[r, bend left, "g", ""{below, name=g}]
	\ar[r, bend right, "g'"', ""{name=gp}]&
	\bullet
	\ar[Rightarrow, from=g, to=gp, "\beta"]
	\end{tikzcd}
	\]
	of $A$, an argument dual with respect to the previous point
	gives us that this is indeed $\SN(G)\bigl(F_{\treeVLeft}(\beta, f)\bigr)$.
	
	\item[\scalebox{0.3}{
		\begin{forest}
			for tree={%
				label/.option=content,
				grow'=north,
				content=,
				circle,
				fill,
				minimum size=3pt,
				inner sep=0pt,
				s sep+=15,
			}
			[ 
			[[[]]]
			]
		\end{forest}
	}] they trivially agree by definition.
\end{description}
\end{paragr}

Summing up the results of this chapter we get the following theorem.

\begin{thm}\label{thm:iso_oplax}
	The class of small $3$-categories and normalised oplax $3$-functors
	are organised in a category $\widetilde{\nCat{3}}$, which
	is isomorphic via the functor $\Nl$ to the category $\nCat{3}_{\cDelta}$
	of $3$-categories and simplicial oplax $3$-morphisms.
\end{thm}

\section{Strictification}\label{sec:tilde}

In this section we are going to explicitly describe the \oo-category
$\Sc(A)$, where $A$ is a $1$-category without split-monos and split-epis.

\begin{paragr}
	We say that a $1$-category $A$ is \ndef{split-free}
	\index{split-free category} if it does
	not have any split-monos or split-epis.
\end{paragr}

\begin{exem}
	Any poset is a split-free category. Moreover, for any category $A$,
	the category $c\,\Sd N(A)$ is split-free.
\end{exem}

\emph{We fix a side-free category $A$.}

\begin{paragr}
	We now define a reflexive $\infty$-graph $\tilde{A}$
	associated to the $1$-category $A$.
	The objects of $\tilde A$ are precisely the objects of $A$.
	For any pair of objects $(a, a')$ of $A$, we then define
	a reflexive $\infty$-graph $\tA(a, a')$ whose objects, \ie
	the $1$-cells of $\tA$ having $a$ as source and $a'$ as target,
	are given by the set of non-degenerate simplices $x \colon \Deltan{n} \to A$
	of $N_1(A)$ such that $x_0 = a$ and $x_n = a'$, for $n \ge 0$;
	that is to say, the objects of $\tA(a, a')$ are the tuples
	$(f_1, \dots, f_n)$ of composable non-trivial arrows of $A$
	such that $s(f_1) = a$ and $t(f_n) = a'$, with $n \ge 0$.
	The $0$-tuple, where necessarily $a = a'$,
	corresponds to the non-degenerate simplex
	$\Deltan{0} \to A$ pointing at $a$, and it is by definition the
	identity $1$-cell of the object $a$ of $\tA$.
	
	Consider two objects $x$ and $y$ of $\tA(a, a')$, \ie two tuples
	$x = (f_1, \dots, f_m)$ and $y = (g_1, \dots, g_n)$ as described above.
	We define $\tA(x, y)$ as follows:
	\begin{description}
		\item[$m=0$] if $x \colon \Deltan{0} \to A$ is a $0$-simplex of $A$, we set
		$\tA(x, y)$ to be the final \emph{\oo-category} $\On{0}$;
		
		\item[$m=1$] we define
		\[
		\tA\bigl( (f), (g_1, \dots, g_p)\bigr)\,,
		\]
		where $s(f) = s(g_1)$ and $t(f) = t(g_p)$, to be the \emph{\oo-category}
		\[
		\tA\bigl( (f), (g_1, \dots, g_p)\bigr) =
		\On{\omega}\bigl(\atom{0, p},\atom{0, 1} + \atom{1, 2} + \dots + \atom{p-1, p}\bigr)\,.
		\]
		\item[$m>1$] otherwise, we set $\tA(x, y)$ to be the \emph{\oo-category}
		\[
		\coprod \tA\bigl((f_1), (g_1, \dots, g_{\phi(1)})\bigr) \times
		\dots \times \tA\bigl((f_m), (g_{\phi(m-1)+1}, \dots, g_n)\bigr)\,,
		\]
		where the sum runs over all the arrows $\phi \colon \Deltan{m} \to \Deltan{n}$
		of $\cDelta$ which are:
		\begin{enumerate}
			\item\label{item:cells-tilde-i} strictly increasing;
			\item we have $\phi(0) = 0$ and $\phi(m) = n$;
			\item\label{item:cells-tilde-ii} such that, for all $1 < i \leq m$,
			we have
			\[
			g_{\phi(i)} \comp_0 \dots \comp_0 g_{\phi(i-1)+1} = f_i
			\]
			in the category $A$.
		\end{enumerate}
		These conditions ensure that $s(f_i) = s(g_{\phi(i-1)+1})$
		and $t(f_i) = t(g_{\phi(i)})$, for all $1 \le i \le m$, so that
		in  particular we have $x_i = y_{\phi(i)}$ for every $0 \le i \le m$;
		notice that the condition imposing that $\phi$ is an active morphism,
		\ie $\phi(0) = 0$ and $\phi(m) = n$ is actually implied by the others.
		We shall sometimes write the above sum as
		\[
		 \coprod_{\phi} \tA_{\phi}(x, y)\,.
		\]
		Using the canonical isomorphism of \oo-categories described in Proposition~A.4
		of~\cite{AraMaltsiCondE}, we shall often identify the \oo-category $\tA_{\phi}(x, y)$
		with
		\[
		\On{\omega}\bigl(\atom{0, \phi(1)}+ \dots + \atom{\phi(m-1), \phi(m)},
		\atom{0, 1} + \dots + \atom{n-1, n}\bigr)\,.
		\]
	\end{description}
	
	Note that if the index of the sum above is empty, that is there is no
	arrow $\phi \colon \Deltan{m} \to \Deltan{n}$ satisfying conditions~\ref{item:cells-tilde-i} and~\ref{item:cells-tilde-ii}, then
	$\tA(x, y)$ is set to be the empty \oo-category. This
	happens in particular every time $m > n$.
	Observe also that condition~\ref{item:cells-tilde-ii} entails
	that if there is a cell between $(f_1, \dots, f_m)$ and
	$(g_1, \dots, g_n)$, then necessarily
	\[
	f_m \comp_0 \dots \comp_0 f_1 = g_n \comp_0 \dots \comp_0 g_1\,.
	\]

	For any $1$-cell $x = (f_1, \dots, f_n)$ of $\tA$, the identity of $x$
	is given by the only trivial $2$-cell of the \oo-category
	\[
	\tA(f_1, f_1) \times \dots \times \tA(f_n, f_n)\,.
	\]
	Indeed, observe that for any arrow $f$ of $A$, the only morphism satisfying~\ref{item:cells-tilde-i}
	and~\ref{item:cells-tilde-ii} is $phi = 1_{\Deltan{n}}$, so that
	$\tA(f, f)$ is isomorphic to $\On{\omega}(\atom{0, 1}, \atom{0, 1})$,
	which is the terminal \oo-category.
\end{paragr}

\begin{rem}
	Without the hypothesis on the category $A$, that is in the general situation
	in which we have split-monos and split-epis, the definition
	of the hom-\oo-category $\tA\bigl((f), (g_1, \dots, g_n)\bigr)$ is more complicated.
	This is due to the fact that, although the simplex $(g_1, \dots, g_n)$ is non-degenerate,
	there could be two consecutive arrows, say $g_i$ and $g_{i+1}$ which compose to the identity.
	When introducing the operations on the \oo-graph $\tA$, this becomes a serious issue.
\end{rem}

\begin{paragr}
	In this paragraph we want to endow the reflexive \oo-graph
	$\tA(a, a')$ with the structure of an \oo-category,
	for any pair $(a, a')$ of objects of $\tA$.
	In order to do so, for any $x = (f_1, \dots, f_\ell)$,
	$y = (g_1, \dots, g_m)$ and $z = (h_1, \dots, h_n)$ of
	$\tA(a, a')$, we want to define an \oo-functor
	\[
	\tA(y, z) \times \tA(x, y) \to \tA(x, z)\,.
	\]
	Without any loss of generality, we can suppose $\ell \le m \le n$
	(see the preceding paragraph) and consider the case $\ell >0$,
	since the other cases are trivial.
	Let us fix two morphisms $\phi \colon \Deltan{\ell} \to \Deltan{m}$
	and $\psi \colon \Deltan{m} \to \Deltan{n}$
	satisfying conditions~\ref{item:cells-tilde-i}
	and~\ref{item:cells-tilde-ii} of the previous paragraph.
	We set
	\[
	 \Phi(i) = \phi(i)- \phi(i-1)
	 \quadet
	 \Psi(j) = \psi(j) - \psi(j-1)
	\]
	for any $1 \le i \le \ell$ and $1 \le j \le m$.
	We have to give an \oo-functor which has
	\begin{gather}\label{eq:mapping-oo-category}
	\tA\bigl((g_1), (h_1, \dots, h_{\psi(1)}\bigr) \times \dots \times
	\tA\bigl((g_m), (h_{\psi(m-1)+1}, \dots, h_n)\bigr) \notag
	\\ \times \phantom{OOOO}\\
	\tA\bigl((f_1), (g_1, \dots, g_{\phi(1)}\bigr) \times \dots \times
	\tA\bigl((f_\ell), (g_{\phi(\ell-1)+1}, \dots, g_m)\bigr) \notag
	\end{gather}
	as source, which by definition is the \oo-category
	\begin{gather*}
	\prod_{i=1}^{m} \On{\omega}\bigl(\atom{0, \Psi(i)}, \atom{0, 1} + \dots + \atom{\Psi(i)-1, \Psi(i)}\bigr)
	\\ \times \\
	\prod_{i=1}^{\ell} \On{\omega}\bigl(\atom{0, \Phi(i)}, \atom{0, 1} + \dots + \atom{\Phi(i)-1, \Phi(i)}\bigr)\,.
	\end{gather*}
	Notice that for every $1 \le p \le \ell$ we have
	\[
	f_p = g_{\phi(p)} \comp_0 g_{\phi(p) -1} \comp_0 \dots
	\comp_0 g_{\phi(p-1)+1}\,,
	\]
	and for every $1 \le q \le m$ we have
	\[
	g_q = h_{\psi(q)} \comp_0 h_{\psi(q) -1} \comp_0 \dots
	\comp_0 h_{\psi(q-1)+1}\,,
	\]
	so that in fact 
	\begin{equation*}
	\begin{split}
	f_p = & \phantom{\comp_0}\ h_{\psi(\phi(p))} \comp_0 h_{\psi(\phi(p)) -1} \comp_0 \dots
	\comp_0 h_{\psi(\phi(p)-1)+1} \\
	&\comp_0 h_{\psi(\phi(p)-1)} \comp_0 h_{\psi(\phi(p)-1)-1} \comp_0 \dots
	\comp_0 h_{\psi(\phi(p)-2)+1} \\
	&\comp_0 \dots \\
	&\comp_0 h_{\psi(\phi(p-1)+1)} \comp_0 h_{\psi(\phi(p-1)+1)-1} \comp_0 \dots
	\comp_0 h_{\psi(\phi(p-1))+1}\,,
	\end{split}
	\end{equation*}
	for every $1 \le p \le \ell$.
	Now, for every $1 \le i \le m$, we have that the \oo-category
	\[
	\On{\omega}\bigl(\atom{0, \Psi(i)}, \atom{0, 1} + \dots + \atom{\Psi(i)-1, \Psi(i)}\bigr)
	\]
	is canonically isomorphic by Corollary~\ref{coro:suboriental} to the \oo-category
	\[
	\On{\omega}\bigl(\atom{\psi(i-1), \psi(i)}, \atom{\psi(i-1), \psi(i-1)+ 1} + \dots + \atom{\psi(i) -1, \psi(i)}\bigr)\,.
	\]
	In order to simplify the notations, let us set
	\[
	b_i = \atom{\psi(i-1), \psi(i)} \quad , \quad
	c_i = \sum_{k= 0}^{\Psi(i)-1} \atom{\psi(i-1) + k, \psi(i-1) + k + 1}\,,
	\]
	for $1 \le i \le m$, and also
	\[
	b = b_1 + b_2 + \dots + b_m
	\quadet
	c = c_1 + c_2 + \dots + c_m\,.
	\]
	There is a canonical \oo-functor
	\[
	\prod_{i=1}^m \On{\omega}(b_i, c_i) \to \On{\omega}(b, c)
	\]
	given by ``horizontal composition'' $\comp_0$, \ie mapping
	a tuple $(x_1, \dots, x_m)$ of $p$-cells to the $p$-cell
	$x_1 \comp_0 x_2 \comp_0 \dots \comp_0 x_m$ of $\On{\omega}(b, c)$.
	Proposition~\ref{prop:2-cells_orientals} actually shows that
	this \oo-functor is an isomorphism of \oo-categories.
	The same argument entails that the \oo-category
	\[
	\prod_{i=1}^{\ell} \On{\omega}\bigl(\atom{0, \phi(i)}, \atom{0, 1} + \dots + \atom{\phi(i)-1, \phi(i)}\bigr)
	\]
	is canonically isomorphic to $\On{\omega}(a', b')$
	via the ``horizontal composition''~$\comp_0$, where we have set
	\[
	a' = \sum_{i=1}^{\ell} \atom{\phi(i-1), \phi(i)}
	\quadet
	b' = \sum_{i=1}^{m} \atom{i-1, i}\,.
	\]
	Applying the increasing morphism $\psi$ and setting
	\[
	 a = \sum_{i=1}^{\ell} \atom{\psi\phi(i-1), \psi\phi(i)}
	\]
	we get, again by Corollary~\ref{coro:suboriental}, a canonical isomorphism of \oo-categories $\On{\omega}(a', b') \cong \On{\omega}(a, b)$.
	The \oo-cat\-e\-gory in~\eqref{eq:mapping-oo-category} is thus
	canonically isomorphic to the \oo-cat\-e\-go\-ry
	\[
	\On{\omega}(b, c) \times \On{\omega}(a, b)\,.
	\]
	On the other hand, the target \oo-category of the \oo-functor
	we are set to construct is 
	\[
	\tA(x, z) = \prod_{i=1}^\ell \tA\bigl((f_i), (h_{\psi\phi(i-1)+1}, \dots, h_{\psi\phi(i)} \bigr)\,,
	\]
	which is by definition
	\[
	\prod_{i=1}^\ell \On{\omega}\bigl(\atom{0, \psi\phi(i)}, \atom{0, 1} + \dots + \atom{\psi\phi(i)-1, \psi\phi(i)}\bigr).
	\]
	The same argument used above gives us that this \oo-category
	is canonically isomorphic to the \oo-category
	\[
	\On{\omega}(a, c)\,.
	\]
	We then define the \oo-functor
	\[
	\On{\omega}(b, c) \times \On{\omega}(a, b) \to \On{\omega}(a, c)
	\]
	to be the ``vertical composition'' $\comp_1$, \ie a pair
	of $p$-cells $(x, y)$ of the source is mapped to the $p$-cell
	$x \comp_1 y$ of $\On{\omega}(a, c)$.
	
	Alternatively, for any $1 \le i \le \ell$ we can consider the \oo-category
	\[
	 \tA\bigl((f_i), (g_{\phi(i-1)+1}, \dots, g_{\phi(i)-1}, g_{\phi(i)})\bigr)\,,
	\]
	which is defined as
	\[
	 \On{\omega}\bigl(\atom{0, \Phi(i)}, \atom{0, 1} + \dots + \atom{\Phi(i)-1, \Phi(i)}\bigr)\,.
	\]
	The latter is canonical isomorphic by Corollary~\ref{coro:suboriental} to
	\[
	 \On{\omega}\bigr(\atom{\phi(i-1), \phi(i)}, \atom{\phi(i-1), \phi(i-1)+1} + \dots + \atom{\phi(i)-1, \phi(i)}\bigr)\,,
	\]
	which in turn is isomorphic to the \oo-category
	\[
	 \On{\omega}\bigl(\atom{\psi\phi(i-1), \psi\phi(i)},
	 \atom{\psi(\phi(i-1)), \psi(\phi(i-1)+1)} + \dots + \atom{\psi(\phi(i)-1), \psi(\phi(i))}\bigr)\,.
	\]
	If we set $a_i = \atom{\psi\phi(i-1), \psi\phi(i)}$ for every $1 \le i \le \ell$,
	then we can write the above \oo-category as
	\[
	 \On{\omega}(a_i, b_{\phi(i-1)+1} + b_{\phi(i-1)+2}\dots + b_{\phi(i)})\,.
	\]
	Similarly, for any $1 \le i \le m$ we have seen above that the \oo-category
	\[
	 \tA\bigl((g_i), (h_{\psi(i-1)}, \dots,h_{\psi(i)})\bigr)
	\]
	is canonically isomorphic to the \oo-category
	\[
	 \On{\omega}\bigl(\atom{\psi(i-1), \psi(i)}, \atom{\psi(i-1), \psi(i-1)+ 1} + \dots + \atom{\psi(i) -1, \psi(i)}\bigr)\,,
	\]
	which we can denote by $\On{\omega}(b_i, c_i)$. Therefore, for a fixed
	$1 \le p \le \ell$ we have that the \oo-category
	 \begin{gather*}
	 \prod_{k=\phi(p-1)+1}^{\phi(p)} \tA\bigl((g_k), (h_{\psi(k-1)+1}, \dots, h_{\psi(k)})\bigr)
	 \\ \times \\
	 \tA\bigl((f_p), (g_{\phi(p-1)+1}, \dots, g_{\phi(p)})\bigr)
	 \end{gather*}
	is canonically isomorphic to the \oo-category
	\[
	 \prod_{k=\phi(p-1)+1}^{\phi(p)} \On{\omega}(b_k, c_k)
	 \ \times\ \On{\omega}(a_p, b_{\phi(p-1)+1} + \dots + b_{\phi(p)})\,.
	\]
	Using Proposition~\ref{prop:2-cells_orientals}, we obtain
	\[
	 \prod_{k=\phi(p-1)+1}^{\phi(p)} \On{\omega}(b_k, c_k)
	  \cong \On{\omega}(b_{\phi(p-1)+1} + \dots + b_{\phi(p)}, c_{\phi(p-1)+1} + \dots + c_{\phi(p)})
	\]
	and hence the former \oo-category is canonically isomorphic to
	\begin{gather*}
	 \On{\omega}(b_{\phi(p-1)+1} + \dots + b_{\phi(p)}, c_{\phi(p-1)+1} + \dots + c_{\phi(p)})
	 \\ \times \\
	 \On{\omega}(a_p, b_{\phi(p-1)+1} + \dots + b_{\phi(p)})\,.
	\end{gather*}
	We set
	\[
	 B_p = b_{\phi(p-1)+1} + \dots + b_{\phi(p)} \quadet
	 C_p  = c_{\phi(p-1)+1} + \dots + c_{\phi(p)}
	\]
	Applying the ``vertical composition'' $\comp_1$ to this product of \oo-categories
	we get an \oo-functor
	\[
	 \On{\omega}(B_p, C_p) \times \On{\omega}(a_p, B_p) \to \On{\omega}(a_p, C_p)\,.
	\]
	Finally, the ``horizontal composition'' $\comp_0$ provides us with an \oo-functor
	\[
	 \prod_{p=1}^\ell \On{\omega}(a_p, C_p) \to \On{\omega}(a, c)\,,
	\]
	since $a = a_1 + \dots + a_\ell$ and $c = C_1 + \dots + C_\ell$.
	
	These two approaches are equivalent by virtue of the exchange law between
	$\comp_0$ and~$\comp_1$.
	
	This endows the reflexive \oo-graph $\tA(a, a')$ with the structure of
	an \oo-category.
\end{paragr}

\begin{paragr}\label{paragr:tA_oo-category}
	In this paragraph we put an \oo-category structure on
	the reflexive \oo-graph~$\tA$. In order to do this,
	we shall define, for any objects $a$, $a'$ and $a''$
	of $\tA$, an \oo-functor
	\[
	\tA(a', a'') \times \tA(a, a') \to \tA(a, a'')\,.
	\]
	
	As \oo-categories are categories enriched in \oo-categories,
	an \oo-functor $F$ between two \oo-categories $C$ and $D$
	can be given by a map $F_0 \colon C_0 \to D_0$ on objects and
	a family of \oo-functors $C(c, c') \to D(Fc, Fc')$, indexed
	by the pairs of objects $(c, c')$ of $C$, satisfying the axioms
	described in paragraph~\ref{def_enriched}.
	
	In light of the above, we have to provide a map
	\[
	\tA(a', a'')_0 \times \tA(a, a')_0 \to \tA(a, a'')_0\,,
	\]
	that we define by sending a pair $(y, x)$ with
	$x \colon \Deltan{m} \to A$ and $y \colon \Deltan{n} \to A$
	to the concatenation simplex
	\[
	y \cdot x \colon \Deltan{m+n} \to A \quad , \quad
	\atom{i, i+1} \mapsto
	\begin{cases}
	x_{\{i, i+1\}}\,, & \text{if $i < m$,}\\
	y_{\{i-m, i+1-m\}}\,, & \text{if $i \ge m$.}
	\end{cases}
	\]
	Furthermore, for any choice of objects $(y, x)$ and $(t, z)$
	of $\tA(a', a'')_0 \times \tA(a, a')_0$, we have to provide
	an \oo-functor
	\[
	\tA(y, t) \times \tA(x, z) \to \tA(y\cdot x, t \cdot z)\,.
	\]
	Notice that if either $\tA(x, z)$ or $\tA(y, t)$ are empty, then the same holds
	for $\tA(y\cdot x, t \cdot z)$.
	If $x$ (resp.~$y$) is a $0$-simplex, then so is $z$ (resp.~$t$) and
	the \oo-functor above is simply the identity on $\tA(y, t)$ (the identity on $\tA(x, z)$).
	We can therefore suppose that~$\tA(x, z)$ and $\tA(y, t)$ are non-empty
	and that $x$ and $y$ are not trivial.
	
	Following the reasoning of the previous paragraph,
	we know that there are integers
	\[
	0 = i_0 < i_1 < \dots < i_m \quadet 0 = j_0 < j_1 < \dots < j_n
	\]
	such that, if we set
	\begin{align*}
	a = \sum_{k=0}^{m-1} \atom{i_k, i_{k+1}}\,, &&
	c = \sum_{p=0}^{i_m -1} \atom{p, p+1}\,, \\
	b = \sum_{k=0}^{n-1} \atom{j_k, j_{k+1}}\,, &&
	d = \sum_{p=0}^{j_n -1} \atom{p, p+1}\,,
	\end{align*}
	then we have canonical isomorphisms
	\[
	\tA(x, z) \cong \On{\omega}(a, c) \quadet \tA(y, t) \cong \On{\omega}(b, d)
	\]
	of \oo-categories. Moreover, setting
	\[
	b' = \sum_{k=0}^{n-1} \atom{i_m + j_k, i_m + j_{k+1}}
	\quadet
	d' = \sum_{p=0}^{j_n -1} \atom{i_m + p, i_m + p+1}\,,
	\]
	we have by Corollary~\ref{coro:suboriental} a canonical isomorphism
	\[
	\On{\omega}(b, d) \cong \On{\omega}(b', d')
	\]
	and by the same argument we can build a further canonical isomorphism
	\[
	\tA(y\cdot x, t \cdot z) \cong \On{\omega}(a + b', c + d')
	\]
	of \oo-categories. We are thus left to provide an \oo-functor
	\[
	\On{\omega}(b', d') \times \On{\omega}(a, c) \to \On{\omega}(a + b', c + d')\,,
	\]
	which we set to be the ``horizontal composition'' by $\comp_0$.
	Notice that by Proposition~\ref{prop:2-cells_orientals}, this \oo-functor
	is in fact an isomorphism.
	
	The identity axioms are trivial from the definition and the associativity
	follows immediately from the associativity of the ``horizontal composition''~$\comp_0$
	as an operation of the \oo-category~$\On{\omega}$.
\end{paragr}

\begin{lemme}
	Let $a$ and $a'$ be two objects of $A$ and consider two elements
	\[x = (f_1, \dots, f_m) \quadtext{and}\quad y = (g_1, \dots, g_n)\]
	of~$\tA(a, a')$.
	Then there is a zig-zag of $2$-cells linking $x$ to $y$ if and only if
	\[
	f_m \comp_0 \dots \comp_0 f_1 = g_n \comp_0 \dots \comp_0 g_1
	\]
	in $A$.
\end{lemme}

\begin{proof}
	This is trivially true if $x$, and then also $y$, is a trivial cell of $A$.
	So let us suppose $m>0$ and $n>0$.
	
	On the one hand, condition~\ref{item:cells-tilde-ii} immediately implies that
	two $1$-cells $x$ and $y$ as above are connected by a zig-zag of $2$-cells only if
	\[
	f_m \comp_0 \dots \comp_0 f_1 = g_n \comp_0 \dots \comp_0 g_1\,.
	\]
	On the other hand, let
	\[
	h = f_m \comp_0 \dots \comp_0 f_1 = g_n \comp_0 \dots \comp_0 g_1
	\]
	and consider the $1$-cell $z = (h)$ of $\tA$.
	It results immediately from the	structure of the oriental $\On{\omega}$ that the \oo-categories
	\[
	\tA\bigl((h), x\bigr) = \On{\omega}\bigl(\atom{0, m}, \atom{0, 1} + \dots + \atom{m-1, m}\bigr)
	\]
	and
	\[
	\tA\bigl((h), y\bigr) = \On{\omega}\bigl(\atom{0, n}, \atom{0, 1} + \dots + \atom{n-1, n}\bigr)
	\]
	are non-empty; hence $x$ and $y$ are connected by a zig-zag of length two.
\end{proof}

\begin{coro}
	We have $\ti{1}(\tA) \cong A$.
\end{coro}

\begin{proof}
	We have a canonical \oo-functor $\eps_A \colon \tA \to A$
	which is the identity on objects and that maps a $1$-cell
	$x = (f_1, \dots, f_n)$ to $f_n \comp_0 \dots \comp_0 f_1$
	if $n>0$ and a $0$-simplex $a \colon \Deltan{0} \to A$ to the
	identity of $a$ in $A$. The identity is clearly preserved,
	the functoriality follows by the definition of $0$-composition
	of $1$-cells of $\tA$ by concatenation and moreover the assignment
	is well-defined by the previous lemma. We are left with showing that
	for any pair of objects $(a, a')$ of $A$, the map
	\[
	\ti{0}(A)(a, a') \to A(a, a')
	\]
	is a bijection. It is clearly surjective, since for any morphism
	$f \colon a \to a'$ of $A$ we have $\eps_A\bigl((f)) = f$
	(and similarly if $f$ is an identity cell of $A$). It results from
	the previous lemma that this map is also injective, hence
	completing the proof of the corollary.
\end{proof}

\begin{paragr}
	We now turn to constructing a normalised oplax $3$-functor
	\[\eta_A \colon A \to \ti{3}(\tA)\,.\]
	\begin{description}
		\item[$\TreeDot\ \ $] The map $(\eta_A)_{\treeDot}$ is defined to be
		the identity map on objects.
		
		\item[$\TreeLog\ \ \:$] The map $(\eta_A)_{\treeLog}$ assigns to any non-trivial morphism $f \colon a \to a'$
		of $A$ the $1$-cell $(f) \colon \Deltan{1} \to A$ of $\tA$ and to any identity $1_a$ of $A$
		the \emph{trivial} $1$-cell $a \colon \Deltan{0} \to A$ of~$\tA$.
		
		\item[$\TreeV\ $] The map $(\eta_A)_{\treeV}$ assigns to any
		pair of composable morphisms
		\[
		\begin{tikzcd}
		a \ar[r, "f"] & a' \ar[r, "g"] &  a''
		\end{tikzcd}
		\]
		of $A$ the unique $2$-cell $(\eta_A)_{\treeV}(g, f)$ of $\tA$ with source $(g \comp_0 f)$ and target $(f, g)$,
		\ie the unique element $\atom{0, 1, 2}$ of the set
		\[
		 \tA\bigl((g\comp_0 f), (g, f)\bigr) = \On{\omega}\bigl(\atom{0, 2}, \atom{0, 1} + \atom{1, 2})\,.
		\]
		
		\item[$\TreeW\ $] The map $(\eta_A)_{\treeW}$,
		assigns to any triple of composable morphisms
		\[
		\begin{tikzcd}[column sep=small]
		a \ar[r, "f"] & a' \ar[r, "g"] & a'' \ar[r, "h"] & a'''
		\end{tikzcd}
		\]
		of $A$ the unique $3$-cell $(\eta_A)_{\treeW}(h, g, f)$ of $\tA$
		with $1$-source $(h\comp_0 g \comp_0 f)$ and $1$-target $(h, g, f)$,
		\ie the unique arrow $\atom{0, 1, 2, 3}$ of the $1$-category
		\[
		\tA\bigl((h\comp_0 g \comp_0 f), (f, g, h)\bigr)\On{\omega}\bigl(\atom{0, 3}, \atom{0, 1} + \atom{1, 2} + \atom{2, 3}\bigr)\,.
		\]
	\end{description}
	
	Notice that by definition we have that $1_{(\eta_A)_{\treeDot}(a)}$, that is the
	$1$-cell $a \colon \Deltan{0} \to A$, is precisely $(\eta_A)_{\treeL}(1_a)$;
	the other conditions of normalisation are trivial. We are left with checking the
	coherence for the tree $\treeVV$.
	
	Consider four composable morphisms of~$A$
	\[
	\begin{tikzcd}[column sep=small]
	\bullet \ar[r, "f"] &
	\bullet \ar[r, "g"] &
	\bullet \ar[r, "h"] &
	\bullet \ar[r, "i"] &
	\bullet
	\end{tikzcd}\ .
	\]
	We have to show that the $3$-cells
	\begin{gather*}
	(\eta_A)_{\treeW}(i, h, g) \comp_0 (\eta_A)_{\treeLog}(f) \comp_1 (\eta_A)_{\treeV}(i\comp_0 h \comp_0 g, f)\\
	\comp_2\\
	(\eta_A)_{\treeLog}(i) \comp_0 (\eta_A)_{\treeV}(h, g) \comp_0 (\eta_A)_{\treeLog}(f) \comp_1
	(\eta_A)_{\treeW}(i, h \comp_0 g, f)\\
	\comp_2 \\
	(\eta_A)_{\treeLog}(i) \comp_0 (\eta_A)_{\treeW}(h, g, f) \comp_1 (\eta_A)_{\treeV}(i, h\comp_0 g \comp_0 f)
	\end{gather*}
	and
	\begin{gather*}
	(\eta_A)_{\treeV}(i, h) \comp_0 (\eta_A)_{\treeLog}(g) \comp_0 (\eta_A)_{\treeLog}(f)
	\comp_1 (\eta_A)_{\treeW}(ih, g, f) \\
	\comp_2\\
	(\eta_A)_{\treeLog}(i) \comp_0 (\eta_A)_{\treeLog}(h) \comp_0 (\eta_A)_{\treeV}(g, f)
	\comp_1 (\eta_A)_{\treeW}(i, h, gf)
	\end{gather*}
	of $\ti{3}(A)$ are equal, which is equivalent to exhibiting a zig-zag of $4$-cells connecting them.
	In fact, they are precisely the target and the source of the unique $2$-cell
	of the $2$-category
	\[\tA\bigl((i\comp_0 h \comp_0 g \comp_0 f), (f, g, h, i)\bigr)\,,\]
	\ie the cell $\atom{0, 1, 2, 3, 4}$ of the $2$-category
	\[
	\On{\omega}\bigl(\atom{0, 4}, \atom{0, 1} + \atom{1, 2} + \atom{2, 3} + \atom{3, 4})\,.
	\]
\end{paragr}

\begin{paragr}
	The construction of the preceding paragraph is in fact the cellular version
	of the  truncation of a	simplicial morphism $\SN(A) \to \SN(\tA)$, which we shall still
	denote by $\eta_A$. We shall dedicate the
	rest of the chapter to define such a map and moreover show
	that it is the unit map of the adjoint pair $(c_\infty, \SN)$ applied to the simplicial set $\SN(A)$, so
	that in particular $\tA \cong \Sc\SN(A)$.
\end{paragr}

\begin{paragr}
	For any object $\atom{i}$ of $\On{m}$, with $0\le i \le m$, we set $\tilde y(\atom{i}) = y(i)$.
	For any $0 < i \le m$, we denote by $f_i$ the arrow $y_{\{i-1, i\}}$ of $A$ and
	consider the $1$-cell 
	\[
	a = \atom{i_0, i_1} + \atom{i_1, i_2} + \dots + \atom{i_{k-1}, i_k}
	\] of $\On{m}$,
	with $0 \le i_0 < i_1 < \dots < i_k \le m$,
	that we can see as a strictly increasing morphism
	$a \colon \Deltan{k} \to \Deltan{m}$.
	To this $1$-cell,
	it is canonically associated the $k$-simplex $z \colon \Deltan{k} \to A$
	of $\SN(A)$ defined by
	\[
	z_{\{p, p+1\}} = f_{i_{p+1}} \comp_0 \dots \comp_0 f_{i_p + 1}\,,
	\]
	that is to say $z = y a$.
	This is a non-degenerate $k$-simplex of $A$ and thus defines
	a $1$-cell of $\tA$. Hence, we set $\tilde y(a) = z = ya$.
\end{paragr}

\begin{paragr}
	Consider two $1$-cells $a$ and $b$ of $\On{m}$
	that we can write as
	two non-degenerate, that is strictly increasing, simplices
	\[
	a \colon \Deltan{p} \to \Deltan{m} \quadet b \colon \Deltan{q} \to \Deltan{m}\,,
	\]
	and suppose they are such that $a(0) = b(0)$ and $a(p) = b(q)$.
	More explicitly, the $1$-cells $a$ and $b$ of $\On{m}$
	correspond respectively to the $1$-cells
	\[
	a = \atom{a_0, a_1} + \dots + \atom{a_{p-1}, a_p}
	\]
	and
	\[
	b = \atom{b_0, b_1} + \dots + \atom{b_{q-1}, b_{q})}
	\]
	such that $a_0 = b_0$ and $a_p = b_{q}$, where we have set
	$a_i = a(i)$, for $0 \le i \le p$ and $b_j = b(j)$ for $0\le j \le q$.
	It results from	Lemma~10.4 of~\cite{AraMaltsiCondE} that there is a $2$-cell
	from $a$ to $b$ if and only if there exists a strictly increasing morphism
	$\phi \colon \Deltan{p} \to \Deltan{q}$ of $\cDelta$ such that
	$a = b \phi$. Notice that if such a morphism $\phi$ exists,
	than it is unique, as $b$ is a monomorphism.
	We suppose that this is the case and we define an \oo-functor
	\[
	\tilde x_{a, b} \colon \On{m}(a, b) \longrightarrow \tA\bigl(\tilde y(a), \tilde y(b)\bigr)\,.
	\]
	The source of this \oo-functor is the \oo-category
	\[
	\On{m}(a, b) = \On{m}\bigl(\atom{a_0, a_1} + \dots + \atom{a_{p-1}, a_p}, \atom{b_0, b_1} + \dots + \atom{b_{q-1}, b_{q})}\bigr)\,,
	\]
	that by virtue of Proposition~\ref{prop:2-cells_orientals} is
	canonically isomorphic to
	\[
	\prod_{i= 1}^{p} \On{m}\bigl(\atom{a_{i-1}, a_{i}},
	\atom{b_{\phi(i-1)+1}, b_{\phi(i-1)+2}} + \dots + \atom{b_{\phi(i) -1}, b_{\phi(i)}}\bigr)\,,
	\]
	while the target \oo-category
	$\tA\bigl(\tilde y(a), \tilde y(b)\bigr) = \tA(ya, yb)$
	is a sum of \oo-categories $\tA_\psi(ya, yb)$
	indexed on strictly increasing
	morphisms $\psi \colon \Deltan{p} \to \Deltan{q}$ of $\cDelta$ such that
	$(ya)_{\{i-1, i\}} = (yb)_{\{\phi(i-1), \dots , \phi(i)\}}$, that is to say verifying
	$ya = yb\psi$.
	Any strictly increasing morphism $\phi \colon \Deltan{p} \to \Deltan{q}$
	such that $a = b \phi$ trivially verifies $ya = yb\phi$ and we observed above
	that there is at most one such morphism. Therefore,
	if such a morphism $\phi$ exists, than it appears as index in the sum of
	\oo-categories defining $\tA(ya, yb)$ and we have
	\[
	\tA_\phi(ya, yb) \cong
	\On{\omega}\bigl(\atom{0, \phi(1)} + \dots + \atom{\phi(p-1), \phi(p)},
	\atom{0, 1} + \dots + \atom{q-1, q}\bigr)\,.
	\]
	Now, this \oo-category is equal to the \oo-category
	\[
	\On{q}\bigl(\atom{0, \phi(1)}, + \dots + \atom{\phi(p-1), \phi(p)},
	\atom{0, 1} + \dots + \atom{q-1, q}\bigr)
	\]
	and the injective morphism $b \colon \Deltan{q} \to \Deltan{m}$
	induces by Corollary~\ref{coro:suboriental} a canonical isomorphism between the latter \oo-category and
	$\On{m}(a, b)$. We set the \oo-functor
	\[
	 \tilde y_{a, b} \colon \On{m}(a, b) \to \tA(ya, yb)
	\]
	to be the composition $\On{m}(a, b) \to \tA_\phi(ya, yb)$ of the isomorphisms
	we have just described
	followed by the embedding $\tA_\phi(ya, yb) \to \tA(ya, yb)$.
		
	We have to check that for any triple $(a, b, c)$ of composable $1$-cells
	of $\On{m}$ we have a commutative diagram
	\begin{equation}\label{dia:functoriality_tilde}
	\begin{tikzcd}
	\On{m}(b, c) \times \On{m}(a, b) \ar[r, "\comp_1"] \ar[d, "\tilde y_{b, c} \times \tilde y_{a, b}"'] &
	\On{m}(a, c) \ar[d, "\tilde y_{a, c}"] \\
	\tA\bigl(\tilde y(b), \tilde y(c)\bigr) \times \tA\bigl(\tilde y(a), \tilde y(b)\bigr)
	\ar[r, "\comp_1"] &
	\tA\bigl(\tilde y(a), \tilde y(c)\bigr)
	\end{tikzcd}
	\end{equation}
	of \oo-categories. Suppose that we have $a \colon \Deltan{p} \to \Deltan{m}$,
	$b \colon \Deltan{q} \to \Deltan{m}$ and $c \colon \Deltan{r} \to \Deltan{m}$,
	with $1 \le p \le q \le r \le m$, and that
	$\phi \colon \Deltan{p} \to \Deltan{q}$ and $\psi \colon \Deltan{q} \to \Deltan{r}$
	are the unique morphisms of $\cDelta$ such that $a = \phi b$ and $b = \psi c$.
	Thus we get
	\[
	\tA_\phi\bigl(\tilde y(a), \tilde y(b)\bigr) =
	\On{\omega}\bigl(\atom{0, \phi(1)} + \dots + \atom{\phi(p-1), \phi(p)},
	\atom{0, 1} + \dots + \atom{q-1, q}\bigr)
	\]
	and
	\[
	\tA_\psi\bigl(\tilde y(b), \tilde y(c)\bigr) =
	\On{\omega}\bigl(\atom{0, \psi(1)} + \dots + \atom{\psi(q-1), \psi(q)},
	\atom{0, 1} + \dots + \atom{r-1, r}\bigr)
	\]
	and
	\[
	\tA_{\psi\phi}\bigl(\tilde y(a), \tilde y(c)\bigr) =
	\On{\omega}\bigl(\atom{0, \psi\phi(1)} + \dots + \atom{\psi\phi(p-1), \psi\phi(p)},
	\atom{0, 1} + \dots + \atom{r-1, r}\bigr)\,.
	\]
	We set
	\[
	a' = \sum_{i=1}^p\atom{\psi\phi(i-1), \psi\phi(i)}\quadet
	b' =  \sum_{i=1}^q \atom{\psi(i-1), \psi(i)}\,.
	\]
	Remember that the \oo-functor
	$tA_\psi\bigl(\tilde y(b), \tilde y(c)\bigr) \times \tA_\phi\bigl(\tilde x(a), \tilde x(b)\bigr) \to	\tA_{\psi\phi}\bigl(\tilde y(a), \tilde y(c)\bigr)$
	is defined by making use
	of the canonical isomorphism between $\tA_\phi\bigl(\tilde y(a), \tilde y(b)\bigr)$
	and the \oo-category $\On{\omega}(a', b')$. We thus have canonical isomorphisms
	\[
	 \tA_\phi\bigl(\tilde y(a), \tilde y(b)\bigr) \cong \On{r}(a', b')\quad , \quad
	 \tA_\psi\bigl(\tilde y(b), \tilde y(c)\bigr) \cong \On{r}(b', \atom{0, 1} + \dots + \atom{r-1, r})
	\]
	and
	\[
	 \tA_{\psi\phi}\bigl(\tilde y(a), \tilde y(c)\bigr) \cong \On{r}(a', \atom{0, 1} + \dots + \atom{r-1, r})\,.
	\]
	Moreover, the morphism $c \colon \Deltan{r} \to \Deltan{m}$ induces by Corollary~\ref{coro:suboriental} canonical
	isomorphisms
	\[
	 \On{r}(a', b') \cong \On{m}(a, b) \quad , \quad
	 \On{r}(b', \atom{0, 1} + \dots + \atom{r-1, r}) \cong \On{m}(b, c)
	\]
	and
	\[
	 \On{r}(a', \atom{0, 1} + \dots + \atom{r-1, r}) \cong \On{m}(a, c)\,.
	\]
	Under this isomorphisms, we claim that the square
	\[
	 \begin{tikzcd}
	  \On{r}(b', c') \times \On{r}(a', b') \ar[r, "\comp_1"] \ar[d] &
	  \On{r}(a', c') \ar[d] \\
	  \On{m}(b, c) \times \On{m}(a, b) \ar[r, "\comp_1"] & \On{m}(a, c)
	 \end{tikzcd}
	\]
	of \oo-categories is commutative, where we have set $c' = \atom{0, 1} + \dots + \atom{r-1, r}$.
	Indeed, consider two $k$-cells $\alpha$ in $\On{r}(a', b')$ and $\beta$ in $\On{r}(b', c')$ and
	suppose that they are $1$-composable, $k>0$. We can express $\alpha$ and $\beta$ as homogeneous
	elements of $\cC(\Deltan{r})_{k+1}$ and thus as sums of atoms, say
	\[
	 \alpha = \sum_{i=0}^s \alpha_i
	 \quadet
	 \beta = \sum_{i=0}^t \beta_i\,.
	\]
	The operation $\comp_1$ at this level is simply the sum $\alpha + \beta$
	and the morphism $\cC(\Deltan{r})_{k+1} \to \cC{\Deltan{m}}_{k+1}$ induced
	by $c \colon \Deltan{r} \to \Deltan{m}$ sends an atom $\atom{j_0, \dots, j_{k+1}}$
	of the source to the atom $\atom{c(j_0), \dots, c(j_{k+1})}$ of the target.
	The morphism $\cC(c) \colon \cC(\Deltan{r}) \to \cC{\Deltan{m}}$
	respects sums, since it is a morphism of augmented directed complexes
	and therefore the square above commutes.
	Observe that (up to canonical isomorphisms of the factors in the line below)
	picking the inverses to the vertical isomorphisms of \oo-categories of the
	commutative square above gives the following commutative square
	\[
	 \begin{tikzcd}
	 \On{m}(b, c) \times \On{m}(a, b) \ar[r, "\comp_1"] \ar[d, "\tilde y_{b, c} \times \tilde y_{a, b}"'] &
	 \On{m}(a, c) \ar[d, "\tilde y_{a, c}"] \\
	 \tA_\psi\bigl(\tilde y(b), \tilde y(c)\bigr) \times \tA_\phi\bigl(\tilde y(a), \tilde y(b)\bigr)
	 \ar[r, "\comp_1"] &
	 \tA_{\psi\phi}\bigl(\tilde y(a), \tilde y(c)\bigr)
	 \end{tikzcd}
	\]
	of \oo-categories. Since the square of embeddings
	\[
	 \begin{tikzcd}
	  \tA_\psi\bigl(\tilde y(b), \tilde y(c)\bigr) \times \tA_\phi\bigl(\tilde y(a), \tilde y(b)\bigr)
	  \ar[r, "\comp_1"] \ar[d]&
	  \tA_{\psi\phi}\bigl(\tilde y(a), \tilde y(c)\bigr) \ar[d]\\
	  \tA_\psi\bigl(\tilde y(b), \tilde y(c)\bigr) \times \tA_\phi\bigl(\tilde y(a), \tilde y(b)\bigr)
	  \ar[r, "\comp_1"] &
	  \tA_{\psi\phi}\bigl(\tilde y(a), \tilde y(c)\bigr)
	 \end{tikzcd}
	\]
	is obviously commutative, we obtain the commutativity of the square depicted
	in~\eqref{dia:functoriality_tilde}. Hence, we have checked that the assignment
	\[\tilde y \colon \On{m}(\atom{i}, \atom{j}) \to \tA(y(i), y(j))\]
	defines an \oo-functor
	for any object $\atom{i}$ and $\atom{j}$ of $\On{m}$.
	
	In order to conclude that $\tilde y \colon \On{m} \to \tA$ is an \oo-functor,
	it remains to show that for any $0 \le i < j < k \le m$ the square
	\[
	 \begin{tikzcd}
	 \On{m}(\atom{j}, \atom{k}) \times \On{m}(\atom{i}, \atom{j}) \ar[r, "\comp_0"] \ar[d, "\tilde y_{j, k} \times \tilde y_{i, j}"'] &
	 \On{m}(\atom{i}, \atom{k}) \ar[d, "\tilde y_{i, k}"] \\
	 \tA\bigl(y(j), y(k)\bigr) \times \tA\bigl(y(i), y(j)\bigr)
	 \ar[r, "\comp_0"] &
	 \tA\bigl(y(i), y(k)\bigr)
	 \end{tikzcd}
	\]
	of \oo-categories is commutative. The proof uses the same strategy adopted
	in paragraph~\ref{paragr:tA_oo-category}. On the objects, that is for any choice
	of composable $1$-cells $a$ and $b$ of $\On{m}$, with $s(a) = i$, $t(a) = s(b) = j$
	and $t(b) = k$, the commutativity of the above diagram is equivalent to the
	equality $yb \cdot ya = y(b\cdot a)$, which is clearly verified.
	Moreover, for any $(a, b)$ and $(c, d)$ in $\On{m}(\atom{i}, \atom{j}) \times \On{m}(\atom{j}, \atom{k})$
	such that the image of $a$ is contained in the image of $c$ and the image of $b$
	is contained in the image of $d$ (the other cases being trivial),
	one easily checks the commutativity of the square
	\[
	 \begin{tikzcd}
	 \On{m}(b, d) \times \On{m}(a, c) \ar[r, "\comp_0"] \ar[d, "\tilde y_{b, c} \times \tilde y_{a, c}"'] &
	 \On{m}(b\cdot a, d \cdot c) \ar[d, "\tilde y_{b\cdot a, d \cdot c}"] \\
	 \tA\bigl(yb, yd\bigr) \times \tA\bigl(ya, yc\bigr)
	 \ar[r, "\comp_0"] &
	 \tA\bigl(y(b\cdot a), y(d\cdot c)\bigr)
	 \end{tikzcd}
	\]
	by reducing to the atoms, as we did for the square~\eqref{dia:functoriality_tilde}.
\end{paragr}

\begin{paragr}
	In this paragraph we show that the assignment sending a functor $x \colon \Deltan{n} \to A$
	to the \oo-functor $\tilde x \colon \On{n} \to \tA$ defines a morphism of simplicial sets
	$\SN(A) \to \SN(\tA)$.
	Let $f \colon \Deltan{p} \to \Deltan{q}$ be a morphism in $\cDelta$ and $x \colon \Deltan{q} \to A$
	a functor. We have to show that the equality $\tilde x \On{f} = \widetilde{xf}$ holds true.
	Consider the Eilenberg--Zilber decompositions $(\pi, x')$ of $x$ and $(\rho, y)$ of $xf$,
	where $x' \colon \Deltan{q'} \to A$ and $y \colon \Deltan{p'} \to A$.
	We can depict the situation with the following diagram
	\[
	 \begin{tikzcd}
	  \Deltan{p} \ar[rr, "f"] \ar[d, "\rho"'] && \Deltan{q} \ar[d, "\pi"] \\
	  \Deltan{p'} \ar[dr, "y"'] && \Deltan{q'} \ar[dl, "x'"] \\
	   & A & 
	 \end{tikzcd}\ .
	\]
	The morphism $\pi f$ of $\cDelta$ admits a decomposition
	of a degeneracy $e \colon \Deltan{p} \to \Deltan{\ell}$ followed by a face
	$g \colon \Deltan{\ell} \to \Deltan{q'}$. Now, the composition $x'g \colon \Deltan{\ell} \to A$
	is a non-degenerate element of $\SN(A)$ and so we must have $\Deltan{\ell} = \Deltan{p'}$,
	$e = \rho$ and $x'g = y$, by the uniqueness of the Eilenberg--Zilber decomposition.
	We have to show that $\tilde y = \widetilde{x'}\On{g}$, so that the following triangle
	\[
	 \begin{tikzcd}
	  \On{p'} \ar[rr, "\On{g}"] \ar[rd, "\tilde y"'] &&
	  \On{q'} \ar[ld, "\widetilde{x'}"] \\
	  & \tA &
	 \end{tikzcd}
	\]
	of \oo-functors is commutative. It trivially commutes at the level of objects.
	For any injective map $a \colon \Deltan{n} \to \Deltan{p'}$, we clearly have
	$ya = x' g a$, so that by definition $\tilde y(a) = \widetilde{x'}\On{g}(a)$
	and therefore the triangle is commutative on $1$-cells.
	Let $a$ and $b$ be two parallel $1$-cells of $\On{p'}$,
	say $a \colon \Deltan{m} \to \Deltan{p'}$ and $b \colon \Deltan{n} \to \Deltan{p'}$.
	Observe that $\On{p'}(a, b)$ is empty if and only if $\On{q'}(ma, mb)$
	is empty if and only if $\tA(\tilde y(a), \tilde y(b))$ is so.
	Otherwise, there exist a unique monomorphism
	$\phi \colon \Deltan{m} \to \Deltan{n}$ of $\cDelta$, an integer $r \ge 0$,
	an unique injective morphism $h \colon \Deltan{r} \to \Deltan{p'}$ of $\cDelta$
	and $1$-cells $a'$ and $b'$
	of $\On{r}$ such that $b \phi = a$, $ha' = a$, $hb' = b$ and
	\[
	 b' = \atom{0, 1} + \atom{1, 2} + \dots + \atom{r-1, r}\,. 
	\]
    By definition,
	$\tA(\tilde y(a), \tilde y(b)) = \On{r}(a', b')$ and we
	have a commutative square of isomorphism
	\[
	 \begin{tikzcd}
	  \On{p'}(a, b) \ar[rr, "(\On{g})_{a, b}"] &&
	  \On{q'}(ma, mb) \\
	  & \On{r}(a', b') \ar[ul, "(\On{h})_{a', b'}"] \ar[ur, "(\On{hg})_{a', b'}"'] &
	 \end{tikzcd}
	\]
	of \oo-categories by Corollary~\ref{coro:suboriental} and this immediately implies that the triangle
	\[
	 \begin{tikzcd}
	  \On{p'}(a, b) \ar[rr, "\On{g}"] \ar[rd, "\tilde y_{a, b}"'] &&
	  \On{q'}(ma, mb) \ar[ld, "\widetilde{x'}_{ma, mb}"] \\
	  & \tA_\phi(\tilde y(a), \tilde y(b)) &
	 \end{tikzcd}
	\]
	of \oo-categories commutes, as $\tilde y_{a, b}$ is defined as
	the inverse of $(\On{h})_{a', b'}$ and $\widetilde{x'}_{ma, mb}$
	as the inverse of $(\On{hg})_{a', b'}$. This concludes the proof,
	showing that the assignment $\eta_A \colon \SN(A) \to \SN(\tA)$ is indeed
	a morphism of simplicial sets.
\end{paragr}

\begin{paragr}
    We now want to show that the morphism $\eta_A \colon \SN(A) \to \SN(\tA)$
    is the counit of the adjoint pair $(\Sc, \SN)$ for the object
    $\SN(A)$. This is equivalent to say that the precomposition by
    $\eta_A$ induces a bijection
    \[
     \Hom_{\ooCat}(\tA, B) \cong \Hom_{\EnsSimp}(\SN(A), \SN(B))
    \]
    of sets for any \oo-category $B$.  In turn, this bijection means
    that for any morphism of simplicial sets $F \colon \SN(A) \to \SN(B)$
    there exists a unique \oo-functor $\bar{F} \colon \tA \to B$ such that
    the triangle
    \begin{equation}\label{dia:lifting_unit}
     \begin{tikzcd}
      \SN(A) \ar[r, "F"] \ar[d, "\eta_A"] & \SN(B) \\
      \SN(\tA) \ar[ru, "\SN(\bar{F})"']
     \end{tikzcd}
    \end{equation}
    is commutative. This would show in particular that $\tA \cong \Sc\SN(A)$.
    We shall first prove the uniqueness and then the existence of such
    an \oo-functor $\bar{F}$.
\end{paragr}

\begin{paragr}[Uniqueness]
    Suppose a functor $G \colon \tA \to B$ such that
    $F = \SN(G) \eta_A$ exists. Object-wise, the
    functor $G$ must coincide with $F_0$.
    A $1$-cell of $\tA$ is a tuple $a = (f_1, \dots, f_n)$
    of non-trivial composable arrows of $A$. Let $f$ be the composite
    $f_n \comp_0 \dots \comp_0 f_1$ of the arrows which are components
    of $a$; we can view $f$ as a non-degenerate $1$-simplex
    $f \colon \Deltan{1} \to \SN(A)$. Then $G(a)$ must be equal
    to $F(f)$. For observe that $\SN(\eps_A) \eta_A$ is the
    identity on $\SN(A)$ and $\eps_A(a) = f$.
    
    Let $a$ and $b$ be two $1$-cells of $\tA$, say
    $a \colon \Deltan{m} \to A$ and $b \colon \Deltan{n} \to A$.
    We can suppose that there is an injective morphism
    $\phi \colon \Deltan{m} \to \Deltan{n}$ of $\cDelta$ such that
    $b \phi = a$, otherwise $\tA(a, b)$ is empty; we fix such a morphism $\phi$.
    By definition,
    \[
     \tA_\phi(a, b) = \On{n}\bigl(\atom{\phi(0), \phi(1)} + \dots + \atom{\phi(m-1), \phi(m)},
     \atom{0, 1} + \dots + \atom{n-1, n} \bigr)
    \]
    and we have an \oo-functor
    \[
     F(b) \colon \On{n} \to B\,.
    \]
    Hence, the \oo-functor $\tA_\phi(a, b) \to B(Ga, Gb)$ is the
    composition of the following \oo-functors
    \[
     \begin{tikzcd}
      \tA_\phi(a, b) \ar[r, "\cong"] &
      \On{n}(a', b') \ar[r, "G(b)"] &
      B(Ga, Gb)
     \end{tikzcd}\ ,
    \]
    where we have set
    \[
     a' = \atom{\phi(0), \phi(1)} + \dots + \atom{\phi(m-1), \phi(m)}
     \quadet
     b' = \atom{0, 1} + \dots + \atom{n-1, n}\,.
    \]
    Varying $\phi$ this gives a unique \oo-functor $\tA(a, b) \to B(Ga, Gb)$,
    thus proving the uniqueness of $G$.
\end{paragr}

\begin{paragr}[Existence]
    The previous paragraph already shows how the functor
    $\bar F \colon \tA \to B$ must be define, if it exists.
    It remains to check that this assignment is indeed an 
    \oo-functor.
    
    Let $x$ and $y$ be objects of $\tA$ and
    $a \colon \Deltan{\ell} \to A$, $b \colon \Deltan{m} \to A$
    and $c \colon \Deltan{n} \to A$ be $1$-cells of $\tA(x, y)$.
    Without loss of generality, we can suppose that there are
    injective morphisms $\phi \colon \Deltan{\ell} \to \Deltan{m}$
    and $\psi \colon \Deltan{m} \to \Deltan{n}$ such that
    $c\psi = b$ and $b \phi = a$. We set
    \begin{align*}
     a' &= \atom{\psi\phi(0), \psi\phi(1)} + \dots + \atom{\psi\phi(\ell-1), \psi\phi(\ell)}\,,
     \\
     b' &= \atom{\psi(0), \psi(1)} + \dots + \atom{\psi(m-1), \psi(m)}\,, \\
     c' &= \atom{0, 1} + \dots + \atom{n-1, n}\,.
    \end{align*}
    We have a diagram
    \[
     \begin{tikzcd}
      \tA_\psi(b, c) \times \tA_\phi(a, b) \ar[r, "\comp_1"] \ar[d, "\cong"'] &
      \tA_{\psi\phi}(a, c) \ar[d, "\cong"] \\
      \On{n}(b', c') \times \On{n}(a', b') \ar[d, "F(c)"'] \ar[r, "\comp_1"] &
      \On{n}(a', c') \ar[d, "F(c)"] \\
      B(\bar{F}b, \bar{F}c) \times B(\bar{F}a, \bar{F}b) \ar[r, "\comp_1"] & B(\bar{F}a, \bar{F}c)
     \end{tikzcd}\ ,
    \]
    where the upper square commutes by definition and the
    lower square commutes by the \oo-functoriality of $F(c) \colon \On{n} \to B$.
    Making the morphisms $\phi$ and $\psi$ varying among the index defining
    the sum $\tA(a, b)$ and $\tA(a, b)$, we get a commutative square
    \[
     \begin{tikzcd}
      \tA(b, c) \times \tA(a, b) \ar[r, "\comp_1"]
      \ar[d, "\bar{F}_{b, c}\times \bar{F}_{a, b}"'] &
      \tA_{\psi\phi}(a, c) \ar[d, "\bar{F}_{a, c}"] \\
      B(\bar{F}b, \bar{F}c) \times B(\bar{F}a, \bar{F}b) \ar[r, "\comp_1"] & B(\bar{F}a, \bar{F}c)
     \end{tikzcd}
    \]
    of \oo-categories.
    
    Let $x$, $y$ and $z$ be three objects of $\tA$. We have to show
	that the square of \oo-categories
	\[
	 \begin{tikzcd}
	  \tA(y, z) \times \tA(x, y) \ar[r, "\comp_0"]
      \ar[d, "\bar{F}_{y, z}\times \bar{F}_{x, y}"'] &
      \tA_{\psi\phi}(x, z) \ar[d, "\bar{F}_{x, z}"] \\
      B(\bar{F}y, \bar{F}z) \times B(\bar{F}x, \bar{F}y) \ar[r, "\comp_0"] &
      B(\bar{F}x, \bar{F}z)
	 \end{tikzcd}
	\]
	is commutative. As for the objects, that is the $1$-cells of $\tA$ and $B$,
	it is clear: indeed, for any $a \colon x \to y$ and $b \colon y \to z$
	of $\tA$, we have on the one hand that $b \comp_0 a$ is just the
	concatenation of the simplices $a$ and $b$ of $A$, while on the
	other hand $\bar{F}(c)$	applied to a $1$-cell $c = (f_1, \dots, f_n)$
	of $\tA$ gives image under $F$ of the composition
	$F(f_n \comp_0 \dots \comp_0 f_1)$ of its components.
	Therefore we have to check that, for any choice
	$(b, a)$ and $(d, c)$ of elements of $\tA(y, z)\times \tA(x, y)$,
	the square
    \[
     \begin{tikzcd}
      \tA(b, d) \times \tA(a, c) \ar[r, "\comp_0"]
      \ar[d, "\bar{F}_{b, d}\times \bar{F}_{a, c}"'] &
      \tA_{\psi\phi}(b\cdot a, d \cdot c) \ar[d, "\bar{F}_{b\cdot a, d\cdot c}"] \\
      B(\bar{F}b, \bar{F}d) \times B(\bar{F}a, \bar{F}c) \ar[r, "\comp_0"] &
      B(\bar{F}(b\cdot a), \bar{F}(d\cdot c))
     \end{tikzcd}
    \]
    of \oo-categories is commutative. This is completely analogous to
    the case of the ``vertical composition'' $\comp_1$ that we showed
    above: we reduce to components $\tA_\psi(b, d)$ and $\tA_\phi(a, c)$,
    for which there is a diagram
    \[
     \begin{tikzcd}
      \tA_\psi(b, d) \times \tA_\phi(a, c) \ar[r, "\comp_0"] \ar[d, "\cong"'] &
      \tA_{\psi\cdot \phi}(a, c) \ar[d, "\cong"] \\
      \On{n}(b', d') \times \On{n}(a', c') \ar[d, "F(d\cdot c)"'] \ar[r, "\comp_0"] &
      \On{n}(b'\cdot a', d'\cdot c') \ar[d, "F(d\cdot c)"] \\
      B(\bar{F}b, \bar{F}d) \times B(\bar{F}a, \bar{F}c) \ar[r, "\comp_0"] &
      B(\bar{F}(b\cdot a), \bar{F}(d\cdot c))
     \end{tikzcd}
    \]
    of \oo-categories in which the upper square commutes by definition
    and the lower square by \oo-functoriality of
    $F(d\cdot c) \colon \On{n} \to B$, and finally we conclude by
    varying among all the morphisms $\phi$ and $\psi$ indexing
    the coproducts $\tA(a, c)$ and $\tA(b, d)$.
    This achieves the proof of the existence of the \oo-functor
    $\bar{F} \colon \tA \to B$ and so this establishes
    the lifting problem depicted in~\eqref{dia:lifting_unit}.
    Equivalently, the precomposition by $\eta_A \colon \SN(A) \to \SN(\tA)$
    gives a bijection
    \[
     \Hom_{\ooCat}(\tA, B) \cong \Hom_{\EnsSimp}(\SN(A), \SN(B))\,,
    \]
    from which we deduce the isomorphism $\tA = \Sc\SN(A)$.
\end{paragr}

\begin{thm}
	Let $A$ be a split-free category. Then the \oo-category $\tA$
	defined in paragraph~\ref{paragr:tA_oo-category} is isomorphic
	to the \oo-category $c_\infty\SN(A)$.
\end{thm}

\begin{exem}
	Let $C$ be a $3$-category and consider the normalised oplax
	$3$-functor $\sup \colon i_{\cDelta}(N_3(C)) \to C$ defined
	in example~\ref{exem:sup}. For any $1$-category $A$, the
	category $c\Sd N(A)$ is split-free and 
	moreover it is shown in Theorem~32 of~\cite{delHoyo} that the
	canonical morphism $c\Sd N(A) \to A$ is a Thomason equivalence.
	Hence, we get a diagram
	\[
	 c\,\Sd N\bigl(i_{\cDelta}(N_3(C))\bigr) \to i_{\cDelta}(N_3(C)) \to C
	\]
	whose composition if still a normalised oplax $3$-functor by
	Theorem~\ref{thm:iso_oplax}. Now, the category
	$C' = c\,\Sd N\bigl(i_{\cDelta}(N_3(C))\bigr)$ is
	split-free and therefore we get a span
	\[
	 C' \leftarrow \ti{3}\widetilde{C'} \to C
	\]
	of $3$-functors.
	We conjecture that both the $3$-functors
	above are Thomason equivalences. Since
	we observed in Example~\ref{exem:sup_we}
	that the morphism $\Nl(\sup)$ of simplicial sets
	is a simplicial weak equivalence and we observed above
	that the functor $C' \to i_{\cDelta}(N_3(C))$ is a Thomason equivalence,
	by a 2-out-of-3 argument
	one of these $3$-functor is a Thomason equivalence if and only
	if the other is so.
	
	This is a partial generalisation to $3$-categories of the
	approach used by Chiche in~\cite{chiche_homotopy} to show
	that the minimal fundamental localiser of $\nCat{2}$ is given
	by the class of Thomason equivalences, thus showing that
	$2$-categories intrinsically model homotopy types.
	In order to generalise this result to higher category, one would need
	to prove that both the $3$-functors of the above span are aspherical,
	\ie they satisfy the $3$\nbd-cat\-egorical generalisation of Quillen's Theorem~A.
	The author does not even know if this is true for $2$-categories.
	In fact, Chiche avoids this problem introducing a notion of asphericity for oplax $2$-functors,
	that seems out of reach for higher dimension.
	We can nonetheless say something interesting
	about the homotopy theory of normalised $3$-functors,
	as pointed out in the following remark.
\end{exem}

\begin{rem}
	The nerve functor $\Nl \colon \lnCat \to \EnsSimp$
	allows us to define a class of weak equivalences on
	$\lnCat$, that we call \emph{Thomason equivalences}.
	More precisely, a normalised oplax $3$-functor is
	a Thomason equivalence if and only if its image via
	$\Nl$ is a weak homotopy equivalence. Since
	the triangle
	\[
	 \begin{tikzcd}[column sep=small]
	 	& \EnsSimp & \\
	 	\Cat \ar[ur, "N"] \ar[rr, hookrightarrow]
	 	&& \lnCat \ar[ul, "\Nl"']
	 \end{tikzcd}
	\]
	commutes, a classical result of Illusie--Quillen
	tells us that the composite functor
	\[
	 \begin{tikzcd}
	  \EnsSimp \ar[r, "i_\cDelta"] &
	  \lnCat \ar[r, "\Nl"] & \EnsSimp
	 \end{tikzcd}
	\]
	is weakly homotopy equivalent to the identity on
	simplicial sets.
	Moreover, Example~\ref{exem:sup_we} gives us that
	a normalised oplax $3$-functor $u \colon A \to B$
	is a Thomason equivalence if and only if the
	functor
	$i_\cDelta(\Nl(u)) \colon i_\cDelta(\Nl(A)) \to i_\cDelta(\Nl(B))$ is so. Hence, the composite
	functor
	\[
	\begin{tikzcd}
	\lnCat \ar[r, "\Nl"] & \EnsSimp
		\ar[r, "i_\cDelta"] & \lnCat
	\end{tikzcd}
	\]
	is homotopic to the identity functor on $\lnCat$.
	We conclude that the nerve functor $\Nl \colon \lnCat \to \EnsSimp$ induces an equivalence at the level of
	the underlying homotopy categories.
\end{rem}

\appendix
\section{Strict higher categories}
\label{app:higher_cats}

\begin{paragr}\label{def_enriched}
	Let \V~be a category. A \emph{\Vn-graph} $X$ is the data of a set $X_0$
	of \emph{objects} and, for any $x, y$ in $X_0$, an object~$X(x, y)$ of~\V.
	A \emph{morphism of \Vn-graphs} $f \colon X \to Y$ is given by a function
	$f_0 \colon X_0 \to Y_0$ between the objects as well as morphisms
	\[
	f_{x, y} \colon X(x, y) \to Y(fx, fy)
	\]
	of \V~for any $x, y$ in $X_0$.
	We denote by \Vn-$\pref{\G_1}$ the category of \Vn-graphs.
	
	Let $I$ be an object of \V.
	A \emph{reflexive (\V, $I$)\nbd-graph} $X$, or simply \emph{reflexive \V-graph}
	if the object $I$ is clear from the context, is a \V-graph endowed with a
	morphism $k_x \colon I \to \Hom_\mathcal{V}(x, x)$ for any element $x$ in $X_0$.
	A \emph{morphism of reflexive \Vn-graphs} $f \colon X \to Y$ is a morphism
	of \Vn-graphs such that $f_{x, x}\, k_x = k_{fx}$ for any $x$ in $X_0$.
	We denote by \Vn-$\pref{\Gr_1}$ the category of reflexive \Vn-graphs.
	
	Let $(\mathcal V, \otimes, I)$ be a monoidal category.
	A $(\mathcal V, \otimes, I)$\nbd-category $A$, or simply \Vn-category
	if the monoidal structure is clear, is a reflexive $(\mathcal V, I)$\nbd-graph
	endowed with morphisms
	\[
	\Hom_A(b, c) \otimes \Hom_A(a, b) \to \Hom_A(a, c)
	\]
	of \V, for any objects $a, b$ and $c$ in $A_0$, satisfying the associativity
	axioms
	\[
	\begin{tikzcd}
	\Hom_A(c, d) \otimes \Hom_A(b, c) \otimes \Hom_A(a, b)
	\ar[r] \ar[d] &
	\Hom_A(c, d) \otimes \Hom_A(a, c) \ar[d] \\
	\Hom_A(b, d) \otimes \Hom_A(a, b) \ar[r] &
	\Hom_A(a, d)
	\end{tikzcd}
	\]
	for any $a$, $b$, $c$ and $d$ in $A_0$, and the identity axioms
	\[
	\begin{tikzcd}[column sep=6em]
	I \otimes \Hom_A(a, b)
	\ar[r, "{k_b \otimes \Hom_A(a, b)}"] &
	\Hom_A(b, b) \otimes \Hom_A(a, b) \ar[d] \\
	\Hom(a, b) \ar[u, "\cong"] \ar[r, equal]&
	\Hom(a, b)
	\end{tikzcd}
	\]
	and 
	\[
	\begin{tikzcd}[column sep=6em]
	\Hom_A(a, b) \otimes I
	\ar[r, "{\Hom_A(a, b) \otimes k_a}"] &
	\Hom_A(a, b) \otimes \Hom_A(a, a) \ar[d] \\
	\Hom(a, b) \ar[u, "\cong"] \ar[r, equal]&
	\Hom(a, b)
	\end{tikzcd}
	\]
	for any $a$ and $b$ in $A_0$.
	A \Vn-category is also widely
	known as \Vn-enriched category or category enriched in \V.
	A \emph{morphism of \Vn-categories} $\phi \colon A \to B$, also called \Vn-enriched functor,
	is a morphism of the underlying reflexive \Vn-graphs which moreover
	commutes with compositions morphisms.
	We denote by \Vn-$\Cat$ the category of \Vn-categories.
	It is easy to see that if the category \V{} has finite products,
	then the category \Vn-$\Cat$ also has finite products.
\end{paragr}

\begin{paragr}
	For $n>1$, the category $\nCat{(n+1)}$ of $(n+1)$-categories
	can be inductively defined as the category
	of reflexive $\nCat{n}$-graphs. This provides a
	canonical functor $\nCat{n} \to \nCat{(n+1)}$
	which admits a left adjoint $\tau^n$.
	The category $\ooCat$ can be defined as the limit of the tower
	\[
	\begin{tikzcd}
	 \dots \ar[r] & \nCat{(n+1)} \ar[r, "\tau^{n}"] &
	 \nCat{n} \ar[r, "\tau^{n-1}"] &
	 \nCat{(n-1)} \ar[r, "\tau^{n-2}"] & \dots
	 \end{tikzcd}
	\]
\end{paragr}

\section{Steiner theory}

In this section we present the theory of
augmented directed complexes, introduced by Steiner in~\cite{Steiner1}.
We follow closely the exposition given by Ara and Maltsiniotis
in~\cite{AraMaltsiCondE} and~\cite{Joint}.

\begin{paragr}\label{paragr:conv_comp}
	Unless explicitly stated, in this section we shall always
	write ``chain complex'' to mean ``chain complex of abelian groups
	in non-negative degrees with homological indexing''.
	We remind that a \ndef{homogeneous element}\index{homogeneous element} of a chain complex~$K$
	is an element of a group $K_n$ for some $n \ge 0$.
	If $x$ is an homogeneous element of~$K$, we shall call
	the \ndef{degree}\index{homogeneous element!degree} of~$x$ the unique $n\ge 0$ for which
	$k$ belongs to $K_n$ and we shall denote it by~$|x|$.
\end{paragr}

\begin{paragr}
	An \ndef{augmented directed complex}\index{augmented directed complex} is a triple $(K, K^\ast, e)$
	where
	\[
	K =
	\begin{tikzcd}\cdots
	\ar[r, "d_{n+1}"] & K_n \ar[r, "d_n"] & K_{n-1}
	\ar[r, "d_{n-1}"] & \cdots \ar[r, "d_2"] & K_1 \ar[r, "d_1"] & K_0
	\end{tikzcd}
	\]
	is a chain complex, $e \colon K_0 \to \Z$ is an augmentation
	(so that we have $e\, d_1 = 0$) and $K^\ast = (K^\ast_i)_{i \ge 0}$
	is a graded set such that for any~$i \ge 0$ the set~$K^\ast_i$
	is a submonoid of the abelian group~$K_i$.
	We shall
	call \ndef{positivity submonoids}\index{augmented directed complex!positivity submonoid} of $K$ the submonoids~$K^\ast_i$,
	with $i \ge 0$.
	
	We will often denote an augmented directed complex simply
	by its underlying chain complex, especially if the augmented
	and positivity structures are clear.
\end{paragr}

\begin{rem}
	We warn the reader that we do not ask any compatibility
	of the submonoids of positivity with respect to the differentials.
\end{rem}

\begin{paragr}\label{paragr:preorder_monoid}
	Let $(K, K^\ast, e)$ be an augmented directed complex.
	For any $i \ge 0$, the submonoid $K^\ast_i$ induces a
	preorder relation $\le$ on~$K_i$,
	compatible with the abelian group structure, defined by
	\[
	x \le y \quaddefssi y - x \in K^\ast_i\ .
	\]
	In particular, we have
	\[
	K^\ast_i = \{x \in K_i : x \ge 0\}\,.
	\]
	More precisely, for an augmented complex $(K, e)$ the additional
	structure given by the collection of submonoids $K^\ast$
	is equivalent to endow, for all $i \ge 0$,
	each abelian group~$K_i$ with a
	preorder relation compatible with the abelian group structure.
\end{paragr}

\begin{paragr}
	A \ndef{morphism of augmented directed complexes} $f$ from
	$(K, K^\ast, e)$ to $(K', K'^\ast, e')$ is a morphism
	of augmented chain complexes which moreover respects
	the submonoids of positivity. That is, $f \colon K \to K'$
	is a morphism of chain complexes such that $e'\, f_0 = e$
	and for any $i \ge 0$ the image $f(K^\ast_i)$ of the submonoid
	$K^\ast_i$ under $f$ is contained in~$K^{\prime\ast}_i$.
	This latter condition can be stated equivalently by saying
	that the morphism $f_i \colon K_i \to K'_i$ preserves the preorder $\le$ on~$K_i$
	for all $i \ge 0$.
	We shall denote by~$\Cda$ the category of augmented directed complexes.
\end{paragr}

\begin{paragr}\label{paragr:def_basis}
	A \emph{basis}\index{augmented directed complex!with basis} of an augmented directed complex $K$ is a graded set
	$B = (B_i)_{i\geq 0}$ such that, for any $i\geq 0$, the set $B_i$
	is both a basis for the $\Z$-module $K_i$ and a set of generators
	for the submonoid $K^\ast_i$ of $K_i$.
	We shall often identify a basis $B = (B_i)_{i \ge 0}$ to the set
	$\coprod_{i \ge 0} B_i$
	
	If an augmented directed complex has a basis, for any $i\geq 0$ the preorder relation
	of positivity on $K_i$ defined in paragraph~\ref{paragr:preorder_monoid}
	is a partial order relation and the elements $B_i$
	of the basis are the minimal elements of the poset $(K_i^\ast\setminus \{0\}, \le)$;
	in particular, if a basis of $K$ exists then it is unique.
	When an augmented directed complex
	has a basis, we shall say
	that the complex is \emph{with basis}.
\end{paragr}


\begin{paragr}\label{paragr:def_support}
	Let $K$ be an augmented directed complex with basis $B$.
	For any $i\geq 0$, a $i$\hyp{}homogeneous element~$x$ can be written
	uniquely as a linear combination of elements of $B_i$
	\[
	x = \sum_{b \in B_i} x_b\, b
	\]
	with integral coefficients. The \emph{support}\index{support} of $x$, denoted by
	$\supp(x)$,
	is the (finite) set of elements of the basis appearing in this
	linear combination with non-zero coefficient. We can
	write the $i$\hyp{}homogeneous elements uniquely as the difference of two
	positive $i$\hyp{}homogeneous elements with disjoint supports $x = x_+ - x_-$,
	where
	\[
	x_+ = \sum_{\substack{b \in B_i \\ x_b > 0}} x_b\,b
	\qquad\text{and}\qquad
	x_- = -\sum_{\substack{b \in B_i \\ x_b < 0}} x_b\,b\,.
	\]
\end{paragr}

\begin{paragr}\label{paragr:def_atome}
	Let $K$ be an augmented directed complex with basis $B = (B_i)_{i\ge 0}$.
	For $i \ge 0$ and $x$ in $K_i$, we define a matrix
	\[
	\atom{x}=\tabll{\atom{x}}{i},
	\]
	where the elements $\atom{x}^\varepsilon_k$ are inductively defined by:
	\begin{itemize}
		\item $\atom{x}^0_i = x = \atom{x}^1_i$\,;
		\smallskip
		
		\item $\atom{x}^0_{k - 1} = d(\atom{x}^0_k)_-$ and $\atom{x}^1_{k - 1} = d(\atom{x}^1_k)_+$\,, \,for $0 < k \leq i$\,.
	\end{itemize}
	We say that the basis $B$ of $K$ is \ndef{unital}
	\index{augmented directed complex!unital basis}
	if, for any $i\geq 0$ and any $x$ in $B_i$, we have the equality
	$e(\atom{x}^0_0) = 1 = e(\atom{x}^1_0)$.
	
	We shall say that an augmented directed complex $K$ is \ndef{with unital basis}
	is it is with basis and its unique basis is unital.
\end{paragr}

\begin{paragr}\label{paragr:def_Steiner}
	Let $K$ be an augmented directed complex with basis $B$.
	For $i\geq 0$, we denote by $\leq_i$ the smallest
	preorder relation on $B = \coprod_i B_i$ satisfying
	\[
	x \leq_i y \quad\text{if}\quad |x|>i, |y|>i\text{ and }\supp{(\atom{x}^1_i)}\cap \supp{(\atom{y}^0_i)} \neq \vide.
	\]
	We say that the basis $B$ is \emph{loop-free}
	\index{augmented directed complex!loop-free basis} if, for any $i \ge 0$,
	the preorder relation $\leq_i$ is a partial order relation.
	
	We shall call \emph{Steiner complex}\index{Steiner complex}
	an augmented directed complex $K$ with unital
	and loop-free basis $B$. 
\end{paragr}

\begin{paragr}\label{paragr:def_le_N}
	Let $K$ be an augmented directed complex with basis $B = \coprod_{i \ge 0} B_i$.
	We shall denote by~$\leN$ the smallest preorder relation on $B$
	satisfying
	\[
	x \leN y \quad\text{if}\quad x \in \supp(d(y)_-) \text{ or }
	y \in \supp(d(x)_+),
	\]
	where we fixed by convention $d(b) = 0$ if $b$ belongs to $B_0$.
	We shall say that a basis $B$ is \emph{strongly loop-free}
	\index{augmented directed complex!strongly loop-free basis}
	if the preorder relation $\leN$ is actually a partial
	order relation.
	
	We shall call an augmented directed complex $K$ 
	a \emph{strong Steiner complex}\index{strong Steiner complex} if it is  with basis and
	its unique basis is unital
	and strongly loop-free. 
\end{paragr}

\begin{prop}[Steiner]
	Let $K$ be an augmented directed complex with basis $B$.
	If the basis $B$ is strongly loop-free, then it is loop-free.
\end{prop}

\begin{proof}
	See Proposition~3.7 of~\cite{Steiner1}.
\end{proof}

\begin{paragr}\label{def:nu}
	We define a functor
	\[
	\nu \colon \Cda \to \ooCat
	\]
	as follows.
	
	Let $K$ be an augmented directed complex.
	For $i \ge 0$, the $i$\hyp{}cells of $\nu(K)$
	are the matrices
	\[
	\tabld{x}{i}
	\]
	such that
	\begin{enumerate}
		\item $x^\epsilon_k$ belongs to $K^\ast_k$ for $\epsilon = 0, 1$
		and $0 \le k \le i$ ;
		\item $d(x^\epsilon_k) = x^1_{k-1} - x^0_{k-1}$ for $\epsilon = 0, 1$
		and $0 < k \le i$ ;
		\item $e(x^\epsilon_0) = 1$ for $\epsilon = 0, 1$ ;
		\item $x_i^0 = x_i^1$.
	\end{enumerate}
	Let us describe the \oo-categorical structure.
	Let
	\[
	x = \tabld{x}{i}
	\]
	be an $i$\hyp{}cell of $\nu(K)$ for $i \ge 0$. If $i > 0$ we
	define the source and the target of $x$ to be respectively
	\[
	s(x) =
	\begin{pmatrix}
	x^0_0 &\dots &x^0_{i-2}
	&x^0_{i-1}\cr\noalign{\vskip 3pt}
	x^1_0 &\dots &x^1_{i-2}
	&x^0_{i-1}
	\end{pmatrix}
	\quadet
	t(x) =
	\begin{pmatrix}
	x^0_0 &\dots &x^0_{i-2}
	&x^1_{i-1}\cr\noalign{\vskip 3pt}
	x^1_0 &\dots &x^1_{i-2}
	&x^1_{i-1}
	\end{pmatrix}\, .
	\]
	The identity of $x$ is given by the matrix
	\[
	1_x = 
	\begin{pmatrix}
	x^0_0 &\dots &x^0_{i-1}
	&x^0_{i} & 0\cr\noalign{\vskip 3pt}
	x^1_0 &\dots &x^1_{i-1}
	&x^1_{i} & 0
	\end{pmatrix}\, .
	\]
	Finally if
	\[
	y = \tabld{y}{i}
	\]
	is another $i$\hyp{}cell which is $j$\hyp{}composable with $x$,
	with $i > j \ge 0$, then we set
	\[
	x \ast_j y =
	\begin{pmatrix}
	y^0_0 &\dots & y^0_j & x^0_{j+1} + y^0_{j+1} & \dots & x^0_{i} + y^0_i
	\cr\noalign{\vskip 3pt}
	x^1_0 &\dots & x^1_j & x^1_{j+1} + y^1_{j+1} & \dots & x^1_{i} + y^1_i
	\end{pmatrix}\, .
	\]
	One checks that this indeed defines an \oo-category.

	If $x$ is an $i$\hyp{}cell of $\nu(K)$, $i \ge 0$, then we shall denote
	by $x^\eps_k$ the component of the matrix defining $x$,
	for $0 \le k \le i$ and $\eps = 0, 1$. We shall simply name by $x_i$
	the element $x^0_i = x^1_i$ and for $k >i$ and $\eps = 0, 1$ we set $x^\eps_k = 0$.
	
	Let $f \colon K \to K'$ be a morphism of augmented directed complexes.
	The collection of functions
	\[
	\tabld{x}{i} \mapsto
	\begin{pmatrix}
	f(x^0_0) &\dots &f(x^0_{i-1})
	&f(x^0_{i})\cr\noalign{\vskip 3pt}
	f(x^1_0) &\dots &f(x^1_{i-1})
	&f(x^1_{i})
	\end{pmatrix}
	\]
	defines an \oo-functor $\nu(f) \colon \nu(K) \to \nu(K')$.
\end{paragr}

\begin{rem}
	Steiner shows in~\cite{Steiner1} that the functor $\nu$ admits a left adjoint $\lambda$, that we will not define since we will not need it.
	In particular, the composition $\nN = \SN \circ \nu \colon \Cda \to \EnsSimp$ defines a nerve functor for augmented directed complexes
	and this has a left adjoint given by $\cC = \lambda \circ c_\infty \colon \EnsSimp \to \ooCat$.
\end{rem}

\begin{paragr}\label{paragr:def_atom_II}
	Let $K$ be an augmented directed complex with basis $B$.
	For any element $x$ of $K_i$, one easily checks
	that the matrix
	\[
	\atom{x} = \tabll{\atom{x}}{i}\,,
	\]
	as defined in paragraph~\ref{paragr:def_atome}, is an $i$\hyp{}cell of $\nu(K)$ if and only if
	the element $x$ belongs to $K^\ast_i$ and we have
	the equalities $e(\atom{x}^0_0) = 1 = e(\atom{x}^1_0)$.
	Setting $\atom{x}_i^\eps = x^\eps_i$ for $k \le i$ and $\atom{x}^\eps_k = 0$
	for all $k > i$, $\eps= 0, 1$, these notations
	are compatible with those of paragraph~\ref{def:nu}
	whenever~$\atom{x}$ is an $i$\hyp{}cell of $\nu(K)$.
	
	If the basis $B$ of $K$ is unital, then for any element $x$ of the basis
	the matrix defined by $\atom{x}$
	is a cell of $\nu(K)$.
	In this case, we call the cell $\atom{x}$ of $\nu(X)$
	the \ndef{atom}\index{atom} associated to $x$ .
\end{paragr}

\begin{thm}\label{thm:Steiner_adj}
	The functors
	\[
	\lambda \colon \ooCat \to \Cda
	\quadet
	\nu \colon \Cda \to \ooCat
	\]
	define a pair of adjoint functors.
\end{thm}

\begin{proof}
	This is Theorem 2.11 of \cite{Steiner1}.
\end{proof}

\begin{thm}[Steiner]\label{thm:equivalence_Steiner}
	For any Steiner complex $K$,
	the counit morphism
	\[
	\lambda(\nu(K)) \to K
	\]
	is an isomorphism.
	In particular, the restriction of the functor
	$\nu \colon \Cda\to \ooCat$
	to the category of Steiner complexes is fully faithful
\end{thm}

\begin{paragr}\label{paragr:def_Steiner_category}
	We shall call \ndef{Steiner \oo-category}\index{Steiner infty-category@Steiner \oo-category}
	(resp.~\ndef{strong Steiner \oo-category}\index{strong Steiner infty-category@strong Steiner \oo-category})
	an \oo-category
	in the essential image of the restriction of
	the functor $\nu \colon \Cda \to \ooCat$ to the 
	full subcategory of Steiner complexes
	(resp.~strong Steiner complexes).
	The preceding theorem states that the functor $\nu$
	induces an equivalence of categories between the
	category of Steiner complexes and that of
	Steiner \oo-categories
	(resp.~between the
	category of strong Steiner complexes and that of
	strong Steiner \oo-cat\-e\-go\-ries).
\end{paragr}

\section{Joyal's {$\Theta$} category}
\label{app:theta}
	
	\begin{paragr}\label{paragr:def_disks}
		For any $i \ge 0$, we shall denote by $\Fl_i$
		the set of $i$-cells of an \oo-category and by
		$\Dn i$ the \oo-category corepresenting 
		the functor $\Fl_i \colon \ooCat \to \Ens$ mapping an \oo-category $A$ to the set
		of its $i$\hyp{}cells.
		In fact, this \oo-category is an $i$\hyp{}category having
		a single non-trivial $i$\hyp{}cell that we shall call its \ndef{principal cell}.
		For any $0 \le k \le i$ the $i$\hyp{}category $\Dn i$ has exactly two
		non-trivial $k$\hyp{}cells which are the $k$\hyp{}dimensional
		iterated source and target of its principal cell.
		This is how the graphs of $\Dn i$ (without identities) for $i=0, 1, 2$
		look like:
		\[
			\Dn 0 =
			\begin{tikzcd}
				\bullet
			\end{tikzcd}
			\ ,\ D_1 =
			\begin{tikzcd}
				\bullet \arrow[r] & \bullet
			\end{tikzcd}
			\ ,\ D_2 = 
			\begin{tikzcd}
				\bullet \arrow[r, bend left=50, ""{name=U, below}] \arrow[r, bend right=50, ""{name=D}]
				& \bullet
				\arrow[Rightarrow,from=U,to=D]
			\end{tikzcd}
			\ ,\ D_3 = 
			\begin{tikzcd}
				\bullet \arrow[r, bend left=50, ""'{name=U}] \arrow[r, bend right=50, ""{name=D}]
				& \bullet
				\arrow[Rightarrow,from=U,to=D, shift right=1ex, bend right=30, ""{name=L}]
				\arrow[Rightarrow, from=U, to=D, shift left=1ex, bend left=30, ""{name=R, left}]
				\arrow[triple, from=L, to=R]{}
			\end{tikzcd}\ .
		\]
		
		For $i>0$, the natural transformations source and target $\Fl_i \to \Fl_{i-1}$
		induces \oo-functors $\Ths{i}, \Tht{i} \colon \Dn{i-1} \to \Dn i$.
		Explicitly the \oo-functor $\sigma_i$ (resp.~$\tau_i$) sends the principal cell of $\Dn{i-1}$
		to the source (resp.~the target) of the principal cell of $\Dn i$.
		
		For $0 \le j < i$ we shall denote by $\Ths[i]{j}, \Tht[i]{j} \colon \Dn{j} \to \Dn{i}$
		the \oo-functors corepresented by the natural transformations $\Gls[i]{j}$ and $\Glt[i]{j}$
		respectively, \ie the \oo-functors
		\[
			\Ths[i]{j} = \Ths{i}\dots\Ths{j+2}\Ths{j+1}
			\quadet
			\Tht[i]{j} = \Tht{i}\dots\Tht{j+2}\Tht{j+1}\,.
		\]
	\end{paragr}
	
	\begin{paragr}\label{paragr:def_tree}
		Let $\ell >0$ and $i_1, \dots, i_\ell$, $j_1, \dots, j_{\ell-1}$ be a collection of positive integers
		satisfying the inequalities
		\[
			i_k > j_k < i_{k+1}\ ,\quad \text{for }0< k <\ell\,.
		\]
		We shall often organise these integers in a matrix, called \ndef{matrix of dimensions}, of the following form
		\[
		 \begin{pmatrix}
		 i_1 && i_2 && \dots && i_{\ell-1} && i_\ell\cr\noalign{\vskip 3pt}
		  & j_1 && j_2 && \dots && j_{\ell-1} &
		 \end{pmatrix}
		\]
		and associate to it the diagram
		\[
			\begin{tikzcd}[column sep=1em]
			\Dn{i_1} &  & \Dn{i_2} &  & \Dn{i_3} &  & \dots &  & \Dn{i_{\ell-1}} &  & \Dn{i_\ell} \\
			& \Dn{j_1} \arrow[lu, "{\Ths[i_1]{j_1}}"] \arrow[ru, "{\Tht[i_2]{j_1}}"'] &  & \Dn{j_2} \arrow[lu, "{\Ths[i_2]{j_2}}"] \arrow[ru, "{\Tht[i_3]{j_2}}"'] &  & \dots &  & \dots &  &
			\Dn{i_{\ell-1}} \arrow[lu, "{\Ths[i_{\ell-1}]{j_{\ell-1}}}"] \arrow[ru, "{\Tht[i_\ell]{j_{\ell-1}}}"'] & 
			\end{tikzcd}
		\]
		in $\ooCat$. We shall call \ndef{globular sum} the colimit of such a diagram and we shall simply denote it by
		\[
			\Dn{i_1} \amalg_{\Dn{j_1}} \Dn{i_2} \amalg_{\Dn{j_2}}\dots \amalg_{\Dn{j_{\ell-1}}} \Dn{i_\ell}\,.
		\]
		We shall call \ndef{globular pasting scheme} any \oo-category that we get this way.
	\end{paragr}

	\begin{paragr}\label{paragr:dimension_globular_scheme}
		Consider a matrix of dimensions
		\[
			\begin{pmatrix}
			i_1 && i_2 && \dots && i_{\ell-1} && i_\ell\cr\noalign{\vskip 3pt}
			& j_1 && j_2 && \dots && j_{\ell-1} &
			\end{pmatrix}\ .
		\]
		The \ndef{dimension} of the globular pasting scheme $T$
		\[ 
			\Dn{i_1} \amalg_{\Dn{j_1}} \Dn{i_2} \amalg_{\Dn{j_2}}\dots \amalg_{\Dn{j_{\ell-1}}} \Dn{i_\ell}
		\]
		is given by the number
		\[
			\sum_{1\le k \le \ell} i_k - \sum_{0 < k < \ell} j_k =
			i_1 - j_1 + i_2 - j_2 + \dots + i_{\ell-1} - j_{\ell -1} + i_\ell\,.
		\]
		The \ndef{height} of the globular pasting scheme $T$
		is defined as the number
        \[
         \text{ht}(T) = \max_{1\le k \le \ell}(i_k)\,.
        \]

	\end{paragr}

	\begin{paragr}\label{paragr:def_Theta}
        Joyal's $\Theta$ category is the full subcategory
        of $\ooCat$ spanned by globular pasting schemes.
        We shall denote by $\Theta_+$ the full subcategory
        of $\ooCat$ obtained by adding the empty \oo-category to $\Theta$.
	\end{paragr}
	
    \begin{paragr}\label{paragr:def_Theta_n}
        The height defined in the previous paragraph defines
        a canonical grading on the objects of $\Theta$.
        For any integer $n \ge 0$, we denote by $\Theta_n$
        the full subcategory of $\Theta$ spanned by the objects
        of height at most $n$. We observe that $\Theta_0$ is
        the category whose only object is $\Dn{0}$ and whose
        only morphism is the identity and that $\Theta_1$
        is canonically isomorphic to the category $\cDelta$ of simplices;
        we thus get a canonical embedding $\cDelta \hookto \Theta$.
	\end{paragr}

	\begin{rem}
        The category $\Theta$ was first introduced by Joyal
        in~\cite{JoyalDisks}, with a definition more geometric in spirit. Berger~\cite{BergerNerve} and Makkai--Zawadowski~\cite{MakkaiZawadowskiDuality}
        later independently showed that the two definitions
        are actually equivalent. Another equivalent definition
        is due to Oury~\cite{Oury}.
	\end{rem}

	\begin{paragr}
        Another convenient and graphical description of the category $\Theta$ can be given
        in terms of planar rooted trees. Let $\mathcal T$ be the category of presheaves
        in finite linearly ordered set on the poset of non-negative integers, \ie
        an element $X$ of $\mathcal T$ is a sequence of finite linearly ordered sets $(X_n)_{n\ge 0}$
        equipped with order-preserving maps $X_n \to X_{n-1}$ for all $n>0$.
        A \ndef{planar rooted tree}, or simply \ndef{tree}, is an object $T$ of $\mathcal T$ such that $T_0$ is a singleton
        and for which $T_i$ is eventually empty for $i$ big enough. The greatest $i$ for which
        $T_i$ is non-empty will be called the \ndef{height} of the tree.
        
        Let us now sketch the correspondence between objects of $\Theta$ and trees.
        Instead of giving a formal framework, we are going to present some
        examples upon which one can easily build the intuition behind this
        bijection. For any $i \ge 0$, we associate to the object $\Dn i$ the linear tree
        $T$ of height $i$, that is, for which $T_k$ is a singleton for all $0 \le k \le i$,
        and we depict it as
        \begin{center}
         \begin{forest}
            for tree={%
                label/.option=content,
                grow'=north,
                content=,
                circle,
                fill,
                minimum size=3pt,
                inner sep=0pt,
            }
            [   [
                    [, edge=dotted
                        []
                    ]
                ]
            ]
         \end{forest}
        \end{center}
        So for instance we have
        \begin{center}
          $\Dn 0 = \bullet$
          \quad , \quad
          $\Dn 1$ =
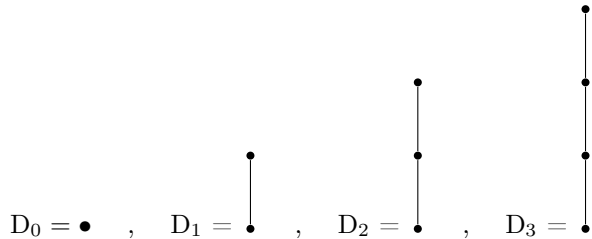
          \begin{forest}
            for tree={%
                label/.option=content,
                grow'=north,
                content=,
                circle,
                fill,
                minimum size=3pt,
                inner sep=0pt,
            }
            [   
                [ ]
            ]
         \end{forest}
         \quad , \quad
         $\Dn 2$ = 
         \begin{forest}
            for tree={%
                label/.option=content,
                grow'=north,
                content=,
                circle,
                fill,
                minimum size=3pt,
                inner sep=0pt,
            }
            [  
                [
                    [ ]
                ]
            ]
         \end{forest}
         \quad , \quad
         $\Dn 3$ = \begin{forest}
            for tree={%
                label/.option=content,
                grow'=north,
                content=,
                circle,
                fill,
                minimum size=3pt,
                inner sep=0pt,
            }
            [   [
                    [
                        []
                    ]
                ]
            ]
         \end{forest}
        \end{center}
        The height $i$ of the tree determines the dimension of the principal cell.
        
        The element of $\cDelta$ corresponds precisely to the trees of height at most 1,
        for instance
        \begin{center}
            $\Deltan{0} = \bullet$
            \quad , \quad
            $\Deltan 1 =$
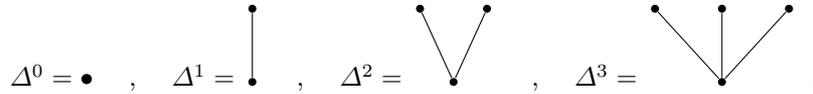
            \begin{forest}
            for tree={%
                label/.option=content,
                grow'=north,
                content=,
                circle,
                fill,
                minimum size=3pt,
                inner sep=0pt,
            }
            [   
                [ ]
            ]
         \end{forest}
         \quad , \quad
            $\Deltan 2 =$
        \begin{forest}
            for tree={%
                label/.option=content,
                grow'=north,
                content=,
                circle,
                fill,
                minimum size=3pt,
                inner sep=0pt,
                s sep+=15,
            }
            [   
                [ ] [ ]
            ]
         \end{forest}
         \quad , \quad
            $\Deltan 3 =$
            \begin{forest}
            for tree={%
                label/.option=content,
                grow'=north,
                content=,
                circle,
                fill,
                minimum size=3pt,
                inner sep=0pt,
                s sep+=15,
            }
            [   
                [ ] [ ] [ ]
            ]
         \end{forest}\, .
        \end{center}
        
        More generally, given an object $S$ of $\Theta$ with matrix of dimensions
        \[
         \begin{pmatrix}
		 i_1 && i_2 && \dots && i_{\ell-1} && i_\ell\cr\noalign{\vskip 3pt}
		  & j_1 && j_2 && \dots && j_{\ell-1} &
		 \end{pmatrix}\,,
        \]
        the corresponding tree is drawn inductively as follows. The terms
        $i_k$ and $i_{k+1}$ correspond to the height of two linear trees that
        are glued together, \ie share the same root and then fork from height $j_k$.
        We proceed from right to left, so that the description agrees with the
        usual way of writing the composition of cells.
        As a first example, consider the following tables of dimensions
        \[
         \begin{pmatrix}
		 2 && 2 \cr\noalign{\vskip 3pt}
		  & 0 &
		 \end{pmatrix}
		 \quad , \quad
		 \begin{pmatrix}
		 2 && 2 \cr\noalign{\vskip 3pt}
		  & 1 &
		 \end{pmatrix}
		 \quad , \quad
		 \begin{pmatrix}
		 2 && 2 && 2 \cr\noalign{\vskip 3pt}
		  & 0 && 1 &
		 \end{pmatrix}
		 \, ,
        \]
        and the associated objects of $\Theta$

        \[
         D_2 \amalg_{D_0} D_2 
		 \quad , \quad
		 D_2 \amalg_{D_1} D_2
		 \quad , \quad
         D_2 \amalg_{D_0} D_2 \amalg_{D_1} D_2\, ,
        \]
        that is the globular pasting schemes
        
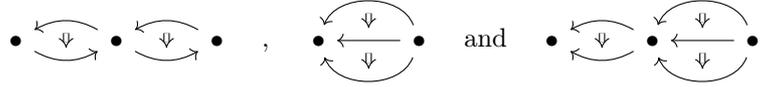
\[
            \begin{tikzcd}
            \bullet \arrow[r, leftarrow, bend left, ""{name=U1, below}] \arrow[r, bend right, ""{name=D1}]
            &  \bullet \arrow[r, leftarrow, bend left, ""{name=U2, below}] \arrow[r, bend right, ""{name=D2}]
            & \bullet
            \ar[Rightarrow, from=U1, to=D1] \ar[Rightarrow, from=U2, to=D2]
            \end{tikzcd}
        \quad , \quad
            \begin{tikzcd}
               \bullet
               \arrow[r, leftarrow, bend left=70, ""'{name=U}]
               \arrow[r, leftarrow, bend right=70, ""{name=D}]
               \arrow[r, leftarrow, ""{name=M1}, ""'{name=M2}] & \bullet
               \ar[Rightarrow, from=U, to=M1] \ar[Rightarrow, from=M2, to=D]
            \end{tikzcd}
        \quadet
            \begin{tikzcd}
                \bullet
                \arrow[r, leftarrow, bend left, ""'{name=U1}]
                \arrow[r, leftarrow, bend right, ""{name=D1}] & 
                \bullet
                \arrow[r, leftarrow, bend left=70, ""'{name=U2}]
                \arrow[r, leftarrow, bend right=70, ""{name=D2}]
                \arrow[r, leftarrow, ""{name=M1}, ""'{name=M2}] & \bullet
                \ar[Rightarrow, from=U1, to=D1]
                \ar[Rightarrow, from=U2, to=M1]
                \ar[Rightarrow, from=M2, to=D2]
            \end{tikzcd}\,.
        \]
        These objects correspond respectively to trees
        
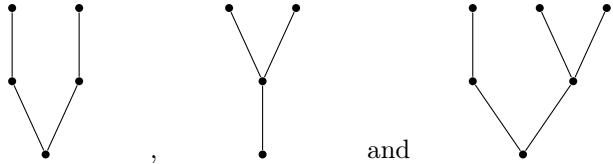
\begin{center}
            \begin{forest}
            for tree={%
                label/.option=content,
                grow'=north,
                content=,
                circle,
                fill,
                minimum size=3pt,
                inner sep=0pt,
                s sep+=15,
            }
            [   
                [ [ ] ]
                [ [ ] ]
            ]
         \end{forest}
         \qquad , \qquad
         \begin{forest}
            for tree={%
                label/.option=content,
                grow'=north,
                content=,
                circle,
                fill,
                minimum size=3pt,
                inner sep=0pt,
                s sep+=15,
            }
            [   
                [ 
                    [ ] [ ]
                ]
            ]
         \end{forest}
         \qquad\text{and}\qquad
         \begin{forest}
            for tree={%
                label/.option=content,
                grow'=north,
                content=,
                circle,
                fill,
                minimum size=3pt,
                inner sep=0pt,
                s sep+=15,
            }
            [   
                [ [ ] ]
                [
                    [ ] [ ]
                ]
            ]
         \end{forest}
        \end{center}
        
        As another more involved example, consider the matrix of dimensions
        \[
         \begin{pmatrix}
		 2 && 2 && 2 && 2 && 2 && 2 && 3 && 2 && 2 \cr\noalign{\vskip 3pt}
		  & 1 && 1 && 1 && 0 && 1 && 0 && 0 && 1 &
		 \end{pmatrix}\,,
        \]
        and the corresponding globular pasting scheme
        
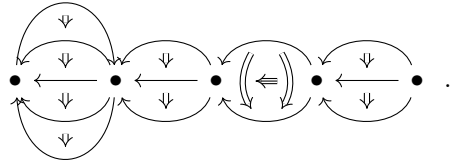
\[
		 \begin{tikzcd}
		  \bullet
		  \ar[r, rightarrow, bend left=85, looseness=2.2, ""{name=D11, below}]
		  \ar[r, leftarrow, bend left=70, looseness=1, ""{name=D12i}, ""{name=D12ii, below}]
		  \ar[r, leftarrow, ""{name=D13i}, ""{name=D13ii, below}]
		  \ar[r, leftarrow, bend right=70, looseness=1, ""{name=D14i}, ""{name=D14ii, below}]
		  \ar[r, leftarrow, bend right=85, looseness=2.2, ""{name=D15}]
		  \ar[Rightarrow, from=D11, to=D12i]
		  \ar[Rightarrow, from=D12ii, to=D13i]
		  \ar[Rightarrow, from=D13ii, to=D14i]
		  \ar[Rightarrow, from=D14ii, to=D15]
		  &
          \bullet
		  \ar[r, leftarrow, bend left=70, ""{name=D21, below}]
		  \ar[r, leftarrow, ""{name=D22i}, ""{name=D22ii, below}]
		  \ar[r, leftarrow, bend right=70, ""{name=D23}]
		  \ar[Rightarrow, from=D21, to=D22i]
		  \ar[Rightarrow, from=D22ii, to=D23]
		  &
		  \bullet
		  \ar[r, leftarrow, bend left=70, ""{name=D31, below}]
		  \ar[r, leftarrow, bend right=70, ""{name=D32}]
          \arrow[Rightarrow,from=D31,to=D32, shift right=1ex, bend right=30, ""{name=L}]
          \arrow[Rightarrow, from=D31, to=D32, shift left=1ex, bend left=30, ""{name=R, left}]
		  \arrow[triple, from=R, to=L]{}
          &
          \bullet
          \ar[r, leftarrow, bend left=70, ""{name=D41, below}]
		  \ar[r, leftarrow, ""{name=D42i}, ""{name=D42ii, below}]
		  \ar[r, leftarrow, bend right=70, ""{name=D43}]
		  \ar[Rightarrow, from=D41, to=D42i]
		  \ar[Rightarrow, from=D42ii, to=D43]
		 & \bullet
		 \end{tikzcd}\ .
        \]
        The associated tree is given by
        
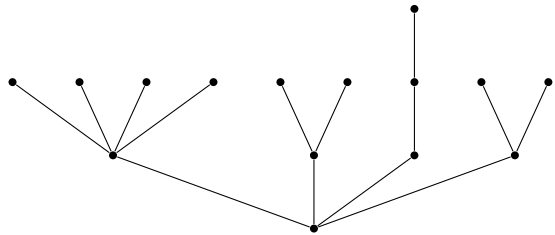
\[
	        \begin{forest}
	        for tree={%
	        	label/.option=content,
	        	grow'=north,
	        	content=,
	        	circle,
	        	fill,
	        	minimum size=3pt,
	        	inner sep=0pt,
	        	s sep+=15,
	        }
	        [   
		        [
			        [ ] [ ] [ ] [ ]
		        ]
		        [
			        [ ] [ ]
		        ]
		        [
			        [ [ ] ]
		        ]
		        [
			        [ ] [ ]
		        ]
	        ]
	        \end{forest}\, .
        \]
	\end{paragr}

\section{Orientals}
In this short section we present a visual intuition
of the first few orientals. Then for every poset $E$
we give a description	of the \oo-category $c_\infty N(E)$,
that we shall simply denote by $\On{E}$ and call the
\ndef{oriental of $E$}, and we deduce some of its basic properties,
following closely~\cite[§6]{AraMaltsiCondE}.

\begin{paragr}\label{paragr:atoms_orientals}
	Fix an integer $n\ge 0$. For any $0\le i \le n$
	the $i$-chains of the strong Steiner complex $\cC (\Deltan n)$
	are the elements of the free abelian groups
	generated by the elements of the set $\nd{(\Deltan n)}{i}$,
	the set of non-degenerate $i$-simplices of the representable
	simplicial set $\Deltan{n}$;
	that is, the generators of $\cC (\Deltan n)$ are the $i$-tuples $(j_0, j_1, \dots, j_i)$
	of non-negative integers such that $0 \le j_\ell < j_{\ell +1} \le n$
	for all $\ell =0, \dots, i-1$. For any such element
	of the basis of $\cC (\Deltan n)_i$, we shall write
	$\atom{j_0j_1\dots j_i}$ for the corresponding atom, instead
	of the more pedantic $\atom{(j_0, j_1, \dots, j_i)}$
	(see paragraph~\ref{paragr:def_atome}).
\end{paragr}

\begin{paragr}\label{paragr:orientals}
	The \oo-categories $\On 0$ and $\Dn 0$ are isomorphic.
	They are both terminal objects for the category $\ooCat$
	of small \oo-categories and they corepresent the functor
	mapping any \oo-category $A$ to the set $\Ob A$ of its objects.
	
	The \oo-categories $\On 1$ and $\Dn 1$ are isomorphic, too.
	They are both generated as \oo-graphs by $\bullet \longrightarrow \bullet$
	and they corepresent the functor mapping any \oo-category
	$A$ to the set $\Fl_1(A)$ of its $1$-cells.
	
	The \oo-category $\On 2$ is a free \oo-category, generated
	by the \oo-graph
	
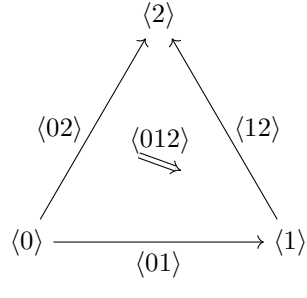
\begin{center}
		\begin{tikzpicture}
		\foreach \i in {0, 1, 2}{
			\tikzmath{\a = 210 + (120 * \i);}
			\node (n\i) at (\a:2) {$\atom{\i}$};
		}
		
		\draw [->] (n0) -- node [below] {$\atom{01}$} (n1);
		\draw [->] (n1) -- node [right] {$\atom{12}$} (n2);
		\draw [->] (n0) -- node [left] {$\atom{02}$} (n2);
		
		\draw [double, double equal sign distance, -implies] (150:0.3) -- node [above] {$\atom{012}$} (-10:0.3);
		\end{tikzpicture},
	\end{center}
	so that the $2$-cell $\atom{012}$ has $\atom{02}$ as source and $\atom{12}\comp_0 \atom{01}$
	as target. With the notations as above and as in paragraph~\ref{paragr:def_atome},
	we have
	\[
	\atom{i} =
	\begin{pmatrix}
	(i) \\ (i)
	\end{pmatrix}\ ,
	\quad\text{for }i=0, 1, 2
	\]
	for the objects,
	\[
	\atom{ij} =
	\begin{pmatrix}(i) & (i, j) \\ (j) & (i, j)\end{pmatrix}\ ,
	\quad\text{for }i, j= 0, 1, 2\text{ and }i <j
	\]
	for the $1$-cells and
	\[
	\atom{012} =%
	\begin{pmatrix}
	(0) & (0, 2) & (0, 1, 2) \\
	(2) & (0, 1) + (1, 2) & (0, 1, 2)
	\end{pmatrix}
	\]
	for the $2$-cell.
	
	The \oo-category $\On 3$ is a free \oo-category generated by the \oo-graph
	
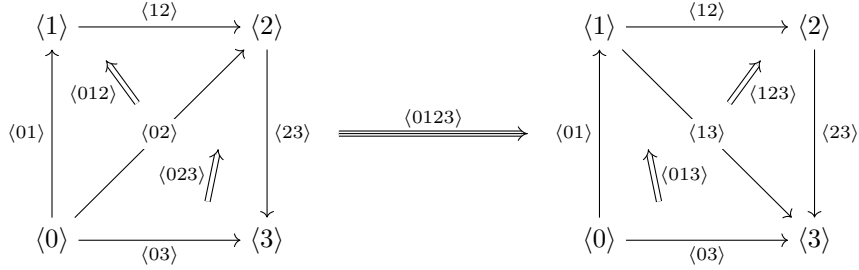
\begin{center}
		\begin{tikzpicture}[scale=2]
		\square{%
			/square/label/.cd,
			0=$\atom{0}$, 1=$\atom{1}$, 2=$\atom{2}$, 3=$\atom{3}$,
			01=$\atom{01}$, 12=$\atom{12}$, 23=$\atom{23}$,
			02=$\atom{02}$, 03=$\atom{03}$, 13=$\atom{13}$,
			012=$\atom{012}$, 023=$\atom{023}$, 123=$\atom{123}$, 013=$\atom{013}$,
			0123=$\atom{0123}$
		}
		\end{tikzpicture}
	\end{center}
	so that the $3$-cell $\atom{0123}$ has the $2$-cell 
	\[\bigl(\atom{23}\comp_0 \atom{012}\bigr) \comp_1 \atom{023}\]
	as source and the $2$-cell
	\[\bigl(\atom{123}\comp_0\atom{01}\bigr) \comp_1 \atom{013}\]
	as target. Indeed we have
	\[
	\atom{0123} =
	\begin{pmatrix}
	(0) & (01) + (12) + (23) & (012) + (023) & (0123) \\
	(3) & (03) & (123) + (013) & (0123)
	\end{pmatrix}\ .
	\]
	
	The \oo-category $\On 4$ is freely generated by the diagram
	displayed in figure~\ref{fig:4-simplex}, where we omitted the
	brackets $\atom{\,\cdot\,}$ for reasons of space.
	
\end{paragr}

\begin{figure}
	\centering
	\begin{tikzpicture}[scale=1.45]
	\pentagon{%
		/pentagon/label/.cd,
		0=$0$, 1=$1$, 2=$2$, 3=$3$, 4=$4$,
		01=$01$, 12=$12$, 23=$23$, 34=$34$, 04=$04$,
		02=$02$, 03=$03$, 13=$13$, 14=$14$, 24=$24$,
		012=$012\phantom{i}$, 034=$034$, 023=$023$, 123=$123$, 134=$134$,
		014=$014$, 024=$024$, 234=$234$, 013=$013$, 124=$124$,
		0123=$0123$, 0124=$0124$, 0134=$0134$, 0234=$0234$, 1234=$1234$,
		01234=$01234$,
	}
	\end{tikzpicture}
	\caption{The \oo-category $\On 4$.}
	\label{fig:4-simplex}
\end{figure}

\begin{paragr}
	Let $E$ be a poset. For any $p \ge 0$
	the non-degenerated $p$-simplices of $\SN(E)$ are strictly
	increasing maps $\Deltan{p} \to E$, \ie the set $\SN(E)_p$ of
	non-degenerate $p$-simplices of the nerve of $E$ consists of
	$(p+1)$\hyp{}tuples
	\[(x_0, x_1,\dots, x_p)\]
	of elements of $E$ such that $x_i < x_{i+1}$ for all $i=0, 1, \dots, p-1$.
	The abelian group $(\cC \SN(E))_p$ (resp. abelian monoid $(\cC \SN(E))^*_p$) 
	is freely generated by the set $\nd{\SN(E)}{p}$.
	The differential is defined by
	\[
	d(i_0, i_1,\dots, i_p) = \sum_{k=0}^p (-1)^k (i_0, i_1, \dots, \widehat{i_k}, \dots, i_p)\,,\quad p>0\,,
	\]
	where $(i_0, \dots, \widehat{i_k}, \dots, i_p) = (i_0, \dots, i_{k-1}, i_{k+1}, \dots, i_p)$,
	and the augmentation by $e(i_0) = 1$.
\end{paragr}

\begin{thm}
	The augmented directed complex $\cC \SN(E)$ is a strong Steiner complex.
\end{thm}

\begin{proof}
	This follows immediately from Theorem~8.6 of~\cite{AraMaltsiCondE}.
	Indeed, using the notation of~\loccit, for any poset $E$ the simplicial set
	$\SN(E)$ is canonically isomorphic to $k^*(E, \xi E)$
	(cf.~paragraph~8.4 of~\cite{AraMaltsiCondE}).
\end{proof}

\begin{paragr}\label{paragr:oriental_poset}
	The previous theorem shows
	that the \oo-category $\nu \cC \SN(E)$ is a strong Steiner category
	and it is more precisely freely generated by the atoms
	\[
	\atom{x_0 x_1 \dots x_p}\ , \quad p\ge 0\text{ and } x_0 < x_1 < \dots < x_p\,.
	\]
	This \oo-category associated to the poset $E$ shall be called the \emph{oriental of $E$}
	and shall be denoted by $\On{E}$.
\end{paragr}

\begin{lemme}
	The functor $\Or \colon \Ord \to \ooCat$ associating to any
	poset $E$ its oriental \oo-category $\On{E}$ preserves monomorphisms.
\end{lemme}

\begin{proof}
	This is Proposition 9.6 of \cite{AraMaltsiCondE}.
\end{proof}

We report here an important result of~\cite{AraMaltsiCondE} giving
a very explicit description of the ``horizontal composition'' of cells
of the orientals.

\begin{prop}\label{prop:2-cells_orientals} 
	Let $n \ge 1$, $m\ge 1$ and $i_0,i_1\dots,i_m$ be integers
	such that
	\[0=i_0<i_1<\cdots<i_{m-1}<i_m=n\,.\]
	Then, the \oo-functor
	\[\textstyle\prod\limits^m_{k=1}\Hom_{\On{n}}(a_k,b_k)\longrightarrow\Hom_{\On{n}}(a,b)\ ,\]
	where
	\[
	a_k=
	\begin{pmatrix}
	(i_{k-1})
	&(i_{k-1},i_k)
	\cr
	\noalign{\vskip 3pt}
	(i_k)
	&(i_{k-1},i_k)
	\end{pmatrix}\,,\quad
	b_k=
	\begin{pmatrix}
	(i_{k-1})
	&\textstyle\sum\limits_{i_{k-1}<l\leq i_k}(l-1,l)
	\cr
	\noalign{\vskip 3pt}
	(i_k)
	&\sum\limits_{i_{k-1}<l\leq i_k}(l-1,l)
	\end{pmatrix}\,,\ \ 1\leq k\leq m\ ,
	\]
	\[\kern -75pt
	a=
	\begin{pmatrix}
	(0)
	&\textstyle\sum\limits^m_{k=1}\kern -3pt(i_{k-1},i_k)
	\cr
	\noalign{\vskip 3pt}
	(n)
	&\sum\limits^m_{k=1}\kern -3pt(i_{k-1},i_k)
	\end{pmatrix}\,,\quad
	b=
	\begin{pmatrix}
	(0)
	&\sum\limits^n_{l=1}(l-1,l)
	\cr
	\noalign{\vskip 3pt}
	(n)
	&\sum\limits^n_{l=1}(l-1,l)
	\end{pmatrix}\,,
	\]
	defined by the ``horizontal composition'' $\comp_0$ of $\On{n}$
	\[(x_1,x_2,\dots,x_m)\longmapsto x_1\comp_0x_2\comp_{0}\cdots\comp_{0}x_m\]
	is an isomorphism of \oo-categories.
\end{prop}

\begin{proof}
	This is Proposition~A.4 of~\cite{AraMaltsiCondE}.
\end{proof}

In section~\ref{sec:tilde} we shall need few more properties
of the hom-\oo-categories of the oriental $\On{E}$ of a poset~$E$
that are proven in~\cite{AraMaltsiCondE}.

\begin{prop}
	Let $E$ be a poset and $s \colon \Deltan{n} \to E$ a non-degenerate $n$-simplex
	of $N(E)$, with $n>0$. Consider the $1$-cell $S$ of $\On{E}$ defined by
	\[
	S = \sum_{i=0}^{n-1} \atom{s_i, s_{i+1}}\,.
	\]
	Then the \oo-functor $\On{s} \colon \On{n} \to \On{E}$ induces an isomorphism
	\[
	\Homi_{\On{n}} (\atom{0, n}, \atom{0, 1} + \dots + \atom{n-1, n}) \to \Homi_{\On{E}}(\atom{s_0, s_n}, S)
	\]
	of \oo-categories.
\end{prop}

\begin{proof}
	This is a particular case of Proposition 1.5 of \cite{AraMaltsiCondE},
	since it is clear that the \oo-categories $\On{n}$ and what they denote
	by $\Or(S)$ are canonically isomorphic.
\end{proof}

\begin{coro}
	Let $E$ be a poset and $s \colon \Deltan{n} \to E$ a non-degenerate $n$-simplex
	of $N(E)$, with $n>0$. Consider the $1$-cell $S$ of $\On{E}$ defined by
	\[
	S = \sum_{i=0}^{n-1} \atom{s_i, s_{i+1}}\,.
	\]
	Then for any $1$-cell
	\[
	f = \sum_{i=0}^{m-1} \atom{j_i, j_{i+1}}
	\]
	with $i_0 = 0$ and $i_m = n$ we have that
	the \oo-functor $\iota_s = \On{s} \colon \On{n} \to \On{E}$ induces an isomorphism
	\[
	\Homi_{\On{n}} (f, \atom{0, 1} + \dots + \atom{n-1, n}) \to \Homi_{\On{E}}(\iota_s(f), S)
	\]
	of \oo-categories.
\end{coro}

\begin{proof}
	This is an equivalent formulation of Proposition 1.5 of~\cite{AraMaltsiCondE}
	and follows immediately from the preceding two propositions.
\end{proof}

\begin{coro}\label{coro:suboriental}
	Let $j \colon E \hookto F$ an inclusion of posets. Then for any two parallel
	$1$-cells $f$ and $g$ of $E$, the \oo-functor $\iota_j = \On{j} \colon \On{E} \to \On{F}$
	induces an isomorphism
	\[
	\Homi_{\On{E}}(f, g) \to \Homi_{\On{F}}(\iota_j(f), \iota_j(g))
	\]
	of \oo-categories.
\end{coro}

\begin{proof}
	This follows immediately from the previous proposition by considering
	the simplex $\bar{g} \colon \Deltan{n} \to E$ of $N(E)$, where
	\[
	g = \sum_{i=0}^{n-1} \atom{\bar{g}_i, \bar{g}_{i+1}}\,. \qedhere
	\]
\end{proof}

\bibliographystyle{myamsplain}
\bibliography{biblio}

\ifx\undefined\bysame
\newcommand{\bysame}{\leavevmode\hbox to3em{\hrulefill}\,}
\fi
\begin{thebibliography}{10}

\bibitem{AraMaltsiNThom}
D.~Ara and G.~Maltsiniotis, {\em Vers une structure de catégorie de modèles
  à la {T}homason sur la catégorie des {$n$}-catégories strictes}, Advances
  in Mathematics {\bf 259} (2014), 557--654.

\bibitem{AraMaltsiCondE}
\bysame, {\em The homotopy type of the $\infty$-category associated to a
  simplicial complex},
  \href{http://arxiv.org/abs/1503.02720}{arxiv:1503.02720}, 2015, preprint.

\bibitem{Joint}
\bysame, {\em Joint et tranches pour les {$\infty$}-catégories strictes},
  2020, to appear in \emph{Mémoires de la SMF}.

\bibitem{FiniteSpaces}
J.~A. Barmak, {\em Algebraic topology of finite topological spaces and
  applications}, Lecture Notes in Mathematics, vol. 2032, Springer, 2011.

\bibitem{BergerNerve}
C.~Berger, {\em A cellular nerve for higher categories}, Adv. Math. {\bf 169}
  (2002), no.~1, 118--175.

\bibitem{BullejosCegarra}
M.~Bullejos and A.~M. Cegarra, {\em On the geometry of 2-categories and their
  classifying spaces}, $K$-Theory {\bf 29} (2003), no.~3, 211--229.

\bibitem{BullejosFaroBlanco}
M.~Bullejos, E.~Faro, and V.~Blanco, {\em A full and faithful nerve for
  2-categories}, Appl. Categ. Structures {\bf 13} (2005), no.~3, 223--233.

\bibitem{Cegarra}
A.~M. Cegarra, {\em Homotopy fiber sequences induced by 2-functors}, J. Pure
  Appl. Algebra {\bf 215} (2011), no.~4, 310--334.

\bibitem{chiche_homotopy}
J.~Chiche, {\em Th\'{e}ories homotopiques des 2-cat\'{e}gories}, Cah. Topol.
  G\'{e}om. Diff\'{e}r. Cat\'{e}g. {\bf 56} (2015), no.~1, 15--75.

\bibitem{CisinskiHigherCats}
D.-C. Cisinski, {\em Higher categories and homotopical algebra}, Cambridge
  Studies in Advanced Mathematics, vol. 180, Cambridge University Press,
  Cambridge, 2019.

\bibitem{delHoyo}
M.~L. del Hoyo, {\em On the subdivision of small categories}, Topology Appl.
  {\bf 155} (2008), 1189–1200.

\bibitem{del-Hoyo}
\bysame, {\em Espacios clasificantes de categor\'{i}as fibradas}, Ph.D. thesis,
  Universidad de Buenos Aires, 2009.

\bibitem{grothendieck_techniques_III}
A.~Grothendieck, {\em Techniques de construction et th\'{e}or\`emes d'existence
  en g\'{e}om\'{e}trie alg\'{e}brique. {III}. {P}r\'{e}schemas quotients},
  S\'{e}minaire {B}ourbaki, vol.~6, Soc. Math. France, Paris, 1995,
  pp.~Exp.~No.~212, 99--118.

\bibitem{GurskiCoherence}
N.~Gurski, {\em Coherence in three-dimensional category theory}, Cambridge
  Tracts in Mathematics, vol. 201, Cambridge University Press, Cambridge, 2013.

\bibitem{cotangent}
L.~Illusie, {\em Complexe cotangent et deformations {I} \& {II}}, Lecture Notes
  in Mathematics, vol. 239 and 283, Springer-Verlag, 1972.

\bibitem{JoyalDisks}
A.~Joyal, {\em Disks, duality and {$\Theta$}\hyp{}categories}, Preprint, 1997.

\bibitem{LackPaoli}
S.~Lack and S.~Paoli, {\em 2-nerves for bicategories}, $K$-Theory {\bf 38}
  (2008), no.~2, 153--175.

\bibitem{kerodon}
J.~Lurie, {\em Kerodon}, \url{https://kerodon.net}, 2018.

\bibitem{MakkaiZawadowskiDuality}
M.~Makkai and M.~Zawadowski, {\em Duality for simple {$\omega$}-categories and
  disks}, Theory Appl. Categ. {\bf 8} (2001), 114--243.

\bibitem{Oury}
D.~Oury, {\em On the duality between trees and disks}, Theory and Applications
  of Categories {\bf 24} (2010), no.~16, 418--450.

\bibitem{OzornovaRovelliDisk}
V.~Ozornova and M.~Rovelli, {\em The duskin nerve of 2-categories in joyal's
  disk category {$\Theta_2$}},
  \href{http://arxiv.org/abs/1910.06103}{arxiv:1910.06103}, 2019, preprint.

\bibitem{GordonPowerStreet}
G.~R., A.~Power, and R.~Street, {\em Coherence for tricategories}, 558 ed.,
  vol. 117, American Mathematical Society, 9 1995.

\bibitem{Steiner1}
R.~Steiner, {\em Omega-categories and chain complexes}, Homology, Homotopy and
  Applications {\bf 6} (2004), no.~1, 175--200.

\bibitem{Street}
R.~Street, {\em The algebra of oriented simplexes}, Journal of Pure and Applied
  Algebra {\bf 49} (1987), no.~3, 283--335.

\bibitem{StreetCatStructures}
\bysame, {\em Categorical structures}, Handbook of algebra, {V}ol.~1, Handb.
  Algebr., vol.~1, Elsevier/North-Holland, Amsterdam, 1996, pp.~529--577.

\bibitem{Conspectus}
\bysame, {\em An australian conspectus of higher categories}, Towards Higher
  Categories (J.~C. Baez and J.~P. May, eds.), The IMA Volumes in Mathematics
  and its Applications, vol. 152, Springer, 2009, pp.~237--264.

\bibitem{Cat_closed}
R.~W. Thomason, {\em Cat as a closed model category}, Cahiers de topologie et
  g{\'e}om{\'e}trie diff{\'e}rentielle cat{\'e}goriques {\bf 21} (1980), no.~3,
  305--324.

\end{thebibliography}

\printindex

\end{document}